\documentclass{amsart}

\usepackage[backend=biber, style=alphabetic,maxbibnames=99,maxalphanames=99]{biblatex}
\usepackage{graphicx}
\usepackage{float}
\usepackage{subfiles}
\usepackage[T1]{fontenc}

\usepackage[a4paper, top=2cm, bottom=2cm, left=3cm, right=3cm, heightrounded, centering]{geometry}

\usepackage[dvipsnames]{xcolor}
\usepackage[bookmarksopen,bookmarksdepth=2]{hyperref}
\hypersetup{colorlinks=true,citecolor=NavyBlue,linkcolor=NavyBlue,urlcolor=Green} 
\hypersetup{
  bookmarksdepth=3
}

\makeatletter
\def\l@section{\@tocline{1}{12pt plus2pt}{0pt}{}{\bfseries}}
\def\l@subsubsection{\@tocline{3}{0pt}{20pt}{}{}}
\makeatother

\usepackage{amsmath,amsthm,amssymb,amscd,mathtools}
\usepackage{amsfonts}
\usepackage{lastpage}
\usepackage{fancyhdr}

\usepackage{mathrsfs}
\usepackage{tikz-cd}
\usepackage{microtype}     
\usepackage{stmaryrd}

\newcommand{\mbb}[1]{\mathbb{#1}}
\newcommand{\mcal}[1]{\mathcal{#1}}

\newcommand{\NN}{\mathbb{N}}
\newcommand{\RR}{\mathbb{R}}
\newcommand{\ZZ}{\mathbb{Z}}
\newcommand{\QQ}{\mathbb{Q}}
\newcommand{\CC}{\mathbb{C}}

\newcommand{\FF}{\mathbb{F}}
\newcommand{\GG}{\mathbb{G}}

\newcommand{\Gal}{\mathrm{Gal}}

\newcommand{\Ind}{\mathrm{Ind}}
\DeclareMathOperator{\im}{Im}

\DeclareMathOperator{\cok}{coker}
\DeclareMathOperator{\Hom}{Hom}

\DeclareMathOperator{\Nm}{Nm}

\DeclareMathOperator{\Spf}{Spf}
\DeclareMathOperator{\Spec}{Spec}
\DeclareMathOperator{\Cl}{Cl}

\DeclareMathOperator{\Lie}{Lie}
\DeclareMathOperator{\Mod}{Mod}
\newcommand{\Res}{\mathrm{Res}}
\newcommand{\ur}{{\mathrm{ur}}}

\newcommand{\GL}{\mathrm{GL}}
\newcommand{\PGL}{\mathrm{PGL}}
\DeclareMathOperator{\Tr}{\mathrm{Tr}}
\newcommand{\mi}{\mathrm{i}}

\DeclareMathOperator{\Fil}{Fil}

\newcommand{\p}{\mathfrak{p}}

\DeclareMathOperator{\pr}{pr}
\newcommand{\mO}{\mathcal{O}}
\newcommand{\mP}{\mathbb{P}}
\newcommand{\mA}{\mathbb{A}}

\newcommand{\mm}{\mathfrak{m}}

\DeclareMathOperator{\SL}{SL}

\DeclareMathOperator{\End}{End}

\newcommand{\id}{\mathrm{id}}

\newcommand{\mF}{\mathcal{F}}

\DeclareMathOperator{\Ker}{Ker}

\newtheorem{thm}{Theorem}[subsection]
\newtheorem{prop}[thm]{Proposition}

\newtheorem{lem}[thm]{Lemma}
\newtheorem{cor}[thm]{Corollary}
\newtheorem{conj}[thm]{Conjecture}
\theoremstyle{remark}
\newtheorem{rmk}[thm]{Remark}
\newtheorem{eg}[thm]{Example}
\theoremstyle{definition}
\newtheorem{dfn}[thm]{Definition}
\newtheorem{notation}[thm]{Notation}
\newtheorem{construction}[thm]{Construction}

\theoremstyle{plain}

\theoremstyle{remark}

\theoremstyle{definition}

\newcommand{\mX}{\mathcal{X}}

\newcommand{\et}{\mathrm{\acute{e}t}}
\newcommand{\proet}{\mathrm{pro\acute{e}t}}
\newcommand{\proket}{\mathrm{prok\acute{e}t}}

\newcommand{\dR}{\mathrm{dR}}

\newcommand{\OBdRlog}{\mO\mbb{B}_{\dR,\log}}

\newcommand{\BdR}{\mbb{B}_\dR}

\newcommand{\BdRKp}{\mbb{B}_{\dR,K^{p}}}
\newcommand{\BdRX}[1]{\mbb{B}_{\dR,#1}}
\DeclareMathOperator{\gr}{gr}
\newcommand{\HT}{\mathrm{HT}}
\newcommand{\GM}{\mathrm{GM}}

\newcommand{\Fl}{{\mathscr{F}\!l}}

\DeclareMathOperator{\Spa}{Spa}
\newcommand{\TT}{\mbb{T}}

\newcommand{\fn}{\mathfrak{n}}
\newcommand{\fb}{\mathfrak{b}}
\newcommand{\fg}{\mathfrak{g}}
\newcommand{\sm}{{\mathrm{sm}}}
\newcommand{\la}{{\mathrm{la}}}
\newcommand{\an}{{\mathrm{an}}}

\newcommand{\Kpp}{K^pK_p}

\newcommand{\fh}{\mathfrak{h}}

\newcommand{\mrm}[1]{\mathrm{#1}}
\newcommand{\mfk}[1]{\mathfrak{#1}}

\newcommand{\RHom}{R\!\Hom}

\newcommand{\mscr}[1]{\mathscr{#1}}

\newcommand{\sen}{\mrm{sen}}

\newcommand{\Sym}{\mrm{Sym}}

\newcommand{\mE}{\mcal{E}}

\newcommand{\Rlim}{R\!\varprojlim}
\newcommand{\Fib}{\mrm{Fib}}

\newcommand{\oInd}{\otimes\!-\!\mrm{Ind}}

\newcommand{\kw}{(\Bbbk,w)}

\DeclareMathOperator{\Fr}{Fr}
\newcommand{\QCoh}{\mrm{QCoh}}
\newcommand{\Vect}{\mrm{Vect}}

\newcommand{\tor}{\mrm{tor}}
\DeclareMathOperator{\Rep}{Rep}

\newcommand{\tVBfu}{\tilde{V\!B}_{K^{p}}}

\DeclareMathOperator{\Spd}{Spd}
\DeclareMathOperator{\AnSpec}{AnSpec}

\newcommand{\CH}{\mrm{CH}}

\DeclareMathOperator{\oIndFQ}{\otimes-\mrm{Ind}^{\QQ}_{F}}

\usepackage{mathtools}

\DeclarePairedDelimiter\floor{\lfloor}{\rfloor}

\newcommand{\mS}{\mcal{S}}
\newcommand{\ti}[1]{\tilde{#1}}
\newcommand{\Perf}{\mrm{Perf}}
\newcommand{\cris}{\mrm{cris}}
\newcommand{\rF}{\mrm{FF}}
\newcommand{\Igs}{\mrm{Igs}}
\newcommand{\Bun}{\mrm{Bun}}
\newcommand{\BL}{\mrm{BL}}
\newcommand{\red}{\mrm{red}}

\newcommand{\Upp}{U^{p}U_{p}}
\newcommand{\tIgs}{\widetilde{\Igs}}
\newcommand{\bc}{\mrm{bc}}

\newcommand{\mPdR}[1]{\mcal{P}^{std,c}_{\mu_{#1},\dR}}
\newcommand{\mPHT}[1]{\mcal{P}_{\mu_{#1}}}

\newcommand{\omegaFL}[2]{\omega^{#2}_{\Fl_{G_{#1},\mu_{#1}}}}
\newcommand{\SGT}[1]{\mS(G_{#1})}
\newcommand{\SGTKp}[1]{\mS_{K^{p}}(G_{#1})}
\newcommand{\SGTK}[2]{\mS_{#2}(G_{#1})}
\newcommand{\SGTKTo}[2]{\mS_{#2}^{\tor}(G_{#1})}
\newcommand{\SGTKpTo}[1]{\mS_{K^{p}}^{\tor}(G_{#1})}
\newcommand{\SGTTo}[1]{\mS^{\tor}(G_{#1})}
\newcommand{\FLT}[1]{\Fl_{G_{#1},\mu_{#1}}}
\newcommand{\PerT}[1]{\Per_{G_{#1},\mu_{#1}}}
\newcommand{\FLTT}[2]{\Fl_{G_{#1},\mu_{#2}}}
\newcommand{\FLTTi}[2]{\Fl_{G_{#1},\mu_{#2}^{-1}}}
\DeclareMathOperator{\alg}{alg}
\newcommand{\Darith}{D_{\mrm{arith}}}

\newcommand{\pdR}{\mrm{alg}.\dR}
\DeclareMathOperator{\JL}{JL}
\newcommand{\hatotimes}{\mathbin{\hat{\otimes}}}
\newcommand{\unl}[1]{\underline{#1}}
\newcommand{\std}{\mrm{std}}
\newcommand{\Per}{\mscr{P}\!er}

\newcommand{\hpp}{\hat{+}\hat{+}}
\newcommand{\hp}{\hat{+}}

\usepackage{subfiles}
\usepackage{xr} 
\usepackage{enumitem}
\begin{document}
\title{Classicality for Hilbert modular forms}
\keywords{Classicality, Hilbert modular varieties, geometric Jacquet-Langlands correspondence}
\author{Yuanyang Jiang}
\address{Institut Mathématiques d'Orsay, Université Paris-Saclay, 307 Rue Michel Magat Bâtiment 307, 91400 Orsay, France}
\email{yuanyang.jiang@universite-paris-saclay.fr}
\subjclass[2020]{Primary 11F41, 11G18, 14G22; Secondary 11F80, 14F10}
\begin{abstract}
Let \(F\) be a totally  real number field.
We prove that a character of the spherical Hecke algebra appearing in the completed cohomology of Hilbert modular varieties is modular if the associated Galois representation of \(\Gal_{F}\) is absolutely irreducible, and de Rham of regular parallel weights. 
As an application, we prove some new cases of the Langlands-Clozel-Fontaine-Mazur conjecture of \(\GL_{2}\) over totally real fields.

For the proof, we generalize the method in \cite{Pan2209.06II}, calculate geometric partial Fontaine operators, and study the cohomology of the associated Koszul-type partial de Rham complexes. 
The key step is the establishment of a locally analytic Jacquet-Langlands transfer, whose proof consists of several novel ingredients including a comparison of Igusa stacks for different quaternionic Shimura data constructed by \cite{DanielVanHoftenKimZhangs2026igusastackscohomologyshimura}, and the Grothendieck-Messing theory for locally analytic infinite level Shimura varieties established in \cite{Jiang2026Shla}.
\end{abstract}
\maketitle

\tableofcontents

\section{Introduction}\label{sectionIntro}

\subsection{Classicality}
Let \(F\) be any number field, and let \( p \) be a prime number. We fix an abstract isomorphism \(\iota:\bar{\QQ}_{p}\cong \CC\). The Langlands-Clozel-Fontaine-Mazur conjecture predicts a bijection between the set of \(n\)-dimensional irreducible geometric representations of \(\Gal_{F}\) and the set of algebraic cuspidal automorphic representations of \(\GL_{n}(\mA_{F})\). Here, we say that a Galois representation \( \rho:\Gal_{F}\to \GL_{n}(\bar{\QQ}_{p}) \) is ``geometric'' if it is continuous, \emph{de Rham} at all places above \(p\) and unramified away from finitely many places.


A highly successful approach to this conjecture is 
by proving a modularity lifting theorem, which establishes the modularity of \( \rho \) under the assumption that its reduction \( \bar{\rho} \) is residually modular.
Based on the previous works on Serre's conjecture (\cite{Bockle2001Density}, \cite{DiamondFlachGuo2004tamagawa}, \cite{KhareWintenberger2009serre1}, \cite{KhareWintenberger2009serre2}, \cite{Kisin2009modularity}),  \cite{Kisin2009fontaine} and \cite{Emerton2011local-global} established the Langlands-Clozel-Fontaine-Mazur conjecture for \(n=2\) and \(F=\QQ\) with mild generic conditions.

The strategy of \cite{Emerton2011local-global} consists of two steps:

(1) 
Assuming that \( \bar{\rho} \) is residually modular,
one proves a ``promodularity'' theorem to realize \(\rho\) in the completed cohomology defined in \cite{Emerton06};

(2) One proves a ``classicality'' result: for Hecke eigenclasses appearing in the completed cohomology, the de Rhamness condition implies modularity.

In \cite{Emerton2011local-global}, both steps rely crucially on the \(p\)-adic local Langlands correspondence for \(\GL_{2}(\QQ_{p})\), which is unavailable for other groups.
 Nevertheless, in the Hilbert case, the remarkable work of \cite{BreuilHerzigHuMorraSchraen2023GKdim},
 combined with the ideas of \cite{GeeNewton2022patching}, 
 establishes the step (1) under a generic condition.

Our work focuses on the step (2), namely classicality, in the Hilbert setting. Throughout the article, we fix \(F\) to be a totally real number field of dimension \(d\). Let \(L\subset\bar{\QQ}_{p}\) be a finite extension of \(\QQ_{p}\) containing all embeddings of \(F\), with ring of integer \( \mO_{L} \) and residue field \( \FF \). For any place \(\ell\) of \(\QQ\), let \(\Sigma_{\ell}\) denote the set of places of \(F\) over \(\ell\). 

Let \(G:=\mrm{Res}_{F/\QQ}(\GL_{2})\). For any neat open compact subgroup \(K\subset G(\mA_{f})\), we have the \(d\)-dimensional Hilbert modular variety \(\mrm{Sh}_{K}(G)\) of level \(K\) defined over \(\QQ\), which is smooth, quasi-projective, and of abelian type (see \S \ref{subsectionQuaternUniSV}).

Following \cite{Emerton06}, we fix a neat open compact subgroup \(K^{p}\subset G(\mA_{f}^{p})\) and define the \(i\)-th \emph{completed cohomology} of the Hilbert modular varieties as
\[
\ti{H}^{i}(K^{p},L):=\left(\varprojlim_{n}\varinjlim_{K_{p}}H^{i}_{\et}(\mrm{Sh}_{\Kpp}(G)_{\bar{\QQ}},\ZZ/p^{n})\right)\hatotimes_{\ZZ_{p}}L,
\]
where the colimit is taken over all open compact subgroups \(K_{p}\subset G(\QQ_{p})\). 

Let \(S\) be a fixed finite set of places of \(\QQ\), including \(p\) and \(\infty\), as well as all places that ramify in \( F \) or \(K^{p}\). Denote by \(\TT^{S}\) the spherical Hecke algebra over \(\ZZ_{p}\) generated by the spherical Hecke operators of \(G\) at finite places away from \(S\). Then \(\ti{H}^{i}(K^{p},L)\) is equipped with an action of \(\TT^{S}\times G(\QQ_{p})\times \Gal_{\QQ}\). The following \emph{classicality} theorem is the main result of this paper.
\begin{thm}[Theorem \ref{thmClassicalityCompleteCoho}]
	\label{IntrothmClassicalityCompleteCoho}
	Let \(\rho:\Gal_{F,S}\to \GL_{2}(L)\) be a continuous representation.
	Let \(\chi:\TT^{S} \to L\) be the associated character determined by the Eichler-Shimura relation (\ref{alignEicherShimuraDetermin}).


	We assume that 
	
	(1) \(\ti{H}^{d}(K^{p},L)[\chi]:=\Hom_{\TT^{S}}(\chi,\ti{H}^{d}(K^{p},L))\ne 0\);

	(2) \(\rho\) is absolutely irreducible;
	
	(3) \(\rho\) is de Rham of regular parallel weights.

	Then the character \(\chi\) is classical, i.e. there exists a Hilbert modular eigenform \(f\),
	such that \(\TT^{S}\) acts on \(f\) via \(\chi\). In other words, \(\rho\cong \rho_{f}\).
\end{thm}

Our proof differs from that of \cite{Emerton2011local-global}, and builds upon the idea of \cite{Pan2209.06II}. In particular, we do not use any \(p\)-adic local Langlands correspondence. 

\begin{rmk}
After the pioneering work of Lue Pan (\cite{Pan2209.06II}) on modular curves, there have been several generalizations: 
\cite{QiuSu2025locallyanalyticvectorscompleted} proves a version for \(\tau\)-locally analytic completed cohomology of Shimura curves, and
\cite{Matsumoto2025classicalitytheoremapplicationsautomorphy} proves a version for the partially completed cohomology of unitary Shimura surfaces assuming that \(F_{\p}\in\{\QQ_{p},\QQ_{p^{2}}\}\) for all \(\p\in\Sigma_{p}\). 

Our work requires no assumptions on \(p\) and \(F\), and \(p\) may even be ramified. Furthermore, we work with the full locally analytic completed cohomology, allowing our method to potentially describe the internal structures of the (conjectural) locally analytic representations arising from the \(p\)-adic local Langlands correspondence (see Remark \ref{rmkInternalWallCrossing}).
\end{rmk}

\begin{rmk}
The technical condition of being of parallel weights is a consequence of current limitations in understanding ``geometric Fontaine operators''. 
I hope to remove this condition in the future.
\end{rmk}

As an application, combining our work with \cite{BreuilHerzigHuMorraSchraen2023GKdim} yields new cases of the Langlands-Clozel-Fontaine-Mazur conjecture: 

\begin{thm}[{Theorem \ref{thmConjecFontaineMazur}}]
	\label{IntothmConjecFontaineMazur}
Assume that \(p\) is \emph{inert} in \(F\).
Let
\(\rho:\Gal_{F,S}\to \GL_{2}(L)\) be a continuous representation.
Let \(\bar{\rho}:\Gal_{F}\to \GL_{2}(\FF)\) be the reduction of \(\rho\).

Assume that 

(1) \(\bar{\rho}\) is residually modular.

(2) \(\bar{\rho}|_{\Gal_{F(\sqrt[p]{1})}}\)
is absolutely irreducible.

(3) \(\bar{\rho}|_{\Gal_{F_{p}}}\) is semisimple generic in the sense of \cite[p.100 (i)]{BreuilHerzigHuMorraSchraen2023GKdim}.

(4) For any \(w\in S\backslash \{p\}\),
the framed deformation ring of \(\bar{\rho}|_{\Gal_{F_{w}}}\) over \(W(\FF)\)
is formally smooth.




(5) \(\rho|_{\Gal_{F_{p}}}\) is de Rham of regular parallel weights \( \{k,1-k\} \) for some \( k\in\ZZ \).

Then there exists a Hilbert modular eigenform \(f\)
such that \(\rho\cong \rho_{f}\).
\end{thm}

This theorem follows from the classicality theorem (Theorem \ref{IntrothmClassicalityCompleteCoho}), the results of \cite{BreuilHerzigHuMorraSchraen2023GKdim}, and an argument utilizing the \(\ell\)-adic Hecke action of \cite{FarguesScholze2021geometrization} (Corollary \ref{corDirectSummandVaryL}). 

\begin{rmk}
This theorem leaves significant room for improvement. 
The technical conditions are imposed to apply the theorem of \cite{BreuilHerzigHuMorraSchraen2023GKdim} directly. 
For example, \( p \) being unramified rather than being inert should suffice, and in (5), we could consider general regular parallel weights rather than the special cases \( \{k,1-k\} \).
\end{rmk}



\subsection{Strategy}
For this introduction, let us focus on the classicality theorem (Theorem \ref{IntrothmClassicalityCompleteCoho}). Our proof follows closely the idea of \cite{Pan2209.06II} to understand the de Rhamness condition.

 By \cite{Fontaine2004arithmetique}, we know that for a finite dimensional Hodge-Tate representation \( \rho \), the de Rhamness is determined by the vanishing of a certain \emph{Fontaine operator} acting on \( D_{\mrm{pdR}}(\rho) \), where the subscript refers to ``presque de Rham''. See \S \ref{subsecFontaineOpe} for details. 

 In \cite{Pan2209.06II}, Pan observes that the analogue of Fontaine operators exists for the infinite dimensional Galois representation given by 
  the locally analytic completed cohomology, 
  and that it arises from a certain \emph{geometric Fontaine operator}, 
  defined by concrete differential operators on the infinite level Shimura varieties. 
Thus the de Rhamness condition can be understood in terms of the kernel of these differential operators, i.e. the cohomolgy of the associated \( D \)-modules.

The construction of these differential operators has been generalized to arbitrary Shimura varieties in \cite{Jiang2026Shla}, and the associated \( D \)-modules is precisely the \emph{Bernstein-Gelfand-Gelfand-Fontaine complex} introduced in \cite[Definition \ref{1-dfnBGGFcomplex}]{Jiang2026Shla}. 

Thus we have the following general proof strategy towards classicality:
\begin{enumerate}
\item{Given an irreducible \( \rho \) with the associated character \( \chi_{\rho}:\TT^{S}\to L \), we want to relate \( \rho \) with the infinite-dimensional Galois representation \( \ti{H}^{d}(K^{p},L)^{\la}[\chi_{\rho}] \), where \((-)^{\la}\) denotes the subspace of \( G(\QQ_{p}) \)-locally analytic vectors.} \label{enumiStep1}
\item{We construct Fontaine operators for \( \ti{H}^{*}(K^{p},L)^{\la} \), and relate them to the differential operators constructed in \cite{Jiang2026Shla}.} \label{enumiStep2}
\item{We study the automorphic property of the cohomology of the Bernstein-Gelfand-Gelfand-Fontaine complexes, and try to answer \cite[Question \ref{1-conjFontaineComplexIsAutomorphic}]{Jiang2026Shla}.}\label{enumiStep3}
\end{enumerate}
In what follows, we will explain how we resolve each step in the Hilbert case, with more emphasis on the step (\ref{enumiStep3}).  
\subsubsection{Plectic Conjecture}
For modular curves, the step (\ref{enumiStep1}) follows immediately from the Eichler-Shimura relation: \[\ti{H}^{1}(K^{p},L)^{\la}[\chi_{\rho}]\cong \rho \otimes \Pi^{\la}(\rho_{\chi})
\] where 
\(\Pi^{\la}(\rho_{\chi})\) is a locally analytic representation of \(G(\QQ_{p})\). However, the general situation is more delicate. 

In the Hilbert setting, based on the classical results obtained via the Langlands-Kottwitz method (\cite{Langlands1979Zeta}, \cite{ReimannHarry1997Tszf}, \cite{BrylinskiLabesse1984cohomologie}), it is rather expected that \[\ti{H}^{d}(K^{p},L)^{\la}[\chi_{\rho}]\cong \oIndFQ\rho \otimes \Pi^{\la}(\rho_{\chi})
\] where \( \oIndFQ \)
denotes the tensor induction from \( \Gal_{F,S} \)
to \( \Gal_{\QQ,S} \). Then we have to face the problem
that \( \oIndFQ\rho \) may not be irreducible even when \( \rho \) is irreducible, and thus the expected isomorphism above will not follow from the Eichler-Shimura relation. 

This problem will be addressed in a forthcoming paper joint with Lue Pan (\cite{JiangPan2025Plectic}), where, motivated by the plectic conjecture of \cite{NekovarScholl2017Plectic}, we will construct (when \( \rho \) is strongly irreducible) an action of a certain plectic Galois group \( \Gal^{plec}_{F} \) on \( \ti{H}^{d}(K^{p},L)^{\la}[\chi_{\rho}] \). With this extra symmetry, the argument will work as \( \oIndFQ\rho \) is irreducible as a representation of \( \Gal^{plec}_{F} \). See \S \ref{subsectionPlecticGaloisAction} for more details.

\subsubsection{Geometric Fontaine operators}
We now introduce the geometric objects, and make precise the step (\ref{enumiStep2}). 
We fix a cone decomposition, and let \(\mrm{Sh}_{K}^{\tor}(G)\) denote the associated toroidal compactification of \(\mrm{Sh}_{K}(G)\). 
Denote by \(\mS_{\Kpp}^{\tor}(G):=(\mrm{Sh}_{\Kpp}^{\tor}(G)\otimes_{\QQ}\CC_{p})^{\an}\) the associated rigid variety over \( \CC_{p} \), which we regard as a diamond over \( \Spd(\CC_{p}) \)
in the sense of \cite{Scholze2017etaleDiamond}. 

We consider the following infinite-level Shimura varieties and Hodge-Tate period maps introduced by \cite{Scholze15}, and subsequently generalized by \cite{CS17}, \cite{Shen2017perfectoid}, \cite{HansenJohansson2020perfectoid}: 
 \[\mS_{K^{p}}^{\tor}(G):=\varprojlim_{K_{p}}\mS_{\Kpp}^{\tor}(G)\xrightarrow{\pi_{\HT}}\Fl\cong \prod_{\tau\in\Sigma_{\infty}}\mP^{1}_{\tau},\] where the limit is taken in the category of diamonds. 
Let \(\mO^{\la}\) denote the subsheaf (on the analytic site of \(\mS^{\tor}_{K^{p}}(G)\)) of \(\mO_{\mS^{\tor}_{K^{p}}(G)}\) consisting of locally analytic vectors. Then it follows from the primitive comparison that  (\cite[Theorem 4.4.6]{Pan22}, \cite[Theorem 6.2.6]{Juan2022.09locallyShi}) 
\[\ti{H}^{i}(K^{p},\QQ_{p})^{\la}\hatotimes \CC_{p}\cong H^{i}_{\an}(\mS_{K^{p}}^{\tor}(G),\mO^{\la}).
\] 

It is a theorem of \cite{Pan22} and \cite{Juan2022.09locallyShi} that the infinite dimensional Galois representation \( \ti{H}^{*}(K^{p},L)^{\la} \) admits an arithmetic Sen operator, and, along the primitive comparison, it can be described via a certain horizontal action \( \theta^{Pan} \) of \( \fh \) on \( \mO^{\la} \). Here \( \fh\cong \bigoplus_{\tau\in\Sigma_{\infty}}(\CC_{p}^{\oplus 2}) \) is the Cartan algebra of \( \Lie(G)\otimes\CC_{p} \). See \cite[Definition \ref{1-dfnHorizontalLAsv}]{Jiang2026Shla} for more details. 
For a character \( \chi:\fh\to\CC_{p} \), we write \(\mO^{\la,\chi}\) for the subsheaf of \(\mO^{\la}\) on which the horizontal action \((\mfk{h},\theta^{Pan})\) acts via \(\chi\). 

From now on,
we fix a \emph{strongly irreducible} representation \( \rho:\Gal_{F,S}\to \GL_{2}(L) \)
(i.e. the restriction to any open subgroup is absolutely irreducible), and let \( \chi_{\rho}:\TT^{S}\to L \) be the associated character. 
For simplicity, we assume (until the end of this introduction) that \( \chi_{\rho} \) is of parallel weights \( 0,1 \).
Then studying the relation among the infinitesimal character, the horizontal action, and the plectic Galois action of \cite{JiangPan2025Plectic} gives the following Hodge-Tate decomposition:
\begin{prop}[Theorem \ref{thmPadicPlectic}]\label{propIntroFixInfCharc}
	Let \( \rho \) and \(\chi_{\rho}\) be as above.
	 Assume that \(\rho\) is strongly irreducible. Then we have a Hodge-Tate decomposition \[\ti{H}^{d}(K^{p},L)^{\la}[\chi_{\rho}]\hatotimes_{L} \CC_{p}\cong \bigoplus_{I\subset \Sigma_{\infty}}H^{d}_{\an}(\SGTKpTo{},\mO^{\la,\chi_{I}})[\chi_{\rho}],
	\] where we define for \( \tau\in\Sigma_{\infty} \),
	 \(\chi_{\tau}:\mfk{h}\to \CC_{p}\), \(\chi_{\tau}:=((\delta_{\tau,\tau'},-\delta_{\tau,\tau'}))_{\tau'\in\Sigma_{\infty}}\),
	and for \(I\subset\Sigma_{\infty}\), \(\chi_{I}:=\sum_{\tau}\chi_{\tau}\).\end{prop}
Since \( \rho \) is a representation of \( \Gal_{F} \), for each \( \tau\in\Sigma_{\infty} \), there is a \( \tau \)-partial Fontaine operator \( N_{\tau} \), which then induces a \( \tau \)-partial Fontaine operator \( N_{\tau} \) acting on \( D_{\mrm{pdR}}(\oIndFQ\rho)  \).
By \cite{Fontaine2004arithmetique}, \( \rho \) is de Rham if and only if \( N_{\tau}=0 \) for all \( \tau \).

Using the plectic Galois action, we can relate  \( \ti{H}^{d}(K^{p},L)^{\la} \) with \( \oIndFQ\rho \). 
Then along the isomorphism in Proposition \ref{propIntroFixInfCharc}, the \( \tau \)-partial Fontaine operator \( N_{\tau} \) on \( D_{\mrm{pdR}}(\oIndFQ\rho) \) is induced by \emph{geometric \(\tau\)-partial Fontaine operators} \[N_{\tau}:\mO^{\la,\chi_{I}}\to \mO^{\la,\chi_{I\cup\{\tau\}}},\;I\subsetneq I\cup\{\tau\}\subset\Sigma_{\infty}, \] 
and as in \cite{Pan2209.06II}, \( N_{\tau} \) has a concrete description (Theorem \ref{IntrothmFontainEqDDbar}) using the following differential operators:
consider the diagram \[
\begin{tikzcd}
\mS_{K}^{\tor}(G)
& \mS_{K^{p}}^{\tor}(G)\arrow[l,"\pi_{K_{p}}"]\arrow[r,"\pi_{\HT}"] & \Fl,
\end{tikzcd}\] 
and one can show that \[
\pi_{K_{p}}^{-1}(\mO_{\mS^{\tor}_{K}(G)})\hatotimes\pi_{\HT}^{-1}(\mO_{\Fl})\to \mO^{\la,\chi_{\emptyset}}
\] is ``\'etale''
in a suitable sense (\cite[Proposition \ref{1-propCartesianLA0}]{Jiang2026Shla}).
Then one considers the relative differentials of \(\mO^{\la,\chi_{\emptyset}}\) over \(\mO_{\Fl}\) and over \(\mO_{\mS^{\tor}_{K}(G)}\) respectively
\begin{align*}
d:\mO^{\la,\chi_{\emptyset}}\to \mO^{\la,\chi_{\emptyset}}\otimes_{\mO_{\mS_{K}^{\tor}(G)}}\Omega^{1}_{\mS_{K}^{\tor}(G)},\;\bar{d}:\mO^{\la,\chi_{\emptyset}}\to \mO^{\la,\chi_{\emptyset}}\otimes_{\mO_{\Fl}}\Omega^{1}_{\Fl},
\end{align*} where \( \Omega^{1}_{\mS_{K}^{\tor}(G)} \) denotes the sheaf of log differential forms of \( \SGTKTo{}{\Kpp} \) over \( \CC_{p} \).

In the Hilbert setting, we have canonical decompositions \[
\Omega^{1}_{\SGTKTo{}{K}}\cong \bigoplus_{\tau\in \Sigma_{\infty}}\Omega^{1}_{\SGTKTo{}{K},\tau},\;\Omega^{1}_{\Fl}\cong \bigoplus_{\tau\in \Sigma_{\infty}}\Omega^{1}_{\Fl,\tau}.\]
We then denote the \(\tau\)-factors of \(d\) and \(\bar{d}\) respectively by \[
d_{\tau}:\mO^{\la,\chi_{\emptyset}}\to \mO^{\la,\chi_{\emptyset}}\otimes_{\mO_{\mS_{K}^{\tor}(G)}}\Omega^{1}_{\mS_{K}^{\tor}(G),\tau},\;\bar{d}_{\tau}:\mO^{\la,\chi_{\emptyset}}\to \mO^{\la,\chi_{\emptyset}}\otimes_{\mO_{\Fl}}\Omega^{1}_{\Fl,\tau}. 
\] Moreover, 
there are twisted versions of \( d_{\tau} \) 
and \( \bar{d}_{\tau} \)
\begin{align*}
\bar{d}_{\tau}&:\mO^{\la,\chi_{I}}\to \mO^{\la,\chi_{I}}\otimes_{\mO_{\Fl}}\Omega^{1}_{\Fl,\tau},\\
d_{\tau}&:\mO^{\la,\chi_{I}}\otimes_{\mO_{\Fl}}\Omega^{1}_{\Fl,\tau}\to \Omega^{1}_{\Fl,\tau}\otimes_{\mO_{\Fl}}
\mO^{\la,\chi_{I}}
\otimes_{\mO_{\SGTKTo{}{K}}}\Omega^{1}_{\SGTKTo{}{K},\tau}\cong \mO^{\la,\chi_{I\cup\{\tau\}}}.
\end{align*}
\begin{thm}[Theorem \ref{thmFontainEqDDbar}]\label{IntrothmFontainEqDDbar}
Up to a scalar in \(\QQ_{p}^{\times}\), we have
\begin{align}
N_{\tau}=d_{\tau}\circ \bar{d}_{\tau}.
\end{align}	
\end{thm}
In \cite{Pan2209.06II}, Pan prove Theorem \ref{IntrothmFontainEqDDbar} in the case \( F=\QQ \) by constructing a non-canonical splitting of the Hodge filtration in the period sheaf \(\OBdRlog\) (\cite{Scholze13} and \cite{DLLZ2022logarithmicJAMS}).
In this paper, we find a more direct and canonical proof, which takes inspiration from \cite{GuoLi2021period}. 
The key observation is that the period sheaf \(\BdR\) is naturally a bi-filtered module over the sheaf of \(\mbb{E}_{\infty}\)-algebras \(\pi_{K_{p}}^{-1}(\dR_{\SGTKTo{}{K}/\CC_{p}})\otimes \pi_{\HT}^{-1}(\dR_{\Fl/\CC_{p}})\). 
Then by construction, \(N_{\tau}\) can be enhanced to a \(G(\mA_{f})\)-equivariant morphism between such bi-filtered modules (Proposition \ref{propFontaineIsdRlinear}).
We then show that such morphisms are unique up to scalars (Proposition \ref{propUniquenessDRLinear})! 
On the other hand, we show that \(N_{\tau}\) is non-zero
by proving it is non-zero on graded pieces,
where the calculation can be reduced via the geometric Sen theory (\cite{Pan22}, \cite{Juan2022.05GeoSen}, \cite{Juan2022.09locallyShi}, \cite{Pilloni22}) to a calculation on the flag variety, which was done in \cite{Jiang2023thetaModCur}.
\begin{rmk}
Although we use the language of derived geometry, it eventually boils down to prove that \(N_{\tau}\) is a differential operator over \(\mO^{\la,\chi_{\emptyset}}\). This idea is also mentioned in \cite[Remark 6.5.14]{Pan2209.06II}.
\end{rmk}

\subsubsection{Bernstein-Gelfand-Gelfand-Fontaine complexes} 
We want to understand the common kernel of the partial Fontaine operators \(N_{\tau}\). Working in the derived category, we consider the Koszul type complex defined by \(\{N_{\tau}:\tau\in\Sigma_{\infty}\}\), which by Theorem \ref{IntrothmFontainEqDDbar} is a special case of the \emph{Bernstein-Gelfand-Gelfand-Fontaine complexes} defined in \cite{Jiang2026Shla}. This leads us to the step (\ref{enumiStep3}).

For simplicity, let us assume from now on (until the end of this introduction) that \(F\) is a quadratic extension of \(\QQ\), 
and denote \(\Sigma_{\infty}=\{\sigma,\tau\}\). 
We write
\(\Fl=\prod_{\tau\in\Sigma_{\infty}}\mP^{1}=:\Fl_{\sigma}\times \Fl_{\tau}\), with \(\Fl_{\sigma}\cong \Fl_{\tau}\cong \mP^{1}\).

In this case, the Bernstein-Gelfand-Gelfand-Fontaine complex \(\mrm{BGGF}\) is of the form \[\mrm{BGGF}:=
\left[
	\mO^{\la,\chi_{\emptyset}}\xrightarrow{(N_{\tau},N_{\sigma})}
	\mO^{\la,\chi_{\{\tau\}}}\oplus 
	\mO^{\la,\chi_{\{\sigma\}}}\xrightarrow{(N_{\sigma},-N_{\tau})}
	\mO^{\la,\chi_{\Sigma_{\infty}}}
\right].
\]

Thus we want to prove the following theorem, which gives an affirmative answer to \cite[Question \ref{1-conjFontaineComplexIsAutomorphic} (1)]{Jiang2026Shla}:
\begin{thm}[Theorem \ref{thmFontainComplexClassical}]\label{IntrothmFontainComplexClassical}
The complex \(\mrm{BGGF}\) has 
	\emph{classical} cohomology groups.
	More precisely, for any maximal ideal \(\mfk{a}\subset \TT^{S}[1/p]\)  that does not correspond to a classical Hilbert modular eigenform of parallel weight \(2\), \[R\Gamma_{\an}(\SGTKpTo{},\mrm{BGGF})_{\mfk{a}}\cong 0.\]
\end{thm}
As a consequence of Theorem \ref{IntrothmFontainEqDDbar}, \(\mrm{BGGF}\) has a bi-filtration of the form 
\begin{equation}\label{equationBiFilter}
\begin{tikzcd}
	\Fib(d_{\tau},\bar{d}_{\sigma}) \arrow[r,dash]\arrow[d,dash] & \Fib(d_{\tau},d_{\sigma})\arrow[d,dash] \\
	\Fib(\bar{d}_{\tau},\bar{d}_{\sigma}) \arrow[r,dash] & \Fib(\bar{d}_{\tau},d_{\sigma}),
\end{tikzcd}
\end{equation}
where in this notation, the left column and the lower row are sub-objects, and the right column and the upper row are quotients, and the terms are the associated Koszul-type complexes. 
Concretely, 
\begin{align*}
\Fib(\bar{d}_{\tau},\bar{d}_{\sigma})&
:=\left[
	\mO^{\la,\chi_{\emptyset}}\xrightarrow{(\bar{d}_{\tau},\bar{d}_{\sigma})}\Omega^{1}_{\Fl}\otimes_{\mO_{\Fl}}
	\mO^{\la,\chi_{\emptyset}}
	\xrightarrow{(\bar{d}_{\sigma},-\bar{d}_{\tau})}\Omega^{2}_{\Fl}\otimes_{\mO_{\Fl}}
	\mO^{\la,\chi_{\emptyset}}
	\right]\cong \mO^{\sm},\\
\Fib(d_{\tau},d_{\sigma})&
:=\Omega^{2}_{\Fl}\otimes_{\mO_{\Fl}}\left[
	\mO^{\la,\chi_{\emptyset}}\xrightarrow{(d_{\tau},d_{\sigma})}
	\mO^{\la,\chi_{\emptyset}}\otimes_{\mO_{\SGTK{}{K}}}\Omega^{1}_{\SGTK{}{K}}
	\xrightarrow{(d_{\sigma},-d_{\tau})}
	\mO^{\la,\chi_{\emptyset}}\otimes_{\mO_{\SGTK{}{K}}}\Omega^{2}_{\SGTK{}{K}}
	\right],\\
\Fib(\bar{d}_{\tau},d_{\sigma})&:=\left[\Omega^{1}_{\Fl,\sigma}\otimes_{\mO_{\Fl}}
	\mO^{\la,\chi_{\emptyset}}\xrightarrow{(\bar{d}_{\tau},d_{\sigma})}(\Omega^{2}_{\Fl}\otimes_{\mO_{\Fl}}
	\mO^{\la,\chi_{\emptyset}})\oplus 
	\mO^{\la,\chi_{\{\sigma\}}}\xrightarrow{(d_{\sigma},-\bar{d}_{\tau})}\Omega^{1}_{\Fl,\tau}\otimes_{\mO_{\Fl}}
	\mO^{\la,\chi_{\{\sigma\}}}
	\right]\\&\cong \Omega^{1}_{\Fl_{\sigma}}\otimes_{\mO_{\Fl_{\sigma}}}\left[
		\mO^{\sigma-\la,\chi_{\emptyset}}\xrightarrow{d_{\sigma}}\mO^{\sigma-\la,\chi_{\emptyset}}\otimes_{\mO_{\SGTKTo{}{K}}}\Omega^{1}_{\SGTKTo{}{K},\sigma}
	\right],
\end{align*}
where \( \mO^{\sm} \) is the subsheaf of \( \mO^{\la} \) consisting of \( G(\QQ_{p}) \)-smooth vectors, and \(\mO^{\sigma-\la,\chi_{\emptyset}}\) denotes the subsheaf of \(\mO^{\la,\chi_{\emptyset}}\) consisting of \(\sigma\)-locally analytic vectors (i.e. smooth in the \(\tau\)-direction).

Now Theorem \ref{IntrothmFontainComplexClassical} follows from the following stronger theorem:
\begin{thm}[Theorem \ref{thmClassicalityCohomoDR}]\label{IntrothmClassicalityCohomoDR}
	The complexes \(\Fib(d_{\tau},\bar{d}_{\sigma}), \Fib(d_{\tau},d_{\sigma}),
	\Fib(\bar{d}_{\tau},\bar{d}_{\sigma}), \Fib(\bar{d}_{\tau},d_{\sigma})\) have classical cohomology groups. Here being ``classical'' is in the same sense as in Theorem \ref{IntrothmFontainComplexClassical}.
\end{thm}
This is the deepest result of this paper. 
The treatments of \( \Fib(\bar{d}_{\tau},\bar{d}_{\sigma})\) and \( \Fib(d_{\tau},d_{\sigma}) \) are similar to those in \cite{Pan2209.06II}: 
the cohomology of \( \Fib(\bar{d}_{\tau},\bar{d}_{\sigma})\cong \mO^{\sm} \) is the colimit of coherent cohomology of \( \SGTKTo{}{\Kpp} \), and for any Newton stratum \( \Fl^{b}\subset \Fl \), and \( x\in \Fl^{b}(\CC_{p}) \), the stalk \( (\pi_{\HT,*}\Fib(d_{\tau},d_{\sigma}))_{x} \) calculates the rigid cohomology of the Igusa varieties \( \Igs^{b}_{K^{p}} \). 
However, the classicality of the rigid cohomology, rather than merely the alternating sum in the Grothendieck group, is already a difficult problem. In fact, as a byproduct of our proof, we show that the rigid cohomology of Igusa varieties are classical as Hecke modules. 

The more mysterious is \(\Fib(\bar{d}_{\tau},d_{\sigma})\), which is some sort of ``partial de Rham complex''. Our strategy is to find a suitable stratification, and define a \emph{locally analytic Jacquet-Langlands correspondence} (Theorem \ref{Introthm:LaJacquetLanglands} (4)) relating this partial de Rham cohomology to coherent cohomology of quaternionic Shimura curves.
We will make some preparations on analytic stacks before stating Theorem \ref{Introthm:LaJacquetLanglands}. 

\begin{rmk}\label{rmkInternalWallCrossing}
As in \cite{Pan2209.06II},
the bi-filtration (\ref{equationBiFilter}) provides evidence for the expected internal structure of the locally analytic representations of \(\GL_{2}(F_{p})\) in the conjectural \(p\)-adic local Langlands correspondence, as predicted by \cite{BreuilHerzig2015OrdinaryRep}, \cite{BreuilHerzigHuMorraSchraen2024conjectures}, \cite{Ding2024wallcrossinglocallymathbbqpanalyticrepresentations}. 
It should be related to
the square appearing in \cite[Conjecture 1.7]{Ding2024wallcrossinglocallymathbbqpanalyticrepresentations} via translation functors.  
In fact, one can calculate the wall-crossing of the graded pieces, following the argument in \cite{Su2025locallyanalytictextext1conjecturetextgl2l}. 
\end{rmk}
\subsubsection{Analytic stacks}
Although Theorem \ref{IntrothmClassicalityCohomoDR} can be stated in terms of sheaves on analytic site of \( \SGTKpTo{} \), it is difficult to manipulate these sheaves involving differential operators within the classical framework.
For one thing, the perfectoid geometry automatically ignores differential information, and for another thing, there does not exist a general six functor formalism for sheaves on analytic sites to the best of our knowledge. Furthermore, we need to take into account the topological properties of these sheaves. 
The theory of analytic stacks of Clausen-Scholze provides a satisfying solution to all these problems. 

In \cite{Jiang2026Shla}, we constructed various realizations of infinite-level Shimura varieties as analytic stacks. Let us recall some constructions here. Readers are referred to the introduction of \cite{Jiang2026Shla} for details. 

Roughly, we define the \emph{locally analytic Hilbert modular variety}
\( \mS_{K^{p}}(G)^{\la} \) (\cite[Definition \ref{1-dfnLocallyAnSV}]{Jiang2026Shla}) as the relative (analytic) spectrum of the sheaf \( \mO^{\la} \), which is a refinement of the perfectoid Hilbert modular variety \( \SGTKp{} \) and could be thought of as a ringed space \( (|\mS_{K^{p}}(G)|,\mO^{\la}) \). Note that in \cite{Jiang2026Shla}, Shimura varieties were denoted by curly X such as \( \mX_{K^{p}}^{\la} \), and we are sorry for changing the notation here. 

On the other hand, we want an enhancement of the topological space \( |\SGTKp{}| \). 
For this, we consider the theory of analytic de Rham stacks of \cite{Juan2024analyticdeRham} and \cite{AnschützBoscoLeBrasCamargoScholze2025analyticrhamstacksfarguesfontaine}, which is an ideal framework for studying differential operators. 
We denote by \( \SGTKp{}^{\dR} \) the analytic de Rham stack of \( \SGTKp{} \), and then it follows by construction that the partial de Rham complexes in Theorem \ref{IntrothmClassicalityCohomoDR} can be naturally enhanced to quasicoherent sheaves on \( \SGTKp{}^{\dR} \). 

The four constructions fit into the following sequence \[
\SGTKp{}\to \SGTKp{}^{\la}\to \SGTKp{}^{\dR}\to |\SGTKp{}|.
\] The partial de Rham complexes arises from the interplay between \( \SGTKp{}^{\la} \)
and \( \SGTKp{}^{\dR} \). In \cite{Jiang2026Shla}, we gave a systematic study of their relation: 
\begin{thm}[{\cite[Theorem \ref{1-thmLAstrucGMHTper}]{Jiang2026Shla}}]\label{thmLAstrucGMHTperIntro}
There exists a \emph{Cartesian} diagram 
\[
\begin{tikzcd}[column sep=0.8in]
\mS_{K^{p}}(G)^{\la}\arrow[r,"\pi_{\GM\HT,K^{p}}^{\la}"]\arrow[d]
& \Per_{G,\mu}/G^{c}\arrow[d]\\
\mS_{K^{p}}(G)^{\dR}\arrow[r,"\pi_{\GM\HT,K^{p}}^{\dR}"]
& \Per_{G,\mu}^{\dR}/G^{c},
\end{tikzcd}
\] where \( \Per_{G,\mu} \) is the Grothendieck-Messing-Hodge-Tate period domain defined in  \cite[Definition \ref{1-dfnGMHTperDomain}]{Jiang2026Shla}.
We refer to the horizontal maps as the \emph{Grothendieck-Messing-Hodge-Tate period maps}.
\end{thm}
\begin{rmk}\label{rmkIntroDescentFib}
As a corollary of Theorem \ref{thmLAstrucGMHTperIntro}, the partial de Rham complexes regarded as quasicoherent sheaves on \( \SGTKp{}^{\dR} \) descend naturally to quasicoherent sheaves on \( \Per^{\dR}_{G,\mu}/G^{c} \), using which we will reduce difficult calculations on Shimura varieties to much simpler calculations on the period domains. 
\end{rmk}
\begin{rmk}
In fact, the constructions of \( \SGTKp{}^{\la} \) and \( \SGTKp{}^{\dR} \) in \cite{Jiang2026Shla} work for any Shimura datum, even without assuming the perfectoidness of \( \SGTKp{} \). Below we will apply them to the quaternionic Shimura varieties. 
\end{rmk}
We will also use the language of analytic stacks for studying \( p \)-adic representations. See for example \cite{JacintoJoaquínJuan2023SolidLARepII}. Given a \( p \)-adic Lie group \( \Gamma \), there are various realizations 
\( \unl{\Gamma} \), \( \Gamma^{\la} \)
and \( \Gamma^{\sm} \) as analytic stacks,
whose classifying stacks parametrize continuous representations, locally analytic representations, and smooth representations of \( \Gamma \) respectively. 
\subsubsection{Jacquet-Langlands correspondence} 
Let \(B\) be a quaternion algebra over \(F\). 
In the theory of automorphic representations, \cite{JacquetLanglands2006automorphic} established a correspondence between the automorphic representations of \((B\otimes_{\QQ}\mA)^{\times} \) and those of \(\GL_{2}(\mA_{F})\). The result that we are going to explain is a geometric version of this relation.

Let \(B_{\{\sigma\}}\) be the quaternion algebra over \(F\) that is ramified precisely at \(\sigma\) and the \(p\)-adic place determined by \( \iota\circ\sigma \). Let \(G_{\{\sigma\}}:=\mrm{Res}_{F/\QQ}B_{\{\sigma\}}^{\times}\) and 
\[h_{\{\sigma\}}:\mbb{S}\to (G_{\{\sigma\}})_{\RR}\cong \GL_{2}\times \mbb{H}^{\times},x+y\cdot \mrm{i}\mapsto \left(\begin{pmatrix}
	x & y \\ -y & x
\end{pmatrix},1\right).\] Denote by \(\mrm{Sh}_{K}(G_{\{\sigma\}})\) the Shimura curve of level \(K\) over \(F\subset^{\tau} \CC\) associated to \( (G_{\{\sigma\}},h_{\{\sigma\}}) \). 
As before, we denote \(\SGTK{\{\sigma\}}{K}:=(\mrm{Sh}_{K}(G_{\{\sigma\}})\otimes_{F}{\CC_{p}})^{\an}\),
where \(F\subset \CC_{p}\) is given \( \iota\circ \tau \). We denote \(\SGTKp{\{\sigma\}}:=\varprojlim_{K_{p}}\SGTK{\{\sigma\}}{\Kpp}\), which is represented by a perfectoid space by \cite{Shen2017perfectoid} or \cite{HansenJohansson2020perfectoid}. 

To be consistent with the notation, we write \( G_{\emptyset}:=G=\mrm{Res}_{F/\QQ}\GL_{2} \).
We identify \(G_{\emptyset}(\mA_{f}^{p})\cong G_{\{\sigma\}}(\mA_{f}^{p})\), fix a neat open compact subgroup \(K^{p}\) and regard it as a prime-to-\(p\) level for both groups.

We want to relate \(\SGTKp{\emptyset}\)
and \(\SGTKp{\{\sigma\}}\). This will be achieved by a local datum. 
Let us consider \[\mu_{\{\tau\}}:
\GG_{m}\to (G_{\emptyset})_{\bar{\QQ}_{p}}\cong \GL_{2,\tau}\times \GL_{2,\sigma},t\mapsto (1,\mrm{diag}\{t,1\}),
\] and let \(b_{\{\tau\}}\)
be the unique basic element in \(B(G_{\emptyset},\mu_{\{\tau\}})\). Then \((G_{\emptyset},\mu_{\{\tau\}},b_{\{\tau\}})\)
forms a local Shimura datum in the sense of \cite{ScholzeWeinstein2020berkeley}, and we denote by \(\mcal{M}_{\emptyset\to \{\sigma\}}\)
the associated (infinite-level) local Shimura variety over \(\CC_{p}\). By \cite{SW13}, this is a perfectoid space equipped with an action of \(G_{\emptyset}(\QQ_{p})\times G_{\{\sigma\}}(\QQ_{p})\).
The constructions of \( \SGTKp{}^{\la} \) and \( \SGTKp{}^{\dR} \) have their local counterparts, which we denote as \( \mcal{M}_{\emptyset\to \{\sigma\}}^{\la} \)
and \( \mcal{M}_{\emptyset\to \{\sigma\}}^{\dR} \) respectively, and there is a Grothendieck-Messing-Hodge-Tate period map \( \mcal{M}_{\emptyset\to\{\sigma\}}^{\dR}\to \Per^{\dR}_{\emptyset\to \{\sigma\}} \) such that the analogue of Theorem \ref{thmLAstrucGMHTperIntro} holds. See \cite[\S \ref{1-subsecLocalAnalogueLASv}]{Jiang2026Shla} for details.
\begin{dfn}[Definition \ref{dfnIbasicLocus}]
We define the \emph{\(\{\sigma\}\)-basic locus} \(\SGTKp{\emptyset}^{\{\sigma\}-\bc}\) to be the open subspace of \(\SGTKp{\emptyset}\) defined by the  Cartesian diagram \[
\begin{tikzcd}
\SGTKp{\emptyset}^{\{\sigma\}-\bc}\arrow[r,hook]\arrow[d]
& \SGTKp{\emptyset}\arrow[d,"\pi_{\HT}"] \\
(\Fl_{\sigma}\backslash p_{\sigma}(\Fl(\QQ_{p})))\times \Fl_{\tau}\arrow[r,hook]& \Fl\cong \Fl_{\sigma}\times \Fl_{\tau}.
\end{tikzcd}\]
Here, \(\Fl=\Fl_{\sigma}\times \Fl_{\tau}\cong (\mP^{1})^{2}\), \(\Fl(\QQ_{p})=\mP^{1}(F_{p})\),
 and \(p_{\sigma}\) is the projection to \(\Fl_{\sigma}\). 

 By the constructions of \( \SGTKp{\emptyset}^{\la} \)
 and \( \SGTKp{\emptyset}^{\dR} \),
 there are corresponding open subspaces \( \SGTKp{\emptyset}^{\la,\{\sigma\}-\bc} \)
 and \( \SGTKp{\emptyset}^{\dR,\{\sigma\}-\bc} \).
\end{dfn}
\begin{thm}[Locally analytic Jacquet-Langlands correspondence]\label{Introthm:LaJacquetLanglands}---

(1) (Corollary \ref{corJLlaIsGlaTorsor})
There is a canonical isomorphism of Tate stacks over \(\CC_{p}\) \[
\JL^{\la}:\left(\SGTKp{\{\sigma\}}^{\la}\times \mcal{M}_{\emptyset\to \{\sigma\}}^{\la}\right)/G_{\{\sigma\}}(\QQ_{p})^{\la}\cong \SGTKp{\emptyset}^{\la,\{\sigma\}-\bc}\subset \SGTKp{\emptyset}^{\la}.
\]

(2) (Theorem \ref{thmJLinTateSt}) There is a canonical isomorphism of Tate stacks over \(\bar{\QQ}_{p}\) \[
\JL^{\dR}:
\left(\SGTKp{\{\sigma\}}^{\dR}\times \mcal{M}_{\emptyset\to \{\sigma\}}^{\dR}\right)/G_{\{\sigma\}}(\QQ_{p})^{\sm}\cong \SGTKp{\emptyset}^{\dR,\{\sigma\}-\bc}\subset \SGTKp{\emptyset}^{\dR}.
\]

(3) (Theorem \ref{thmJLCompatiblePeriodMaps}) There is a commutative diagram \[
\begin{tikzcd}
    \SGTKp{\{\sigma\}}^{\dR}\times_{\bar{\QQ}_{p}} \mcal{M}_{\emptyset\to \{\sigma\}}^{\dR} \arrow[r,"\JL^{\dR}"] \arrow[d]
    & \SGTKp{\emptyset}^{\dR}\arrow[d] \\
    (\PerT{\{\sigma\}}^{\dR}/G^{c,\an}_{\{\sigma\}})\times_{\bar{\QQ}_{p}} \Per^{\dR}_{\emptyset\to \{\sigma\}} \arrow[r,"\JL^{per,\dR}"] & \PerT{\emptyset}^{\dR}/G^{c,\an}_{\emptyset},
\end{tikzcd}
\] where the vertical maps are the corresponding Grothendieck-Messing-Hodge-Tate period maps. 

(4) (Theorem \ref{thmJacquetLanglandsLA}) We have a canonical isomorphism \[(\JL^{\dR})^{*}(\Fib(\bar{d}_{\tau},d_{\sigma}))\cong \mO^{\sm}_{\SGTKp{\{\sigma\}}}\boxtimes \pi_{\HT,\emptyset}^{*}(\Omega^{1}_{\Fl_{\sigma}})^{G_{\{\sigma\}}(\QQ_{p})-\sm},
	\] where \(\mO^{\sm}_{\SGTKp{\{\sigma\}}}\) denotes the subsheaf of \(\mO_{\SGTKp{\{\sigma\}}}\) consisting of \( G_{\{\sigma\}}(\QQ_{p}) \)-smooth vectors, and \(\pi_{\HT,\emptyset}\) denotes the \(G_{\emptyset}(\QQ_{p})\)-equivariant Hodge-Tate period map \(\mcal{M}_{\emptyset\to\{\sigma\}}\to 
	\Fl_{\sigma}\). Here \( \mO_{\SGTKp{\{\sigma\}}} \)
	is regarded as a quasicoherent sheaf on \( \SGTKp{\{\sigma\}}^{\dR} \) via \( * \)-pushforward along \( \SGTKp{\{\sigma\}}\to \SGTKp{\{\sigma\}}^{\dR} \), and a similar convention is applied to \( \pi_{\HT,\emptyset}^{*}(\Omega^{1}_{\Fl_{\sigma}}) \).
\end{thm}
Theorem \ref{Introthm:LaJacquetLanglands} (4) shows that the Jacquet-Langlands transfer changes partial de Rham cohomology over Hilbert modular surfaces
into coherent cohomology over quaternionic curves.
This type of phenomena was first observed in \cite[Theorem 5.3.20]{Pan2209.06II}. This is the crucial technical ingredient for the proof of Theorem \ref{IntrothmClassicalityCohomoDR} in \S \ref{subsecProofOfClassicality}. 
See \S \ref{subsubsec:Intro-Pf-Cla} for a sketch of the argument. 

Let us say a few words about the proof of Theorem \ref{Introthm:LaJacquetLanglands}. 
Using Remark \ref{rmkIntroDescentFib}, the part (4) follows immediately from the part (3) via an explicit calculation along \( \JL^{per,\dR} \). The part (1) follows from the parts (2) and (3) by Theorem \ref{thmLAstrucGMHTperIntro}. 

For the parts (2) and (3), we will first prove the following perfectoid version of Jacquet-Langlands correspondence,
which is of independent interest:
\begin{thm}[Theorem \ref{thmProductFormulaQuate}]
	\label{IntrothmthmProductFormulaQuate}
There is a canonical isomorphism of perfectoid spaces over \(\CC_{p}\) \[
\left(\SGTKp{\{\sigma\}}\times \mcal{M}_{\emptyset\to \{\sigma\}}\right)/\underline{G_{\{\sigma\}}(\QQ_{p})}\cong \SGTKp{\emptyset}^{\{\sigma\}-\bc}\subset \SGTKp{\emptyset}.
\] 
\end{thm} 




The proof of Theorem \ref{IntrothmthmProductFormulaQuate} is strongly motivated by \cite{TianXiao2016goren}, where a similar result is proven for special fibers of different unitary Shimura varieties by constructing explicit \(p\)-quasi isogenies.
In this article, using the Igusa stacks introduced by \cite{Zhang2023pel} and further developed by \cite{DanielsHoftenKimZhang2024igusa}, \cite{Kim2025uniquenessfunctorialityigusastacks}, \cite{DanielVanHoftenKimZhangs2026igusastackscohomologyshimura},
we give a more streamlined proof. Roughly, the Igusa stack parametrizes abelian schemes (with PEL structures) up to \(p\)-quasi isogenies, and we establish Theorem \ref{IntrothmthmProductFormulaQuate} by constructing an open immersion from the Igusa stack of \(\SGTKp{\{\sigma\}}\) to that of \(\SGTKp{\emptyset}\).

Given Theorem \ref{IntrothmthmProductFormulaQuate}, the part (2) of Theorem \ref{Introthm:LaJacquetLanglands} follows immediately from the properties of the functor of taking analytic de Rham stacks established in \cite{AnschützBoscoLeBrasCamargoScholze2025analyticrhamstacksfarguesfontaine}. 
The proof of Theorem \ref{Introthm:LaJacquetLanglands} (3) is more involved. It is easy to construct a perfectoid version of the commutative diagram, but as an extra ingredient, one needs to compare the de Rham \( G^{c} \)-torsors on \( \SGTKp{\emptyset}^{\dR} \) and \( \SGTKp{\{\sigma\}}^{\dR} \), which encodes information about connections. 
The proof of this theorem uses \cite[Corollary \ref{1-corReformulateRiemmanHilbPerfdCase}]{Jiang2026Shla}, which is a reformulation of the logarithmic Riemann-Hilbert correspondence of \cite{LiuZhu2017rigidity} and \cite{DLLZ2022logarithmicJAMS}. 
\begin{rmk}
The theorems in the main text are more general,
relating the \( I \)-basic locus of the \( (d-|T|) \)-dimensional quaternionic Shimura varieties \( \SGTKp{T} \) (Definition \ref{dfnIbasicLocus}) to the \( (d-|T|-|I|) \)-dimensional quaternionic Shimura varieties \( \SGTKp{T\coprod I} \). If \( T\coprod I=\Sigma_{\infty} \), this specializes to the \(p\)-adic uniformization of \cite{RapoportZink1996period}.
 Our generalization is crucial for the following reason: in the curve case, the complement of \(\mP^{1}(\QQ_{p})\) is precisely the basic locus, but the similar statement does not hold in the Hilbert setting; instead, the complement of \(\mP^{1}(F_{p})\) is precisely the union of all the \(I\)-basic locus for varying \(I\subset\Sigma_{\infty}\), \(I\ne \emptyset\). This is the key to the inductive proof of Theorem \ref{IntrothmClassicalityCohomoDR} in \S \ref{subsecProofOfClassicality}.
\end{rmk}

\subsubsection{Proof of classicality}
\label{subsubsec:Intro-Pf-Cla}
For curious readers, let us sketch the proof of Theorem \ref{IntrothmClassicalityCohomoDR} using Theorem \ref{Introthm:LaJacquetLanglands}, which also reveals the internal structure of the locally analytic representations. See the proof of Theorem \ref{thmClassicalityCohomoDR} in \S \ref{subsecProofOfClassicality} for details. 

We consider the excision sequence with respect to the partition \[j:(\Fl_{\sigma}\backslash p_{\sigma}(\Fl(\QQ_{p})))\times \Fl_{\tau}\hookrightarrow \Fl\cong \Fl_{\sigma}\times \Fl_{\tau}\hookleftarrow p_{\sigma}(\Fl(\QQ_{p}))^{\dagger}\times \Fl_{\tau}:i,
\] 
where \( p_{\sigma}(\Fl(\QQ_{p}))^{\dagger} \) denotes the dagger neighborhood of \( p_{\sigma}(\Fl(\QQ_{p})) \) in \( \Fl_{\sigma} \).
We also denote by \((j,i)\) the pull-back of this partition to \(\SGTKpTo{}\) along \(\pi_{\HT}\). 
 So we have a fiber sequence \[j_{!}j^{*}(\Fib(\bar{d}_{\tau},d_{\sigma}))\to \Fib(\bar{d}_{\tau},d_{\sigma})\to i_{*}i^{*}(\Fib(\bar{d}_{\tau},d_{\sigma})).
\] 
Thus \( R\Gamma(\SGTKpTo{}^{\dR},\Fib(\bar{d}_{\tau},d_{\sigma})) \)
is filtered (in the derived sense) by the following two pieces:
\begin{itemize}
\item \(R\Gamma(\SGTKpTo{}^{\dR},j_{!}j^{*}(\Fib(\bar{d}_{\tau},d_{\sigma}))) \): we can understand this part by reducing to the case of quaternionic Shimura curves via the \emph{locally analytic Jacquet-Langlands transfer} (Theorem \ref{Introthm:LaJacquetLanglands} (4)). 

Concretely, \begin{align*}\label{alignExampleJLIntro}
R\Gamma(j_{!}j^{*}(\Fib(\bar{d}_{\tau},d_{\sigma})))\cong C_{*}\left(G_{\{\sigma\}}(\QQ_{p})^{\sm},
N_{\{\sigma\}}^{\sm}
\hatotimes^{L} T_{\{\sigma\}\to \emptyset}\right),
\end{align*} where 
\begin{align*}
&N_{\{\sigma\}}^{\sm}:=
\varinjlim_{K_{p}\subset G_{\{\sigma\}}(\QQ_{p})}R\Gamma(\SGTK{\{\sigma\}}{\Kpp})\in D(G_{\{\sigma\}}(\QQ_{p})^{\sm}\times \TT^{S}),\\
&T_{\{\sigma\}\to \emptyset}:=R\Gamma_{c}(\mcal{M}_{\emptyset\to \{\sigma\}}^{\dR},\pi_{\HT,\emptyset}^{*}(\Omega^{1}_{\Fl_{\sigma}})^{G_{\{\sigma\}}(\QQ_{p})-\sm})\in D(G_{\{\sigma\}}(\QQ_{p})^{\sm}\times G_{\emptyset}(\QQ_{p})^{\la}),
\end{align*}
and
\(C_{*}(G_{\{\sigma\}}(\QQ_{p})^{\sm},-)\) denotes the homology complex in the category of smooth \(G_{\{\sigma\}}(\QQ_{p})\)-representations. 
\item \( R\Gamma(\SGTKpTo{}^{\dR},i_{*}i^{*}(\Fib(\bar{d}_{\tau},d_{\sigma}))) \): this is a \( \sigma \)-locally analytic parabolic induction. More precisely, let \( Q:=\mrm{Stab}_{G(\QQ_{p})}(\infty_{\sigma})\subset G(\QQ_{p}) \) for \( \infty_{\sigma}\in\Fl_{\sigma} \). Then \begin{align*}
R\Gamma(i_{*}i^{*}(\Fib(\bar{d}_{\tau},d_{\sigma})))
\cong \Ind_{Q}^{G(\QQ_{p})}(M^{\sm}\otimes \Omega^{1}_{\Fl_{\sigma}}|_{\infty_{\sigma}})^{\sigma-\la},
\end{align*}
where \[
M^{\sm}:=R\Gamma(\SGTKpTo{}|_{\infty_{\sigma}\times \Fl_{\tau}},\dR_{\{\sigma\}}(\mO^{\sm}))
\] with \[ \dR_{\{\sigma\}}(\mO^{\sm}):=\left[
	\mO^{\sm}\xrightarrow{d_{\sigma}}\mO^{\sm}\otimes_{\mO_{\SGTKTo{}{K}}}\Omega^{1}_{\SGTKTo{}{K},\sigma}
\right]. \]
\end{itemize}
The situation is illustrated in Figure \ref{figureIntro} below.
\begin{figure}
\includegraphics[width=\linewidth]{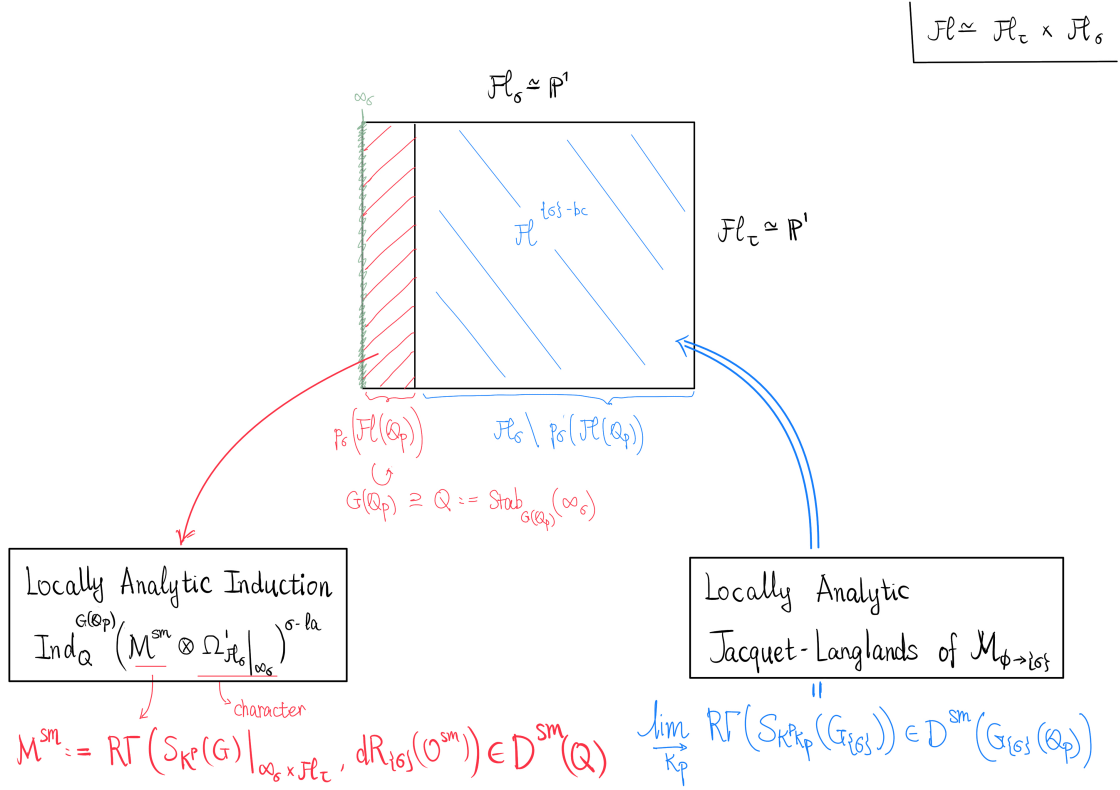}
\caption{Parabolic Induction or Jacquet-Langlands transfer}
\label{figureIntro}
\end{figure}

In general, the locally analytic Jacquet-Langlands transfer
will allow us to reduce \( I \)-basic parts to lower dimensional Shimura varieties, which are well understood by the inductive hypothesis. Thus the remaining mysterious part is the smooth \( Q \)-representation \( M^{\sm} \)
that appears in the locally analytic parabolic induction.

Here we employ a trick by replacing \( \Fib(\bar{d}_{\tau},d_{\sigma}) \) with \( \dR_{\{\sigma\}}(\mO^{\sm}) \), and repeating the whole argument of excision. Now \( R\Gamma(j_{!}j^{*}(\dR_{\{\sigma\}}(\mO^{\sm}))) \) is governed by the smooth Jacquet-Langlands transfer. Moreover, \[ R\Gamma(\SGTKpTo{}^{\dR},\dR_{\{\sigma\}}(\mO^{\sm}))\cong \varinjlim_{K_{p}}R\Gamma(\SGTKTo{}{\Kpp}^{\dR},\dR_{\{\sigma\}}(\mO_{\SGTKTo{}{\Kpp}})), \]
with self-explanatory notation on the RHS. 
This allows us to reverse-engineer information about \( M^{\sm} \)
from the excision sequence.

\subsection{Organization of the paper}
The paper is organized as follows:

In Section \ref{sectionGeoSV}, we review the general theory of Igusa stacks following \cite{Zhang2023pel}, \cite{DanielsHoftenKimZhang2024igusa}, \cite{Kim2025uniquenessfunctorialityigusastacks}, \cite{DanielVanHoftenKimZhangs2026igusastackscohomologyshimura}. Then we define the quaternionic Shimura varieties, and the related unitary Shimura varieties, 
we will construct their Igusa stacks, and show that the Igusa stacks for different quaternionic Shimura varieties can be embedded into that for the Hilbert modular varieties.
This implies Theorem \ref{IntrothmthmProductFormulaQuate}. We also give another application to \( \ell \)-adic cohomology, which implies in particular that the completed cohomolgy of certain lower dimensional quaternionic Shimura varieties can be realized as a \emph{direct summand} of that of higher dimensional ones.

In Section \ref{sectionJLTransfer}, we combine the results in \cite{Jiang2026Shla} and Section \ref{sectionGeoSV} to prove the locally analytic Jacquet-Langlands transfer between different quaternionic Shimura varieties (Theorem \ref{thmJacquetLanglandsLA}, Corollary \ref{corJLlaIsGlaTorsor}), using which we give an inductive proof of the classicality theorem for partial de Rham cohomology (Theorem \ref{thmClassicalityCohomoDR}). 

In Section \ref{sectionFontaineOperator}, we define the Fontaine operators, and prove Theorem \ref{IntrothmFontainEqDDbar}.
Then we define the Bernstein-Gelfand-Gelfand-Fontaine complex to be the Koszul complex defined by the Fontaine operators, and show it has classical cohomolgy using Theorem \ref{thmClassicalityCohomoDR}.

In Section \ref{sectionAppplicationClassicality}, we state the theorem on the plectic Lie algebra action (Theorem \ref{thmPadicPlectic}) in a joint work with Lue Pan, and combine the ingredients to finish the proof of Theorem \ref{IntrothmClassicalityCompleteCoho}. 

In Appendix \ref{secUniversalDeterminant}, we will define the \(p\)-adic Hecke algebra \(\hat{\TT}^{S}\) acting on completed cohomology of Hilbert modular varieties, 
and define a \(2\)-dimensional determinant of \(\Gal_{F}\)
valued in \(\hat{\TT}^{S}\) in the sense of \cite{Chenevier2014determinants}. 
The result should be well-known, but we give an account here due to the lack of references.

\subsection{Acknowledgment}
I would like to express my gratitude to my advisor Vincent Pilloni for many invaluable discussions and for his constant encouragement and support, without which this work wouldn't be possible. 
I want to thank Lue Pan for his interest in my work, and for many discussions that lead to essential improvements of the results. 
I would like to thank Johannes Anschütz, Juan Esteban Rodr\'iguez Camargo, Dustin Clausen, Xingzhu Fang, Valentin Hernandez, Pol van Hoften, Dongryul Kim, Arthur-C\'esar Le Bras,
Shizhang Li, Benchao Su, Matteo Tamiozzo, Longke Tang,
Qixiang Wang, Xiangqian Yang, Mingjia Zhang, and Daming Zhou for helpful discussions.
Additional thanks to Lue Pan, Vincent Pilloni, and Benchao Su for many corrections on the early draft.
Part of this work was done during my visits to Princeton University and to Beijing International Center for Mathematical Research,
and I want to thank both institutes and my hosts Lue Pan, Yiwen Ding and Liang Xiao for their hospitality. 
The work was done while the author was a PhD student at Université Paris-Saclay, under the
specific doctoral contract for normaliens (CDSN), and I would like to thank Université Paris-Saclay for its support.

\subsection{Notation}
We will work with solid modules in the sense of \cite{CS19}. To ease the notation, all the tensor products will be taken in the solid sense, which we will simply denote as \(-\hatotimes-\).

We fix \(\bar{\QQ}\subset\CC\).
We fix \(F\) to be
a totally real number field of dimension \(d\),
and we fix an isomorphism \(\iota:\CC\cong\bar{\QQ}_{p}\). Let
\(\Sigma_{\infty}\)
and \(\Sigma_{p}\) be 
the sets of places of \(F\)
over \(\infty\)
and over \(p\) respectively. The isomorphism \(\iota\)
induces a map \(i_{p}:\Sigma_{\infty}\to\Sigma_{p}\),
and for $\p\in\Sigma_{p}$,
define \(\Sigma_{\infty/\p}:=i_{p}^{-1}(\p)\subset \Sigma_{\infty}\). 

For any number field \(K\), 
 \(\hat{\mO}_{K}\) denotes the profinite completion of 
\(\mO_{K}\), \(\hat{\mO}_{K,(p)}:=\hat{\mO}_{K}\otimes_{\ZZ}\ZZ_{(p)}\), and \(\mA_{K,f}:=\hat{\mO}_{K}\otimes \QQ\). If \(K=\QQ\), we write \(\hat{\ZZ}\), \(\hat{\ZZ}_{(p)}\) and \(\mA_{f}\) respectively. 

Without defining otherwise, \(L\subset \bar{\QQ}_{p}\) is a sufficiently large finite extension of \(\QQ_{p}\), which is the coefficient field of our Galois representations, and contains all the embeddings of \(F\). 

We will sometimes fix a neat tame level \(K^{p}\subset G(\mA_{f}^{p})\). In that case, we 
will further fix \(S\) to be a finite set of places of \(F\),
	including the places above \(p\) and \(\infty\)
	and all the ramified places of \(K^{p}\).
	Denote by \(\TT^{S}\) the spherical Hecke algebra over \(\ZZ_{p}\) generated by the spherical Hecke operators of \(G\) at finite places away from \(S\). 

We fix an algebraic closure \(\bar{\QQ}\) of \(\QQ\). For each place \(l\) of \(\QQ\), we fix an algebraic closure \(\bar{\QQ}_{l}\) of \(\QQ_{l}\), and an embedding \(\bar{\QQ}\subset\bar{\QQ}_{l}\). This determines an embedding \(\Gal_{\QQ_{l}}\hookrightarrow\Gal_{\QQ}\). We fix a similar set of data for \(F\) instead of \(\QQ\). Denote by \(\CC_{p}\) the completion of \(\bar{\QQ}_{p}\) for the \(p\)-adic topology.
By Frobenius, we will always refer to the geometric Frobenius. 
Let \(\chi_{\mrm{cycl}}:\Gal_{\QQ}\to\ZZ_{p}^{\times}\)
denote the cyclotomic character, sending \(\Fr_{l}\) to \(l^{-1}\). By our convention, \(\chi_{\mrm{cycl}}\)
has Hodge-Tate weight \(-1\). Our convention for local class field theory takes the uniformizer to a lift of the geometric Frobenius. 

For a Shimura datum \((G,X)\), and for a neat open compact subgroup \(K\subset G(\mA_{f})\), we denote by \(\mrm{Sh}_{K}(G)\) the Shimura variety defined over the reflex field \(E(G,X)\subset \CC\). Then we embed \(E(G,X)\) into \(\CC_{p}\) via \(\iota\), and we denote by \(\mS_{K}(G)\) the rigid analytification of the scheme \(\mrm{Sh}_{K}(G)\times_{E(G,X)}\CC_{p}\). The tower \(\{\mS_{K}(G)\}\)
carries a \emph{right} action of \(G(\mA_{f})\).





\section{Perfectoid Jacquet-Langlands correspondence}\label{sectionGeoSV}

In this section, we will prove a product formula (Theorem \ref{thmProductFormulaQuate}), which relates the geometry of different quaternionic Shimura varieties. We first learn a result of this type in \cite{TianXiao2016goren}. The result loc. cit. is concerned with the special fibers, while our result is about the rigid generic fibers. In particular, our results are not conditional on the unramifiedness of \(p\) in \(F\),
and our results hold on the infinite level at \(p\). 

Roughly, quaternionic Shimura varieties are related to certain moduli problems of abelian schemes with PEL structures, and different choices of the quaternion algebras give rise to different Kottwitz signature conditions. 
The proof strategy of \cite{TianXiao2016goren} consists of constructing morphisms between different quaternionic Shimura varieties by constructing explicit \( p \)-quasi-isogenies to alter the signatures. Nowadays, this procedure can be streamlined using the Igusa stacks, constructed in \cite{Zhang2023pel}, \cite{DanielsHoftenKimZhang2024igusa}, \cite{Kim2025uniquenessfunctorialityigusastacks}, \cite{DanielVanHoftenKimZhangs2026igusastackscohomologyshimura},
which are vaguely moduli problems of abelian schemes with PEL structures \emph{up to \( p \)-quasi-isogenies}. 

Hence, our proof strategy is to show that all the relevant quaternionic Shimura varieties have the same Igusa stack; more precisely, we will show that the Igusa stacks of various quaternionic Shimura varieties admit open immersions into the Igusa stack of Hilbert modular varieties. 

This fact is well expected by experts, and it could be proven using the relation of the Igusa stacks with the canonical models shown in \cite{DanielsHoftenKimZhang2024igusa}, and by constructing certain ``exotic correspondences'' over the canonical models. Examples of such results are in \cite{XiaoZhu2017cycles}.

Our proof is more elementary. In fact, our method is based on the following simple observation: for Shimura varieties of PEL type, whose Igusa stacks are explicitly constructed in \cite{Zhang2023pel}, the Kottwitz condition does not play a role in the definition, since it is already remembered by the part of \( \Fl_{G,\mu} \) in the Zhang's Cartesian diagram. Thus we give a modified definition of the Igusa stack by simply forgetting the Kottwitz signature condition, and then we show that the desired Zhang's Cartesian diagram still holds.

Apart from an application to the product formula (Theorem \ref{thmProductFormulaQuate}),
the identification of the Igusa stacks also allow additional flexibility in relating \( \ell \)-adic cohomology of different quaternionic Shimura varieties. We will show in \S \ref{subsecAppliLadicJLcorr} how to realize the completed cohomolgy of lower dimensional 
Shimura varieties as a \emph{direct summand} as
 that of higher dimensional 
Shimura varieties, which should be of some independent interest. This result will be used later for proving promodularity.

The section is structured as follows: in \S \ref{subsecGeneralIgusa}, we introduce the general notion of Igusa stacks following \cite{Kim2025uniquenessfunctorialityigusastacks}. 
In \S \ref{subsectionQuaternUniSV},
we introduce the Shimura data that define the quaternionic Shimura varieties and the unitary Shimura varieties. In \S \ref{subsectionGoodRedLoc}
we define the ``good reduction locus", over which we are going to define the reduction map to the Igusa stack. In \S \ref{subsecIgusaStack}, we construct an Igusa stack of PEL type and prove a Cartesian diagram in our setting. We follow the original idea of \cite{Zhang2023pel} with a slight modification, such that the comparison of the resulting Igusa stack between different Shimura varieties becomes obvious (see Remark \ref{rmkDifferentFromZhang}). 
We then construct Igusa stacks for the unitary Shimura varieties in \S \ref{subsecUnitaryIgusa}, and for the quaternionic Shimura varieties in \S \ref{subsecQuaternionicIgusa}.
As an application, we show in \S \ref{subsecAppliLadicJLcorr} a version of \( \ell \)-adic Jacquet-Langlands correspondence, relating the completed cohomology of different quaternionic Shimura varieties (Corollary \ref{corDirectSummandVaryL}).
In \S \ref{subsecProductFormula}, via analyzing the local 
data, we show a version of the product formula (Theorem \ref{thmProductFormulaQuate}) involving different quaternionic Shimura varieties and basic local Shimura varieties. 

\subsection{General formalism of Igusa stacks}\label{subsecGeneralIgusa}
\begin{notation}
Let \(\Perf\) denote the category of perfectoid spaces over \(\FF_{p}\), equipped with the \(v\)-topology (\cite{Scholze2017etaleDiamond}).
\end{notation}
\begin{notation}\label{notationGenealShimura}
Let \((G,X)\) be a general Shimura datum, and \(E\subset\CC\) be the reflex field. Let \(\mS(G)_{E}:=\varprojlim_{K}(\mrm{Sh}_{K}(G)_{E_{\p}})^{\diamond}\) as a diamond over \(\Spd(E_{\p})\), where \(E_{\p}\) denotes \(p\)-adic completion of \(E\) along \(E\subset\CC\cong\bar{\QQ}_{p}\). Then \(\mS(G)_{E}\)
is equipped with an action of \(G(\mA_{f})\). We will denote by \(\mS(G)\)
its base change along \(\Spd(\CC_{p})\to\Spd(E_{\p})\). 

We denote \(\Fl_{G,\mu,E}:=G/P_{\mu}\), and \(\Fl_{G,\mu,E}^{\mrm{std}}:=G/\bar{P}_{\mu}\). Then both \(\Fl_{G,\mu,E}\) and \(\Fl_{G,\mu,E}^{\mrm{std}}\) descend to schemes over \(E\).
We denote by the same notations the diamonds over \(\Spd(E_{\p})\) that they represent. 
We will denote by \(\Fl_{G,\mu}\) and \(\Fl_{G,\mu}^{\mrm{std}}\)
the corresponding diamonds given by pulling back to \(\Spd(\CC_{p})\).  
They are equipped with a \(G(\QQ_{p})\)-action, which we inflate to an action of \(G(\mA_{f})\)
along the projection \(G(\mA_{f})\to G(\QQ_{p})\).

Then there is a \(G(\mA_{f})\)-equivariant \emph{Hodge-Tate period map} (\cite{Scholze15}, followed by \cite{CS17}, \cite{Shen2017perfectoid}, \cite{HansenJohansson2020perfectoid}, \cite{Hansen2016period}) 
\[\pi_{\HT}:\mS(G)_{E}\to \Fl_{G,\mu,E}. 
\]
\end{notation}


\begin{dfn}[{\cite[Definition III.1.2]{FarguesScholze2021geometrization}}]
For any reductive group \(G\) over \(\QQ_{p}\), we define \(\Bun_{G}\) to be the \(v\)-stack over \(\Perf\)
given by \[\Bun_{G}(S):=\{\text{\'etale \(G\)-torsors over \(\mX_{\rF}(S)\)}\}.\]
We write \(\mcal{G}_{0}\in \Bun_{G}(*)\) for the trivial \(G\)-torsor.
\end{dfn}
\begin{dfn}[Beauville-Laszlo morphism]\label{dfnBLmap}
By \cite[Theorem 3.4.5]{CS17},
we identify \[\Fl_{G,\mu}\cong \mrm{Gr}^{\BdR^{+}}_{G,\mu}\times \Spd(\CC_{p}),
\] where the RHS represents the moduli problem sending \(S\in\Perf\)
to the groupoid of pairs \( (S^{\sharp},i) \), where \( S^{\sharp} \) is an untilt of \( S \) over \( \CC_{p} \), and \( i \) is a modification \(\mcal{G}_{0}\to \mcal{G}\) of \(G\)-torsors over \(\mX_{\rF}(S)\) along \( S^{\sharp} \)
bounded by \(\mu\). 
The action of \(G(\QQ_{p})\) on \(\Fl_{G,\mu}\)
is induced by its action on \(\mcal{G}_{0}\).

Then we define the Beauville-Laszlo morphism \[\BL=\BL_{G,\mu}:\Fl_{G,\mu}\to [\Fl_{G,\mu}/G(\QQ_{p})]\to \Bun_{G,\mu},
\] sending \(\mcal{G}_{0}\to \mcal{G}\)
to \(\mcal{G}\). 
\end{dfn}
\begin{dfn}[{\cite[Lecture 24]{ScholzeWeinstein2020berkeley}}] Given \(b\in B(G,\mu)\), let \(\mcal{G}_{b}\) be the corresponding \(G\)-torsor in \(\Bun_{G}(*)\), and define 
\(\mcal{M}_{G,b,\mu}\) to be the \(v\)-stack sending \(S\in \Perf\) to the groupoid of modifications from \(\mcal{G}_{0}\) to \(\mcal{G}_{b}\) that is bounded by \(\mu\). Then \(\mcal{M}_{G,b,\mu}\) carries an acion of \(G(\QQ_{p})\times \ti{G}_{b}\), and the diagonal \(Z(\QQ_{p})\hookrightarrow G(\QQ_{p})\times\ti{G}_{b}\) acts trivially. Here \(Z\) denotes the center of \(G\).
\end{dfn}
\begin{dfn}[{\cite[Definition 5.5]{Kim2025uniquenessfunctorialityigusastacks}}]\label{dfnGeneralIgus}
Let \(U\) be an open subspace of \(\mS(G)_{E}\) that is closed under the action of \(G(\mA_{f})\). An \emph{Igusa stack} \(\Igs^{U}\) for \(U\) is a datum consisting of 

(1) A \(v\)-stack \(\Igs^{U}\) equipped with a \(G(\mA_{f}^{p})\)-action;

(2) A \(G(\mA_{f})\) -equivariant morphism (``\emph{the reduction map}'') \[
\red:U\to \Igs^{U}\]
where the \(G(\mA_{f})\)-action on \(\Igs^{U}\) is induced by the projection \(G(\mA_{f})\to G(\mA_{f}^{p})\);

(3)
A morphism (``\emph{the reduced Hodge-Tate period map}'') \[\bar{\pi}_{\HT}:[\Igs^{U}/G(\mA_{f}^{p})]\to \Bun_{G,\mu};\] 

(4) A coherent Cartesian diagram \[
\begin{tikzcd}
{[U/G(\mA_{f})] \arrow[r,"\pi_{\HT}"]}\arrow[d,"\red"]
& {[\Fl_{G,\mu,E}/G(\QQ_{p})]}\arrow[d,"\BL"]\\
{[\Igs^{U}/G(\mA_{f}^{p})]}
\arrow[r,"\bar{\pi}_{\HT}"]
& \Bun_{G,\mu},
\end{tikzcd},
\] 
and the datum satisfies the \emph{condition} that the absolute Frobenius
and the identity coincide in \(\pi_{0}(\End(\Igs^{U}))\).
\end{dfn}
\begin{rmk}
Later we will also use this terminology for \(\mS(G):=\mS(G)_{E}\times_{\Spd(E_{\p})}\Spd(\CC_{p})\),
where \(\Fl_{G,\mu,E}\)
in the Cartesian diagram is replaced by \(\Fl_{G,\mu}:=\Fl_{G,\mu,E}\times_{\Spd(E_{\p})}\Spd(\CC_{p})\).
\end{rmk}
This definition has
the benefit that \(\Igs^{U}\) (if exists)
exists uniquely up to a unique isomorphism by \cite[Theorem 10.13]{Kim2025uniquenessfunctorialityigusastacks}.
\begin{rmk}\label{rmkIgsToGlobalUniformization}
Since \(\BL\) is a \(v\)-surjection by \cite[Proposition III.3.1]{FarguesScholze2021geometrization}, \(\red:U\to \Igs^{U}\)
is also a \(v\)-surjection by Definition \ref{dfnGeneralIgus} (4). It thus suffices to understand the groupoid object defined by \[
U\times_{\Igs^{U}}U\cong U\times_{\Bun_{G,\mu}}\Fl_{G,\mu},
\] which carries an action of \(G(\mA_{f}^{p})\times G(\QQ_{p})\times G(\QQ_{p})\). Note that \(\BL\) is \(0\)-truncated, and \(\pr_{i}:U\times_{\Igs^{U}}U\to U\) (\(i=1,2\))
are also \(0\)-truncated. In particular, \(U\times_{\Igs^{U}}U\) is \(0\)-truncated. 
\end{rmk}

We now prove some technical lemmas that will be needed later. 
\begin{lem}\label{lemCenterFactorization}
Assume that we have an Igusa stack \(\Igs^{U}\) as in Definition \ref{dfnGeneralIgus}.
Consider the action of \[\overline{Z(\QQ)}\subset G(\mA_{f})\xhookrightarrow{1_{\mA_{f}^{p}}\times\Delta_{\QQ_{p}}}G(\mA_{f}^{p})\times G(\QQ_{p})\times G(\QQ_{p})\times\cdots\times G(\QQ_{p}).\] Then the diagonal action of \(\overline{Z(\QQ)}\)
on \(U\times_{\Igs^{U}}U\times_{\Igs^{U}}\cdots\times_{\Igs^{U}}U\)
is trivial.
\end{lem}
\begin{proof}
Since \(U\times_{\Igs^{U}}U\cong U\times_{\Bun_{G}}\Fl_{G,\mu}\) is 0-truncated,
it suffices to show that the action of \({Z(\QQ)}\)
is trivial on \(\Spd(C,C^{+})\)-points, where \(C\)
is a complete algebraically closed field over \(\FF_{p}\). In particular, we can fix \(b\in B(G,\mu)\), and work over \(\Bun_{G}^{b}\cong [*/\ti{G}_{b}]\) (\cite[Proposition III.5.3]{FarguesScholze2021geometrization}). 

Denote \(\iota_{b}:\Bun_{G}^{b}\hookrightarrow \Bun_{G}\), and \(f_{b}:*\to [*/\ti{G}_{b}]\hookrightarrow \Bun_{G,\mu}\),
and \(\Igs^{b}:=f_{b}^{*}(\Igs^{U})\), which is also \(0\)-truncated by Definition \ref{dfnGeneralIgus} (4), and carries an action of \(\ti{G}_{b}\times G(\mA_{f})\). We also denote \(U^{b}:=f_{b}^{*}(U))\subset U\).
Then it follows from Definition \ref{dfnGeneralIgus} (4)
that \[U^{b}\cong (\Igs^{b}\times \mcal{M}_{G,b,\mu})/\ti{G}_{b}.
\] 

The action of \(\overline{Z(\QQ)}\) is trivial on \(U\) by considering the archimedean uniformization, and thus it also acts trivially on \(U^{b}\). 
Since the diagonal action of \(Z(\QQ_{p})\) is trivial on \(\mcal{M}_{G,b,\mu}\),
we know that the diagonal action of \(\overline{Z(\QQ)}\hookrightarrow G(\mA_{f}^{p})\times \ti{G}_{b}\) is also trivial on \(\Igs^{b}\). 

Now \[\iota_{b}^{*}(U\times_{\Igs^{U}}U\times_{\Igs^{U}}\cdots\times_{\Igs^{U}}U)\cong (\Igs^{b}\times \mcal{M}_{G,b,\mu}^{i})/\ti{G}_{b}. \]
By the argument above, the diagonal action of \(\overline{Z(\QQ)}\) that combines the actions of \(\ti{G}_{b}\), \(G(\QQ_{p})\) and \(G(\mA_{f}^{p})\)
is trivial on the RHS. Since we are quotienting out the diagonal action of \(\ti{G}_{b}\), the diagonal action of \(\overline{Z(\QQ)}\) that combines the actions of \(G(\QQ_{p})\) and \(G(\mA_{f}^{p})\)
is trivial, as desired.
\end{proof}
\begin{cor}\label{corFunctorCenter}
The action of \(G(\mA_{f})\)
on \(\Igs^{U}\)
factors naturally through an action of \(G(\mA_{f})/\overline{Z(\QQ)}\), such that \(\red\)
is \(G(\mA_{f})/\overline{Z(\QQ)}\)-equivariant. 

Moreover, this action is functorial in terms of the Shimura data in the sense that for \((G,X)\to (G',X')\)
a morphism of Shimura data as in \cite[Theorem 10.13]{Kim2025uniquenessfunctorialityigusastacks}, 
\(\Igs^{U}\to \Igs^{U'}\)
is \(G(\mA_{f})/\overline{Z(\QQ)}\)-equivariant.
\end{cor}
\begin{proof}
By Remark \ref{rmkIgsToGlobalUniformization},
\(\Igs^{U}\)
is the simplicial colimit of \(U^{i/(\Igs^{U})}\). The simplicial system of \(G(\mA_{f})\)-equivariant, and on each term, there is a natural factorization of the \(G(\mA_{f})\)-action into a \(G(\mA_{f})/\overline{Z(\QQ)}\)-action by Lemma \ref{lemCenterFactorization}. 
By taking colimit, we obtain the desired \(G(\mA_{f})/\overline{Z(\QQ)}\) on \(\Igs^{U}\). The functoriality follows from \cite[Theorem 10.13]{Kim2025uniquenessfunctorialityigusastacks} and the functoriality of the simplicial system \(U^{\bullet/\Igs^{U}}\).
\end{proof}

\begin{notation}
Let \(G\) be a reductive group over \(\QQ\). We define \(\mscr{A}_{G}\) as \[
\mscr{A}_{G}:=G(\mA_{f})/\overline{Z(\QQ)}*^{G(\QQ)_{+}} G^{\mrm{ad}}(\QQ)^{+}:=(G(\mA_{f})/\overline{Z(\QQ)}\rtimes G^{\mrm{ad}}(\QQ)^{+})/(G(\QQ)_{+}/\ZZ(\QQ)), 
\] where \(G^{\mrm{ad}}(\QQ)^{+}:=G^{\mrm{ad}}(\QQ)\cap G^{\mrm{ad}}(\RR)^{+}\), \(G(\QQ)_{+}\) denotes the preimage of \(G^{\mrm{ad}}(\QQ)^{+}\) along \(G(\QQ)\to G^{\mrm{ad}}(\QQ)\), and
 \(G^{\mrm{ad}}(\QQ)\) acts on \(G(\mA_{f})\) by conjugation, and \[G(\QQ)_{+}/\ZZ(\QQ)\hookrightarrow G(\mA_{f})/\overline{Z(\QQ)}\rtimes G^{\mrm{ad}}(\QQ)^{+}\]
is given by \(g\mapsto (g,g^{-1})\). Note that if \(H^{1}(\QQ,\ZZ)\cong 0\), then \(\mscr{A}_{G}\cong G(\mA_{f})/\overline{Z(\QQ)}\). 
\end{notation}
\begin{prop}\label{propRelateTwoIgusa}
Let \((G_{1},X_{1})\hookrightarrow (G_{2},X_{2})\) be an embedding of Shimura data that induces \(G_{1}^{\mrm{ad}}\cong G_{2}^{\mrm{ad}}\). Let \(Z_{i}\) denote the center of \(G_{i}\). We assume further that \(G_{i}(\QQ)\to G_{i}^{\mrm{ad}}(\QQ)\) is surjective for \(i=1,2\), so \(\mscr{A}_{G_{i}}\cong G_{i}(\mA_{f})/\overline{Z_{i}(\QQ)}\).

 For \(i=1,2\), let \(U_{i}\subset \mS(G_{i})\) be as in Definition \ref{dfnGeneralIgus}, and assume that it has an Igusa stack \(\Igs^{U_{i}}\). 
The embedding of Shimura data induces \(\mS(G_{1})\hookrightarrow\mS(G_{2})\), and we assume that \(U_{1}=U_{2}\cap \mS(G_{1})\).

Then we have a natural isomorphism \[
\Igs^{U_{2}}\times_{\Bun_{G_{2}},\mu_{2}}\Bun_{G_{1},\mu_{1}}\cong \left[(\Igs^{U_{1}}\times \mscr{A}_{G_{2}})/\mscr{A}_{G_{1}}\right],
\] such that we have a Cartesian diagram \[
\begin{tikzcd}
U_{1}\arrow[r,hook]\arrow[d,"\red_{1}"]
& U_{2}\arrow[d,"\red_{2}"]\\
\Igs^{U_{1}}
\arrow[r,hook] & \Igs^{U_{2}}\times_{\Bun_{G_{2}},\mu_{2}}\Bun_{G_{1},\mu_{1}}.
\end{tikzcd}
\]
\end{prop}
\begin{proof}
The Cartesian diagram follows from \cite[Theorem 10.13]{Kim2025uniquenessfunctorialityigusastacks}, by noting that \(\BL_{2}:\Fl_{G_{2},\mu_{2}}\to \Bun_{G_{2},\mu_{2}}\)
factors naturally through \(\Bun_{G_{1},\mu_{1}}\). By \cite[\S 2]{Deligne1979varietesInterpretationModulaire}, \begin{align}\label{alignRelateTwoSV}
[(U_{1}\times \mscr{A}_{G_{2}})/\mscr{A}_{G_{1}}]\cong U_{2} .
\end{align} 
Moreover, by Corollary \ref{corFunctorCenter},
there is a natural map \[
\left[(\Igs^{U_{1}}\times \mscr{A}_{G_{2}})/\mscr{A}_{G_{1}}\right]\to \Igs^{U_{2}}\times_{\Bun_{G_{2}},\mu_{2}}\Bun_{G_{1},\mu_{1}},\] 
that is compatible with (\ref{alignRelateTwoSV}). Thus it 
is an isomorphism, since \(\red_{2}\) is a \(v\)-surjection. 
\end{proof}

\subsection{Quaternionic \& Unitary Shimura varieties}\label{subsectionQuaternUniSV}

\subsubsection{Quaternionic Shimura varieties}\label{subsubQuaternion}
We fix \(F\) to be
a totally real number field of dimension \(d\).
We  denote by \(\Sigma_{l}\) the set of places of \(F\) over \(l\), where \(l\) is a (finite or infinite) place of \(\QQ\). We fix an abstract isomorphism \(\iota:\bar{\QQ}_{p}\cong \CC\).

For any subset $T\subset \Sigma_{\infty}$, define 
\(r_{\p,T}:=\#(\Sigma_{\infty/\p}\cap T)\),
and \(T_{p}:=\{\p\in\Sigma_{p}:2\nmid r_{\p,T}\}\).
Define \(B_{T}\)
as the quaternion algebra over \(F\) that ramifies precisely at $T\cup T_{p}$. Note that by construction, this is an even set, so \(B_{T}\)
is well-defined. We put \(G_{T}:=\mrm{Res}_{F/\QQ}B^{\times}_{T}\).
For example,
we have $B_{\emptyset}\cong \mrm{Mat}_{2\times 2,F}$ and $G_{\emptyset}\cong G=\mrm{Res}_{F/\QQ}\GL_{2,F}$.
We define \[h_{T}:\mbb{S}\to G_{T,\RR}\cong \prod_{\tau\in T^{c}}\GL_{2,\RR}\times \prod_{\tau\in T}\mbb{H}^{\times},\;x+y\mi\mapsto \left((\begin{pmatrix}
	x&y\\-y&x
\end{pmatrix})_{\tau\in T^{c}},(1)_{\tau\in T}\right),\]
where \(T^{c}:=\Sigma_{\infty}\backslash T\).
Then \((G_{T},h_{T})\)
defines a Shimura datum when \(T\ne \Sigma_{\infty}\), and a weak Shimura datum when \(T=\Sigma_{\infty}\). 

Note that $G_{T,\mA^{p}_{f}}$
is independent of $T$, so we fix such isomorphisms \(G_{T,\mA^{p}_{f}}\cong G_{\mA^{p}_{f}}\). We also fix a neat open compact 
subgroup $K^{p}\subset G(\mA^{\infty,p})$, and 
for any open compact subgroup $K_{p}\subset G_{T}(\QQ_{p})$,
we have the \emph{quaternionic Shimura variety}
\(\mrm{Sh}_{\Kpp}(G_{T})\), which is a quasi-projective smooth variety of dimension \(d-\#T\)
 defined over a reflex field \(E_{0,T}\subset \CC\).

We base change \(\mrm{Sh}_{\Kpp}(G_{T})\) from \(E_{0,T}\)
to \(\CC_{p}\) using \(\iota\), and denote by \(\mS_{\Kpp}(G_{T})\) the associated adic space over \(\Spa(\CC_{p},\mO_{\CC_{p}})\). 

Note that \(\mrm{Sh}_{\Kpp}(G_{\emptyset})\)
is precisely the Hilbert modular variety. Also note that \(\mS_{\Kpp}(G_{T})\)
is proper when \(T\ne \emptyset\). 

As in Notation \ref{notationGenealShimura},
\((G_{T},h_{T})\) determines a Hodge cocharacter 
\(\mu_{T}\), which in turn determines a flag variety \(\FLT{T}\).
Note that there is a canonical decomposition \(G_{T,\CC_{p}}\cong \prod_{\tau\in\Sigma_{\infty}}
G_{T,\tau}\),
and after fixing such a non-canonical isomorphism \(G_{T,\tau}\cong\GL_{2,\tau}\),
we can identify the parabolic \(P_{\mu_{T}}\) defined by \(\mu_{T}\) with 
\(\prod_{\tau\in T^{c}}B_{\tau}\times \prod_{\tau\in T}\GL_{2,\tau}\),
and then \[\FLT{T}\cong (P_{\mu_{T}}\backslash G_{T,\CC_{p}})^{\an}\cong \prod_{\tau\in T^{c}}(B_{\tau}\backslash G_{T,\tau})^{\an}.\]

\begin{notation}
Denote for \(\tau\in T^{c}\), \(\Fl_{G_{T},\mu_{T},\tau}:=\Fl_{G_{T},\mu_{\{\tau\}^{c}}}:=(B_{\tau}\backslash G_{T,\tau})^{\an}\), then \(\FLT{T}\cong \prod_{\tau\in T^{c}}\Fl_{G_{T},\mu_{T},\tau}\), and (non-canonically), \(\Fl_{G_{T},\mu_{T},\tau}\cong \mP^{1}_{\CC_{p}}\).

For any \(I\subset T^{c}\),
\[\Fl_{G_{T},\mu_{T},I}:=\FLTT{T}{I^{c}}:=\prod_{\tau\in I}\Fl_{G_{T},\mu_{T},\tau}.\]
\end{notation}

The Shimura varieties \(\SGTK{T}{K}\) are of abelian type. By \cite{Scholze15}, followed by \cite{CS17}, \cite{Shen2017perfectoid}, \cite{HansenJohansson2020perfectoid}, we know that there exists a unique perfectoid Shimura variety 
\(\mS_{K^{p}}(G_{T})\) over \(\Spa(\CC_{p},\mO_{\CC_{p}})\)
such that \[\mS_{K^{p}}(G_{T})\sim \varprojlim_{K_{p}}\mS_{\Kpp}(G_{T}),\]
equipped with a Hodge-Tate period map
\[\pi_{\HT,T}:\mS_{K^{p}}(G_{T})\to \FLT{T}.\]

\subsubsection{Unitary Shimura varieties}
The quaternionic Shimura varieties \(\mS_{\Kpp}(G_{T})\) are of abelian type. We can construct the following related class of Shimura varieties. We learn the following construction from \cite{TianXiao2016goren}.

Let $E/F$ be a CM extension, that is \emph{inert} over all $\p\in\Sigma_{p}$. Let \(\Sigma_{E,\infty}:=\{\tilde{\tau}:E\hookrightarrow\CC\}\),
then there is a 2-to-1 morphism \(r_{E/F}:\Sigma_{\infty,E}\to \Sigma_{\infty}\)
given by restriction. We fix two liftings $i,i^{c}:\Sigma_{\infty}\to \Sigma_{E,\infty}$ such that for 
any \(\tau\in\Sigma_{\infty}\),
\(r_{E/F}^{-1}(\tau)=\{i(\tau),i^{c}(\tau)\}\).
Denote \(\tilde{T}:=i(T)\subset \Sigma_{E,\infty}\).

Given such \(\tilde{T}\), 
we define \(T_{E,\tilde{T}}:=\mrm{Res}_{E/\QQ}\GG_{m}\),
and we have an isomorphism \(T_{E,\tilde{T},\RR}\cong \prod_{\tilde{\tau}\in i(\Sigma_{\infty})}\mbb{S}\), where the last isomorphism is induced by \(\tilde{\tau}\in i(\Sigma_{\infty})\).
Then we define \[h_{E,\tilde{T}}:\mbb{S}\to T_{E,\tilde{T},\mbb{R}},z\mapsto (z_{E,\tilde{\tau}})_{\tilde{\tau}\in i(\Sigma_{\infty})},\]
where \(z_{E,\tilde{\tau}}=\tilde{\tau}(z)\)
if \(\tilde{\tau}\in \tilde{T}\)
and \(z_{E,\tilde{\tau}}=1\)
if \(\tilde{\tau}\notin \tilde{T}\).
Then 
\((T_{E,\tilde{T}},h_{E,\tilde{T}})\)
defines a Shimura set 
\(\mrm{Sh}_{K_{E}}(T_{E,\tilde{T}})\) for neat \(K_{E}\subset T_{E,\tilde{T}}(\mA)\). 

Now we consider \(G'_{\ti{T}}:=(G_{T}\times T_{E,\tilde{T}})/T_{F}\), where $T_{F}$ acts on \(G_{T}\times T_{E,\tilde{T}}\) via \(t\cdot (g,z)=(gt^{-1},tz)\),
where \(T_{F}:=\mrm{Res}_{F/\QQ}\GG_{m}\),
and consider \[h_{\ti{T}}:\mbb{S}\xrightarrow{(h_{T},h_{E,\tilde{T}})}G_{T,\RR}\times T_{E,\tilde{T}}\to G'_{\ti{T}},  \]
then \((G'_{\ti{T}},h_{\ti{T}})\)
also gives rise to Shimura varieties \(\mrm{Sh}_{U}(G'_{\ti{T}})\) of dimension \(\#T\) for any neat \(U\subset G'_{\ti{T}}(\mA_{f})\),
which we will refer to loosely as ``\emph{unitary Shimura varieties}".


We fix a totally negative element \(\mfk{d}\in F\), such that \(E=F(\mfk{d})\).
Now by the choice of \(E\),
we know \(B_{T}\otimes_{F}E\)
splits over \(E\), we denote 
\(D_{T}:=B_{T}\otimes_{F}E\), and fix such an isomorphism \[i_{T}:D_{T}\cong M_{2}(E),\]
where \(M_{2}(-)\)
refers to the algebra of \(2\times 2\)-matrices.
\(D_{T}\) has a natural involution, which we denoted as \(l\mapsto \bar{l}\). Take some \(\delta_{T}\in (D_{T}^{\mrm{sym}})^{\times},\)
and define an involution \(l^{*_{T}}:=\delta_{T}^{-1}\bar{l}\delta_{T}\),
then \(l\mapsto l^{*_{T}}\)
is a new involution on \(D_{T}\). By \cite[Lemma 5.4]{TianXiao2016goren}, we can choose them properly so that we have an identification \((D_{T},*_{T})\cong (D_{\emptyset},*_{\emptyset})\). We will denote them simply as \((D,*)\).

Let \(V:=D_{T}\) be the free \(D_{T}\)-module of rank 1, and define \[\psi_{T,E}:V\times V\to E,\;\psi_{T,E}(v,w):=\Tr_{D_{T}/E}(\sqrt{\mfk{d}}v\delta_{T} w^{*})=\Tr_{D_{T}/E}(\sqrt{\mfk{d}}v\bar{w}\delta_{T}),\]
then \(\psi_{T,E}\)
is skew-Hermitian in the sense that \(\overline{\psi_{E}(v,w)}=-\psi_{E}(w,v)\)
and \(\psi_{E}(lv,w)=\psi_{E}(v,l^{*}w)\). Let us define \[\psi_{T,F}:=\mrm{Tr}_{E/F}\circ \psi_{T,E}:V\times V\to F;\] Then for any \(\QQ\)-algebra \(R\), 
\begin{align}\label{alignTheGrpIsSimilit}
	G'_{\ti{T}}(R)\cong \left\{
	(l\in (D_{T}\otimes_{\QQ} R)^{\times},c(l)\in (F\otimes_{\QQ}R)^{\times}):
	\psi_{T}(vl,wl)=c(l)\psi_{T}(v,w),\forall v,w\in V
\right\},
\end{align}
where the map is given by \(l\mapsto (l^{-1},\nu(l)^{-1})\).

Note that we allow \(c(l)\in (F\otimes R)^{\times}\),
so this still gives Shimura varieties of abelian type.

Fixing a neat open compact subgroup \(U^{p}\subset G'_{\ti{T}}(\mA^{p}_{f})\),
we know that the diamond
\(\varprojlim_{U_{p}}\mS_{\Upp}(G'_{\ti{T}})\)
is represented by a perfectoid space (\cite{HansenJohansson2020perfectoid})
which we denote as \(\mS_{U^{p}}(G'_{\ti{T}})\). We also write \(\mS(G'_{\ti{T}}):=\varprojlim_{U^{p}}\mS_{U^{p}}(G'_{\ti{T}})\)
We have a Hodge-Tate period map \(\pi_{\HT}:\mS_{U^{p}}(G'_{\ti{T}})\to \Fl_{  G_{\ti{T}}',   \mu_{\ti{T}}    }. \)

Note that the natural map \(G_{T}\times T_{E}\rightarrow G'_{\ti{T}}\) induces an isomorphism \(\FLT{T}\cong \Fl_{  G_{\ti{T}}',   \mu_{\ti{T}}    }\), and we have a natural commutative diagram \[\begin{tikzcd}
	\mS_{K^{p}}(G_{T})\times \mS_{T^{p}}(T_{E})\arrow[r]\arrow[d,"\pi_{\HT}"]
	&\mS_{U^{p}}(G'_{\ti{T}})\arrow[d,"\pi_{\HT}"]\\
	\FLT{T}\arrow[r,equal]
	& \Fl_{  G_{\ti{T}}',   \mu_{\ti{T}}    }
\end{tikzcd}\] 
if \(K^{p}\times T^{p}\subset U^{p}\). By standard argument as in \cite[\S 2]{Deligne1979varietesInterpretationModulaire}, we have the following:
\begin{lem}\label{lemUnitaryToQuaternion}
There is a \(G_{\ti{T}}'(\mA_{f})\)-equivariant isomorphism \[(\mS(G_{T})\times \mscr{A}_{G'_{\ti{T}}})/\mscr{A}_{G_{T}}\cong \mS(G'_{\ti{T}}).
\] Moreover, \(\mscr{A}_{G_{T}}\cong G_{T}(\mA_{f})/\overline{F^{\times}}\)
and \(\mscr{A}_{G'_{\ti{T}}}\cong G_{\ti{T}}'(\mA_{f})/\overline{E^{\times}}\).
\end{lem}

\subsubsection{Moduli interpretation}
We fix \(\mO_{D_{T}}\subset D_{T}\)
an order that is stable under 
\(*\), and \(\Lambda\subset V\)
an \(\mO_{D_{T}}\)-lattice, such that \(\psi_{T}(\hat{\Lambda},\hat{\Lambda})\subset \hat{\mO}_{F}\),
where \(\hat{\Lambda}:=\Lambda\otimes_{\ZZ}\hat{\ZZ}\).

We learn the following construction from \cite[\S 3.3]{BoxerCalegariGeePilloni2021abelian}. In the Hilbert case,it is due to {\cite[\S 2.2.2]{DiamondSasakiKassaei2022mod}}. Note that we define the moduli problem over characteristic 0 field, so the level at \(p\) can be arbitrary. 
\begin{prop}[{\cite[\S 2.2.2-2.2.3]{DiamondSasakiKassaei2022mod}}]
\label{propModuliShGpp}
Let \(E_{0,\tilde{T}}\subset \CC\) be the reflex field of \((G'_{\ti{T}},h_{\ti{T}})\). Define Kottwitz signature for \(\ti{T}\): 
for \(\tilde{\tau}\in \Sigma_{E,\infty}\), \(s_{\tilde{\tau},\tilde{T}}=0\) if \(\tilde{\tau}\in\tilde{T}=i(T)\), \(s_{\tilde{\tau},\tilde{T}}=2\) if \(\tilde{\tau}\in i^{c}(T)\), and \(s_{\tilde{\tau},\tilde{T}}=1\) if \(\tilde{\tau}\notin r_{E/F}^{-1}(T)\).

We consider the following moduli problem. Let \(U=U^{p}U_{p}\subset G'_{\ti{T}}(\mA_{f})\) be an neat open compact subgroup such that \(U_{p}\subset G'_{\ti{T}}(\ZZ_{p})\), 
we define an fppf stack \(\widetilde{\mrm{Sh}}_{U}(G'_{\ti{T}})\), sending any \(E_{0,\ti{T}}\)-scheme \(S\) to the groupoid \(\widetilde{\mrm{Sh}}_{U}(G'_{\ti{T}})(S):=\{(A,\iota,\lambda,\tilde{\alpha}_{U})\}\), where 

(1) \(A\) is an abelian scheme over \(S\)
of dimension \(4d\);

(2) \(\iota:D_{T}\hookrightarrow \End_{S}(A)\otimes\ZZ_{(p)}\), which satisfies 
Kottwitz signature condition for \(\tilde{T}\): for any \(b\in\mO_{E}\),
the action of \(\iota(b)\)
on \(\Lie(A/S)\)
has characteristic polynomial \(\prod_{\tilde{\tau}\in\Sigma_{E,\infty}}(x-\tilde{\tau}(b))^{2s_{\tilde{\tau},\tilde{T}}}\);

(3) \(\lambda:A\to A^{\vee}\) is an \(\mO_{D_{T}}\)-quasi-polarization (in the sense of ``polarisation faible" in \cite[\S 4.4]{Deligne71}) whose induced Rosati involution is compatible with \(*\) on \(\mO_{D_{T}}\);

(4) \(\ti{\alpha}_{U}\)
is a \emph{\(U\)-level structure over \(F\)}, in the sense that \(\ti{\alpha}_{U}\) is a global section of the \'etale sheaf  \[ 
\left\{
	\left(\alpha^{0}:
\Lambda\otimes_{\ZZ}  \hat{\ZZ}_{(p)}\cong (\hat{T}A)_{(p)},\ti{c}\in \hat{\mO}_{F,(p)}^{\times}(1)\right):
\begin{tikzcd}
	\Lambda\otimes \hat{\ZZ}_{(p)}\times\Lambda\otimes \hat{\ZZ}_{(p)}\arrow[r,"\alpha^{0}"]
	\arrow[d,"\psi_{T,F}"]
	& \hat{V}A\times \hat{V}A \arrow[d,"\psi_{F}^{Weil}"] \\ 
	\hat{\mO}_{F,(p)}\arrow[r,"\ti{c}"] & \hat{\mO}_{F,(p)}(1)
\end{tikzcd}
\right\}/U.
\]

The morphisms \(\Hom\left((A_{1},\iota_{1},\lambda_{1},\tilde{\alpha}_{K,1}=(\alpha^{0}_{1},\ti{c}_{1})),(A_{2},\iota_{2},\lambda_{2},\tilde{\alpha}_{K,2}=(\alpha^{0}_{2},\ti{c}_{2}))\right)\) 
are the set of 
prime-to-\(p\) \(D_{T}\)-equivariant quasi-isogenies
\(f:A_{1}\to A_{2}\)
and locally constant \(r:S\to \ZZ_{(p)}^{\times,+}\),
such that \(f^{*}(\lambda_{2})\cong r\lambda_{1}\)
and \(f^{*}(\alpha_{2}^{0},\ti{c}_{2})=(\alpha_{1}^{0},r\ti{c}_{1})\).

Then
\(\widetilde{\mrm{Sh}}_{U}(G'_{\ti{T}})\)
is represented by a scheme, which we also denote by \(\widetilde{\mrm{Sh}}_{U}(G'_{\ti{T}})\).

There is an action of \(\mO_{F,(p)}^{\times,+}/\ZZ_{(p)}^{\times,+}\)
on \(\widetilde{\mrm{Sh}}_{U}(G'_{\ti{T}})\)
given by \(x\in \mO_{F,(p)}^{\times,+}\),
\(x\cdot (A,\iota,\lambda,(\alpha_{1}^{0},\ti{c}_{1})):=(A,\iota,x\lambda,(\alpha_{1}^{0},x\ti{c}_{1}))\),
which factors through \(\Delta_{U}:=\mO_{F,(p)}^{\times,+}/(\ZZ_{(p)}^{\times,+}\Nm_{E/F}(E^{\times}\cap U))\),
and we have an isomorphism \[\mrm{Sh}_{U}(G'_{\ti{T}})\cong [\widetilde{\mrm{Sh}}_{U}(G'_{\ti{T}})/\Delta_{U}].\]

Let \(\mS_{U}(G'_{\ti{T}})\) (resp. \(\ti{\mS}_{U}(G'_{\ti{T}})_{E_{0,\ti{T}}}\)) be the analytification of \(\mrm{Sh}_{U}(G'_{\ti{T}})\times_{E_{0,\tilde{T}}}\CC_{p}\) (resp. of \(\widetilde{\mrm{Sh}}_{U}(G'_{\ti{T}})\times_{E_{0,\ti{T}}}\CC_{p}\)).
Then for any \((R,R^{+})\)
complete Huber pair over \(\Spa(\CC_{p},\mO_{\CC_{p}})\),
we have by definition \[\ti{\mS}_{U}(G'_{\ti{T}})_{E_{0,\ti{T}}}(R,R^{+})\cong \widetilde{\mrm{Sh}}_{U}(G'_{\ti{T}})(R). \]
Moreover, 
there is a free action of \(\Delta_{U}\)
on \(\ti{\mS}_{U}(G'_{\ti{T}})_{E_{0,\ti{T}}}\)
such that \({\mS}_{U}(G'_{\ti{T}})\cong [\ti{\mS}_{U}(G'_{\ti{T}})_{E_{0,\ti{T}}}/\Delta_{U}]\).
\end{prop}
\begin{rmk}\label{rmkActuallyPEL}
\(\widetilde{\mrm{Sh}}_{U}
(G'_{\ti{T}})\)
is the disjoint union of the Shimura varieties associated to \(G^{\prime,*}_{\ti{T}}:=\nu^{-1}(T_{\QQ})\) for \(T_{\QQ}\cong \GG_{m}\hookrightarrow T_{F}\), parametrized by 
\(\ti{c}\pmod{\ZZ_{(p)}^{\times}\nu(U)}\), which are of PEL type.
\end{rmk}

Now that \(\ti{\mS}_{U^{p}}(G'_{\ti{T}}):=\varprojlim_{U_{p}}\ti{\mS}_{\Upp}(G'_{\ti{T}})\) is also perfectoid, since it is pro\'etale over \(\mS_{U^{p}}(G'_{\ti{T}})\).
Moreover, as \(\mS_{\Upp}(G'_{\ti{T}})\cong [\ti{\mS}_{\Upp}(G'_{\ti{T}})/\Delta_{\Upp}]\),we know by passing to limit that \begin{align}\label{alignFromtiStoSQuotientInfi}
	\mS_{U^{p}}(G'_{\ti{T}})\cong [\ti{\mS}_{U^{p}}(G'_{\ti{T}})/\underline{\Delta_{U^{p}}}].
\end{align}
where \(\Delta_{U^{p}}:=\varprojlim_{\Upp}\Delta_{\Upp}
\) is a locally profinite group.

\subsection{Good reduction locus}\label{subsectionGoodRedLoc} 
We consider the following moduli \(Y_{U^{p}}\)
over \(\Spec(\mO_{\CC_{p}})\) given by \[{Y}_{U^{p}}(R^{+}):=\{(A,\iota,\lambda,\alpha_{U^{p}})\},\;\forall R^{+}/\mO_{\CC_{p}},\]
where \(A\) is an abelian scheme over \(R^{+}\) up to prime-to-\(p\) quasi-isogeny,
and \(\iota:\mO_{D_{T}}\hookrightarrow \End_{R^{+}}(A)\otimes\ZZ_{(p)}\),
and \(\lambda:A\to A^{\vee}\) is a prime-to-\(p\) \(\mO_{D_{T}}\)-quasi-polarization, and \(\alpha_{U^{p}}\) is a \(U^{p}\)-level structure 
\[\alpha_{U^{p}}\in \left\{
	(\alpha^{0}:\hat{\Lambda}^{p}\cong \hat{T}^{p}A,\ti{c}\in \mA^{p,\times}_{F,f}(1)):
	\begin{tikzcd}
		(\hat{\Lambda}^{p}\otimes\ZZ_{(p)})\times(\hat{\Lambda}^{p}\otimes\ZZ_{(p)})\arrow[r,"\alpha^{0}"]
		\arrow[d,"\psi_{T,F}"]
		& \hat{V}^{p}A\times \hat{V}^{p}A \arrow[d,"\psi_{F}^{Weil}"] \\ 
		\mA^{p,\times}_{F,f}\arrow[r,"\ti{c}"] & \mA^{p,\times}_{F,f}(1)
	\end{tikzcd}
\right\}/U^{p}.\]

The morphisms \(\Hom\left((A_{1},\iota_{1},\lambda_{1},\tilde{\alpha}_{U^{p},1}=(\alpha^{0}_{1},\ti{c}_{1})),(A_{2},\iota_{2},\lambda_{2},\tilde{\alpha}_{U^{p},2}=(\alpha^{0}_{2},\ti{c}_{2}))\right)\) 
are the set of \((f,r)\)
where \(f:A_{1}\to A_{2}\)
is a \(D_{T}\)-equivariant prime-to-\(p\) quasi-isogeny, and \(r:S\to \ZZ_{(p)}^{\times,+}\) locally constant,
such that \(f^{*}(\lambda_{2})\cong r\lambda_{1}\)
and \(f^{*}(\alpha_{2}^{0},\ti{c}_{2})=(\alpha_{1}^{0},r\ti{c}_{1})\).

Then when \(U^{p}\subset G'_{\ti{T}}(\mA_{f}^{p})\) is neat, \(Y_{U^{p}}\)
is represented by a scheme 
 over \(\mO_{\CC_{p}}\), which we also denote by \(Y_{U^{p}}\).

Then we can consider the analytification of the rigid generic fiber \(\mcal{Y}_{U^{p}}:=(Y_{U^{p},\CC_{p}})^{\an}\) of the scheme \(Y_{U^{p},\CC_{p}}\),
as well as the rigid generic fiber of the associated formal scheme \((Y_{U^{p}})^{\wedge}_{p}\), 
which we denote as \(\mcal{Y}_{U^{p}}^{\circ}\).
By the general theory of rigid generic fibers, we know that we have a natural map \(j:\mcal{Y}_{U^{p}}^{\circ}\hookrightarrow \mcal{Y}_{U^{p}}\), which is an \emph{open immersion}.
Explicitly, \[\mcal{Y}_{U^{p}}(R,R^{+})=Y_{U^{p}}(R)\supset \mcal{Y}^{\circ,pre}_{U^{p}}(R,R^{+}):=Y_{U^{p}}(R^{+}),\] and \(
\mcal{Y}^{\circ}_{U^{p}}
\) is the analytic sheafification of the presheaf \(
\mcal{Y}^{\circ,pre}_{U^{p}}
\).

We have a natural map \(\ti{\mS}_{U}(G'_{\ti{T}})\to \mcal{Y}_{U^{p}}\), given by forgetting the level structure at \(p\); more precisely, given \((A,\iota,\lambda,\ti{\alpha}_{U^{p}})\),
using the level structure at \(p\), we can associate another element \((A',\iota',\lambda',\ti{\alpha}'_{U^{p}})\) in the same prime-to-\(
p
\) quasi-isogeny class such that \(\ti{\alpha}_{U^{p}}=(\alpha^{0},\ti{c})\), where
\(\alpha^{0}\) induces an isomorphism \(\Lambda\otimes\ZZ_{p}\cong T_{p}A\), and  \(\ti{c}'_{p}\in(\mO_{F}\otimes\ZZ_{p})^{\times}\).
\begin{dfn}
	We define the \emph{good reduction locus} for \(\ti{\mS}_{U}(G'_{\ti{T}})\)
as the fiber product \(\ti{\mS}^{\circ}_{U}(G'_{\ti{T}}):=\ti{\mS}_{U}(G'_{\ti{T}})\times_{\mcal{Y}_{U^{p}}}\mcal{Y}^{\circ}_{U^{p}}\).
We see by definition that there is a natural open embedding \(j:\ti{\mS}_{U}^{\circ}(G'_{\ti{T}})\hookrightarrow \ti{\mS}_{U}(G'_{\ti{T}}).\)

It is clear that the locus is closed under the action of \(\mO_{F,(p)}^{\times,+}\),
and we define the \emph{good reduction locus}
for \(\mS_{U}(G'_{\ti{T}})\)
as the open subspace \[j:\mS_{U}^{\circ}(G'_{\ti{T}}):=[\ti{\mS}_{U}(G'_{\ti{T}})/\Delta_{U}]\hookrightarrow \mS_{U}(G'_{\ti{T}}).\]

We also define \[\mS_{U^{p}}^{\circ}(G'_{\ti{T}}):=\varprojlim_{U_{p}}\mS_{U^{p}U_{p}}^{\circ}(G'_{\ti{T}}),\;
\mS^{\circ}(G'_{\ti{T}}):=\varprojlim_{U^{p}}\mS_{U^{p}}^{\circ}(G'_{\ti{T}}).
\] Note that \(\mS^{\circ}(G'_{\ti{T}})\)
is an open subspace of \(\mS(G'_{\ti{T}})\)
that is stable under the action of \(G'_{\ti{T}}(\mA_{f})\).
\end{dfn}

Explicitly, \(\ti{\mS}_{U}^{\circ}(G'_{\ti{T}})\)
is the (analytic) sheafification of the following presheaf \(\ti{\mS}_{U}^{\circ,pre}(G'_{\ti{T}})\): for any Huber pair 
\((R,R^{+})\)
over \((\CC_{p},\mO_{\CC_{p}})\), define
\[\ti{\mS}_{U}^{\circ,pre}(G'_{\ti{T}}):=\{(A,\iota,\lambda,\alpha_{K'})\},\]
where \(A\) is an abelian scheme over \(R^{+}\),
and \(\iota:\mO_{D_{T}}\hookrightarrow \End_{R^{+}}(A)\) satisfies Kottwitz signature condition for \(\tilde{T}\),
and \(\lambda:A\to A^{\vee}\) is an \(\mO_{D_{T}}\)-polarization, and \(\alpha_{U^{p}}\) is a \(U^{p}\)-level structure.


\begin{prop}\label{propGoodReducForNonOrd}
	The natural open embedding \(j:\mS_{U}^{\circ}(G'_{\ti{T}})\hookrightarrow \mS_{U}(G'_{\ti{T}})\)
	is an isomorphism when \(T\ne \emptyset\).
	If \(T=\emptyset\),
	then \(\mS^{\circ}_{U^{p}}(G'_{\emptyset})\)
	contains all the non-ordinary strata, in the sense that \[\mS^{\circ}_{U^{p}}(G'_{\emptyset})\supset \pi_{\HT}^{-1}(\Fl_{  G_{\ti{T}}',   \mu_{\ti{T}}    }\backslash \Fl_{  G_{\ti{T}}',   \mu_{\ti{T}}    }(\QQ_{p})).\]
\end{prop}
\begin{proof}
	Since \(\mS^{\circ}_{U}(G'_{\ti{T}})\) 
	is open in \(\mS_{U}(G'_{\ti{T}})\),
	it is generalizing, so it suffices to check they have the same set of rank \(1\) points, and thus it suffices to check \((C,\mO_{C})\)-points, where \(C\)
	is a complete algebraically closed extension of \(\QQ_{p}\).
	Assume \((A,\iota,\lambda,\alpha_{K'})\in \ti{\mS}_{U}(G'_{\ti{T}})(C,\mO_{C})\backslash \ti{\mS}_{U}^{\circ}(G'_{\ti{T}})(C,\mO_{C})\). 
	Then \(A\)
	is an abelian scheme over \(C\), and we can consider its Néron model \(\mfk{A}\) over \(\Spf\mO_{C}\),
	which lies in a canonical sequence \[0\to \mfk{T}\to \mfk{A}\to \mfk{B}\to 0,\]
	where \(\mfk{T}\)
	is a torus, and \(\mfk{B}\)
	is an abelian scheme over \(\mO_{C}\). We will be done if \(\mfk{T}=0\), as then \(\mfk{A}\)
	is an abelian scheme over \(\mO_{C}\) with all the additional structure.
	
	We now assume \(\mfk{T}\ne 0\).
	Since 
	the Néron model is canonical, we have an action of \(\mO_{D_{T}}\)
	on \(\mfk{A}\) as well as on \(\mfk{T}\).
	We can consider the cocharacter group \(X_{*}(\mfk{T})\), which is a \(\mO_{D_{T}}\)-module.
	Note that by counting the dimension  \(X_{*}(\mfk{T})\)
	is of rank at most \(4d\), while 
	\(\mO_{D_{T}}\)
	is of rank \(8d\).
Then by counting rank, \(X_{*}(\mfk{T})_{\QQ}\)
	is isomorphic to the unique simple \(D_{T}\)-module. This implies that the action of \(\mO_{D_{T}}\) satisfies the Kottwitz signature condition for \(\emptyset\),
	which contradiction the Kottwitz condition for \(\ti{T}\)
	unless \(T=\emptyset\). So we have finished the proof in the case where \(T\ne \emptyset\).
	
	So now we assume \(T=\emptyset\), and the argument above shows that  that \(\mfk{A}=\mfk{T}\). 
	Now by the theory of Néron model, we 
	know that \(\mfk{T}\)
	has a canonical algebrization \(T\)
	over \(\Spec\mO_{C}\),
	such that there is a natural covering map \[p:T^{\an}_{\eta}\to A,\]
	where the kernel is identified with 
	\(X_{*}(T)\).
Now if we consider \(p^{n}\)-torsion points, we have an exact sequence \[0\to T_{\eta}[p^{n}]\to A[p^{n}]\to X_{*}(T)/p^{n}\to 0,\]
and taking limit gives \[0\to T_{p}(T_{\eta})\to T_{p}A\to X_{*}(T)\otimes_{\ZZ}\ZZ_{p}\to 0.\]
Moreover, the sequence is \(\mO_{D_{T}}\)-equivariant, 
and after base change to \(C\), it can be identified with the Hodge-Tate sequence \[0\to \Lie(A) \to T_{p}A\otimes C\to \omega_{A^{\vee}} \to 0.
\] This implies that \(\pi_{\HT}((A,\iota,\lambda,\alpha_{U^{p}}))\in\Fl(\QQ_{p}),\)
as we desire to see.	
\end{proof}


\subsection{Igusa stack of PEL type}
\label{subsecIgusaStack}
We define an Igusa stack for \(\ti{\mS}^{\circ}_{U^{p}}(G'_{\ti{T}})\) in this subsection.
Our definition is a mild modification of those in \cite{Zhang2023pel}.
A version in the case of general Shimura varieties of Hodge type has been obtained in \cite{DanielsHoftenKimZhang2024igusa}. 

We insist on giving an independent construction here, 
because we will need a comparison of Igusa stacks
associated to different Shimura varieties. 
Note that \cite{DanielsHoftenKimZhang2024igusa} uses crucially the theory of integral models, and it is possible to prove the same comparison using exotic correspondence as in \cite{XiaoZhu2017cycles}, but the results there seem to require some technical conditions. 
\begin{rmk}[Difference with \cite{Zhang2023pel}]\label{rmkDifferentFromZhang}
Our definition of Igusa stacks will be slightly different from the previous work. 
More precisely, we remove the Kottwitz signature condition in \cite[Definition 8.1]{Zhang2023pel}. 
In this way, the comparison of Igusa stacks becomes only a matter of definition. 
The key observation is that the Kottwitz signature condition is remembered by the flag variety and is forgotten by the Igusa stack.
\end{rmk}

\subsubsection{Construction of the Igusa stack}
As before, we fix isomorphisms between \(i_{T}:D_{T}\cong M_{2}(E)\)
from now on, and denote it simply as \(D\). We further fix \(\delta_{T}\in D_{T}^{\mrm{sym},\times}\) such that \(*_{T}\) coincides for all \(T\subset\Sigma_{\infty}\) along the identification \(i_{T}\). 
Note that over \(\mA^{p}_{f}\), 
\(G'_{\tilde{T},\mA^{p}_{f}}\) are isomorphic. 
We fix a neat compact subgroup 
\(U^{p}\subset G'_{\tilde{T}}(\mA^{p}_{f})\). 

Recall that
we have also fixed 
a lattice \(\Lambda_{T}\subset D\)
such that \(\psi_{T}(\hat{\Lambda}_{T},\hat{\Lambda}_{T})\subset \hat{\ZZ}\). Recall that level structures were defined by considering isomorphisms of the Tate modules and \(\hat{\Lambda}_{T}\).

\begin{lem}\label{lemRatHermIndep}
	There is an \(\mO_{D}\)-linear isomorphism \((\Lambda_{T}\otimes_{\ZZ}\ti{E},\psi_{T})\cong (\Lambda_{\emptyset}\otimes_{\ZZ}\ti{E},\psi_{\emptyset})\),
	where \(\ti{E}\)
	is the Galois closure of \(E\).
\end{lem}
\begin{proof}
	This is because \(\mO_{D}\otimes_{\ZZ}\ti{E}\cong \prod_{\tilde{\tau}:E\hookrightarrow \ti{E}}D\otimes_{E,\ti{\tau}}\ti{E} \),
	and the action of \(*\) on the left-hand side will switch the factors \(\ti{\tau}\)
	and \(\ti{\tau}^{c}\).
	Hence we conclude by noting that there is a unique free Hermitian space of rank \(1\) over \(\mO_{D}\otimes\ti{E}\). 
\end{proof}

\begin{dfn}[Igusa stack of the good reduction locus]\label{dfnIgusaStack}
We define the following prestack \(\widetilde{\mrm{Igs}}_{U^{p}}^{\circ,pre}(G'_{\tilde{T}})\)
as follows: for any \((R,R^{+})\subset \Perf\) and \(\varpi\in R^{+}\)
a pseudo-uniformizer, 
define the groupoid
 \[\widetilde{\mrm{Igs}}_{U^{p}}^{\circ,pre}(G'_{\tilde{T}})(R,R^{+}):=\{(\bar{A},\bar{\iota},\bar{\lambda},\alpha_{U^{p}})\}\]
where

(1)
\(\bar{A}\) is an abelian scheme of dimension \(4d\) up to quasi-isogenies over \(R^{+}/\varpi\);

(2)
\(\iota:\mO_{D}\hookrightarrow \End_{R^{+}/\varpi}(\bar{A})_{\QQ}\);

(3) \(\bar{\lambda}:\bar{A}\to \bar{A}^{\vee}\)
a \(\mO_{D}\)-quasi-polarization, whose Rosati involution coincides with \(*\) on \(D\);

(4) \(\alpha_{U^{p}}\) is a \(U^{p}\)-level structure as in \S \ref{subsectionGoodRedLoc};



The morphisms in this groupoid are quasi-isogenies of \(\bar{A}\)
that are compatible with the PEL structure: \(\Hom((\bar{A}_{1},\bar{\iota}_{1},\bar{\lambda}_{1},\alpha_{U^{p},1}
=(\alpha_{1}^{0},c_{1}
),(\bar{A}_{2},\bar{\iota}_{2},\bar{\lambda}_{2},\alpha_{U^{p},2}
=(\alpha_{2}^{0},c_{2})
)\) consists of 
pairs \((f,r)\) where \
\(\mO_{D}\)-equivariant quasi-isogenies
 \(f:\bar{A}_{1}\to \bar{A}_{2}\) and \(r:\Spa(R,R^{+})\to \QQ^{\times,+}\) locally constant,
  such that 
  \(f^{*}(\alpha_{2}^{0})=\alpha_{1}^{0}\), 
  \(c_{2}=rc_{1}\) 
  and \(f^{*}(\bar{\lambda}_{2})\cong r\bar{\lambda}_{1}\).

We define \(\widetilde{\Igs}_{U^{p}}^{\circ}(G'_{\ti{T}})\)
as the \(v\)-sheafification of \(\widetilde{\mrm{Igs}}_{U^{p}}^{\circ,pre}(G'_{\tilde{T}})\).
%
\end{dfn}
\begin{rmk}\label{rmkEACrystallineDieudonne}
	Let us recall the construction of \(\mscr{E}(\bar{A})\):
	starting from \(\bar{A}\), we construct \(M:=H^{1}_{\cris}(\bar{A}/A_{\cris}(R^{+}/\varpi))^{\vee}[1/p]\)
	which is a \(\varphi\)-module over \(B^{+}_{\cris}(R^{+}/\varpi),\) and
	then using the standard argument, it can be extended to a \(\varphi\)-module over \(\mcal{Y}_{\rF}(R)\) using the Frobenius action,
	and hence it gives rise to a vector bundle \(\mscr{E}(\bar{A})\)
	over \(\mX_{\rF}(R)\). See also \cite{Zhang2023pel}
	for a construction using the algebraic Fargues-Fontaine curve.
\end{rmk}
\begin{rmk}
	We have implicitly used the Serre-Tate theory (\cite[Theorem 2.4.1]{CaraianiScholze2024generic}) to show that  the definition is  indeed independent of the choice of the pseudo-uniformizer \(\varpi\).
	Note that all the PEL data are 
	defined only up to quasi-isogenies.
\end{rmk}
\begin{notation}
	By definition, we have \[\widetilde{\Igs}^{\circ}_{U^{p}}(G'_{\tilde{T}})\cong \widetilde{\Igs}^{\circ}_{U^{p}}(G'_{\emptyset}).\]
	We will denote \(\widetilde{\Igs}^{\prime,\circ,pre}_{U^{p}}:=\widetilde{\Igs}^{\circ,pre}_{U^{p}}(G'_{\emptyset})\cong \widetilde{\Igs}^{\circ,pre}_{U^{p}}(G'_{\ti{T}})\),
	and \(\widetilde{\Igs}^{\prime,\circ}_{U^{p}}:=\widetilde{\Igs}^{\circ}_{U^{p}}(G'_{\emptyset})\cong \widetilde{\Igs}^{\circ}_{U^{p}}(G'_{\ti{T}})\).
\end{notation}

We have the following reduction morphism relating \(\ti{\mS}_{U^{p}}^{\circ}(G'_{\ti{T}})\)
and \(\widetilde{\Igs}_{U^{p}}^{\prime,\circ}\).

\begin{dfn}\label{dfnReductionMap}
	For any \(S=\Spa(R,R^{+})\in \Perf\) with an untilt 
	\(S^{\#}=\Spa(R^{\#},R^{\#,+})\in\Spd(\CC_{p},\mO_{\CC_{p}})(S)\), we map \((A,\iota,\lambda,\alpha_{p},\alpha_{U^{p}})\in \ti{\mS}^{\circ}_{U^{p}}(G'_{\ti{T}})(S^{\#})\) to \((\bar{A},\bar{\iota},\bar{\lambda},\alpha_{U^{p}})\in\widetilde{\Igs}_{U^{p}}^{\prime,\circ}(S)\), where \(\bar{A}:=A\otimes_{R^{\#,+}}R^{\#,+}/\varpi\), and \((\bar{\iota},\bar{\lambda},\alpha_{U^{p}})\)
	is the 
	induced PEL structure, then this defines a well-defined morphism \[\red_{\ti{T}}:\ti{\mS}^{\circ}_{U^{p}}(G'_{\ti{T}})\to \widetilde{\Igs}_{U^{p}}^{\prime,\circ}.\]
	
	The tower of morphisms \(\{\mrm{red}_{\ti{T}}:\ti{\mS}^{\circ}_{U^{p}}(G'_{\ti{T}})\to \widetilde{\Igs}^{\prime,\circ}_{U^{p}}\}_{U^{p}}\)
	is also \(G'_{\ti{T}}(\mA_{f}^{p})\)-equivariant.
\end{dfn}

\subsubsection{Stack of \texorpdfstring{\(G\)}{G}-torsors on Fargues-Fontaine curves}
We now prepare ourselves for 
the Cartesian diagram (Theorem \ref{thmCartesianDiagram}).
\begin{notation}
	To simplify the notation, we will write \(G':=G'_{\emptyset}\),
	\(G:=G_{\emptyset}=\mrm{Res}_{F/\QQ}\GL_{2}\),
	and \(\mu':=\mu_{\emptyset}\) from now on.
\end{notation}
\begin{dfn}\label{dfnFFcurve}
	We denote \(\mX_{\rF,K}(S):=\mX_{\rF}(S)\otimes_{\QQ_{p}} K\) for any \(K\) finite extension of \(\QQ_{p}\).
For any number field \(F\), we denote \(\mX_{\rF,F}(S):=\coprod_{\p}\mX_{\rF,F_{\p}}(S)\),
where \(\coprod\)
is taken over all the places of \(F\) over \(p\).
\end{dfn}

\(\Bun_{G'_{\emptyset}}\)
has the following explicit description:

\begin{lem}\label{lemStackUnitarySimilitudeVB}
The stack \(\Bun_{G'}\) has the following alternative moduli interpretation: for any \(S\),
\(\Bun_{G'}(S)\)
is equivalent to the groupoid of tuples 
\((\mcal{E},\mcal{L},\lambda_{\mcal{E}})\), where \(\mcal{E}\) is a
rank 2 vector bundles over \(\mX_{\rF,F}(S)\), \(\mcal{L}\) is a line bundle over \(\mX_{\rF,F}(S)\), and 
a similitude symplectic form induced by \(\lambda_{\mcal{E}}:\mcal{E}\to \mcal{E}^{*}\otimes \pi^{*} \mcal{L}\) (with its similitude factor in \(\mO_{\rF,F}(S)^{\times}\)).
Here, \(\mcal{E}^{*}:=c^{*}\mcal{E}^{\vee}\) 
where \(c:\mX_{\rF,F}(S)\to \mX_{\rF,F}(S)\)
is induced by the complex conjugation \(c:E\to E\),
and \(\pi:\mX_{\rF,E}(S)\to \mX_{\rF,F}(S)\). 

The map \(G'=\Res_{F/\QQ}\GL_{2}\times^{T_{F}}T_{E}\to \Res_{E/\QQ}\GL_{2}\) induces \(\Bun_{G'}\to \Bun_{\Res_{E/\QQ}\GL_{2}}\)
which can be described by \((\mcal{E},\mcal{L},\lambda_{\mcal{E}})\mapsto \mcal{E}\). 

The map \(\det\times 1:G'=\Res_{F/\QQ}\GL_{2}\times^{T_{F}}T_{E}\to C':=T_{F}\times^{T_{F}}T_{E}\) induces 
\(\Bun_{G}\to \Bun_{C'}\), which can be described by \((\mcal{E},\mcal{L},\lambda_{\mcal{E}})\mapsto (\det(\mcal{E}),\mcal{L},\det(\lambda_{E}):\det(\mcal{E})\to \det(\mcal{E})^{*}\otimes \mcal{L}^{\otimes 2})\).
\end{lem}
\begin{proof}
	By (\ref{alignTheGrpIsSimilit}), \(\Bun_{G'}(S)\) is equivalent to 
	the groupoids of rank \(2\) venctor bundles \(\mcal{E}\) over \(\mX_{\rF,E}(S)\) equipped with a polarization
\(\lambda_{\mcal{E}}:\mcal{E}^{\oplus 2}\cong (\mcal{E}^{\oplus 2})^{*}\) up to \(\mO_{\mX_{\rF,F}(S)}^{\times}\)  that is \(M_{2}(E)\otimes_{E}\mO_{\mX_{\rF,E}}\)-equivariant, where \(M_{2}(E)\otimes_{E}\mX_{\rF,E}\) acts on \((\mcal{E}^{\oplus 2})^{*}\) via \(*\). Explicitly,
	we define a functor \(\mcal{G}'\mapsto \mcal{G}'(\Lambda,\psi_{\emptyset})\), where \((\Lambda,\psi_{\emptyset})\)
	is regarded as a representation of \(G'\),
	and the inverse is given by \((\mcal{E},\lambda_{\mcal{E}})\mapsto \underline{\mrm{Isom}}((\mcal{E},\lambda_{\mcal{E}}),(\Lambda,\psi_{\emptyset}))\),
	which defines a \(G'\)-torsor by Lemma \ref{lemRatHermIndep}.

	Now the rest follows immediately from Morita equivalence. Note that we have fixed the identification \(D_{T}\cong M_{2}(E)\),
	and \(M_{2}(E)\)
	is endowed with an induced convolution \(*\) with \(\begin{pmatrix}
	a&b\\c&d
	\end{pmatrix}^{*}=\begin{pmatrix}
	\bar{a}&\bar{c}\\\bar{b}&\bar{d}
	\end{pmatrix}\), and in particular, \(\begin{pmatrix}
	1 & 0 \\0 &0 
	\end{pmatrix}\) is invariant under \(*\). 
\end{proof}

We can identify \(\Bun_{G'_{\ti{T}}}\) with \(\Bun_{G'}\), and identify their connected components.
\begin{lem}\label{lemIdenfificationBL} Consider the subgroup \(G'_{\ti{T},1}:=\Ker(\nu:G'_{\ti{T}}\to T_{F})\subset G'_{\ti{T}}\). Recall that we have defined 
	\(h_{\ti{T}}:\mbb{S}\to G'_{\ti{T}}\), and we denote the associated Hodge cocharacter as \(\mu_{\ti{T}}\).

(1) There is an isomorphism \(\pi_{1}(G'_{\tilde{T},1})_{\Gal_{\QQ_{p}}}\cong (\ZZ/2)^{\oplus\Sigma_{p}}\). 

The sequence \(1\to G'_{\tilde{T},1}\to G'_{\ti{T}}\to T_{F}\to 1\) induces a
canonical exact sequence  \[0\to (\ZZ/2)^{\oplus \Sigma_{p}}\to  \pi_{1}(G'_{\ti{T}})_{\Gal_{\QQ_{p}}}\to  \ZZ^{\oplus\Sigma_{p}}\to 0.\]

(2) For any \(T\subset \Sigma_{\infty}\),
\(\mcal{G}'_{\ti{T},1}:=\underline{\mrm{Isom}}(\Lambda_{T}\otimes\mO,\Lambda_{\emptyset}\otimes\mO)\)
defines an \'etale \(G'_{\emptyset,1}\)-torsor over \(\Spa(\QQ_{p},\ZZ_{p})\).
For any \(S=\Spa(R,R^{+})\in \Perf\), we we denote by \(\mcal{G}'_{\ti{T},1}(S)\in \Bun_{G'_{\emptyset,1}}(S)\) the pull-back of \(\mcal{G}'_{\ti{T},1}\)
along \(\mX_{\rF}(S)\to \Spa(\QQ_{p})\). 
We denote by \(\mcal{G}'_{\ti{T}}(S)\in\Bun_{G'_{\emptyset}}(S)\)
the induced \(G'_{\emptyset}\)-torsor. 

Then \(\mcal{G}'_{\ti{T}}(S)\in \Bun_{G'_{\emptyset},b_{\ti{T}}^{bsc}}(S)\),
where \(b_{\ti{T}}^{bsc}\)
is the basic element in \(B(G'_{\emptyset},\mu_{\ti{T}}^{-1}\mu_{\emptyset})\),
Moreover, under the bijection \(\pi_{1}(G'_{\emptyset})_{\Gal_{\QQ_{p}}}\cong \pi_{0}(\Bun_{G'_{\emptyset}})\cong B(G'_{\emptyset})^{bsc}\), \(b^{bsc}_{\ti{T}}\)
corresponds to \[(r_{\p,T}\!\!\!\!\mod{2})_{\p\in\Sigma_{p}}\in(\ZZ/2\ZZ)^{\oplus\Sigma_{p}}\cong \pi_{1}(G'_{\ti{T},1})_{\Gal_{\QQ_{p}}}\subset \pi_{1}(G'_{\ti{T}})_{\Gal_{{\QQ}_{p}}}\] as in (1), with \(r_{\p,T}:=\#(T\cap \Sigma_{\infty/p})\).

(3) \(G'_{\ti{T}}\) acts on $\mcal{G}'_{\ti{T}}$ by acting on \(\Lambda_{T}\otimes\mO\).	The morphism \[\iota_{\ti{T}}:\Bun_{G'_{\ti{T}}}\to\Bun_{G'_{\emptyset}},\;\mcal{G}'\mapsto \mcal{G}'\times^{G'_{\ti{T}}} \mcal{G}'_{\ti{T}}\]
induces an isomorphism between two stacks. Moreover, \(\iota_{\ti{T}}|_{\Bun_{G'_{\ti{T}},\mu_{\ti{T}}}}\)
has a factorization 
\[\iota_{\ti{T}}|_{\Bun_{G'_{\ti{T}},\mu_{\ti{T}}}}:\Bun_{G'_{\ti{T}},\mu_{\ti{T}}}\to \Bun_{G'_{\emptyset},\mu_{\emptyset}},\]
which is an open immersion. 
\end{lem}
\begin{proof}
	For (1),
	recall that \(G'_{\ti{T}}\)
	and \(G'_{\ti{T},1}\)
	has the same derived subgroups as \(G_{T}\),
	which is an inner form of \(\mrm{Res}_{F/\QQ}\SL_{2}\),
	so they are simply connected.
	So 
	\(\pi_{1}(G'_{\ti{T},1})\cong \pi_{1}(G'_{\ti{T},1}/G^{\prime,der}_{\ti{T},1})\),
	and \(G'_{\ti{T},1}/G^{\prime,der}_{\ti{T},1}\cong \Ker(\nu:(T_{F}\times T_{E})/T_{F}\to T_{F})\),
	where
\(T_{F}\)
acts on \(T_{F}\times T_{E}\)
by \(a\cdot (b,c)=(a^{2}b,ca^{-1})\),
and \(\nu(b,c)=bc\bar{c}\).
Note that \(\nu\) has a splitting given by \(b\mapsto (b,1)\in (T_{F}\times T_{E})/T_{F}\), 
so the kernel is isomorphic to \(T_{E}/T_{F}\).
Now we take \(X_{*}(-)\) of the sequence 
\[1\to T_{F}\to T_{E}\to T_{E}/T_{F}\to 1,\]
which is given by 
\[0\to \ZZ^{\Sigma_{\infty}}\to \ZZ^{\Sigma_{E,\infty}}\to \ZZ^{\Sigma_{E,\infty}}/\ZZ^{\Sigma_{\infty}}\to 0,\]
and the first map is given by 
\(1_{\tau}\mapsto 1_{\ti{\tau}}+1_{\ti{\tau}^{c}}\).
Now taking \(\Gal_{\QQ_{p}}\)-coinvariant,
it becomes \[\ZZ^{\Sigma_{p}}\to \ZZ^{\Sigma_{E,p}}\to \ZZ^{\Sigma_{E,p}}/\ZZ^{\Sigma_{p}}\to 0.\]
Now by our assumption on $E$, we know \(\Sigma_{p}=\Sigma_{E,p}\),
and the map is given by \(1_{\p}\mapsto 2\cdot 1_{\p}\), which implies that \(\pi_{1}(G'_{\ti{T},1})_{\Gal_{\QQ_{p}}}\cong (\ZZ/2)^{\oplus \Sigma_{p}}\). The rest part of (1) follows easily.

For (2), note that \(\mcal{G}'_{\ti{T}}\) is pulled back from \(\Spa(\QQ_{p})\),
so by \cite[\S 4.5]{kottwitz1985isocrystals}, its Newton point is zero, and it is basic. 
Let \(\kappa(\mcal{G}'_{\ti{T}})\) be the corresponding element in \(\pi_{1}(G'_{\emptyset})_{\Gal_{\QQ_{p}}}\). 
\(\mcal{G}'_{\ti{T}}(S)\)
has a \(G'_{\emptyset,1}\)reduction \(\mcal{G}'_{\ti{T},1}(S)\),
so \(\kappa(\mcal{G}'_{\ti{T}})\in \pi_{1}(G'_{\emptyset,1})_{\Gal_{\QQ_{p}}}\cong (\ZZ/2)^{\oplus \Sigma_{\infty}}\).
Moreover, there is a canonical decomposition \(G'_{\ti{T},1,\QQ_{p}}\cong \prod_{\p\in\Sigma_{p}}G'_{\ti{T},1,\p}\),
so we also have a decomposition \(\mcal{G}'_{\ti{T},1}=\prod_{\p\in\Sigma_{p}}\mcal{G}'_{\ti{T},1,\p}\),
and \(\mcal{G}'_{\ti{T},1,\p}\) is basic, with the associated inner form given by \(G'_{\ti{T},1,\p}\).

From the proof of (1),
we see \(\kappa(\mcal{G}'_{\ti{T}})=(\kappa(\mcal{G}'_{\ti{T},1,\p}))\in (\ZZ/2)^{\oplus \Sigma_{p}}\)
with \(\kappa(\mcal{G}'_{\ti{T},1,\p})\in\pi_{1}(G'_{\emptyset,1,\p})_{\Gal_{\QQ_{p}}}\cong \ZZ/2\). 
Note that \(\mcal{G}'_{\ti{T},1,\p}\)
is determined by \(\kappa(\mcal{G}'_{\ti{T},1,\p})\)
by \cite[Proposition 5.6]{kottwitz1985isocrystals}.
So \(\mcal{G}'_{\ti{T},1,\p}\)
has only two possible choices. 
\(G'_{\ti{T},1,\p}\) is quasi-split when \(2\mid r_{\p,T}\),
and is not quasi-split when \(2\nmid r_{\p,T}\), so the torsors \(\mcal{G}'_{\ti{T},1,\p}\) in the two cases 
corresponds to the two distinct elements in \(\pi_{1}(G'_{\emptyset,1,\p})_{\Gal_{\QQ_{p}}}\cong \ZZ/2\). As \(G'_{\emptyset}\)
is quasi-split, the quasi-split case corresponds to \(0\in \pi_{1}(G'_{\emptyset,1,\p})\).


Finally, we need to show that \(\mcal{G}'_{\ti{T}}(S)\in \mrm{Bun}_{G'_{\emptyset},\mu_{\ti{T}}^{-1}\mu_{\emptyset}}\). Since \(\mcal{G}'_{\ti{T}}\) is basic, it suffices to show that \(\kappa(\mcal{G}'_{\ti{T}})=(\mu_{\ti{T}}^{-1}\mu_{\emptyset})^{\sharp}\) (\cite[Definition 3.1.2]{CS17}). 
We consider the composition \[\nu_{\ti{T}}:\GG_{m,\bar{\QQ}_{p}}\xrightarrow{\mu_{\ti{T}}^{-1}\mu_{\emptyset}} G'_{\ti{T},1,\bar{\QQ}_{p}}\to (T_{E}/T_{F})_{\bar{\QQ}_{p}}\cong \prod_{\tau\in\Sigma_{\infty}}(\Res_{E/F}\GG_{m,E}/\GG_{m,F})\times_{F,\tau}\bar{\QQ}_{p}. \]
Then for \(\tau\in T\), the \(\tau\)-component of \(\nu_{\ti{T}}\)
induces an isomorphism,
while for \(\tau\notin T\), the
\(\tau\)-component of \(\nu_{\ti{T}}\) is trivial. Combining with the construction of the isomorphism \(\pi_{1}(G'_{\ti{T},1})_{\Gal_{\QQ_{p}}}\cong \pi_{1}(T_{E}/T_{F})_{\Gal_{\QQ_{p}}}\cong (\ZZ/2)^{\bigoplus \Sigma_{p}}\) in (1), we conclude that the image of \(\nu_{\ti{T}}^{\#}\) is precisely \((r_{\p,T}\mod{2})_{\p}\in \pi_{1}(G'_{\ti{T},1})_{\Gal_{\QQ_{p}}}\).

(3) follows easily from (2).
\end{proof}

\subsubsection{Reduced Hodge-Tate period maps}

\begin{lem}
Let \((R,R^{+})\in \Perf\), and \((\bar{A},\bar{\iota},\bar{\lambda},\alpha_{U^{p}})\in\widetilde{\Igs}^{pre}_{U^{p}}(G'_{\ti{T}})(R,R^{+})\). Then \(\underline{\mrm{Isom}}(\mscr{E}(\bar{A}),(\Lambda_{T}\otimes \mO_{\mX_{\rF}(R)},\psi_{T}))\) is an \'etale \(G'_{\ti{T}}\)-torsor over \(\mX_{\rF}(R)\), where \(\underline{\mrm{Isom}}(-)\)
	is the internal \(\mO_{D}\)-linear similitude symplectic  isomorphisms (with similitude factors in \((\mO_{\mX_{\rF}(R)}\otimes F)^{\times}\)), and \(G'_{\ti{T}}\)
	acts by acting on \(\Lambda_{T}\otimes \mO_{\mX_{\rF}(R)}\).
\end{lem}
\begin{proof}
It suffices to show that \'etale locally over \(\mX_{\rF}(R)\), such isomorphisms exist. 
	This follows from Lemma \ref{lemStackUnitarySimilitudeVB} and  Lemma \ref{lemRatHermIndep}. 
\end{proof}
\begin{dfn}\label{dfnReducedHTPeriod}
	We define the \emph{reduced Hodge-Tate period map} \(\bar{\pi}_{\HT,\ti{T}}:\widetilde{\Igs}_{U^{p}}^{\prime,\circ}\to \Bun_{G'_{\ti{T}}}
	\) as the \(v\)-sheafification of \[\bar{\pi}_{\HT,\ti{T}}:\widetilde{\Igs}_{U^{p}}^{\circ,pre}(G'_{\ti{T}})\to \Bun_{G'_{\ti{T}}},
	\;
	(\bar{A},\bar{\iota},\bar{\lambda},\alpha_{U^{p}})
	\mapsto \underline{\mrm{Isom}}(\mscr{E}(\bar{A}),(\Lambda_{T}\otimes \mO_{\mX_{\rF}(R)},\psi_{T})).\]
\end{dfn}
\begin{rmk}
We are actually not sure whether \(\bar{\pi}_{\HT}\)
should factor through \(\Bun_{G'_{\emptyset},\mu_{\emptyset}}\)
or not. This might be related to incoherent definite spaces in \cite{Gross2020incoherentdefinitespacesshimura}. 
\end{rmk}
\begin{notation}
We denote by \(\widetilde{\Igs}^{\prime,\circ,coh}_{U^{p}}\) the preimage of \(\Bun_{G_{\emptyset}',\mu_{\emptyset}}\) along 
\(\bar{\pi}_{\HT}:\widetilde{\Igs}_{U^{p}}^{\prime,\circ}\to \Bun_{G_{\emptyset}'}\).
\end{notation}

We can identify \(\bar{\pi}_{\HT,\ti{T}}\) for different \(T\subset \Sigma_{\infty}\). 

\begin{lem}
We have a natural isomorphism \(\bar{\pi}_{\HT,\emptyset}\cong \iota_{\ti{T}}\circ\bar{\pi}_{\HT,\ti{T}}\). In particular, under the isomorphism \(\iota_{\tilde{T}}\),
we have a Hodge-Tate period map independent of \(T\),
which we simply denote as \[\bar{\pi}_{\HT}:\widetilde{\Igs}^{\prime,\circ}_{U^{p}}\to \Bun_{G'}.\]
\end{lem}
\begin{proof}
This follows from Lemma \ref{lemIdenfificationBL} (3).
\end{proof}

\subsubsection{Cartesian diagram}
Now we state and prove the main theorem about the Cartesian diagram:
\begin{thm}\label{thmCartesianDiagram}
	There is a coherent diagram \[\begin{tikzcd}
		\ti{\mS}^{\circ}_{U^{p}}(G'_{\ti{T}})\arrow[r,"\pi_{\HT}"]\arrow[d,"\red_{\ti{T}}"]
		& \Fl_{  G_{\ti{T}}',   \mu_{\ti{T}}    }\arrow[d,"\BL_{\ti{T}}"]\arrow[r] & \Spd(\CC_{p},\mO_{\CC_{p}})
		\\ 
		\tIgs^{\prime,\circ}_{U^{p}} \arrow[r,"\bar{\pi}_{\HT}"]
		&
		\Bun_{G'_{\ti{T}}},
	\end{tikzcd}\] where \(\pi_{\HT}\) is induced by the Hodge-Tate period map of \(\mS_{U^{p}}(G'_{\ti{T}})\), and \(\red_{\ti{T}}\), \(\BL_{\ti{T}}\), and \(\bar{\pi}_{\HT}\)
	are defined respectively in Definition \ref{dfnReductionMap}, Definition \ref{dfnBLmap}
	and Definition \ref{dfnReducedHTPeriod}.

	Moreover, the left square is a Cartesian diagram, and the 
	\(G'_{\ti{T}}(\mA_{f})\)-action on 
	\(\ti{\mS}^{\circ}_{U^{p}}(G'_{\ti{T}})\) is induced by the \(G'_{\ti{T}}(\mA_{f}^{p})\)-action on \(\widetilde{\Igs}^{\prime,\circ}_{U^{p}}\) and the \(G'_{\ti{T}}(\QQ_{p})\)-action on \(\Fl_{G'_{\ti{T}},\mu_{\ti{T}}}\)..
\end{thm}
\begin{proof}	
	We first check that the diagram is coherent.
	If we are given 
	\(S^{\#}=\Spa(R^{\#},R^{\#,+})\in\Spd(\CC_{p})(S)\)
	and \((A,\iota,\lambda,\alpha_{p},\alpha_{U^{p}})\in \mS^{\circ}_{U^{p}}(G'_{\ti{T}})(S^{\#})\),
	then there is a sequence \[0\to T_{p}A\otimes \mO_{\mX_{\rF}(R)}\to \mscr{E}(\bar{A})\to i_{S^{\#},*}\Lie(A)\to 0.\] On the other hand, we define \(\mscr{E}'(A):=\BL_{\ti{T}}\circ \pi_{\HT}(A,\iota,\lambda,\alpha_{p},\alpha_{U^{p}}) \), which  is the vector bundle induced by the Hodge-Tate sequence \[0\to \Lie(A)(1)\to T_{p}A\otimes R^{\#}\to \omega_{A}\to 0\]
	via the pull-back diagram  \[\begin{tikzcd}
		\mscr{E}'(A)\arrow[r,hook] \arrow[d] & T_{p}A\otimes\mO_{\mX_{\rF}(R)}(1) \arrow[d]
		\\ i_{S^{\#},*}\Lie(A)\arrow[r,hook]
		&{T_{p}A(-1)\otimes i_{S^{\#},*}}\mO_{S^{\#}},
	\end{tikzcd}\]
	Moreover, via de Rham comparison, there is a natural isomorphism \(\mscr{E}(\bar{A})\cong \mscr{E}'(A)\),
	and by functoriality,
	the isomorphism is compatible with the \(\mO_{D}\)-action and  the polarization,
	which provides the 2-morphism making the diagram coherent.

	Now we have defined a morphism \[\mS^{\circ}_{U^{p}}(G'_{\ti{T}})\to \tIgs_{U^{p}}^{\prime,\circ}\times_{\Bun_{G'_{\ti{T}}}}\Fl_{  G_{\ti{T}}',   \mu_{\ti{T}}    }.\]
	To verify that this is an isomorphism, we construct an inverse. 
	Since we are working over \(v\)-stacks, it suffices to verify the statement for ``products of points" by \cite[Remark 1.3]{Gleason2024specialization}, that is, \(S=\Spa(R,R^{+})\)
	where 
	\(R^{+}=\prod_{i}C_{i}^{+}\)
	and \(R=R^{+}[1/\varpi]\), 
	with \(C_{i}^{+}\subset C_{i}\) being a complete algebraically closed perfectoid field over \(\bar{\FF}_{p}\), and \(\varpi=(\varpi_{i})\) for  some
	pseudo-uniformizer \(\varpi_{i}\) in \(C_{i}^{+}\).
	We know that \(\tIgs_{U^{p}}^{\prime,\circ}\times_{\Bun_{G'_{\ti{T}}}}\Fl_{  G_{\ti{T}}',   \mu_{\ti{T}}    }\) is the \(
    v
    \)-stackification of the prestack sending 
	\(S=\Spa(R,R^+)\in \Perf\) with an untilt \(S^{\#}=\Spa(R^{\#},R^{\#,+})\) to the groupoid whose objects are given by \((\bar{A},\bar{\lambda},\bar{\iota},\alpha_{U^{p}},p:\Lambda_{\ti{T}}\otimes R^{\#}\twoheadrightarrow L,i)\)
	where
	
	(1) \((\bar{A},\bar{\lambda},\bar{\iota},\alpha_{U^{p}})\in\widetilde{\Igs}_{U^{p}}^{\prime,\circ}(S)\);

	(2) \(p\)
	is a surjection of \(R^{\#}\)-modules,
	where \(L\) is a finite projective \(R^{\#}\)-modules equippped
	with an action of \(\mO_{D}\),
	such that \(\psi_{\ti{T}}(\ker(p),\ker(p))=0\),
	and the action of \(\mO_{E}\subset \mO_{D}\)
	on 
	\(L\) satisfies the Kottwitz signature condition for \(\ti{T}\);

	(3) \(i\) is an \(\mO_{D}\)-linear similitude symplectic isomorphism \(i:\mscr{E}(\bar{A})\cong \BL_{\ti{T}}(p), \)
	where we recall concretely that \(\BL_{\ti{T}}(p)\)
	is defined via the pull-back diagram \[\begin{tikzcd}
		\BL_{\ti{T}}(p)\arrow[r,hook] \arrow[d] & \Lambda_{\ti{T}}\otimes\mO_{\mX_{\rF}(R)}(1) \arrow[d]
		\\ i_{S^{\#},*}\ker(p)(-1)\arrow[r,hook]
		&{\Lambda_{\ti{T}}(-1)\otimes i_{S^{\#},*}}\mO_{S^{\#}},
	\end{tikzcd}\]
	with the induced \(\mO_{D}\)-module symplectic structure.

Given such a datum \((\bar{A},\bar{\lambda},\bar{\iota},\alpha_{U^{p}},p,i)\), we are going to construct a lift of \(\bar{A}\) to \(R^{+}\). The construction is very similar to the proof of \cite[Theorem 25.1.2]{ScholzeWeinstein2020berkeley}.
	Given such \(p:\Lambda_{\ti{T}}\otimes R^{\#}\to L\),
	we can 
	construct a vector bundle over \(\mcal{Y}_{[0,\infty)}(R^{+})=\Spa(W(R^{+}))\backslash V([\varpi])\)
	as follows: 
we have an exact sequence \[0\to \Lambda_{\ti{T}}\otimes \mO_{\mX_{\rF}(R)}\to \mrm{BL}_{\ti{T}}(p)\to i_{S^{\#},*}\ker(p)\to 0.\]
Pulling back along \(p:\mcal{Y}_{(0,\infty)}(R)\to \mX_{\rF}(R)\),
we have \[0\to \Lambda_{\ti{T}}\otimes \mO_{\mcal{Y}_{(0,\infty)}(R)}\to p^{*}\mrm{BL}_{\ti{T}}(p)\to \bigoplus_{n\in \ZZ} i_{\varphi^{n}(S^{\#}),*}\ker(p)\to 0.\]
Then we can consider the pull-back along 
\( \bigoplus_{n\in \NN} i_{\varphi^{n}(S^{\#}),*}\ker(p)\hookrightarrow  \bigoplus_{n\in \ZZ} i_{\varphi^{n}(S^{\#}),*}\ker(p)\)
to obtain a new vector bundle \[0\to \Lambda_{\ti{T}}\otimes \mO_{\mcal{Y}_{(0,\infty)}(R)}\to M(p)\to  \bigoplus_{n\in \NN} i_{\varphi^{n}(S^{\#}),*}\ker(p)\to 0\] 
over \(\mcal{Y}_{(0,\infty)}(R)\).
In particular, \(M(p)|_{\mcal{Y}_{(0,r)}}\cong \Lambda_{\ti{T}}\otimes \mO_{\mcal{Y}_{(0,r)}(R)}\)
for \(r\) sufficiently small,
and thus we can glue \(M(p)\)
and \(\Lambda_{\ti{T}}\otimes \mO_{\mcal{Y}_{[0,r)}(R)}\)
to obtain a vector bundle over \(\mcal{Y}_{[0,\infty)}(R)\), which we also denote as \(M(p)\).
We also note that there is a natural \[\varphi_{M(p)}:\varphi^{*}M(p)[\frac{1}{\varphi(\xi)}]\to M(p)[\frac{1}{\varphi(\xi)}]\] such that \(\varphi_{M(p)}(M(p))\subset \frac{1}{\varphi(\xi)}M(p)\).

On the other hand, for \(r\) sufficiently large,
we have \(M(p)|_{\mcal{Y}_{(r,\infty)}(R)}\cong p^{*}\BL_{\ti{T}}(p)|_{\mcal{Y}_{(r,\infty)}(R)}\cong p^{*}\mscr{E}(\bar{A})|_{\mcal{Y}_{(r,\infty)}(R)}\), where the last map is induced by \(i\). Now
\(p^{*}\mscr{E}(\bar{A})\) has a unique extension from \(\mcal{Y}_{(r,\infty)}(R)\)
to \(\mcal{Y}_{(r,\infty]}(R)\) induced by the \(\varphi\)-module \(H_{\cris}^{1}(\bar{A}/A_{\cris}^{+}(R^{+}/\varpi))^{\vee}[1/p]\) over \(B_{\cris}^{+}(R^{+}/\varpi)\). We denote this extension as \(M'(\bar{A})\). 

So we can glue \(M(p)\)
and \(M'(\bar{A})\)
 together 
to obtain a vector bundle \(M(\bar{A},p,i)\)
over \(\mcal{Y}_{[0,\infty]}(R)\). 
Now by \cite[Theorem 3.8]{Kedlaya2020some} and \cite[Proposition 2.7]{Gleason2021geometric},
	\(\tilde{\mscr{E}}(p)\)
	is the restriction of a unique finite projective \(W(R^{+})\)-module, which we also denote as \(M(\bar{A},p,i)\). 
	Moreover, we have an induced isomorphism \(\varphi_{M(\bar{A},p,i)}:M(\bar{A},p,i)[\frac{1}{\xi}]\cong M(\bar{A},p,i)[\frac{1}{\varphi(\xi)}]\),
	and \[\varphi_{M(\bar{A},p,i)}(M(\bar{A},p,i))\subset \frac{1}{\varphi(\xi)}\cdot \Lambda_{\ti{T}}\otimes W(R^{+})\subset \frac{1}{\varphi(\xi)}\cdot M(\bar{A},p,i).\]
	Hence by \cite[Theorem 17.5.2]{ScholzeWeinstein2020berkeley},
	this corresponds to a unique \(p\)-divisible group \(\mcal{G}\)
	over \(R^{+}\). By our construction,
	the base change of \(M(\bar{A},p,i)\)
	to \(B_{\cris}^{+}(R^{+}/\varpi)\)
	will coincide with the (covariant) Dieudonn\'e module of \(\bar{A}[p^{\infty}]\).
	Thus we have a \(p\)-quasi-isogeny \[\bar{A}[p^{\infty}]\cong \mcal{G}\times_{R^{+}}R^{+}/\varpi.\]

By Serre-Tate theory, we can find a unique abelian scheme \(A\) (up to \(p\)-quasi-isogeny) over \(R^{+}\) such that \(A[p^{\infty}]\) is \(p\)-quasi-isogeny to \(\mcal{G}\), and
\(A\) admits a \(p\)-quasi-isogeny \(A\times_{R^{+}}R^{+}/\varpi\cong \bar{A}\).
By functoriality, we have an action \(\iota\) of \(\mO_{D_{T}}\) on \(A\),
which coincides with \(\bar{\iota}\)
after reduction. Similarly, we have a polarization induced by the symplectic structure on \(\Lambda_{\ti{T}}\) and on \(\bar{A}\). Note 
that  the induced map \(f:A\to A^{\vee}\) over \(R^{\#,+}\) satisfies that \(f=f^{\vee}\)
and that \(f/\varpi:\bar{A}\to \bar{A}^{\vee}\) is a polarization over \(R^{\#,+}/\varpi\), so \(f\)
is also a polarization.

Finally, we note that we can identify \(\Lambda_{\ti{T}}[1/p]\) with \(T_{p}A[1/p]\), and \(\Lambda_{\ti{T}}\otimes R^{\#}\twoheadrightarrow L\)
can be identified with the Hodge-Tate sequence \(T_{p}A\otimes R^{\#}\twoheadrightarrow \Lie(A)\). This follows from the same proof of \cite[Theorem 14.1.1]{ScholzeWeinstein2020berkeley}. In particular, we can choose the representative in the \(p\)-isogeny class such that we  have a fixed isomorphism \(T_{p}A\cong \Lambda_{\ti{T}}\). It also follows that the desired Kottwitz signature condition holds.
The \(U^{p}\)-level structure on \(\bar{A}\) also induces naturally 
a \(U^{p}\)-level structure on \(\bar{A}\).

The construction 
is clearly functorial. 
\end{proof} 
\begin{cor}\label{corSurjectiveShToIgs}
The reduction map \(\mrm{red}_{\ti{T}}\)
factors through the open substack \(\widetilde{\Igs}^{\prime,\circ,coh}_{U^{p}}\). Moreover, \(\mrm{red}_{\emptyset}\) induces a \(v\)-surjection \(\mrm{red}_{\emptyset}:\ti{\mS}_{U^{p}}^{\circ}(G'_{\emptyset})\to \widetilde{\Igs}_{U^{p}}^{\prime,\circ,coh}\).
\end{cor}
\begin{proof}
This follows from Lemma \ref{lemIdenfificationBL}, Theorem \ref{thmCartesianDiagram} and the fact that \(\BL_{\ti{T}}\) is a \(v\)-surjection onto \(\Bun_{G'_{\ti{T}},\mu_{\ti{T}}}\) 
by \cite[Proposition III.3.1]{FarguesScholze2021geometrization}.
\end{proof}
\begin{notation}
We define a \(v\)-stack \(\widetilde{\Igs}^{\prime,\circ,coh}\) as \[
\widetilde{\Igs}^{\prime,\circ,coh}:=\varprojlim_{U^{p}}\widetilde{\Igs}^{\prime,\circ,coh}_{U^{p}},\;\ti{\mS}^{\circ}(G'_{\ti{T}}):=\varprojlim_{U^{p}}\ti{\mS}^{\circ}_{U^{p}}(G'_{\ti{T}})
\] Then by Theorem \ref{thmUnitaryCartesianDiagram}, we have a Cartesian diagram \[
\begin{tikzcd}
\ti{\mS}^{\circ}(G'_{\ti{T}})\arrow[r,"\pi_{\HT}"]\arrow[d,"\red_{\ti{T}}"]
		& \Fl_{  G_{\ti{T}}',   \mu_{\ti{T}}    }\arrow[d,"\BL_{\ti{T}}"]\arrow[r] & \Spd(\CC_{p},\mO_{\CC_{p}})
		\\ 
		\widetilde{\Igs}^{\prime,\circ,coh}|_{\Bun_{G'_{\ti{T}},\mu_{\ti{T}}}} 
		\arrow[r,"\bar{\pi}_{\HT}"]
		&
		\Bun_{G'_{\ti{T}}},
\end{tikzcd}
\]
\end{notation}

\begin{cor}\label{corIgusaStackPELGeneralDfn}
For any \(T\subset\Sigma_{\infty}\), \(\tIgs^{\prime,\circ,coh}|_{\Bun_{G'_{\ti{T}},\mu_{\ti{T}}}}\)
is an Igusa stack for \(\ti{\mS}^{\circ}(G'_{\ti{T}})_{E_{0,\ti{T}}}\)
in the sense of Definition \ref{dfnGeneralIgus}, which is applicable by Remark \ref{rmkActuallyPEL}. 

Here \(E_{0,\ti{T}}\subset \CC\)
denotes the reflex field of \((G'_{\ti{T}},h_{\ti{T}})\), \( E_{0,\ti{T},\p_{\ti{T}}} \) denotes the \( p \)-adic completion along \( E_{0,\ti{T}}\subset\CC\cong^{\iota} \bar{\QQ}_{p} \), and \(\ti{\mS}^{\circ}(G'_{\ti{T}})_{E_{0,\ti{T}}}\)
is the model of \(\ti{\mS}^{\circ}(G'_{\ti{T}})\) over \(\Spd(E_{0,\ti{T},\p_{\ti{T}}})\) as in Notation \ref{notationGenealShimura}.
\end{cor}
\begin{proof}
By Theorem \ref{thmCartesianDiagram}, it suffices to verify the assumption about the absolute Frobenius. This is because for \((R,R^{+})\in\Perf\),  
the absolute Frobenius induces \(R^{+}/\varpi\to R^{+}/\varpi,x\mapsto x^{p}\),
which coincide with \(R^{+}/\varpi\twoheadrightarrow R^{+}/\varpi^{1/p}\). Then the induced functor on the groupoids coincide with base change along \(R^{+}/\varpi^{p}\to R^{+}/\varpi^{1/p}\), which is equivalent to the identity functor by Serre-Tate theory (\cite[Theorem 2.4.1]{CaraianiScholze2024generic}).
\end{proof}
By \cite{Kim2025uniquenessfunctorialityigusastacks}, we then have an Igusa stack for \(\ti{\mS}(G'_{\ti{T}})_{E_{0,\ti{T}}}\), i.e. beyond the good reduction locus. 
\begin{cor}	\label{corIgusaStackPELGeneralDfnExtend}
There is a \(v\)-stack \(\tIgs'_{U^{p}}\to \Bun_{G'_{\emptyset},\mu_{\emptyset}}\) that contains
\(\tIgs^{\prime,\circ,coh}\) as an open subspace, such that if we put \[\tIgs':=\varprojlim_{U^{p}}\tIgs'_{U^{p}},\] then for any \(T\subset \Sigma_{\infty}\), \(\tIgs'|_{\Bun_{G'_{\ti{T}},\mu_{\ti{T}}}}\)
gives an Igusa stack for \(\ti{\mS}(G'_{\ti{T}})_{E_{0},\ti{T}}\). 
\end{cor}
\begin{proof}
We first consider \(\ti{\mS}(G'_{\emptyset})\). By \cite[Corollary 11.5]{Kim2025uniquenessfunctorialityigusastacks} and Remark \ref{rmkActuallyPEL}, there is a \(v\)-stack \(\tIgs'_{U^{p}}\) such that \(\tIgs':=\varprojlim_{U^{p}}\tIgs'_{U^{p}}\)
gives an Igusa stack for \(\ti{\mS}(G'_{\ti{T}})_{E_{0,\ti{T}}}\) for \(T=\emptyset\).
By \cite[Theorem 10.13]{Kim2025uniquenessfunctorialityigusastacks} and Corollary \ref{corIgusaStackPELGeneralDfn}, there is a canonical open immersion \(\tIgs^{\prime,\circ,coh}_{U^{p}}\subset \tIgs^{\prime}_{U^{p}}\), 
which is an isomorphism away from the 
\(\mu_{\emptyset}\)-ordinary locus by Proposition \ref{propGoodReducForNonOrd}. 
Since \(\Bun_{G'_{\ti{T}},\mu_{\ti{T}}}\) is disjoint from the \(\mu_{\emptyset}\)-locus for \(T\ne \emptyset\), \(\tIgs'|_{\Bun_{G'_{\ti{T}},\mu_{\ti{T}}}}\) gives an Igusa stack for \(\ti{\mS}(G'_{\ti{T}})_{E_{0,\ti{T}}}\)
by Corollary \ref{corIgusaStackPELGeneralDfn}.
\end{proof}

\subsection{Igusa stack for unitary Shimura varieties}\label{subsecUnitaryIgusa}
In this subsection, we construct Igusa stacks for \(\mS(G'_{\ti{T}})_{E_{0,\ti{T}}}\) using (\ref{alignFromtiStoSQuotientInfi}).
 These are Igusa stacks for Shimura varieties of abelian type, and we borrow ideas from \cite{DanielsHoftenKimZhang2024igusa} and \cite{Kim2025uniquenessfunctorialityigusastacks}. 
 Recall that 
\(\Delta_{U}:=\mO^{\times,+}_{F,(p)}/(\ZZ_{p}^{\times,+}\mrm{Nm}_{F/E}(E^{\times}\cap U))\), and \(\Delta_{U}\) acts on \(\ti{\mS}_{U}(G'_{\ti{T}})_{E_{0,\ti{T}}}\) as in Proposition \ref{propModuliShGpp}. We have \(\Delta_{U^{p}}:=\varprojlim_{U^{p}U_{p}}\Delta_{U^{p}U_{p}}\),
which is a locally profinite group, and admits a natural dense inclusion \(\mO_{F,(p)}^{\times,+}\hookrightarrow \Delta_{U^{p}}\).
\begin{prop}\label{propDeltaAction}
We have an action of \(\mO_{F,(p)}^{\times,+}/\ZZ_{(p)}^{\times,+}\)
on \(\widetilde{\Igs}^{\prime,\circ}_{U^{p}}\) via \[x:(\bar{A},\iota,\bar{\lambda},\alpha_{U^{p}}=(\alpha_{1}^{0},\ti{c}_{1})\!\!\!\!\mod{U^{p}})
\mapsto (\bar{A},\iota,x\bar{\lambda},(\alpha_{1}^{0},x\ti{c}_{1})\!\!\!\!\mod{U^{p}})
\] for \(x\in \mO^{\times,+}_{F,(p)}\), which preserve the subspace \(\widetilde{\Igs}^{\prime,\circ,coh}_{U^{p}}\).

Moreover, the action extends naturally to an action of \(\Delta_{U^{p}}\)
on \(\widetilde{\Igs}^{\prime,\circ,coh}_{U^{p}}\) such that for any \(T\subset \Sigma_{\infty} \) the reduction map \(\mrm{red}_{\ti{T}}:\ti{\mS}^{\circ}_{U^{p}}(G'_{\ti{T}})_{E_{0,\ti{T}}}\to \widetilde{\Igs}^{\prime,\circ,coh}_{U^{p}}\) is \(\Delta_{U^{p}}\)-equivariant. 
\end{prop}
\begin{proof}
By definition, the action of \(\mO_{F,(p)}^{\times,+}/\ZZ_{(p)}^{\times,+}\) on \(\widetilde{\Igs}^{\prime,\circ}_{U^{p}}\) is compatible with that on \(\ti{\mS}_{U^{p}}(G'_{\ti{T}})_{E_{0,\ti{T}}}\) along \(\mrm{red}_{\ti{T}}\). 
For \(i\in\NN\), let \(\mrm{Isog}^{\circ,i}\) denote the \(i\)-th fiber product  of \(\ti{\mS}^{\circ}_{U^{p}}(G'_{\ti{T}})_{E_{0,\ti{T}}}\)
over \(\tIgs^{\prime,\circ,coh}_{U^{p}}\). Then \(\mrm{Isog}^{\circ,i}\)
is also equipped with an action of \(\mO_{F,(p)}^{\times,+}/\ZZ_{(p)}^{\times,+}\). By Corollary \ref{corIgusaStackPELGeneralDfn}, we have an isomorphism \[\mrm{Isog}^{\circ,i}_{U^{p}}\cong \ti{\mS}_{U^{p}}(G'_{\ti{T}})_{E_{0,\ti{T}}}\times_{\Bun_{G}}\Fl_{G,\mu}^{(i-1)/\Bun_{G}},
\] which is \(\mO_{F,(p)}^{\times,+}/\ZZ_{(p)}^{\times,+}\)-equivariant. Since the action extends to \(\Delta_{U^{p}}\) on \(\ti{\mS}_{U^{p}}(G'_{\ti{T}})_{E_{0,\ti{T}}}\), it also extends on \(\mrm{Isog}^{\circ,i}_{U^{p}}\). Thus by taking simplicial colimit, we have an extension of the action of \(\Delta_{U^{p}}\)
on \(\tIgs^{\prime,\circ,coh}_{U^{p}}\).
\end{proof}
We also want to extend this action beyond the good reduction locus.
\begin{cor}\label{corExtendDeltaAction}
The action of \(\Delta_{U^{p}}\)
on \(\tIgs^{\prime,\circ,coh}_{U^{p}}\) extends to \(\tIgs^{\prime}_{U^{p}}\)
such that \(\tIgs'_{U^{p}}\to \Bun_{G'_{\emptyset},\mu_{\emptyset}}\) factors through \(\tIgs_{U^{p}}\to \tIgs_{U^{p}}/\Delta_{U^{p}}\),
and for any \(T\subset \Sigma_{\infty} \) the reduction map \(\mrm{red}_{\ti{T}}:\ti{\mS}_{U^{p}}(G'_{\ti{T}})_{E_{0,\ti{T}}}\to \widetilde{\Igs}^{\prime}_{U^{p}}\) is \(\Delta_{U^{p}}\)-equivariant. 
\end{cor}
\begin{proof}
We will use notation from the proof of Proposition \ref{propDeltaAction}. For \(i\in\NN\), let \(\mrm{Isog}^{i}\) denote the \(i\)-th fiber product  of \(\ti{\mS}_{U^{p}}(G'_{\ti{T}})_{E_{0,\ti{T}}}\)
over \(\tIgs^{\prime}_{U^{p}}\). 
By Corollary \ref{corIgusaStackPELGeneralDfnExtend}, \[\mrm{Isog}^{2}_{U^{p}}\cong \ti{\mS}_{U^{p}}(G'_{\ti{T}})_{E_{0,\ti{T}}}\times_{\Bun_{G}}\Fl_{G,\mu}\cong \Fl_{G,\mu}\times_{\Bun_{G}}\ti{\mS}_{U^{p}}(G'_{\ti{T}})_{E_{0,\ti{T}}}.\]
Using the first isomorphism, we have an action 
of \(\Delta_{U^{p}}\)
on \(\mrm{Isog}_{U^{p}}^{2}\) induced from the first factor \(\ti{\mS}_{U^{p}}(G'_{\ti{T}})_{E_{0,\ti{T}}}\).
We claim that the projection to the second factor \[\mrm{pr}_{2}:\mrm{Isog}^{2}_{U^{p}}
\to \ti{\mS}_{U^{p}}(G'_{\ti{T}})_{E_{0,\ti{T}}}
\] is \(\Delta_{U^{p}}\)-equivariant. Equivalently, we need to show that for \(x\in\Delta_{U^{p}}\), 
the map \[(\pr_{2},\pr_{2}\circ x):\mrm{Isog}^{2}_{U^{p}}\to \ti{\mS}_{U^{p}}(G'_{\ti{T}})_{E_{0,\ti{T}}}\times_{\Spd(E_{0,\ti{T},\p_{\ti{T}}})} \ti{\mS}_{U^{p}}(G'_{\ti{T}})_{E_{0,\ti{T}}}
\] factors through the closed subspace given by the graph of \(x\) \[\Gamma_{x}:\ti{\mS}_{U^{p}}(G'_{\ti{T}})_{E_{0,\ti{T}}}\to 
\ti{\mS}_{U^{p}}(G'_{\ti{T}})_{E_{0,\ti{T}}}\times_{\Spd(E_{0,\ti{T},\p_{\ti{T}}})} \ti{\mS}_{U^{p}}(G'_{\ti{T}})_{E_{0,\ti{T}}}
\]

Consider the natural map \(\mrm{Isog}^{2}_{U^{p}}\to \Spd(E_{0,\ti{T},\p_{\ti{T}}})\times \Spd{E_{0,\ti{T},\p_{\ti{T}}}}\). Then we can prove the factorization using the argument of \cite{Kim2025uniquenessfunctorialityigusastacks}. More precisely, using the action of \((\varphi,1)\)
on \(\mrm{Isog}_{U^{p}}^{2}\), it suffices to show the factorization when restricted to the fiber of \(\mrm{Isog}^{2}_{U^{p}}\) over \(\Spd(E_{0,\ti{T},\p_{T}})\times_{\Spd(k_{\ti{T}})} \Spd(E_{0,\ti{T},\p_{T}})\)
where \(k_{\ti{T}}\)
denotes the residue field of \(E_{0,\ti{T},\p_{T}}\).
Now by \cite[Proposition 10.8]{Kim2025uniquenessfunctorialityigusastacks}, 
it suffices to show the factorization when restricted to the fiber of \(\mrm{Isog}^{2}_{U^{p}}\) over \(\Delta:\Spd(E_{0,\ti{T},\p_{T}})\hookrightarrow \Spd(E_{0,\ti{T},\p_{T}})\times \Spd(E_{0,\ti{T},\p_{T}})\), and the latter is a limit of smooth rigid varieties over \(\Spa(E_{0,\ti{T},\p_{T}})\) by \cite[Proposition 10.2]{Kim2025uniquenessfunctorialityigusastacks}. Moreover, we have the factorization over the good reduction locus by Proposition \ref{propDeltaAction}, and thus by \cite[Lemma 2.1.4]{Conrad1999irreducible}, we have the desired factorization.
\end{proof}

\begin{dfn}
We define the \emph{Igusa stack} for \(\mS_{U^{p}}(G'_{\ti{T}})\) 
to be the \(v\)-stack defined
by \[\Igs_{U^{p}}^{\prime}:=\left[
\widetilde{\Igs}_{U^{p}}^{\prime}/\Delta_{U^{p}}
\right], 
\] where the \(\Delta_{U^{p}}\)-action is given 
by 
Corollary \ref{corExtendDeltaAction}.
\end{dfn}
\begin{thm}\label{thmUnitaryCartesianDiagram}
We define \(\Igs^{\prime}:=\varprojlim_{U^{p}}\Igs^{\prime}_{U^{p}}\).
Then for any \(T\subset\Sigma_{\infty}\), \(\Igs^{\prime}|_{\Bun_{G'_{\ti{T}},\mu_{\ti{T}}}}\)	
gives an Igusa stack for \(\mS(G'_{\ti{T}})_{E_{0,\ti{T}}}\) (Definition \ref{dfnGeneralIgus}). 

\end{thm}
\begin{proof}
This follows immediately from Corollary 
\ref{corIgusaStackPELGeneralDfnExtend}, Corollary \ref{corExtendDeltaAction} and (\ref{alignFromtiStoSQuotientInfi}).
\end{proof}

\subsection{Igusa stack for quaternionic Shimura varieties}\label{subsecQuaternionicIgusa}

We now construct an Igusa stack \(\Igs_{K^{p}}\) that works for \(\mS(G_{T})\) for all \(T\subset \Sigma_{\infty}\) using Proposition \ref{propRelateTwoIgusa}. We start by identifying the local objects.
\begin{notation}
We write \(G:=G_{\emptyset}\) and \(\mu:=\mu_{\emptyset}\).
\end{notation}
\begin{lem}
For \(S\in\Perf\),
\(\Bun_{G}(S)\) is equivalent to the groupoid of rank \(2\) vector bundles over \(\mX_{\rF,F}(S)\). Here \(\mX_{\rF,F}(S)\)
is defined in Definition \ref{dfnFFcurve}. 
\end{lem}
\begin{proof}
This follows from \(B(\mrm{Res}_{F/\QQ}\GL_{2})\cong \mrm{Res}_{F/\QQ}(B\GL_{2})\). 
\end{proof}
\begin{notation}
We have by definition, \(\mX_{\rF,F}(S)\cong \coprod_{\p\in\Sigma_{p}}\mX_{\rF,F_{\p}}(S)\).
For any \(T\subset \Sigma_{\infty}\), recall that \(r_{\p,T}:=\#(T\cap \Sigma_{\infty/\p})\). We define \(\mF_{T}\)
to be the rank \(2\) vector bundle over \(\mX_{\rF,F}(S)\) such that
\[\mF_{T}(S)|_{\mX_{\rF,F_{\p}}(S)}:=
\begin{cases}
	\mO_{\mX_{\rF,F_{\p}}(S)}(r_{\p,T}/2)^{\oplus 2},&2\mid r_{\p,T},\\
	\mO_{\mX_{\rF,F_{\p}}(S)}(r_{\p,T}/2),&2\nmid r_{\p,T},
\end{cases}\]
\end{notation}
\begin{rmk}
Here we recall that the indecomposable vector bundles on the Fargues-Fontaine curve is classified by \(\lambda\in\QQ\),
by Dieudonné-Manin classification. See \cite[Théorème 8.2.10]{FarguesFontaine2019courbes}.
\end{rmk}

\begin{lem}\label{lemIdentifyQuaternion}
For any \(T\subset \Sigma_{\infty}\),
there is an isomorphism \(\Bun_{G}\cong\Bun_{G_{T}}\), mapping \(\mF_{T}\in \Bun_{G}\) to the trivial \(G_{T}\)-torsor in \(\Bun_{G_{T}}\), which induces an open immersion \(\iota_{T}:\Bun_{G_{T},\mu_{T}}\hookrightarrow \Bun_{G,\mu}\).
\end{lem}
\begin{proof}
This is a general result (\cite[Corollary III.4.3]{FarguesScholze2021geometrization}). Concretely,
the functor is given by \(\mF\in \Bun_{G}(S)\)
to \(\underline{\mrm{Isom}}_{\mX_{\rF,F}(S)}(\mF,\mF_{T})\), which carries an action of \(G_{T}\) induced by the action on \(\mF_{T}\).
\end{proof}
\begin{thm}\label{thmCartesianDiagramQaternion}
There is a \(v\)-stack \(\Igs_{K^{p}}\) with a reduced Hodge-Tate morphism \(\bar{\pi}_{\HT}:\Igs_{K^{p}}\to \Bun_{G,\mu}\) such that for any 
\(T\subset \Sigma_{\infty}\), there is a Cartesian diagram \[\begin{tikzcd}
\mS_{K^{p}}(G_{T})_{E_{0,T}}\times_{E_{0,T,\p_{T}}}\Spd(E_{0,\ti{T},\p_{\ti{T}}})
\arrow[r,"\pi_{\HT,T}"]\arrow[d,"\red_{T}"] & \Fl_{G_{T},\mu_{T},E_{0,T}}\times_{E_{0,T,\p_{T}}}\Spd(E_{0,\ti{T},\p_{\ti{T}}})\arrow[d,"\BL_{T}"]\\
\iota_{T}^{*}
(\Igs_{K^{p}})
\arrow[r,"\bar{\pi}_{\HT,T}"]
& \Bun_{G_{T},\mu_{T}}.
\end{tikzcd}
\]
Here
\(E_{0,T}\subset \CC\)
denotes the reflex field of \((G_{T},h_{T})\), \( E_{0,T,\p_{T}} \) denotes the \( p \)-adic completion along \( E_{0,T}\subset\CC\cong^{\iota} \bar{\QQ}_{p} \), \(\mS(G_{T})_{E_{0,T}}\)
is the model of \(\mS(G_{T})\) over \(\Spd(E_{0,T,\p_{T}})\) as in Notation \ref{notationGenealShimura}, and
\(E_{0,T,\p_{T}}\) is as in Corollary \ref{corIgusaStackPELGeneralDfn}.
\end{thm}
\begin{proof}
Applying \cite[Theorem 11.4]{Kim2025uniquenessfunctorialityigusastacks} to the embedding \((G,\mu)\hookrightarrow (G',\mu')\), by Theorem \ref{thmUnitaryCartesianDiagram}, we have an Igusa stack \(\Igs\) 
for \(\mS(G)_{\QQ}\) which admits a natural immersion \(\Igs\hookrightarrow \Igs'\). 
Then the Cartesian diagram holds when \(T=\emptyset\).
Put \(\Igs:=\varprojlim_{K^{p}}\Igs_{K^{p}}\). 
By Proposition \ref{propRelateTwoIgusa}, we have an action of 
\(\mscr{A}_{G}:=G(\mA_{f})/\overline{Z(\QQ)}\) on \(\Igs\), and
\[(\Igs\times \mscr{A}_{G_{\emptyset}'})/\mscr{A}_{G_{\emptyset}}\cong \Igs'\times_{\Bun_{G'}}\Bun_{G}.
\] In particular, we have a \(G'(\mA_{f}^{p})\times E_{p}^{\times}\)-equivariant morphism \[\Igs'\times_{\Bun_{G'}}\Bun_{G}\to \mscr{A}_{G'_{\emptyset}}/\mscr{A}_{G_{\emptyset}}\cong \mA_{E,f}^{\times}/(E^{\times}\mA_{F,f}^{\times}),\]
such that the fiber along \(1\hookrightarrow \mA_{E,f}^{\times}/(E^{\times}\mA_{F,f}^{\times})\)
is isomorphic to \(\Igs\). 

Along the embeddings \(\iota_{T}:\Bun_{G_{T},\mu_{T}}\hookrightarrow \Bun_{G,\mu}\)
and \(\iota_{\ti{T}}:\Bun_{G'_{\ti{T}},\mu_{\ti{T}}}\hookrightarrow \Bun_{G',\mu}\) (Lemma \ref{lemIdenfificationBL} (3)), the map \(\Bun_{G,\mu}\to \Bun_{G',\mu}\)
induces \(\Bun_{G_{T},\mu_{T}}\to \Bun_{G'_{\ti{T}},\mu'_{T}}\).

By Theorem \ref{thmUnitaryCartesianDiagram}, we have an induced \(G'_{\ti{T}}(\mA_{f})\)-equivariant morphism \[\mS(G'_{\ti{T}})\to \mA_{E,f}^{\times}/(E^{\times}\mA_{F,f}^{\times}),
\] such that the fiber at \(1\) is isomorphic to \(U:=\iota_{T}^{*}(\Igs)\times_{\Bun_{G_{T},\mu_{T}}}\FLTT{T}{T}\). Then \(U\hookrightarrow \mS(G'_{\ti{T}})\)
is an immersion, is \(G_{T}(\mA_{f})\)-stable, and \(U\cong \pi_{0}(U)\times_{\pi_{0}(\mS(G_{\ti{T}}'))}\mS(G'_{\ti{T}})\).

By \cite[\S 2]{Deligne1979varietesInterpretationModulaire}, \((\mS(G_{T})\times \mscr{A}_{G'_{\ti{T}}})/\mscr{A}_{G_{T}}\cong \mS(G'_{\ti{T}})\), and \(G'_{\ti{T}}(\mA_{f})\)
acts transitively on \(\pi_{0}(\mS(G_{T}))\). This implies that we can find \(g\in Z'(\mA_{f})\), such that \(g(\mS(G_{T}))\subset U\). This induces a \(G'_{\ti{T}}\)-equivariant morphism \(\mA_{E,f}^{\times}/(E^{\times}\mA_{F,f}^{\times})\twoheadrightarrow \mA_{E,f}^{\times}/(E^{\times}\mA_{F,f}^{\times})\), which is thus forced to be an isomorphism. This shows that \(g(\mS(G_{T}))\subset U\) is actually an isomorphism, as desired.

The descent to \( \Spd(E_{0,\ti{T},\p_{\ti{T}}}) \) follows from the property of canonical models.
See \cite[\S 2.8]{TianXiao2016goren} for a discussion of canonical model of weak Shimura datum.
\end{proof}



\begin{lem}\label{lemCokerAGtoAGprime}
	Let \(C\) and \(C'\) denote the cocenters of \(G_{T}\) and \(G'_{\ti{T}}\) respectively.
	The natural maps \(\mscr{A}_{G_{T}}\to \mscr{A}_{G_{\ti{T}}'}\), \(\mscr{A}_{Z}\to \mscr{A}_{Z'}\) and \(\mscr{A}_{C}\to \mscr{A}_{C'}\)
	is injective,
	and we have natural isomorphisms \[
	\mscr{A}_{G_{\ti{T}}'}/\mscr{A}_{G_{T}}\cong \mscr{A}_{C'}/\mscr{A}_{C}\cong \mA_{E,f}^{\times}/(E^{\times}\mA_{F,f}^{\times}).\] 
\end{lem}
\begin{proof}
	Note that \(G_{\ti{T}}'/G_{T}\cong Z'/Z\cong C'/C\cong T_{E}/T_{F}\).
	Consider 
	\[\begin{tikzcd}
		1\arrow[r] &
		Z(\mA_{f})\arrow[r] 
		&
		Z'(\mA_{f})\arrow[r] 
		& Z'/Z(\mA_{f})
		\arrow[r]
		& 1
		\\
		1\arrow[r] &
		Z(\QQ)\arrow[r] \arrow[u,hook]
		&
		Z'(\QQ)\arrow[r]\arrow[u,hook]
		& Z'/Z(\QQ)\arrow[u,hook]
		\arrow[r]
		& 1.
	\end{tikzcd}\]
	Note that by Lemma \ref{lemDiscreteSubgroup}, \(Z'/Z(\QQ)\cong E^{\times}/F^{\times}\)
	is discrete in \(Z'/Z(\mA_{f})\),
	so \[Z'(\QQ)/Z(\QQ)\cong (Z'/Z)(\QQ)\cong \overline{(Z'/Z)(\QQ)}\cong \overline{Z'(\QQ)}/\overline{Z(\QQ)},\]
	and thus 
	we have \(\overline{Z'(\QQ)}\cap Z(\mA_{f})\cong \overline{Z(\QQ)}\).
	This implies the desired injectivity.

	Since \(G^{\prime,der}_{\ti{T}}\)
	is simply connected, we know by \cite{Kottwitz1984StableTF} 
	that \(H^{1}(\QQ_{p},G^{\prime,der}_{\ti{T}})\cong 0\), and thus \(G'_{\ti{T}}(\mA_{f})\to C'(\mA_{f})\)
	is surjective. 	
	We are left to show that \(\overline{Z'(\QQ)}G_{T}(\mA_{f})\to\overline{C'(\QQ)}C(\mA_{f})\) along \(G'_{\ti{T}}\to C'\) is surjective. By the discreteness of \((Z'/Z)(\QQ)\) in \((Z'/Z)(\mA_{f})\),
	we know that \(\overline{Z'(\QQ)}G_{T}(\mA_{f})=Z'(\QQ)G_{T}(\mA_{f})\), 
	and similarly, \(\overline{C'(\QQ)}C(\mA_{f})=C'(\QQ)C(\mA_{f})\). As \(Z'\times^{Z}C\cong C'\), for any \(x\in C'(\QQ)\),
	we can find \(y\in Z'(\QQ)\),
	such that \(xy^{-1}\in C(\QQ)\),
	which then admits a lift in \(G_{T}(\mA_{f})\).
\end{proof}

\begin{lem} \label{lemDiscreteSubgroup}
Let \(E/F\) be a CM extension, then \(E^{\times}/F^{\times}\subset \mA_{E,f}^{\times}/\mA_{F,f}^{\times}\)
	is a discrete subgroup.
\end{lem}
\begin{proof}
	By the unit theorem, \(\mO_{F}^{\times}\) and \(\mO_{E}^{\times}\)
	are finitely generated abelian groups of the same rank, and thus
	\(\mO_{E}^{\times}/\mO_{F}^{\times}\)
	is a finite group. Therefore, 
	\(\mO_{E}^{\times}/\mO_{F}^{\times}\cong (E^{\times}\cap F^{\times}\hat{\mO}_{E}^{\times})/F^{\times}\subset \hat{\mO}_{E}^{\times}\mA_{F,f}^{\times}/\mA_{F,f}^{\times}\)
	is discrete. We know \(\Cl_{F}=\mA_{F,f}^{\times}/F^{\times}\hat{\mO}_{F}^{\times}\)
	is finite, so \(F^{\times}\hat{\mO}_{E}^{\times}\subset \mA_{F,f}^{\times}\hat{\mO}_{E}^{\times}\) is of finite index. Thus 
	\((E^{\times}\cap \hat{\mO}_{E}^{\times}\mA_{F,f}^{\times})/F^{\times}\subset \hat{\mO}_{E}^{\times}\mA_{F,f}^{\times}/\mA_{F,f}^{\times}\)
	is discrete, so \(E^{\times}/F^{\times}\subset \mA_{E,f}^{\times}/\mA_{F,f}^{\times}\)
	is discrete.
\end{proof}
\begin{cor}\label{corIgsuaArtin}
(1)
The \(v\)-stack \(\Igs_{K^{p}}\) in Theorem \ref{thmCartesianDiagramQaternion}
is an Artin \(v\)-stack, is \(\ell\)-cohomologically smooth of dimension \(0\) for any \(\ell\ne p\).

(2) The reduced Hodge-Tate period map \(\bar{\pi}_{\HT}:\Igs_{K^{p}}\to \Bun_{G}\) is separated, representable by locally spatial diamonds, partially proper and locally \(\dim.\mrm{tr}(f)<\infty\). 

(3) The dualizing complex in \(D_{\et}(\Igs_{K^{p}}\times \Spd(\bar{\FF}_{p}),\ZZ/N)\) is isomorphic to \(\ZZ/N[0]\) for \(p\nmid N\). 
\end{cor}
\begin{proof}
The proof is verbatim the same as \cite[Corollary 8.18]{Zhang2023pel} and \cite[Theorem 8.3.1]{DanielsHoftenKimZhang2024igusa}. 
The partial properness of \(\bar{\pi}_{\HT}\) follows from the partial properness of \(\mS_{K^{p}}(G_{\emptyset})\)
and \(\Fl_{G,\mu}\), and \cite[Proposition 18.7 (v)]{Scholze2017etaleDiamond}. 
\end{proof}


\subsection{Application: \texorpdfstring{\(\ell\)}{l}-adic Jacquet-Langlands correspondence}\label{subsecAppliLadicJLcorr}
We thank Xiangqian Yang for very helpful discussions around this topic.
The goal of this subsection is to prove the following theorem:

\begin{thm}\label{thmDirectSummandDifferentShiV}
Let \(T\subset T'=T\coprod I\subset \Sigma_{\infty}\). Assume that for any \(\p\in\Sigma_{p}\), \(\#(I\cap \Sigma_{\infty/\p})\) is even. In particular, we can identify \(G_{T,\QQ_{p}}\cong G_{T',\QQ_{p}}\). 

Let \(\Lambda=\ZZ/N[\sqrt{p}]\) for \(p\nmid N\).
Then \(R\Gamma_{c,\et}(\mS_{K^{p}}(G_{T'}),\Lambda)[d-\#(T')]\) can be realized as a direct summand of 
\(R\Gamma_{c,\et}(\mS_{K^{p}}(G_{T}),\Lambda)[d-\#(T)]\) as smooth \(G_{T}(\QQ_{p})\)-representations over \(\TT^{S}\). Moreover, the construction is natural in \(\Lambda\) and in \(K^{p}\). 
\end{thm}
We will use the spectral action of \cite{FarguesScholze2021geometrization}. Let us review the results loc. cit.
\begin{notation}
Let \(G\) be a reductive group over \(\QQ_{p}\). 
We denote 
\(\Bun_{G,\bar{\FF}_{p}}:=\Bun_{G}\times_{\Spd(\FF_{p})}\Spd(\bar{\FF}_{p})\). Equivalently, following \cite{FarguesScholze2021geometrization}, we are working over \(\Perf_{\bar{\FF}_{p}}\)
instead of \(\Perf\).
\end{notation}
\begin{dfn}
We define \(\mfk{I}_{K^{p}}:=j_{\mu,!}\circ\bar{\pi}_{\HT,!}\Lambda \in D_{\et}(\Bun_{G},\Lambda)\) for \(\Igs_{K^{p}}\xrightarrow{\bar{\pi}_{\HT}} \Bun_{G,\mu}\xhookrightarrow{j_{\mu}} \Bun_{G}\)
as in Theorem \ref{thmCartesianDiagramQaternion}. Note that \(j_{\mu}\)
is an open immersion, and we will identify \(D_{\et}(\Bun_{G,\mu},\Lambda)\)
as a subcategory of \(D_{\et}(\Bun_{G},\Lambda)\)
via \(j_{\mu,!}\). By abuse of notation, we denote by \(\mfk{I}_{K^{p}}\in D_{\et}(\Bun_{G,\bar{\FF}_{p}},\Lambda)\) the pull-back of \(\mfk{I}_{K^{p}}\)
to \(\Bun_{G,\bar{\FF}_{p}}\).

By functoriality, \(\mfk{I}_{K^{p}}\)
carries an action of \(\TT^{S}\),
and is natural in \(K^{p}\).
\end{dfn}
Consider the diagram \[
\begin{tikzcd}
{[\mS_{K^{p}}(G_{T})/G_{T}(\QQ_{p})]}
\arrow[r,"\pi_{\HT,T}"]\arrow[d,"\red_{T}"] & {[\FLTT{T}{T}/G_{T}(\QQ_{p})]}
\arrow[d,"\BL_{T}"]\arrow[r,"h"] & {[\Spd(\CC_{p})/G_{T}(\QQ_{p})]} \\
\iota_{T}^{*}
(\Igs_{K^{p}})
\arrow[r,"\bar{\pi}_{\HT,T}"]
& \Bun_{G_{T},\mu_{T}}.
\end{tikzcd}
\]
Then by proper base change (\cite[Proposition 22.15]{Scholze2017etaleDiamond} and Corollary \ref{corIgsuaArtin} (2)), up to identifying \(D_{\et}([\Spd(\CC_{p})/G_{T}(\QQ_{p})],\Lambda)\)
with the category of smooth \(G_{T}(\QQ_{p})\)-representations,
\begin{align}\label{alignCohoToCorrespondence}
R\Gamma_{c,\et}(\mS_{K^{p}}(G_{T}),\Lambda)\cong h_{!}\circ \pi_{\HT,T,!}(\Lambda)
\cong h_{!}\circ \pi_{\HT,T,!}\circ \red_{T}^{*}(\Lambda)
\cong h_{!}\circ \BL_{T}^{*}(\mfk{I}_{K^{p}}). 
\end{align}
\begin{dfn}[Hecke action]
For \(I\) a finite set, \(V\in\mrm{Rep}_{\Lambda}(\hat{G}^{I})\),
let \(T_{V}:D_{\et}(\Bun_{G,\bar{\FF}_{p}},\Lambda)\to D_{\et}(\Bun_{G,\bar{\FF}_{p}},\Lambda)\) denote the Hecke action  defined in \cite[Proposition IX.2.1]{FarguesScholze2021geometrization}.

Note that \cite{FarguesScholze2021geometrization} uses the category \(D_{\mrm{lis}}\), which we can identify with \(D_{\et}\) using \cite[Proposition VII.6.6]{FarguesScholze2021geometrization}. By the convention of the Hecke action on \cite[p.268 \& Proposition VII.5.2]{FarguesScholze2021geometrization}, for \(\mF\in D_{\et}(\Bun_{G,\bar{\FF}_{p}},\Lambda)\) and \(V\in\mrm{Rep}_{\Lambda}(\hat{G}^{I})\),
\begin{align}\label{alignHeckeAction}
u^{*}(T_{V}(\mF))\cong p_{!}(r^{*}(\mcal{S}_{V})\otimes q^{*}(u^{*}(\mF)))\in D_{\et}(\Bun_{G}\times \Spd(\CC_{p})),
\end{align}
where \(\mcal{S}_{V}\)
is the Satake sheaf associated to \(V\)
via \cite[Theorem VI.0.2]{FarguesScholze2021geometrization}, and 
 the morphisms are as in the diagram \[
\begin{tikzcd}
	&
	\mrm{Hck}_{G}^{I}\times_{\mrm{Div}^{1}_{\QQ_{p}}}\Spd(\CC_{p})\arrow[rd,"p"]
	\arrow[ld,"q"]\arrow[r,"r"]
	& \mcal{H}\!\mrm{ck}^{I}_{G}\times_{\mrm{Div}^{1}_{\QQ_{p}}}\Spd(\CC_{p})
	\\
\Bun_{G}\times \Spd(\CC_{p})\arrow[d,"u"]
& & \Bun_{G}\times \Spd(\CC_{p})\arrow[d,"u"]\\
\Bun_{G,\bar{\FF}_{p}}
& & \Bun_{G,\bar{\FF}_{p}}.
\end{tikzcd}
\] 
Following \cite{FarguesScholze2021geometrization}, \(\mcal{H}\!\mrm{ck}^{I}_{G}\)
denotes the local Hecke stack (\cite[Definition VI.1.6]{FarguesScholze2021geometrization}),
and \(\mrm{Hck}_{G}^{I}\)
denotes the global Hecke stack (\cite[16]{FarguesScholze2021geometrization}). 
Note that (\ref{alignHeckeAction}) characterizes \(T_{V}(\mF)\),
as \(u^{*}\) is an equivalence by \cite[Proposition V.2.3]{FarguesScholze2021geometrization}.
\end{dfn}
\begin{eg}\label{egMinusculHecke}
If \(V=V_{\mu}\in\mrm{Rep}_{\Lambda}(\hat{G})\)
for a minuscule \(\mu\in X_{\bullet}(T)^{+}\), then \(j_{\mu}:\mcal{H}\!\mrm{ck}_{G,\mu,\bar{\FF}_{p}}\hookrightarrow \mcal{H}\!\mrm{ck}_{G,\bar{\FF}_{p}}\) is a closed embedding, and thus 
\[{}^{p}j_{\mu,!}(\Lambda[d_{\mu}])\cong j_{\mu,!*}(\Lambda[d_{\mu}])\cong {}^{p}j_{\mu,*}(\Lambda[d_{\mu}])\;(d_{\mu}:=\langle 2\rho,\mu \rangle)
\] is a simple object in the Satake category, which is then forced to be the Satake sheaf associated to \(V_{\mu}\) by considering the constant term functor \(\mrm{CB}_{T}[\deg]\). Note that by the proof of \cite[Theorem VI.11.1]{FarguesScholze2021geometrization}, \(\mrm{CB}_{T}[\deg]\)
corresponds via Satake equivalence to the embedding \(\hat{T}\hookrightarrow \hat{G}\), from which one can read off the highest weight. Therefore, \[j_{1}^{*}T_{V_{\mu}}(\mF)\cong h_{!}\circ \BL_{\mu}^{*}(\mF)[d_{\mu}]
\] for the diagram \[
\begin{tikzcd}
{[\mrm{Gr}_{G,\mu}\times\Spd(\CC_{p})/G(\QQ_{p})]}
\arrow[r,"h"]\arrow[d,"\BL_{\mu}"]
& {[\Spd(\CC_{p})/G(\QQ_{p})]}\arrow[r,hook,"j_{1}"]
& \Bun_{G}\times\Spd(\CC_{p})\\
\Bun_{G}\times \Spd(\CC_{p}).
\end{tikzcd}
\]
\end{eg}
\begin{eg}\label{egCenterTwistHecke}
If \(\mu\in X_{\bullet}(T)^{+}\)
factors through the center of \(G\), then \(\mu\) is minuscule,
and \(\mrm{Gr}_{G,\mu,\bar{\FF}_{p}}\cong \mrm{Div}^{1}_{\QQ_{p}}\).  
Then \( T_{V} \) 
induces an equivalence that is the pull-back along the automorphism \( \Bun_{G}\cong \Bun_{G} \)
induced by twisting by \( \mu \). 
\end{eg}

\begin{dfn}[Spectral action]
By \cite[Theorem X.0.2]{FarguesScholze2021geometrization}, if \(\ell\nmid |\pi_{0}(Z(G))|\),
we have an action of \(\Perf(Z^{1}(W_{\QQ_{p}},\hat{G})_{\Lambda}/\hat{G})\) on \(D_{\et}(\Bun_{G,\bar{\FF}_{p}},\Lambda)\), which we denote as \[\mrm{Act}_{M}:D_{\et}(\Bun_{G,\bar{\FF}_{p}},\Lambda)\to D_{\et}(\Bun_{G,\bar{\FF}_{p}},\Lambda),\;M\in \Perf(Z^{1}(W_{\QQ_{p}},\hat{G})_{\Lambda}/\hat{G}).\]

For any \(V\in\mrm{Rep}_{\Lambda}(\hat{G})\), 
let \(\ti{V}\)
denote the vector bundle on \(Z^{1}(W_{\QQ_{p}},\hat{G})_{\Lambda}/\hat{G}\)
given by pulling back along \[Z^{1}(W_{\QQ_{p}},\hat{G})_{\Lambda}/\hat{G}\to \Spec(\Lambda)/\hat{G}.\]
Then \(\mrm{Act}_{\ti{V}}=T_{V}
\). 
\end{dfn}

We go back to the setting of Shimura varieties. In our setting, \(G_{T,\QQ_{p}}\cong \prod_{\p\in\Sigma_{p}}\mrm{Res}_{F_{\p}/\QQ_{p}}B_{T}^{\times}\). Then we have the following:
\begin{lem}\label{lemWeilRestrictionLpar}
Let \(G:=G_{T,\QQ_{p}}\). 
We have an identification \(Z^{1}(W_{\QQ_{p}},\hat{G})/\hat{G}\cong \prod_{\p} (Z^{1}(W_{F_{\p}},\GL_{2})/\GL_{2})\)
such that the following diagram commutes: \[\begin{tikzcd}\prod_{\p} (Z^{1}(W_{F_{\p}},\GL_{2})/\GL_{2})
\arrow[r,"\sim"]\arrow[d] & 
Z^{1}(W_{\QQ_{p}},\hat{G})/\hat{G}\arrow[d]\\
\prod_{\p\in\Sigma_{p}}\Spec(\Lambda)/\GL_{2}
\arrow[r]
& \Spec(\Lambda)/\hat{G},
\end{tikzcd} 
\] where the bottom arrow is induced by \(\prod_{\p\in\Sigma_{p}}\GL_{2}\hookrightarrow \hat{G}\cong \prod_{\tau\in\Sigma_{\infty}}\GL_{2}\), which is obtained by taking product over \(\p\in\Sigma_{p}\) of
the diagonal embedding \(\GL_{2}\hookrightarrow \prod_{\tau\in\Sigma_{\infty/\p}}\GL_{2}\).
\end{lem}
\begin{proof}
This is the content of \cite[Proposition IX.6.3]{FarguesScholze2021geometrization}. 
\end{proof}
Consider the diagram \[
\begin{tikzcd}
{[\FLTT{T}{T}/G_{T}(\QQ_{p})]}
\arrow[d]\arrow[r,"h"] & {[\Spd(\CC_{p})/G_{T}(\QQ_{p})]}\arrow[d,"u"] \\
{[\mrm{Gr}_{G,\mu,\bar{\FF}_{p}}/G_{T}(\QQ_{p})]}
\arrow[d,"\BL_{T}"]\arrow[r,"q"] & {[\Spd(\bar{\FF}_{p})/G_{T}(\QQ_{p})]}\arrow[r,hook,"j_{b_{\mu_{T}}}"] & \Bun_{G_{T},\bar{\FF}_{p}}\cong \Bun_{G,\bar{\FF}_{p}}\\
\Bun_{G_{T},\mu_{T},\bar{\FF}_{p}},
\end{tikzcd}
\] and by (\ref{alignCohoToCorrespondence}) and Example \ref{egMinusculHecke}, 
\begin{align}\label{alignCohoVSSpectral}
R\Gamma_{c,\et}(\mS_{K^{p}}(G_{T}),\Lambda)[d_{\mu_{T}}]
\cong u^{*}\circ j_{b_{\mu_{T}}}^{*}(T_{V_{\mu_{T}}}(\mfk{I}_{K^{p}}))
\cong u^{*}\circ j_{b_{\mu_{T}}}^{*}(\mrm{Act}_{\ti{V}_{\mu_{T}}}(\mfk{I}_{K^{p}})).
\end{align}
Here \(d_{\mu_{T}}=d-\#T\), and \(V_{\mu_{T}}\) is the highest weight representation of \(\hat{G}\cong \prod_{\tau\in\Sigma_{\infty}}\GL_{2}\). Concretely, \(V_{\mu_{T}}\cong \bigotimes_{\tau\in T^{c}}\mrm{Std}\circ \pr_{\tau}\)
where \(\pr_{\tau}\)
denotes the projection to the \(\tau\)-factor, and \(\mrm{Std}\)
denotes the standard \(2\)-dimensional representation. 
Note that \(Z\cong \Res_{F/\QQ}\GG_{m}\)
is connected, so the condition \(\ell\nmid|\pi_{0}(Z(G))|\) 
is satisfied.

\begin{proof}[Proof of Theorem \ref{thmDirectSummandDifferentShiV}]
By induction, we can reduce to the case that \(\#(I)=2\), and \(I\cap \Sigma_{\infty/\p_{0}}\ne \emptyset\) for a unique \(\p_{0}\in\Sigma_{p}\). 
Using Lemma \ref{lemWeilRestrictionLpar}, we denote \[Z^{1}(W_{\QQ_{p}},\hat{G})/\hat{G}\xrightarrow{s} \prod_{\p\in\Sigma_{p}}\Spec(\Lambda)/\GL_{2}\xrightarrow{\pr_{\p}}\Spec(\Lambda)/\GL_{2}.
\] Then 
\[\ti{V}_{\mu_{T}}\cong s^{*}\left(\bigotimes_{\p}\pr_{\p}^{*}(\mrm{Std}^{\otimes\#(T^{c}\cap \Sigma_{\infty/\p})})\right)
\]
Using \(\mrm{Std}^{\otimes 2}\cong \det\oplus \mrm{Sym}^{2}(\mrm{Std})\), we have a decomposition \[\ti{V}_{\mu_{T}}\cong 
\left(\ti{V}_{\mu_{T'}}\otimes s^{*}(\pr_{\p_{0}}^{*}(\det))\right)\oplus 
\left(\ti{V}_{\mu_{T'}}\otimes s^{*}(\pr_{\p_{0}}^{*}(\Sym^{2}\mrm{Std}))\right),
\]
and thus \[
\mrm{Act}_{\ti{V}_{\mu_{T}}}(\mfk{I}_{K^{p}})
\cong 
\mrm{Act}_{s^{*}(\pr_{\p_{0}}^{*}(\det))}\circ
\mrm{Act}_{\ti{V}_{\mu_{T'}}}(\mfk{I}_{K^{p}})
\oplus \mrm{Act}_{s^{*}(\pr_{\p_{0}}^{*}(\Sym^{2}\mrm{Std}))}\circ
\mrm{Act}_{\ti{V}_{\mu_{T'}}}(\mfk{I}_{K^{p}}).
\]
Note that \(s^{*}(\pr^{*}_{\p_{0}}(\det))\)
corresponds to the representation \(\det\circ \pr_{\p_{0}}\), which factors through the cocenter. We are thus in the setting of Example \ref{egCenterTwistHecke},
which implies that we can essentially identify 
\(\mrm{Act}_{s^{*}(\pr_{\p_{0}}^{*}(\det))}\circ
\mrm{Act}_{\ti{V}_{\mu_{T'}}}(\mfk{I}_{K^{p}})\)
with \(
\mrm{Act}_{\ti{V}_{\mu_{T'}}}(\mfk{I}_{K^{p}})\). 
By (\ref{alignCohoVSSpectral}),
this implies that as smooth \(G_{T}(\QQ_{p})\)-representations, \(R\Gamma_{c,\et}(\SGTKp{T},\Lambda)[d-\#T]\)
contains \(R\Gamma_{c,\et}(\SGTKp{T'},\Lambda)[d-\#T']\)
as a direct summand. 
\end{proof}

As a corollary of the proof, we also have the following: 
\begin{cor}\label{corDifferentSVfromHecke}
Let \(T\subset T'=T\coprod I\subset \Sigma_{\infty}\).
Then
there exists a \(\TT^{S}\)-linear isomorphism \(M_{I}\in \mrm{Perf}(Z^{1}(W_{\QQ_{p}},\hat{G})/\hat{G})\)
such that \[j_{b_{\mu_{T}},!}R\Gamma_{c,\et}(\SGTKp{T},\Lambda)[d-\#T]\cong \mrm{Act}_{M_{I}}(
j_{b_{\mu_{T'}},!}R\Gamma_{c,\et}(\SGTKp{T'},\Lambda)[d-\#T']).
\] 

In particular, for any ideal \(J\subset \TT^{S}\),
if \[R\Gamma_{c,\et}(\SGTKp{T},\Lambda)_{J}\ne 0,\]
then \[R\Gamma_{c,\et}(\SGTKp{T'},\Lambda)_{J}\ne 0.\] The converse is also true if for any \(\p\in\Sigma_{p}\), \(\#(I\cap \Sigma_{\infty/\p})\)
is even. 
\end{cor}
\begin{proof}
We take \(M_{I}:=s^{*}\left(\bigotimes_{\p}\pr_{\p}^{*}(\mrm{Std}^{\otimes\#(I\cap \Sigma_{\infty/\p})})\right)\). The last statement follows from Theorem  \ref{thmDirectSummandDifferentShiV}.
\end{proof}

We can go to finite level using the following lemma.
\begin{lem}\label{lemDescentCompactSupport}
We have a natural isomorphism \[R\Gamma_{c,\et}(\mS_{K^{p}}(G_{T}),\Lambda)\cong \varinjlim_{K_{p}\subset G_{T}(\QQ_{p})}R\Gamma_{c,\et}(\mS_{\Kpp}(G_{T}),\Lambda).
\] Moreover, for any pro-\(p\)-subgroup \(K_{p}\), \[R\Gamma_{c,\et}(\mS_{\Kpp}(G_{T}),\Lambda)\cong R\Gamma(K_{p},R\Gamma_{c,\et}(\mS_{K^{p}}(G_{T}),\Lambda))
\]
\end{lem}
\begin{proof}
we can write \(\SGTKp{T}\) as an increasing union of quasicompact 
open subspaces \(U_{i}\) for \(i\in\NN\), such that \(\overline{U_{\Kpp,i}}\subset U_{\Kpp,i+1}\). Denote \(f_{\Kpp,i}:\overline{U_{\Kpp,i}}\hookrightarrow U_{\Kpp,i+1}\),
and then \[R\Gamma_{\et}(\SGTKp{T},\Lambda)\cong \varinjlim_{i\in\NN}R\Gamma_{\et}(U_{\Kpp,i+1},f_{\Kpp,i,!}\Lambda).\]
Then we apply \cite[Corollary 7.18]{Scholze12}
to the tower defined by the preimage of \(U_{\Kpp,i}\), and we finish the proof by taking colimits along \(i\). 
If we denote the preimage along \(\SGTKp{T}\to\SGTK{T}{\Kpp}\) as \(f_{K^{p},i}:\overline{U_{K^{p},i}}\hookrightarrow U_{K^{p},i}\), then by
pro\'etale descent, 
\[R\Gamma_{\et}(U_{\Kpp,i},f_{\Kpp,i,!}(\Lambda))\cong R\Gamma(K_{p},R\Gamma_{\et}(U_{K^{p},i},f_{K^{p},!}(\Lambda))).\]
If \(K_{p}\)
is pro-\(p\), then \(R\Gamma(K_{p},-)\)
commutes with colimits, and thus by taking colimits along \(i\),
we have \[R\Gamma_{c,\et}(\mS_{\Kpp}(G_{T}),\Lambda)\cong R\Gamma(K_{p},R\Gamma_{c,\et}(\mS_{K^{p}}(G_{T}),\Lambda)).
\]
\end{proof}
\begin{cor}
\label{corDirectSummandDifferentShiVFiniteLevel}
Let \(T\subset T'=T\coprod I\subset \Sigma_{\infty}\). Assume that for any \(\p\in\Sigma_{p}\), \(\#(I\cap \Sigma_{\infty/\p})\) is even. In particular, we can identify \(G_{T,\QQ_{p}}\cong G_{T',\QQ_{p}}\). 

Let \(\Lambda=\ZZ/N[\sqrt{p}]\) for \(p\nmid N\), and \(K_{p}\subset G_{T}(\QQ_{p})\) is an open compact pro-\(p\) subgroup.
Then \(R\Gamma_{c,\et}(\mS_{\Kpp}(G_{T'}),\Lambda)[d-\#(T')]\) can be realized as a direct summand of 
\(R\Gamma_{c,\et}(\mS_{\Kpp}(G_{T}),\Lambda)[d-\#(T)]\). Moreover, the construction is natural in \(\Lambda\) and in \(\Kpp\). 
\end{cor}
\begin{proof}
This follows from Theorem \ref{thmDirectSummandDifferentShiV}
and Lemma \ref{lemDescentCompactSupport}.
\end{proof}

Note that the cohomology \(R\Gamma(\mS_{K^{p}}(G_{T}),\ZZ/\ell^{n})[d-\#(T)]\) does not depend 
on the choice of \(p\). So we have the freedom of choosing a suitable \(p\). This gives the following:
\begin{cor}\label{corDirectSummandVaryL}
Let \(S\) be any finite set of places of \(\QQ\).
Let \(T\subset T'=T\coprod I\subset \Sigma_{\infty}\). Assume that \(I\) is even. 
Then there exists a finite set \(S'\)
of finite places of \(\QQ\) that is disjoint from \(S\), such that the following holds:

Let \(N\in \NN\) that is prime to any \(\ell\in S'\), and \(\Lambda:=\ZZ/N\left[\prod_{\ell\in S'}\sqrt{\ell}\right]\).
For any neat open compact subgroup \(K=K^{S'}\prod_{\ell\in S'}K_{\ell}\) for \(K^{S'}\subset G(\mA_{f}^{S})\)
and \(K_{\ell}\subset G(\QQ_{\ell})\), if \(K_{\ell}\) is pro-\(\ell\) for any \(\ell\in S'\), and if \(K\) is unramified away from \(S\cup S'\), then
\(R\Gamma_{c,\et}(\mrm{Sh}_{K}(G_{T'})_{\bar{\QQ}},\Lambda)[d-\#(T')]\)
 can be realized as a direct summand of 
\(R\Gamma_{c,\et}(\mrm{Sh}_{K}(G_{T})_{\bar{\QQ}},\Lambda)[d-\#(T)]\) as \(\TT^{S\cup S'}\)-modules. Moreover, the construction is natural in \(\Lambda\) and \(K\).
\end{cor}
\begin{proof}
By induction, it suffices to consider the case when \(\#(I)=2\). Let \(I=\{\tau,\sigma:F\hookrightarrow \bar{\QQ}\subset\CC\}\). Then since \(\Gal_{\QQ}\)
acts smoothly and transitively on \(\Sigma_{\infty}\),
we can find \(g\in \Gal_{\QQ}\)
such that \(\sigma=\tau\circ g\). By Chebotarev density theorem,
we can find a prime \(p\notin S\), an embedding \(\bar{\QQ}\hookrightarrow \bar{\QQ}_{p}\) (which induces \(\Gal_{\QQ_{p}}\subset \Gal_{\QQ}\)) and \(g\in \Gal_{\QQ_{p}}\)
such that \(\sigma=\tau\circ g\). Then if we fix an identification \(\bar{\QQ}_{p}\cong \CC\)
that is compatible with the embeddings of \(\bar{\QQ}\) (which exists by the axiom of choice),
then the induced \(\tau,\sigma:F\hookrightarrow \bar{\QQ}\subset \bar{\QQ}_{p}\)
differs by an action of \(g\in \Gal_{\QQ_{p}}\), which implies that there exists \(\p\in\Sigma_{p}\)
such that \(\tau,\sigma\in \Sigma_{\infty/\p}\). We can then conclude by Corollary \ref{corDirectSummandDifferentShiVFiniteLevel}.
\end{proof}

\subsection{Product formula}
\label{subsecProductFormula}
In this section, using the Cartesian diagram,
we prove a version of the product formula (Theorem \ref{thmProductFormulaQuate}) that will be useful for us.
We will need to first understand the local picture. 




We consider the natural map \(f_{F}:\mX_{\rF,F}(S)\to \mX_{\rF}(S)\),
which
is finite \'etale of degree \(d\).

Recall that for \( \p\in\Sigma_{p} \), the stack of degree \(1\) Cartier divisor \(\mrm{Div}^{1}_{F_{\p}}\) on \(\mX_{\rF,F_{\p}}(S)\) 
as defined in \cite[Definition II.1.19]{FarguesScholze2021geometrization}
is isomorphic to \(\Spd(\breve{F}_{\p})/\varphi_{F_{\p}}^{\ZZ}\), where \( \breve{F}_{\p} \) is the completion of the maximal unramified extension of \( F_{\p} \).
Given any \(S^{\#}\in\Spd(\CC_{p})(S)\),
we have a divisor 
\(S^{\#}\hookrightarrow \mX_{\rF}(S)\).
For \(\tau\in\Sigma_{\infty}\), using \(\iota:\CC\cong\bar{\QQ}_{p}\),
we have an induced map \(S\to \Spd(F_{p})\),
and thus a divisor \(S_{\tau}^{\#}\subset\mX_{\rF,F}(S)\),
and \(f_{F}(S_{\tau})^{\#}\subset S^{\#}\),
so we know that
\[f_{F}^{-1}(S^{\#})=\coprod_{\tau\in\Sigma_{\infty}}S^{\#}_{\tau}.\]
This coincides with the decomposition that appears in \(S^{\#}\times_{\Spa\QQ_{p}}\Spa F_{p}\).

\begin{lem}\label{lemFLidentifyModification}
We have a canonical decomposition \(\FLT{T}\cong \prod_{\tau\in T^{c}} \Fl_{G_{T},\mu_{T},\tau}\). In terms of the moduli problem, 
	for any \(S\in \Perf\), \(\FLTT{T}{T}(S)\) 
	is equivalent to the groupoid of tuples \[\left(S^{\#}\in\Spd(\CC_{p})(S),\mF'\in \Bun_{G}(S),f:\mF_{T}\xrightarrow{\coprod_{\tau\in T^{c}}S^{\#}_{\tau}} \mF' \right),\]
	where \(S^{\#}\)
	is an untilt of \(S\)
	over \(\CC_{p}\),
	and \(f:\mF_{T}\to \mF'\) is a modification along the divisor \(\coprod_{\tau\in T^{c}}S^{\#}_{\tau}\),
	such that for any \(\tau\in T^{c}\),
	\(\cok(f|_{\Spa(\BdR^{+}(S^{\#}_{\tau}))})\) is of rank \(1\) over \(\mO_{S^{\#}_{\tau}}\)
for \(\tau\notin T\).

Then for \(\tau\in T^{c}\), for any \(S\in \Perf\), \(\Fl_{G_{T},\mu_{T},\tau}(S)\) 
	is equivalent to the groupoid of tuples \[\left(S^{\#}\in\Spd(\CC_{p})(S),\mF'\in \Bun_{G}(S),f_{\tau}:\mF_{T}\xrightarrow{S^{\#}_{\tau}} \mF' \right),\]
	where \(S^{\#}\)
	is an untilt of \(S\)
	over \(\CC_{p}\),
	and \(f:\mF_{T}\to \mF'\) is a modification along the divisor \(S^{\#}_{\tau}\),
	such that 
	\(\cok(f|_{\Spa(\BdR^{+}(S^{\#}_{\tau}))})\) is of rank \(1\) over \(\mO_{S^{\#}_{\tau}}\).
\end{lem}
\begin{proof}
The first part follows from Lemma \ref{lemIdentifyQuaternion}. If we define \(\Fl_{G_{T},\mu_{T},\tau}\) as above, the natural map \(\FLT{T}\to \prod_{\tau}\Fl_{G_{T},\mu_{T},\tau}\)
is given by restricting the modification to each divisor \(S_{\tau}^{\#}\).
\end{proof}
\begin{rmk}
	We know that \(\Fl_{  G_{\ti{T}}',   \mu_{\ti{T}}, \tau  }\cong \mP^{1}_{\CC_{p}}\)
	for any \(\tau\in T^{c}\), but not canonically.
\end{rmk}
\begin{notation}\label{notationFLGtMutI}
For any \(I\subset T^{c}\), we write \[\Fl_{G_{T},\mu_{T},I}:=\FLTT{T}{I^{c}}:=\prod_{\tau\in I}\Fl_{G_{T},\mu_{T},\tau}.
\]
\end{notation}

\begin{dfn}[\(I\)-basic locus]\label{dfnIbasicLocus}
	Given \(T\subset\Sigma_{\infty}\)
	and \(I\subset T^{c}\),	consider the Beauville-Laszlo morphism \[\BL_{T,I}:\Fl_{G_{T},\mu_{I^{c}}}\to \Bun_{G_{T},\mu_{I^{c}}}.
	\]
	%
	We define the preimage of the basic locus as \(\FLTT{T}{I^{c}}^{bc}\subset \FLTT{T}{I^{c}}\).
	We define the \emph{\(I\)-basic locus} \(\Fl_{G_{T},\mu_{T}}^{I-\bc}\) as 
	the preimage of \(\FLTT{T}{I^{c}}^{\bc}\) along the projection \(\FLT{T}\to \FLTT{T}{I^{c}}\).
	For \(I=T^{c}\), we write \(\Fl^{\bc}_{G_{T},\mu_{T}}:=\Fl^{I-\bc}_{G_{T},\mu_{T}}\). 

	We denote the composition \(\mS_{K^{p}}(G_{T})\xrightarrow{\pi_{\HT}}\FLT{T}\to \FLTT{T}{I^{c}}\)
	as \(\pi_{\HT,T,I}\).
\end{dfn}

\begin{lem} \label{lemNoEmptyBasicLocus}
	(1) For any \(T\subset\Sigma_{\infty}\)
	and \(I\subset T^{c}\),
	\(\FLTT{T}{I^{c}}^{\bc}\ne \emptyset\) and is an open dense connected subspace of \(\FLTT{T}{I^{c}}\);

	(2) Let \(T\subset\Sigma_{\infty}\), \(\p\in\Sigma_{p}\), and  \(I\subset T^{c}\cap \Sigma_{\infty/\p}\).
	Assume that \(2|r_{\p,T}\),
	and  \(1\le\#I\le 2\),
	then \(\FLTT{T}{I^{c}}^{\bc}\cong \FLTT{T}{I^{c}}\backslash \im(\mP^{1}(F_{\p})\to \FLTT{T}{I^{c}})\). 
	
	Here the embedding \(\mP^{1}(F_{\p})\to \FLTT{T}{(\Sigma_{\infty/\p})^{c}}\) is given as follows: fixing an isomorphism \(\mF_{T}|_{\mX_{\rF,F_{\p}}}\cong \mO_{\mX_{\rF,F_{\p}}}(r_{\p,T}/2)^{\oplus 2}\), and for each \(x\in\mP^{1}(F_{\p})\),
	there is a sub-bundle \(\iota_{x}:\mO_{\mX_{\rF,F_{\p}}}\hookrightarrow \mO_{\mX_{\rF,F_{\p}}}(r_{\p,T}/2)\hookrightarrow \mF_{T}|_{\mX_{\rF,F_{\p}}}\), which gives rise to an element in \(\FLTT{T}{(\Sigma_{\infty/\p})^{c}}\). Note that the subspace \(\mP^{1}(F_{\p})\) is independent of the choice of the isomorphism as \(\mrm{Aut}(\mF_{T}|_{\mX_{\rF,F_{\p}}})\cong \underline{\GL_{2}(F_{\p})}\).


	(3) Let \(T\subset\Sigma_{\infty}\), \(\p\in\Sigma_{p}\), and \(I\subset T^{c}\cap \Sigma_{\infty/\p}\). Assume that \(2\nmid r_{\p,T}\), and \(\#I=1\),
	then \(\FLTT{T}{I^{c}}^{bc}=\FLTT{T}{I^{c}}\).
\end{lem}
\begin{proof}

Now to prove (1),
we can reduce inductively to the case where \(\#I=1\), and then it suffices to prove (2) and (3). 
More precisely, let \(\tau\in I\). If we are given any \(U\subset \Fl_{G_{T}, \mu_{T},I}\), we want to prove that \(U\cap \Fl^{bc}_{G_{T},\mu_{T},\tau}\ne \emptyset\).
We know that the projection of \(U\) to \(\Fl_{G_{T},\mu_{T},\tau}\)
is also open,
so by the case for 
\(\{\tau\}\),
there exists \(x\in \Fl_{G_{T},\mu_{T},\tau}^{bc}\),
such that
\((\{x\}\times \Fl_{G_{T},\mu_{T},I\backslash\{\tau\}})\cap U\ne \emptyset \).
Finally,
we note that \(x\) induces an isomorphism
\(\{x\}\times \Fl_{G_{T},\mu_{T},I\backslash\{\tau\}}\cong \Fl_{G_{T\cup\{\tau\}},\mu_{T\cup \{\tau\},I\backslash \{\tau\}}}\), that identifies
\(\BL_{T,I}|_{\{x\}\times \Fl_{G_{T},\mu_{T},I\backslash\{\tau\}}}\)
with
\(\BL_{T\cup\{\tau\},I\backslash\{\tau\}}\). In particular, it also identifies 
\((\{x\}\times \Fl_{G_{T},\mu_{T},I\backslash\{\tau\}})\cap \Fl^{bc}_{G_{T},\mu_{T},I}\)
with \(\Fl^{bc}_{\mu_{T\cup\{\tau\},I\backslash\{\tau\}}}\).
So we can continue with induction. 
Now we prove that \(\Fl_{G_{T}, \mu_{T},I}^{bc}\)
is connected. If \(\Fl_{G_{T}, \mu_{T},I}^{bc}=U_{1}\coprod U_{2}\)
is a decomposition into open subspaces of \(\Fl_{G_{T}, \mu_{T},I}\),
then for any \(x\in \Fl^{bc}_{G_{T},\mu_{T},\{\tau\}}\), we also have 
\[\Fl_{G_{T\cup\{\tau\}},\mu_{T\cup\{\tau\}},I\backslash\{\tau\}}^{bc}\cong 
(\{x\}\times \Fl_{G_{T},\mu_{T},I\backslash\{\tau\}})\cap \Fl^{bc}_{G_{T},\mu_{T},I}=\coprod_{i=1}^{2}((\{x\}\times \Fl_{G_{T},\mu_{T},I\backslash\{\tau\}})\cap U_{i})\ne \emptyset.\]
By induction,
\(\Fl_{G_{T\cup\{\tau\}},\mu_{T\cup\{\tau\}},I\backslash\{\tau\}}^{bc}\)
is connected,
so
for any \(x\in\Fl_{G_{T},\mu_{T},\tau}^{bc}\),
there exists a unique \(i_{x}\in\{1,2\}\),  \((\{x\}\times \Fl_{G_{T},\mu_{T},I\backslash\{\tau\}})\cap U_{i_{x}}\ne\emptyset\). Then since \(\Fl^{bc}_{G_{T},\mu_{T},\{\tau\}}\)
is connected,
we know \(i_{x}\) is constant, say \(i_{x}=1\), then 
\(U_{2}=\emptyset\). This shows that \(\Fl^{bc}_{G_{T},\mu_{T},I}\)
is connected.

Now we are left to prove (2) and (3). 
In the case where \(\#I=1\), up to twists by \(\mO(1)\),
it suffices to show the following:

(a) For \(\mcal{E}=\mO(1/2)\)
rank \(2\) vector bundle over \(\mX_{\rF,F_{\p}}(S)\),
and any \(f:\mcal{E}\to\mcal{E}'\)
which is a modification along a divisor \(S^{\#}\) over \(F_{\p}\),
such that \(\mE'\otimes \BdR^{+}(S^{\#})/\mE\otimes\BdR^{+}(S^{\#})\)
is of rank \(1\),
then automatically \(\mE'\cong \mO(1)^{\oplus 2}\).

(b) For \(\mcal{E}=\mO^{\oplus 2}\) over \(\mX_{\rF,F_{\p}}(S)\),
and any \(f:\mcal{E}\to\mcal{E}'\) as in (a), which corresponds to a unique element \(x\in \mP^{1}(S^{\#})\),
then \(\mE'\cong \mO\oplus \mO(1)\),
if and only if \(x\in \mP^{1}(F_{\p})\),
and \(\mE'\cong \mO(1/2)\)
if and only if \(x\in \mP^{1}\backslash \mP^{1}(F_{\p})\).

(a) follows immediately as \(B(\mrm{Res}_{F_{\p}/\QQ_{p}}B_{\p}^{\times},\mu_{\{\tau\}^{c}})\) has only one element,
where \(B_{\p}\)
denotes the (non-split) quaternion algebra over \(F_{\p}\), and \(\tau:F_{\p}\hookrightarrow\bar{\QQ}_{p}\)
and \(\mu_{\{\tau\}^{c}}:\GG_{m}\to (\mrm{Res}_{F_{\p}/\QQ_{p}}B_{\p}^{\times})_{\bar{\QQ}_{p}}\cong \prod_{\tau}\GL_{2,\bar{\QQ}_{p}}\)
that is non-trivial only at the \(\tau\)-factor.
Geometrically, this can also be seen by noting that there does not exist an injection \(\mO(1/2)\to \mO(1-i)\oplus \mO(1+i)\) unless \(i=0\).
For (b), there are only two elements 
in \(B(\mrm{Res}_{F_{\p}/\QQ_{p}}\GL_{2},\mu_{\{\tau\}^{c}})\).
If \(\mcal{E}'\cong \mO\oplus \mO(1)\),
then \(f:\mO^{\oplus 2}\to \mO\oplus \mO(1)\) is determined by its 
composition along \(\mO\oplus \mO(1)\to \mO\),
whose kernel defines a unique element in \(\mP(H^{0}(\mO^{\oplus 2}))\cong \mP^{1}(F_{\p})\), which by definition coincides with \(x\in\mP^{1}\).

Finally, we prove the case where \(2|r_{\p,T}\) and \(\#I=2\). We note that
\(B(G,\mu_{I^{c}})\)
has two elements, corresponding to \(\mO(1)^{\oplus2}\)
and \(\mO\oplus \mO(2)\). Then again we note that any modification \(f:\mO^{\oplus 2}\to \mO\oplus \mO(2)\)
is determined by its projection to \(\mO\),
and again defines an element \(x\in\mP^{1}(F_{\p})\).
\end{proof}
\begin{dfn} \label{dfnLocalShimuraTtoT'}
	Let \(T\subset T':=T\coprod I\subset \Sigma_{\infty}\). 
	We denote by \(\mcal{M}_{T\to T'}\)
	the local Shimura variety at infinite level in the sense of \cite[Lecture 24]{ScholzeWeinstein2020berkeley} corresponds to the local Shimura data \((G_{T,\QQ_{p}},\mu_{I^{c}},b_{T,I})\),
	where \(b_{T,I}\)
	denotes the unique basic element in \(B(G_{T,\QQ_{p}},\mu_{I^{c}})\). 

	Explicitly, \(\mcal{M}_{T\to T'}(S)\)
	parametrizes pairs \(\left(S^{\#}\in\Spd\CC_{p}(S),f:\mF_{T}\to \mF_{T'}\right)\)
	where \(f\) is a modification along 
	\(\coprod_{\tau\in I}S_{\tau}^{\#}\) such that the modification at \(S_{\tau}^{\#}\) is bounded by \(\mu_{\{\tau\}^{c}}\).

	By \cite[Theorem D]{SW13},
	\(\mcal{M}_{T\to T'}\)
	is represented by a perfectoid space over \(\Spa(\CC_{p},\mO_{\CC_{p}})\),
	which we denote with the same notation.

	There are two period maps for \(\mcal{M}_{T\to T'}\), \[\begin{tikzcd}
		& \mcal{M}_{T\to T'}\arrow[ld,"\pi_{\HT,T}=\pi_{\GM,T'}"']\arrow[rd,"\pi_{\HT,T'}=\pi_{\GM,T}"] & \\
		\FLTT{T}{I^{c}}^{\bc} & & 
		\FLTTi{T'}{I^{c}}^{\bc}
	\end{tikzcd}\] given by forgetting the trivialization of \(\mF_{T}\) and \(\mF_{T'}\)
	respectively,
	where \(\pi_{\GM}\)
	denotes the Grothendieck-Messing period map,
	and \(\pi_{\HT}\)
	denotes the Hodge-Tate period map.

	Let
	\(E_{T,I}\) denote the minimal subfield of \(\bar{\QQ}_{p}\) over which the conjugacy classes of \(\mu_{T}\) and of \(\mu_{T'}\) are defined.
\end{dfn}
\begin{lem}\label{lemRZisAtorsorOverFl}
\(\pi_{\HT,T}=\pi_{\GM,T'}:\mcal{M}_{T\to T'}\to \Fl^{\bc}_{G_{T},\mu_{I^{c}}}\)
is a pro\'etale \(G_{T'}(\QQ_{p})\)-torsor and is \(G_{T}(\QQ_{p})\)-equivariant. 
Similarly, \(\pi_{\HT,T'}=\pi_{\GM,T}:\mcal{M}_{T\to T'}\to \FLTTi{T'}{I^{c}}^{\bc}\)
is a pro\'etale \(G_{T}(\QQ_{p})\)-torsor  and is \(G_{T'}(\QQ_{p})\)-equivariant.
\end{lem}
\begin{proof}
	By definition,
	there is a Cartesian diagram \[\begin{tikzcd}
		\mcal{M}_{{T}\to{T}'}\arrow[r,"\pi_{\HT,{T}}"]\arrow[d]
		& \Fl^{\bc}_{G_{T},\mu_{I^{c}}}\arrow[d,"\BL_{{T},I}"]\\
		* \arrow[r,"\mF_{T'}"]
		& \Bun_{G,\mu_{{T}^{\prime,c}}}^{bc},
	\end{tikzcd}\]
	and the result follows by noting that that \(\Bun_{G,\mu_{{T}^{\prime,c}}}^{bc}\cong\left[*/\underline{G_{T'}(\QQ_{p})}\right]\) by \cite[Theorem III.4.5]{FarguesScholze2021geometrization}.
\end{proof}
\begin{lem}\label{lemIdentifyFlVaroverM}
	The local Shimura variety \(\mcal{M}_{T\to T'}\)	
	induces a \(G_{T}(\QQ_{p})\times G_{T'}(\QQ_{p})\times\Gal(\bar{\QQ}_{p}/E_{T,I})\)-equivariant isomorphism \[\FLT{T}\times_{\FLTT{T}{I^{c}}}\mcal{M}_{T\to T'}\cong \Fl_{G_{T'},\mu_{T'}}\times_{\CC_{p}}\mcal{M}_{T\to T'}.\] 
	\end{lem}
	\begin{proof}
		By Lemma \ref{lemFLidentifyModification},
		\(\FLT{T}\)
		parametrizes \((S^{\#},\mF,\mF_{T}\xrightarrow{\coprod_{\tau\in T^{c}}S^{\#}_{\tau}}\mF)\). Thus \(\FLT{T}\times_{\FLTT{T}{I^{c}}}\mcal{M}_{T\to T'}\)
		parametrizes \((S^{\#},\mF,\mF_{T}\xrightarrow{\coprod_{\tau\in I}S_{\tau}^{\#}}
		\mF_{T'}
		\xrightarrow{\coprod_{\tau\in (T')^{c}}S^{\#}_{\tau}}\mF)\).
		This is the same as the data 
		of the tuple \((S^{\#},\mF,\mF_{T'}\xrightarrow{\coprod_{\tau\in (T')^{c}}S^{\#}_{\tau}}\mF)\)
		plus the data of \((S^{\#},\mF_{T}\xrightarrow{\coprod_{\tau\in I}S_{\tau}^{\#}}
		\mF_{T'})\),
		which is precisely parametrized by \(\Fl_{G_{T'},\mu_{T'}}\times_{\CC_{p}}\mcal{M}_{T\to T'}\). 
	\end{proof}
\begin{thm}\label{thmProductFormulaQuate}Fix a neat open compact subgroup \(K^{p}\subset G(\mA^{p}_{f})\).
We consider \(T\subset T'=T\coprod I\subset T''=T\coprod J\subset \Sigma_{\infty}\).
Let \( E_{\ti{T},\ti{I}}:=E_{0,\ti{T},\p_{\ti{T}}}E_{0,\ti{T}',\p_{\ti{T}'}}\subset\bar{\QQ}_{p} \), where the RHS is as in Corollary \ref{corIgusaStackPELGeneralDfn}.

	(1) (Global version)
	There is a \(G_{T}(\QQ_{p})\times G_{T'}(\QQ_{p})\times\Gal(\bar{\QQ}_{p}/E_{\ti{T},\ti{I}})\)-equivariant isomorphism of perfectoid spaces over \(\CC_{p}\)
	\[\mS_{K^{p}}(G_{T})\times_{\FLTT{T}{I^{c}}}\mcal{M}_{T\to T'}\cong 
	\mS_{K^{p}}(G_{T'})\times_{\CC_{p}}
	\mcal{M}_{T\to T'},
	\]
	where
	the group \(G_{T}(\QQ_{p})\)
	acts on \(\mS_{K^{p}}(G_{T})\)
	and on \(\mcal{M}_{T\to T'}\),
	and the group \(G_{T'}(\QQ_{p})\)
	acts on \(\mS_{K^{p}}(G_{T'})\)
	and on \(\mcal{M}_{T\to T'}\).

	Moreover, when varying \(K^{p}\),
	we have an isomorphism of towers
	\[\left\{\mS_{K^{p}}(G_{T})\times_{\FLTT{T}{I^{c}}}\mcal{M}_{T\to T'}\right\}_{K^{p}}\cong 
	\left\{\mS_{K^{p}}(G_{T'})\times_{\CC_{p}}
	\mcal{M}_{T\to T'}\right\}_{K^{p}},\]
	which is \(G(\mA_{f}^{p})\)-equivariant.

	Quotienting out \(G_{T'}(\QQ_{p})\)-action,
	we have a \(G_{T}(\mA_{f})\times \Gal(\bar{\QQ}_{p}/E_{\ti{T},\ti{I}})\)-equivariant isomorphism 
	\[\JL_{T,I}:=\JL_{T\to T'}:
	\left\{\left(\mS_{K^{p}}(G_{T'})\times_{\CC_{p}}
	\mcal{M}_{T\to T'}\right)/\underline{G_{T'}(\QQ_{p})}\right\}_{K^{p}}
	\cong 
	\left\{\mS_{K^{p}}(G_{T})|_{\FLTT{T}{I^{c}}^{bc}}\right\}_{K^{p}}.\]

	(2)
	(Compatibility of Hodge-Tate period map) 
	The isomorphism has the  the following diagram commutes: 
	\[\begin{tikzcd}
		\mS_{K^{p}}(G_{T})\times_{\FLTT{T}{I^{c}}}\mcal{M}_{T\to T'}\arrow[r,"\sim"] \arrow[d,"\pi_{\HT,T}\times\id"]
		&
		\mS_{K^{p}}(G_{T'})\times_{\CC_{p}}
		\mcal{M}_{T\to T'}\arrow[d,"\pi_{\HT,T'}\times\id"]
		\\
		\FLT{T}\times_{\FLTT{T}{I^{c}}}\mcal{M}_{T\to T'}
		\arrow[r,"\sim"]
		&\Fl_{G_{T'},\mu_{T'}}\times_{\CC_{p}}\mcal{M}_{T\to T'}
	\end{tikzcd}\]
	where the second row is given by Lemma \ref{lemIdentifyFlVaroverM}. 

\end{thm}
\begin{proof}
This follows formally from Theorem \ref{thmCartesianDiagramQaternion}: for (1), 
we have \begin{align*}
\mS_{K^{p}}(G_{T})\times_{\FLTT{T}{I^{c}}}\mcal{M}_{T\to T'}\cong \Igs_{K^{p}}\times_{\Bun_{G}}\FLT{T}\times_{\FLTT{T}{I^{c}}}\mcal{M}_{T\to T'}\\
\cong \Igs_{K^{p}}\times_{\Bun_{G}}\FLT{T'}\times_{\CC_{p}}\mcal{M}_{T\to T'}\cong \SGTKp{T'}\times_{\CC_{p}}\mcal{M}_{T\to T'},
\end{align*} where the second isomorphism is given by Lemma \ref{lemIdentifyFlVaroverM}. 
(2) holds by our construction.
\end{proof}

\section{Partial de Rham cohomology}\label{sectionJLTransfer}

In this section, we work with quaternionic Shimura varieties as in Section \ref{sectionGeoSV}. The goal is to define the differential operators \(d^{-k_{\tau}+1}_{\tau}\) and \(\bar{d}^{-k_{\tau}+1}_{\tau}\),  generalizing those in \cite{Pan2209.06II},
and to prove that the ``partial de Rham complexes'' determined by these operators have ``classical'' cohomology as a Hecke module (Theorem \ref{thmClassicalityCohomoDR}). 

Let us give a sketch here with the trivial coefficient. For a general Shimura variety, we have the following diagram: \[
\mX^{\sm}_{K^{p}}
\xleftarrow{\pi^{\sm}_{\la,0}}\mX^{\la,0}_{K^{p}}\xrightarrow{\pi^{\la,0}_{\HT}}\Fl.
\] We show in \cite[Corollary \ref{1-corCartesianFlandSMGeneral}]{Jiang2026Shla} that the morphism \(\pi^{\sm}_{\la,0}\times \pi^{\la,0}_{\HT}\) is ``\'etale'' in a suitable sense. We can then consider the ``relative differentials'' along $\pi^{\sm}_{\la,0}$ and along \(\pi^{\la,0}_{\HT}\), which we denote as \[\bar{d}:\mO^{\la,0}_{K^p}\to \mO^{\la,0}_{K^p}\otimes_{\mO_{\Fl}}\Omega^1_{\Fl},\text{ and }d:\mO^{\la,0}_{K^p}\to \mO^{\la,0}_{K^p}\otimes_{\mO_{K^{p}}^\sm}\Omega^{1,\sm}_{K^p}
\] respectively. 

Using either \(
\bar{d}
\) or \(
d
\), we can define a ``de Rham complex'', i.e. a Koszul complex. 
By studying \(
\bar{d}
\), we will see that the Koszul complex that it defines concentrates in degree 0, and is isomorphic to \(
\mO^{\sm}_{K^p}
\). Hence its cohomology is the direct limit of the coherent cohomology of Shimura varieties on the finite level. 

In the case of Hilbert modular varieties, there are decompositions \(
\Omega^1_{\Fl}=\bigoplus_{\tau\in\Sigma_{\infty}}\Omega^1_{\Fl,\tau}
\) and \(
\Omega^{1,\sm}_{K^p}=\bigoplus_{\tau\in\Sigma_{\infty}}\Omega^{1,\sm}_{K^p,\tau}
\). Therefore, given any subsets \(I,J\subset \Sigma_{\infty}\), we can consider \[
\bar{d}_{J}\oplus d_{I}:\mO^{\la,0}_{K^p}\to (\bigoplus_{\tau\in J}\mO^{\la,0}_{K^p}\otimes\Omega^1_{\Fl,\tau})\oplus (\bigoplus_{\tau\in I}\mO^{\la,0}_{K^p}\otimes \Omega^{1,\sm}_{\tau}),
\] and we denote by \(\dR_{I,J}(\mO^{\la,0}_{K^p})
\) the associated Koszul complex, which we will refer to as the  ``partial de Rham complex''. See Definition \ref{dfnDiffeOperatorsdanddbar} for a rigorous and more general definition. 

As mentioned above, if \(
(I,J)=(\emptyset,\Sigma_{\infty})
\), then the cohomology of \(
\dR_{I,J}(\mO^{\la,0}_{K^p})
\) is \emph{classical} in the sense that the cohomology groups as Hecke modules are supported at points that correspond to Hilbert modular forms. The main theorem of this section (Theorem \ref{thmClassicalityCohomoDR}) shows  that the same ``classicality'' holds whenever \(I\cup J=\Sigma_{\infty}
\). 
In Section \ref{sectionFontaineOperator}, we will see that the ``Bernstein-Gelfand-Gelfand-Fontaine complex'' is filtered by such \(\dR_{I,J}(\mO^{\la,0}_{K^p})\) (or its variant with non-trivial coefficient), and we will use Theorem \ref{thmClassicalityCohomoDR}
to show that the Bernstein-Gelfand-Gelfand-Fontaine complex has classical cohomology (Theorem \ref{thmFontainComplexClassical}), which in turn will be used in Section \ref{sectionAppplicationClassicality} to show Theorem \ref{IntrothmClassicalityCompleteCoho}. 

Note that such classicality is proven in \cite{Pan2209.06II} in the case of modular curves. To treat the case where \((I,J)=(\Sigma_{\infty},\emptyset)\), the idea loc. cit. is  to consider the Newton stratification, and over each Newton stratum, the complex \(\dR_{\Sigma_{\infty},\emptyset}(\mO^{\la,0}_{K^p})
\) is related to the rigid cohomology of the associated rigid variety. 
However, the cases when neither \(I\) nor \(J\) is an empty set need significantly new ideas.

The idea of our proof of Theorem \ref{thmClassicalityCohomoDR}
is to use induction. For this purpose, we actually prove the theorem not only for Hilbert modular varieties, but also for all the quaternionic Shimura varieties. 

The cohomology of the ``partial de Rham complex''
in our setting is a mixture of coherent cohomology and rigid cohomology. In particular, in the extreme case where \(J=\Sigma_{\infty}\), the desired classicality is obvious by the argument above. In the general case, we use the product formula (Theorem \ref{thmProductFormulaQuate}), which somehow translates the operator \(d\) on one Shimura variety to the operator \(\bar{d}\) on another Shimura variety of lower dimension, which we will refer to as \emph{locally analytic Jacquet-Langlands transfer} (Theorem \ref{thmJacquetLanglandsLA}). This idea already appears in \cite{Pan2209.06II}. 

However, in our setting, 
the proof becomes more complicated because Theorem \ref{thmProductFormulaQuate} is proven using the perfectoid method. The new ingredients of our proof is provided by the results of \cite{Jiang2026Shla}, notably Theorem \ref{1-thmLAstrucGMHTper} and Corollary \ref{1-corReformulateRiemmanHilbPerfdCase} loc. cit., which show that the information about connections can indeed be recovered from the perfectoid data.



The section is structured as follows:
we set up the notation in \S \ref{subsecSetUpTcase}.
In \S \ref{subsectionConstrDiffeOper}, 
using the results from \cite{Jiang2026Shla} (especially Corollary \ref{1-corCartesianFlandSMGeneral} loc. cit.),
we construct the differential operators (Definition \ref{dfnDiffeOperatorsdanddbar}), define the ``partial de Rham complexes'' (Definition \ref{dfnDRComplexNotation}) and state the main theorem (Theorem \ref{thmClassicalityCohomoDR}). 
In \S \ref{subsecCompareUnderSpaces}, based on Theorem \ref{thmProductFormulaQuate}, we prove the locally analytic Jacquet-Langlands correspondence (Corollary \ref{corJLlaIsGlaTorsor} and  Theorem \ref{thmJacquetLanglandsLA}).
Finally, we conclude the proof of the main theorem (Theorem \ref{thmClassicalityCohomoDR}) in \S \ref{subsecProofOfClassicality}.

\subsection{Set-up}\label{subsecSetUpTcase}
We will use the notation from \S \ref{subsectionQuaternUniSV} and \cite{Jiang2026Shla}. 
\begin{notation}\label{notationVariousSVResume}
We will use the results of \cite[Section \ref{1-sectionLAVectorSV}]{Jiang2026Shla}. 
Since we are going to work with various Shimura varieties, for a Shimura datum \( (G,X) \), 
we write
\begin{align*}
\mS^{\tor}_{K}(G)&:=\mX^{\tor}_{K} \text{ Shimura variety at finite level (Notation \ref{1-notationShimuraSetUp})},\\
\mS^{\tor}_{K^{p}}(G)&:=\mX^{\tor}_{K^{p}} \text{ Completed infinite level Shimura variety (Lemma \ref{1-lemTorCompacAsTateStack})},\\
\mS^{\tor}_{K^{p}}(G)^{\sm}&:=\mX^{\tor,\sm}_{K^{p}}\text{ Smooth infinite level Shimura variety (Definition \ref{1-dfnNonCompletedInfiniteLevel})},\\
\mS^{\tor}_{K^{p}}(G)^{\la}&:=\mX^{\tor,\la}_{K^{p}}\text{ Locally analytic Shimura variety (Definition \ref{1-dfnLocallyAnSV})},\\
\mS^{\tor}_{K^{p}}(G)^{\pdR/\CC_{p},?,\sm}_{\log}&:=(\mX^{\tor}_{K^{p}})^{\pdR/\CC_{p},?,\sm}_{\log}\text{ Smooth log de Rham stack (Definition \ref{1-dfnSMratConstruction})},\\
\mS^{\tor}_{K^{p}}(G)^{\la,0}&:=\mX^{\tor,\la,0}_{K^{p}}\text{ Horizontal locally analytic Shimura variety (Definition \ref{1-dfnHorizontalLAsv})},\\
(\mS^{\tor}_{K^{p}}(G)^{\la,0})_{\log}^{\pdR/\CC_{p},??}&:=
(\mX^{\tor,\la,0}_{K^{p}})_{\log}^{\pdR/\CC_{p},??},??\in\{++,\hpp\}\\&\text{ Smooth bi-filtered algebraic de Rham stack (Definition \ref{1-dfnBiFilAgLa0})},
\end{align*}
where the numbers refer to the respective results in \cite{Jiang2026Shla}. 
We will remove the superscript \( ()^{\tor} \) and the subscript \( ()_{\log} \)  when we work away from the toroidal boundaries.
\end{notation}
\begin{rmk}\label{rmkApplySMtorToTower}
There is no conflict of notation for \( \SGTKp{T} \) (Notation \ref{notationVariousSVResume} versus \S \ref{subsectionQuaternUniSV}) by \cite[Lemma \ref{1-lemTorCompacAsTateStack} (6)]{Jiang2026Shla}.
\end{rmk}
\begin{notation}
Let \( F \) be a totally real number field of dimension \(d\), and let \( \Sigma_{\infty} \) 
be the set of archimedean places.
Let \(T\subset \Sigma_{\infty}\), and let \((G_{T},h_{T})\)
be the Shimura datum in \S \ref{subsubQuaternion}. Recall that the associated Hodge cocharacter \( \mu_{T}:\GG_{m,\bar{\QQ}_{p}}\to \prod_{\tau\in\Sigma_{\infty}}\GL_{2,\bar{\QQ}_{p}} \) 
is trivial at the \( \tau \)-factor for \( \tau\in T \), and is \( t\mapsto \mrm{diag}\{t,1\} \) for \( \tau\notin T \).   

Let \( E_{T}\subset\CC \) denote the reflex field of \( (G_{T},h_{T}) \), and \( E_{T,\p_{T}} \) its \( p \)-adic completion along \( E_{T}\subset\CC\cong^{\iota}\bar{\QQ}_{p} \).

For fixed \( T \), we denote the morphisms as 
\[
\begin{tikzcd}
\SGTKpTo{T}^{\la}\arrow[rr,"\pi^{\la}_{\sm,T}",bend left]\arrow[rdd,"\pi^{\la}_{\HT,T}"]
\arrow[r,"\pi^{\la}_{0,T}"]&
\SGTKpTo{T}^{\la,0}\arrow[dd,"\pi_{\HT,T}^{\la,0}"]
\arrow[r,"\pi^{\la,0}_{\sm,T}"]&
\SGTKpTo{T}^{\sm}
\arrow[r,"h^{\sm}_{T}"]\arrow[d,"\pi^{\sm}_{K,T}"]
&
\SGTKpTo{T}_{\log}^{\pdR/\CC_{p},\sm}\arrow[d,"\pi^{\pdR,\sm}_{K,T}"]\arrow[dd,"\pi_{\HT,T}^{\pdR}",bend left=80]
\\
&  & \SGTKTo{T}{K}\arrow[r,"h_{K,T}"] & \SGTKTo{T}{K}_{\log}^{\pdR/\CC_{p}}.\\
& \FLT{T}\arrow[rr,"\beta_{\Fl,T}"] &&\FLT{T}^{\dR}
\end{tikzcd}
\] and  \[
\begin{tikzcd}[column sep=large]
{\mS^{\tor}_{K^{p}}(G_{T})^{\la,0}\times [{\mA}^{1}/\GG_{m}]^{2}}
\arrow[d,"\pi^{\la,0}_{\HT,T}\times \pi^{\la,0}_{\sm,T}"]
\arrow[r,"h^{\la,0,++}_{T}"] &
(\mS^{\tor}_{K^{p}}(G_{T})^{\la,0})^{\pdR/\CC_{p},++}_{\log}
\arrow[d,"\pi^{\pdR,++}_{\GM\HT,T}"] 
\\ {(\FLT{T}\times [\mA^{1}/\GG_{m}])}
\times  
{(\mS^{\tor}_{K^{p}}(G_{T})^{\sm}\times [\mA^{1}/\GG_{m}])}
\arrow[r,"h_{\Fl,T}^{+}\times h_{T}^{\sm,+}"]\arrow[d,"1\times \pi_{\GM,T}^{\sm}"]
& \FLT{T}^{\pdR/\CC_{p},+}\times  (\SGTKpTo{T})^{\pdR/\CC_{p},{+},\sm}_{\log}
\arrow[d,"1\times \pi_{\GM,T}^{\sm,\pdR}"]
\\
{(\FLT{T}\times [\mA^{1}/\GG_{m}])}
\times  
{(\FLT{T}^{\std}\times [\mA^{1}/\GG_{m}])}/G_{T}^{c}
\arrow[r,"h_{\Fl,T}^{+}\times h_{\Fl^{\std},T}^{+}"]
& \FLT{T}^{\pdR/\CC_{p},+}\times  \FLT{T}^{\std,\pdR/\CC_{p},+}/G^{c}_{T}.
\end{tikzcd}
\] We will write \( (-)^{\hp} \)
(resp. \( (-)^{\hpp} \)) for the corresponding restrictions to \( [\hat{\mA}^{1}/\GG_{m}] \) (resp. \( [\hat{\mA}^{1}/\GG_{m}]^{2} \)).

We denote \( \mO^{\la}_{\SGTKp{T}}:=(\pi^{\la}_{\sm,T})_{*}(\mO_{\SGTKpTo{T}}^{\la})\in \QCoh(\SGTKpTo{T}^{\sm}) \), which is the same as the sheaf \( \mO^{\la}_{K^{p}} \) in \cite[Notation \ref{1-notationOlaIla}]{Jiang2026Shla}. By abuse of notation, we will also write \( \mO^{\la}_{\SGTKp{T}}:=(\pi^{\la}_{0,T})_{*}(\mO_{\SGTKpTo{T}}^{\la})\in \QCoh(\SGTKpTo{T}^{\la,0}) \).
\end{notation}



To treat the prime-to-\( p \) Hecke action, we will also take limits along prime-to-\( p \) 
levels:
\begin{notation}
Denote \[\mS(G_{T}):=\varprojlim_{K^{p}}\mS_{K^{p}}(G_{T})\subset \SGTTo{T}:=\varprojlim_{K^{p}}\SGTKpTo{T},\]
with the limit taken \emph{in the category of solid stacks}.
We define \(\SGTTo{T}^{\sm}\), \(\SGTTo{T}^{\la}\), and \(\SGTTo{T}^{\la,0}\) similarly, i.e. by taking limits along \(K^p\) in the category of Tate stacks. These stacks are equipped with an action of \( \unl{G_{T}(\QQ_{p})}\times G_{T}(\mA_{f}^{p})^{\sm}\times \unl{\Gal_{E_{T,\p_{T}}}} \). 
\end{notation}

\begin{notation}
We denote \(G_{T,\CC_{p}}\cong\prod_{\tau}G_{\tau}\) with \(G_{\tau}\cong \GL_{2}\), which we identify for all \(T\). 
Let \(\fg:=\Lie(G_{T})\otimes\CC_{p}\), which has a decomposition \(\fg\cong \bigoplus_{\tau\in\Sigma_{\infty}}\fg_{\tau}\), 
with \(\fg_{\tau}:=\Lie(G_{\tau})\cong\mrm{gl}_{2}\).

Using the isomorphism \(\bar{\QQ}_{p}\cong \CC\),
we have \(\mu_{T}:\GG_{m}\to G_{T,\CC_{p}}\cong\prod_{\tau}\GL_{2}\), which we conjugate to \( \mu_{T}:t\mapsto ((1)_{\tau\in T},(\mrm{diag}(t,1))_{\tau\notin T}) \).
For each \(\tau\),
we choose the Borel \(B_{\tau}\) (resp. \( \bar{B}_{\tau} \)) of upper-triangular (resp. lower-triangular) matrices in \(G_{\tau}\),
such that \[P_{\mu_{T}}=\prod_{\tau\in T}G_{\tau}\times \prod_{\tau\in T^{c}}B_{\tau},\;
P_{\mu_{T}}^{\std}=\prod_{\tau\in T}G_{\tau}\times \prod_{\tau\in T^{c}}\bar{B}_{\tau}. \]
Denote its Levi decomposition as \(B_{\tau}=H_{\tau}N_{\tau}\subset G_{\tau}\),
and  their Lie algebras as \(\fb_{\tau}=\fh_{\tau}\oplus\fn_{\tau}\).

For any \(I\subset \Sigma_{\infty}\), 
we will write \(\fg_{I}:=\bigoplus_{\tau\in I}\fg_{\tau},\)
and similarly define \(\fb_{I}\), \(\fh_{I}\)
and \(\fn_{I}\). 

Denote the Levi decomposition of \(P_{\mu_{T}}\) as \(P_{\mu_{T}}=M_{\mu_{T}}N_{\mu_{T}}\),  
and their Lie algebras as \(\mfk{p}_{\mu_{T}}=\mfk{m}_{\mu_{T}}\oplus \fn_{\mu_{T}}\). Then \(\fn_{\mu_{T}}= \fn_{T^{c}}\), \(\mfk{m}_{\mu_{T}}= \fh_{T^{c}}\oplus \fg_{T}\), and \(\mfk{p}_{\mu_{T}}= \fb_{T^{c}}\oplus \fg_{T}\).

We fix a basis \(\{\mrm{Z}_{\tau},\mrm{H}_{\tau},\mrm{E}_{\tau},\mrm{F}_{\tau}\}\) of \(\fg_{\tau}\),
with \(\mrm{Z}_{\tau}\) in the center \(\mfk{z}_{\tau}\subset \fg_{\tau}\), \([\mrm{H}_{\tau},\mrm{E}_{\tau}]=2\mrm{E}_{\tau}\), \(H_{\tau}\in\fh_{\tau}\), \([\mrm{H}_{\tau},\mrm{F}_{\tau}]=-2\mrm{F}_{\tau}\), and \([\mrm{E}_{\tau},\mrm{F}_{\tau}]=\mrm{H}_{\tau}\), and \(\mrm{E}_{\tau}\in\fn_{\tau}\). 
We denote the Casmir element by \(\Omega_{\tau}=\frac{1}{2}\mrm{H}_{\tau}^{2}+\mrm{H}_{\tau}+\mrm{F}_{\tau}\mrm{E}_{\tau}\in Z(\fg_{\tau})\subset U(\fg_{\tau})\). 
Then \(Z(\fg_{\tau})=\CC_{p}[\mrm{Z}_{\tau},\Omega_{\tau}]\). 

As a small note, \(H_{\tau}\) denotes the group, while \(\mrm{H}_{\tau}\in \fh_{\tau}\) is an element.

Let \(Z\cong \mrm{Res}_{F/\QQ}(\GG_{m})\) denote the center of \(G_{T}\), and
\(Z_{c}\) denote the maximal \(\RR\)-split, \(\QQ\)-anisotropic subtorus of \(Z\). Then \(Z_{c}\cong \Ker(\mrm{Nm}_{F/\QQ}:Z\to \GG_{m})\). We define \(G^{c}_{T}:=G_{T}/Z_{c}\) and \(P_{\mu_{T}}^{c}:=P_{\mu_{T}}/Z_{c}\). Let \(\mfk{z}_{c}:=\Lie(Z_{c})\otimes\CC_{p}\). Recall that for any open compact subgroup \(K_{p}\subset G_{T}(\QQ_{p})\), \(\widetilde{K_{p}}:=K_{p}/\overline{(K^{p}K_{p}\cap Z(\QQ))}\twoheadrightarrow K_{p}/(K^{p}K_{p}\cap Z_{c}(\QQ_{p}))\). We denote \(\ti{\fg}:=\Lie(\widetilde{K_{p}})\otimes\CC_{p}\twoheadrightarrow \fg^{c}:=\fg/\mfk{z}_{c}\). We define \(\ti{\mfk{p}}_{\mu_{T}}\) and \(\mfk{p}_{\mu_{T}}^{c}\)
similarly.
\end{notation}
\begin{rmk}
To avoid confusions, let us emphasize that for \(T\subset \Sigma_{\infty}\), we define \(
\fg_{T}=\bigoplus_{\tau\in T}\fg_{\tau}
\), and is \emph{not} the Lie algebra of \(G_{T}\).
\end{rmk}
Let us now fix the notations for some algebraic representations.
\begin{dfn}[Multiweight] 
	Let \(I\subset\Sigma_{\infty}\).
	We say that a tuple \(\kw\)
	is a \emph{multiweight} (resp. \(I\)-\emph{multiweight}), if \(\Bbbk=(k_{\tau})_{\tau\in\Sigma_{\infty}}\) (resp. \(\Bbbk=(k_{\tau})_{\tau\in I}\))
	with \(k_{\tau},w\in\ZZ\),
	and \(k_{\tau}\equiv w\pmod{2}\).

	Let \(I'\subset \Sigma_{\infty}\) (resp. \(I'\subset I\)).
	We will say that \(\kw\) is \emph{\(I'\)-regular}, if \(k_{\tau}\ne 1\) for \(\tau\in I'\). We will simply say that it is \emph{regular}
	if \(I'=\Sigma_{\infty}\) (resp. 
	\(I'=I\)).

        We say that a multiweight \(\kw\) is \emph{\(T\)-parallel} if \(\#\{k_{\tau}:\tau\in T
        \}=1\). We say that \(\kw\) is \emph{parallel} if it is \(
        \Sigma_{\infty}
        \)-parallel.
	
	Given 
	\(\kw\) an \(I\)-multiweight, and \(I'\subset I\),
	we denote by \((\Bbbk|_{I'},w)\) the \(I'\)-multiweight obtained by forgetting \(k_{\tau}\) for \(\tau\in I\backslash I'\).
	\end{dfn}



\begin{notation}\label{notationSomeAlgRep}
Given \((k_{\tau},w)\in\ZZ^{\oplus2}\) with \(k_{\tau}\equiv w\pmod{2}\), \(k_{\tau}\ne 1\). 
Denote by \(V_{\tau}^{(k_{\tau},w)}\) the algebraic representation of \(G_{\tau}\) \[V_{\tau}^{(k_{\tau},w)}:=\mrm{Sym}^{\max\{-k_{\tau},k_{\tau}-2\}}\mrm{Std}\otimes\det{}^{\frac{-w-\max\{-k_{\tau},k_{\tau}-2\}}{2}},\] 
and denote by \(\lambda^{(k_{\tau},w)}_{\tau}:Z(\fg_{\tau})\to \CC_{p}\) its infinitesimal character.

Denote by \(\chi_{\tau}^{(k_{\tau},w)}:=\left(\frac{w-k_{\tau}}{2},\frac{w+k_{\tau}}{2}\right)\) the algebraic representation of \(H_{\tau}\), such that
 \(\chi_{\tau}^{(k_{\tau},w)}(\mrm{H}_{\tau})=-k_{\tau}\),
and \(\chi_{\tau}^{(k_{\tau},w)}(\mrm{Z}_{\tau})=w\).

For any \(I\subset \Sigma_{\infty}\), and any \(I\)-multiweight \(\kw\), 
denote by   \(\chi^{\kw}_{I}:=\bigotimes_{\tau\in I}\chi_{\tau}^{(k_{\tau},w)}\circ \mrm{pr}_{\tau}\)
the algebraic representation of \(\prod_{\tau\in I}H_{\tau}\).

If in addition \(\kw\) is regular, we
denote by \(V_{I}^{\kw}:=\bigotimes_{\tau\in I}V_{\tau}^{(k_{\tau},w)}\circ \mrm{pr}_{\tau}\)
the algebraic representation of \(\prod_{\tau\in I}G_{\tau}\).

Given any \(T\)-regular multiweight \(\kw\) and \(T\subset \Sigma_{\infty}\),
we denote by \[W^{\kw}_{T}:=V_{T}^{\kw_{T}}\boxtimes \chi_{T^{c}}^{(-\Bbbk,-w)_{T^{c}}}\]
the representation of \(M_{\mu_{T}}\cong \prod_{\tau\in T}G_{\tau}\times \prod_{\tau\in T^{c}}H_{\tau}\), which is in fact a representation of \(M_{\mu_{T}}^{c}\),  which we inflate to a representation of \(P_{\mu_{T}}^{c}\). 
\end{notation}



\begin{notation}[Weyl group actions]
Let \(W\) denote the absolute Weyl group of \(G_{T}\), and let \(s_{\tau}\) denote the reflection induced by the simple positive root of \(G_{\tau}\).
We denote \(s_{T}:=\prod_{\tau\in T}s_{\tau}\in W\cong (\ZZ/2)^{\Sigma_{\infty}}\),
and denote \(s_{\tau}:\Bbbk\mapsto s_{\tau}\Bbbk\) (resp. \(\Bbbk\mapsto s_{\tau}\cdot\Bbbk\)) acting on \(\ZZ^{\oplus \Sigma_{\infty}}\)
by only changing at \(\tau\)-component via \(k_{\tau}\mapsto -k_{\tau}\) (resp. \(k_{\tau}\mapsto 2-k_{\tau}\)). 
\end{notation}
\begin{dfn} 
	(1) (over \( \FLT{T}^{\std} \))
	For any \(T\)-regular multiweight \(\kw\), we define \[
	\omega^{\kw}_{\FLT{T}^{\std}}:=\mcal{P}^{\mrm{std},c}_{\mu_{T}}(W^{(s_{T^{c}}\Bbbk,-w)}_{T})
	\in \QCoh(\FLT{T}^{\std}),
	\] where \( \mcal{P}^{\mrm{std},c}_{\mu_{T}} \) is as in \cite[Notation \ref{1-notationEquShvOnFl}]{Jiang2026Shla}, and \( W^{(s_{T^{c}}\Bbbk,-w)}_{T}=V_{T}^{(\Bbbk,-w)_{T}}\boxtimes \chi_{T^{c}}^{\kw_{T^{c}}} \).

	(2) (over \( \SGTKTo{T}{K} \))
	For neat \(K\subset G_{T}(\mA_{f})\) and any \(T\)-regular multiweight \(\kw\),
	we denote by \(\mcal{P}^{\mrm{std},c}_{\mu_{T},\dR,K}\) the de Rham \(P_{\mu_{T}}^{\mrm{std},c}\)-torsor over \(\mS_{K}^{\tor}(G_{T})\) as in \cite[Notation \ref{1-notationDRtorsor}]{Jiang2026Shla}.
	Define \[\omega^{\kw}_{\mS_{K}(G_{T})}:=\mcal{P}^{\mrm{std},c}_{\mu_{T},\dR,K}(W^{(s_{T^{c}}\Bbbk,-w)}_{T})\cong \pi_{\GM,K}^{*}(\omega^{\kw}_{\FLT{T}^{\std}})
	\in \QCoh(\SGTKTo{T}{K}),
	\] where the isomorphism is given by \cite[Lemma \ref{1-lemDRtorsorDescends}]{Jiang2026Shla}. 

	We denote by \(\mPdR{T}\) and \(\omega^{\kw,\sm}_{\SGTKp{T}}\) respectively the pull-back's of \(\mcal{P}^{\mrm{std},c}_{\mu_{T},\dR,K}\) and \(\omega^{\kw}_{\mS_{K}(G_{T})}\) from \(\SGTKTo{T}{K}\) to \(\SGTKpTo{T}^{\sm}\). 

	
	(3) (over \( \FLT{T} \)) For any \(T^{c}\)-multiweight \(\kw\), we define \[\omega^{\kw}_{\Fl_{G_{T},\mu_{T}}}:=\mPHT{T}(\chi_{T^{c}}^{(-\Bbbk,w)})\in \QCoh(\FLT{T}).
	\]
	Here \(\mcal{P}_{\mu_{T}}:=N_{\mu_{T}}\backslash G_{T}\) as in \cite[Notation \ref{1-notationEquShvOnFl}]{Jiang2026Shla}, and
	\(\chi^{(-\Bbbk,w)}_{T^{c}}\)
	is inflated to a representation of \(P_{\mu_{T}}\)
	via \(P_{\mu_{T}}\twoheadrightarrow \prod_{\tau\in T^{c}}H_{\tau}\).
\end{dfn}
\begin{rmk}
Note that
	\(\omega^{\kw}_{\FLT{T}}\)
	is equipped with the horizontal action of \(\mm_{\mu_{T}}^{0}\), which when restricted to \(\fh_{\tau}\) for \(\tau\in T^{c}\) is given by \(\chi^{(-\Bbbk,w)}_{T^{c}}|_{\fh_{\tau}}\).

	Note that \(\omega^{\kw}_{\Fl_{G_{T},\mu_{T}}}\) is of rank \(1\), and \(\omega^{\kw}_{\SGT{T}}\) is of rank \(\dim(V_{T}^{\kw,\sm})=\prod_{\tau\in T}|k_{\tau}-1|\). 
\end{rmk}
\begin{rmk}[Normalization]
	Let us clarify the normalization by making the construction explicit
	in the case of modular curves i.e. \(F=\QQ\) and \(T=\emptyset\). In this case, \(\mrm{Sh}_{K}(G_{\emptyset})\) parametrizes elliptic curves with level structures, and \(P_{\mu_{T}}^{\mrm{std},c}\cong \bar{B}\) (resp. \(P_{\mu_{T}}^{c}\cong B\)), the subgroup of \(\GL_{2/\QQ}\) consisting of lower (resp. upper) triangular matrices. Then by our normalization, the pro\'etale \(K\)-torsor over \(\mrm{Sh}_{K}(G_{\emptyset})\) sends the \(2\)-dimensional standard representation corresponds to the Tate modules of the universal elliptic curve \(f:E^{univ}\to \mrm{Sh}_{K}(G_{\emptyset})\). Therefore, \(\mcal{P}^{\mrm{std},c}_{\mu_{T},\dR,K}\)
	sends the standard representation to the dual of the relative de Rham cohomology \(H^{1}_{\dR}(E^{univ}/\mrm{Sh}_{K}(G_{\emptyset}))\). 
	The Hodge filtration on \(H^{1}_{\dR}(E^{univ}/\mrm{Sh}_{K}(G_{\emptyset}))\) is of the form \[0\to R^{0}f_{*}(\Omega^{1}_{E^{univ}})
	\to H^{1}_{\dR}(E^{univ}/\mrm{Sh}_{K}(G_{\emptyset}))
	\to R^{1}f_{*}(\mO_{E^{univ}})\to 0,
	\] which corresponds to the sequence of \(\bar{B}\)-representations \[
	0\to \chi^{(-1,0)}\to \mrm{Std}^{\vee}\to \chi^{(0,-1)}\to 0.
	\] Thus by our notation,  \[\omega^{((1),-1)}_{\mrm{Sh}_{K}(G_{\emptyset})}:=\mcal{P}^{\std,c}_{\mu_{T},\dR,K}(\chi^{(-1,0)})\cong 
	R^{0}f_{*}(\Omega^{1}_{E^{univ}})\cong \omega_{E^{univ}/\mrm{Sh}_{K}(G_{\emptyset})},\]
	\[\omega^{((-1),-1)}_{\mrm{Sh}_{K}(G_{\emptyset})}:=\mcal{P}^{\std,c}_{\mu_{T},\dR,K}(\chi^{(0,-1)}) \cong R^{1}f_{*}(\mO_{E^{univ}})\cong (\omega_{(E^{univ})^{\vee}/\mrm{Sh}_{K}(G_{\emptyset})})^{\vee}. 
	\]

	On the other hand, over \(\Fl\cong \mP^{1}\), the sequence of \(B\)-representations \[0\to \chi^{(0,-1)}\to \mrm{Std}^{\vee}\to \chi^{(-1,0)}\to 0
	\] which gives rise to the short exact sequence \[
	0\to \omega_{\Fl}^{((1),-1)} \to \mrm{Std}^{\vee}\otimes\mO_{\Fl}\to \omega_{\Fl}^{((-1),-1)}\to 0.
	\] Pulling back the sequence to \(\mS(G_{\emptyset})\), we obtain the Hodge-Tate sequence \[0\to \omega^{((1),-1)}_{\mS_{K}(G_{\emptyset})}\hatotimes_{\mO_{\mS_{K}(G_{\emptyset})}} \hat{\mO}\to \mrm{Std}^{\vee}\otimes \hat{\mO}
	\to \omega^{((-1),-1)}_{\mS_{K}(G_{\emptyset})}\hatotimes_{\mO_{\mS_{K}(G_{\emptyset})}}\hat{\mO}(-1)\to 0.
	\] Also note that the horizontal action of \(d\mu=(1,0)\)
	on \(\omega^{\kw}_{\Fl}\)
	is given by \(\frac{k+w}{2}\)
	for \(\kw=((k),w)\). 
\end{rmk}

\begin{notation}\label{notationOlakw}
Let \(\kw\) be a \(T\)-regular multiweight.
We denote \[\mO^{\la,\kw}_{\SGTKp{T}}:=\RHom_{\ti{\mfk{p}}_{\mu_{T}}}\left(W_{T}^{\kw},\mO^{\la}_{\SGTKp{T}}\right)\in \QCoh(\SGTKpTo{T}^{\la,0})\]
where \(\mfk{p}_{\mu_{T}}\cong \fg_{T}\oplus \fb_{T^{c}}\) acts on \(\mO^{\la}_{\SGTKp{T}}\) by combining of \((\fg_{T},*_{\mrm{const}})\)
and \((\fb_{T^{c}},\theta^{Pan})\) as in \cite[Definition \ref{1-dfnHorActionPan}]{Jiang2026Shla},
and the action factors through \(\ti{\mfk{p}}_{\mu_{T}}\). 
Note that by definition, 
\(\mO^{\la,(\underline{0},0)}_{\SGT{T}}=\mO^{\la,0}_{\SGT{T}}:=\mO_{\SGTKpTo{T}^{\la,0}}\). 

\end{notation}
\begin{lem}\label{lemHodgeTateCompareOnTcase}
There is a natural isomorphism over \(\SGTKpTo{T}^{\la,0}\),
\begin{align}
(\pi^{\la,0}_{\sm,T})^{*}\left(\omega^{\kw,\sm}_{\SGTKp{T}}\right)\otimes 
(\pi_{\HT,T}^{\la,0})^{*}\left(
\omegaFL{T}{(\Bbbk,-w)_{T^{c}}}\right)
\cong  \mO^{\la,\kw}_{\SGTKp{T}}
\left(-\chi^{\kw}(d\mu_{T})\right),
\end{align} where 
\(\left(-\chi^{\kw}(d\mu_{T})\right)\) refers to the Tate twist. Concretely, \( 
-\chi^{\kw}(d\mu_{T})=-\sum_{\tau\in T^{c}}\frac{w-k_{\tau}}{2}. \)

\end{lem}
\begin{proof}
By Hodge-Tate comparison (see for example \cite[Theorem \ref{1-thmIdenTwoTorsorsOverLa}]{Jiang2026Shla}), 
we have 
\[(\pi^{\la}_{\sm,T})^{*}\left(\omega^{\kw,\sm}_{\SGTKp{T}}\right)
\cong 
(\pi_{\HT,T}^{\la})^{*}\left(
\omegaFL{T}{(-\Bbbk,w)_{T^{c}}}\right)\otimes_{\CC_{p}} V_{T}^{\kw,\vee}
\left(-\chi^{\kw}(d\mu_{T})\right)\]
over \(\SGTKpTo{T}^{\la}\). 
Now by applying \( (h^{\la}_{0,T})_{*} \) and taking \(\RHom_{\ti{\p}_{\mu_{T}}}(1,-)\)
on both sides, 
with the action as in Notation \ref{notationOlakw}, we obtain \[
(\pi^{\la,0}_{\sm,T})^{*}\left(\omega^{\kw,\sm}_{\SGTKp{T}}\right)
\cong 
(\pi_{\HT,T}^{\la,0})^{*}\left(
\omegaFL{T}{(-\Bbbk,w)_{T^{c}}}\right)\otimes \mO^{\la,\kw}_{\SGTKp{T}}
\left(-\chi^{\kw}(d\mu_{T})\right)
\] Since \(\omega^{(-\Bbbk,w)_{T^{c}}}_{\FLT{T}}\)
is a line bundle over \(\FLT{T}\), we can dualize it to the other side and obtain the desired isomorphism. 
\end{proof}
\begin{rmk}\label{rmkOlaKWConcentrate}
In particular, combined with \cite[Lemma \ref{1-lemHorizalActNotSame}]{Jiang2026Shla},
this shows that \(\mO^{\la,\kw}_{\SGTKp{T}}\) is concentrated in degree \(0\) when regarded as a sheaf on the analytic site of \(\SGTKp{T}\).
\end{rmk}
We also need to twist \(\mO^{\la,\kw}_{\SGTKp{T}}\) by line bundles on \(\FLT{T}
\), 
so we introduce the following more general notation.
\begin{notation}\label{notationNablaKW}
Let \(\kw\) be a \(T\)-regular multiweight,
and \((\Bbbk',w')\) be a \(T^{c}\)-multiweight. We denote \[\nabla^{\kw,(\Bbbk',w')}_{\SGTKp{T}}:=(\pi^{\la,0}_{\sm,T})^{*}(\omega^{\kw,\sm}_{\SGTKp{T}})\otimes (\pi^{\la,0}_{\HT,T})^{*}(
\omegaFL{T}{(\Bbbk',w')})
\in \QCoh(\SGTKpTo{T}^{\la,0}).\]
In particular, by Lemma \ref{lemHodgeTateCompareOnTcase},
	\[\nabla_{\SGTKp{T}}^{(\Bbbk,w),(\Bbbk,-w)_{T^{c}}}\cong \mO^{\la,(\Bbbk,w)}_{\SGTKp{T}}\left(-\chi^{\kw}(d\mu_{T})\right).\]
\end{notation}


\subsection{Differential operators}\label{subsectionConstrDiffeOper} 
We now define various (log) differential operators on \( \nabla_{\SGTKp{T}}^{\kw,(\Bbbk',w')} \). Recall that
(log) differential operators can be realized in the category of quasicoherent sheaves on the analytic de Rham stacks by taking suitable push-forward (\cite[Lemma \ref{1-lemDiffOpeToLogStack}]{Jiang2026Shla}). 

We recall the following dual BGG resolution on \( \FLT{T}^{\std} \) and \( \FLT{T} \) from \cite[Corollary \ref{1-corFilterDualBGG}]{Jiang2026Shla}. Recall that \( \mcal{G}^{\std,c,+}_{\mu_{T}}:=\mcal{G}^{c,+}_{\mu_{T}^{-1}} \) 
over \( \FLT{T}^{\std}=\Fl_{G_{T},\mu_{T}^{-1}} \).
\begin{notation}[Theta operators]\label{notationDiffHilber}---

(1) (over \( \FLT{T}^{\std} \)) Let \( \kw \) be a regular multiweight with \( k_{\tau}\in\ZZ_{\le0} \)
for \( \tau\in T^{c} \). Then 
\begin{align}\label{alignStdBGGCompelxHilb}
\mcal{G}_{\mu_{T}}^{\std,c,+}(V^{(\Bbbk,-w)})\cong 
\left[M^{\std,0}\to M^{\std,1}\cdots \to M^{\std,|T^{c}|} \right]\in \QCoh((\FLT{T}^{\std})^{\pdR/\CC_{p},+})/G_{T}^{c}
\end{align}
where \( M^{\std,i}\cong \bigoplus_{I\subset T^{c},|I|=i}M_{I}^{\std}
\) is in cohomological degree \( i \), with \[
M_{I}^{\std}\cong (h_{\Fl^{\std},T}^{+})_{*}(\omega^{(s_{I}\cdot\Bbbk,w)}_{\FLT{T}^{\std}})\left\{\sum_{\tau\in I}(1-k_{\tau})+\sum_{\tau\in\Sigma_{\infty}}\frac{w+k_{\tau}}{2}\right\}. 
\] 

Moreover, by \cite[Corollary \ref{1-corPartialBGGDual}]{Jiang2026Shla}, the differential from \( M_{I}^{\std}\subset M^{\std,i} \) for \( |I|=i \)
to \( M_{J}^{\std}\subset M^{\std,j} \) for \( |J|=i+1 \) is zero unless \( I\subset J \). We denote the differential from \( M_{I}^{\std} \)
to \( M_{I\cup\{\tau\}}^{\std} \) for \( \tau\in T^{c}\backslash I \)
by \[
\theta^{1-k_{\tau}}_{\tau}:M_{I}^{\std}\to M_{I\cup\{\tau\}}^{\std}.
\]

(2) (over \( \SGTKTo{G}{K} \))
Let \( \bar{\mcal{G}}_{\mu_{T},\dR,K}^{c,+} \) be the de Rham \( G^{c} \)-torsor over \( \SGTKTo{T}{K}^{\pdR/\CC_{p},+}_{\log} \) (\cite[Lemma \ref{1-lemDRtorsorDescends}]{Jiang2026Shla}). 
By \cite[Proposition \ref{1-propBGGFilSV}]{Jiang2026Shla},
we have the dual BGG resolution \begin{align}\label{alignBGGCompelxSVHilb}
\bar{\mcal{G}}_{\mu_{T},\dR,K}^{c,+}(V^{(\Bbbk,-w)})\cong 
\left[M_{K}^{0}\to M_{K}^{1}\cdots \to M_{K}^{|T^{c}|} \right]\in \QCoh(\SGTKpTo{T}^{\pdR/\CC_{p},+,\sm}_{\log}),
\end{align} where \( M_{K}^{i}\cong \bigoplus_{I\subset T^{c},|I|=i}M_{K,I} \) is in cohomological degree \( i \), with \[
M_{K,I}\cong (h_{K,T}^{+})_{*}(\omega^{(s_{I}\cdot\Bbbk,w)}_{\SGTK{T}{K}})\left\{\sum_{\tau\in I}(1-k_{\tau})+\sum_{\tau\in\Sigma_{\infty}}\frac{w+k_{\tau}}{2}\right\}
\] for \( h_{K,T}^{+}:\SGTKTo{T}{K}\times [\mA^{1}/\GG_{m}]\to \SGTKTo{T}{K}^{\pdR/\CC_{p},+}_{\log} \).

As in (1), by \cite[Corollary \ref{1-corPartialBGGDual}]{Jiang2026Shla} and the construction in Proposition \ref{1-propBGGFilSV} loc. cit., the differential \( M_{K,I}\to M_{K,J} \) for \( |J|=|I|+1 \)
is zero unless \( J=I\cup \{\tau\} \), in which case we denote the differential as \( \theta^{1-k_{\tau}}_{\tau}:M_{K,I} \to M_{K,I\cup\{\tau\}}\).

(2)' (over \( \SGTKpTo{G}^{\sm} \)) Let \( \bar{\mcal{G}}_{\mu_{T},\dR}^{c,+} \)
denote the pull-back of \( \bar{\mcal{G}}_{\mu_{T},\dR,K}^{c,+} \)
to \( \SGTKpTo{T}^{\pdR/\CC_{p},+,\sm}_{\log} \). Then by \cite[Proposition \ref{1-propBGGFilSVsm}]{Jiang2026Shla},
we have \begin{align}\label{alignBGGCompelxSVsmHilb}
\bar{\mcal{G}}_{\mu_{T},\dR}^{c,+}(V^{(\Bbbk,-w)})\cong 
\left[M^{\sm,0}\to M^{\sm,1}\cdots \to M^{\sm,|T^{c}|} \right]\in \QCoh(\SGTKpTo{T}^{\pdR/\CC_{p},+,\sm}_{\log}),
\end{align} 
where \( M^{\sm,i}\cong \bigoplus_{I\subset T^{c},|I|=i}M_{I}^{\sm}
\) is in cohomological degree \( i \), with \[
M_{I}^{\sm}\cong (h_{T}^{\sm,+})_{*}(\omega^{(s_{I}\cdot\Bbbk,w),\sm}_{\SGTKp{T}})\left\{\sum_{\tau\in I}(1-k_{\tau})+\sum_{\tau\in\Sigma_{\infty}}\frac{w+k_{\tau}}{2}\right\}. 
\] 

As in (1) and (2), by \cite[Corollary \ref{1-corPartialBGGDual}]{Jiang2026Shla} and the construction in Proposition \ref{1-propBGGFilSVsm}, the differential \( M^{\sm}_{I}\to M^{\sm}_{J} \) for \( |J|=|I|+1 \)
is zero unless \( J=I\cup \{\tau\} \), in which case we denote the differential as \( \theta^{1-k_{\tau}}_{\tau}:M^{\sm}_{I} \to M^{\sm}_{I\cup\{\tau\}}\).

(3) Let \( (\Bbbk',w') \) be a \( T^{c} \)-multiweight with \( k_{\tau}\in\ZZ_{\le0} \) 
for \( \tau\in T^{c} \). Then similar to (1), by \cite[Proposition \ref{1-propFilterBGG}]{Jiang2026Shla}
\begin{align}\label{alignBGGCompelxHilb}
\mcal{G}_{\mu_{T}}^{+}(V^{(\Bbbk',w')}_{T^{c}})
\cong 
\left[M^{0}\to M^{1}\cdots \to M^{|T^{c}|} \right]\in \QCoh(G_{T}\backslash \FLT{T}^{\pdR/\CC_{p},+})
\end{align}
where \( M^{i}\cong \bigoplus_{I\subset T^{c},|I|=i}M_{I}
\) is in cohomological degree \( i \), with \[
M_{I}\cong (h_{\Fl,T}^{+})_{*}(\omega^{(s_{I}\cdot\Bbbk',-w')}_{\FLT{T}})\left\{\sum_{\tau\in I}(1-k_{\tau}')+\sum_{\tau\in\Sigma_{\infty}}\frac{w'+k_{\tau}'}{2}\right\}. 
\] 

Similarly to (1), the differential from \( M_{I}\to M_{J}\subset M^{j} \) for \( |J|=|I|+1 \) is zero unless \( I\subset J \), and we denote the differential from \( M_{I} \)
to \( M_{I\cup\{\tau\}} \) for \( \tau\in T^{c}\backslash I \)
by \[
\bar{\theta}^{1-k_{\tau}}_{\tau}:M_{I}\to M_{I\cup\{\tau\}}.
\]
Here we adopt the notation \( \bar{\theta} \) to parallel with the notation \( \bar{d} \)
in \cite{Pan2209.06II}.
\end{notation}

We will now use \cite[Lemma \ref{1-lemKunnFlSm}]{Jiang2026Shla} to extend both theta operators to \(\nabla_{\SGTKp{T}}^{\kw,(\Bbbk',w')}\):
\begin{dfn}[Differential operators \(d_{\tau}^{1-k_{\tau}}\) and \(\bar{d}_{\tau}^{1-k'_{\tau}}\)]\label{dfnDiffeOperatorsdanddbar}
Let \(\kw\) be a regular multiweight,
and \((\Bbbk',w')\) be a \(T^{c}\)-multiweight. Recall that by \cite[Lemma \ref{1-propNablaKWasPullBackDR} (2)]{Jiang2026Shla}, we have \[(h^{\la,0,\hpp}_{T})_{*}
\nabla^{\kw,(\Bbbk',w')}_{\SGTKp{T}}\cong (\pi_{\GM\HT,T}^{\pdR})^{*}\left(
	(h^{\sm,\hp}_{T})_{*}(\omega^{\kw,\sm}_{\SGTKp{T}})\boxtimes (h^{\hp}_{\Fl,T})_{*}(\omega^{(\Bbbk',w')}_{\FLT{T}})\right)\] in \( \QCoh((\SGTKpTo{T}^{\la,0})^{\pdR/\CC_{p},\hpp}_{\log}) \).
Consider the structure map \[s:
(\SGTKpTo{T}^{\la,0})^{\pdR/\CC_{p},\hpp}_{\log}\to [\mA^{1}/\GG_{m}]^{2},
\] and we write \( M\{a,b\}:=M\otimes s^{*}(\mO\{a\}\boxtimes \mO\{b\}) \) (\cite[Notation \ref{1-notationTwistByAGm}]{Jiang2026Shla}).

(1) If \(\tau\in T^{c}\)
and \(k_{\tau}\in\ZZ_{\le 0}\),
we define the \(d^{1-k_{\tau}}_{\tau}\)-operator  \[d_{\tau}^{1-k_{\tau}}:(h^{\la,0,\hpp}_{T})_{*}\nabla^{\kw,(\Bbbk',w')}_{\SGTKp{T}}\to (h^{\la,0,\hpp}_{T})_{*}\nabla^{(s_{\tau}\cdot\Bbbk,w),(\Bbbk',w')}_{\SGTKp{T}}\{1-k_{\tau},0\},\]
as the morphism induced by \(\theta_{\tau}^{1-k_{\tau}}\) in Notation \ref{notationDiffHilber} (2).

(2) If \(\tau\in T^{c}\), and \(k_{\tau}'\in\ZZ_{\le 0}\),
we denote the \(\bar{d}^{1-k'_{\tau}}_{\tau}\)-operator \[\bar{d}_{\tau}^{1-k_{\tau}'}:(h^{\la,0,\hpp}_{T})_{*}\nabla^{\kw,(\Bbbk',w')}_{\SGTKp{T}}\to (h^{\la,0,\hpp}_{T})_{*}\nabla^{(\Bbbk,w),(s_{\tau}\cdot\Bbbk',w')}_{\SGTKp{T}}\{0,1-k_{\tau}'\},\]
as the morphism induced by \(\bar{\theta}^{1-k'_{\tau}}_{\tau}\) in Notation \ref{notationDiffHilber} (3).
\end{dfn}

\begin{lem} \label{lemBasicNablaDiffe}
 All the operators \(d^{1-k_{\tau}}_{\tau}\)
	and \(\bar{d}^{1-k_{\tau'}'}_{\tau'}\) (when defined) commute with each other up to a sign.

	In particular, if \(k_{\tau}\in\ZZ_{\le 0}\),
	we have a commutative diagram over \((\SGTKpTo{T}^{\la,0})^{\pdR/\CC_{p},\hpp}_{\log}\) \[\begin{tikzcd}
		(h^{\la,0,\hpp}_{T})_{*}\mO^{\la,\kw}_{\SGTKp{T}}\left(-\chi^{\kw}(d\mu_{T})\right)
		\arrow[r,"d^{1-k_{\tau}}_{\tau}"]\arrow[d,"\bar{d}^{1-k_{\tau}}_{\tau}"]
		& (h^{\la,0,\hpp}_{T})_{*}\nabla_{\SGTKp{T}}^{(s_{\tau}\cdot\Bbbk,w),(\Bbbk,-w)}\{1-k_{\tau},0\}
		\arrow[d,"\bar{d}^{1-k_{\tau}}_{\tau}"]
		\\ (h^{\la,0,\hpp}_{T})_{*}\nabla_{\SGTKp{T}}^{(\Bbbk,w),(s_{\tau}\cdot\Bbbk,-w)}\{0,1-k_{\tau}\}
		\arrow[r,"-d^{1-k_{\tau}}_{\tau}"]
			& (h^{\la,0,\hpp}_{T})_{*}\mO^{\la,(s_{\tau}\cdot\Bbbk,w)}_{\SGTKp{T}}\left(-\chi^{(s_{\tau}\cdot\Bbbk,w)}(d\mu_{T})\right)\{1-k_{\tau},1-k_{\tau}\}.
	\end{tikzcd}\]

	Moreover, the composition \( \bar{d}_{\tau}^{1-k_{\tau}}\circ d_{\tau}^{1-k_{\tau}}\cong -d_{\tau}^{1-k_{\tau}}\circ \bar{d}_{\tau}^{1-k_{\tau}} \) satisfies the property that the map induced by adjuntion \[(h^{\la,0,\hpp}_{T})^{*}
	(h^{\la,0,\hpp}_{T})_{*}\mO^{\la,\kw}_{\SGTKp{T}}\left(-\chi^{\kw}(d\mu_{T})\right)\to 
	\mO^{\la,(s_{\tau}\cdot\Bbbk,w)}_{\SGTKp{T}}\left(-\chi^{(s_{\tau}\cdot\Bbbk,w)}(d\mu_{T})\right)\{1-k_{\tau},1-k_{\tau}\}
	\] induces a surjection on \[
	\gr^{(1-k_{\tau},1-k_{\tau})}\left((h^{\la,0,\hpp}_{T})^{*}
	(h^{\la,0,\hpp}_{T})_{*}\mO^{\la,\kw}_{\SGTKp{T}}\left(-\chi^{\kw}(d\mu_{T})\right)\right)
	\twoheadrightarrow \mO^{\la,(s_{\tau}\cdot\Bbbk,w)}_{\SGTKp{T}}\left(-\chi^{(s_{\tau}\cdot\Bbbk,w)}(d\mu_{T})\right).
	\]
\end{lem}
\begin{proof}
By Definition \ref{dfnDiffeOperatorsdanddbar}, it suffices to verify it over 
\( \FLT{T}^{\pdR/\CC_{p},\hp}\times (\mX^{\tor}_{K^{p}})^{\pdR/\CC_{p},+,\sm} \).
The commutativity  among \( \bar{d}_{\tau}^{1-k'_{\tau}} \) and that between \( \bar{d}_{\tau}^{1-k'_{\tau}} \) and \( d_{\tau}^{1-k_{\tau}} \) follow easily since they comes from different factors. The commutativity among \( d_{\tau}^{1-k_{\tau}} \) follows from the commutativity among \( \theta_{\tau}^{1-k_{\tau}} \) as they form a complex (\ref{alignBGGCompelxSVsmHilb}).
The last part follows from \cite[Example \ref{1-egThetaP1}]{Jiang2026Shla}.
\end{proof}

\subsubsection{Algebraic partial de Rham complexes}\label{subsecPartialDRComplex}

The main theorem that we want to prove is about the following partial de Rham complexes, which arises from the \( \ZZ^{T^{c}} \)-filtration in \cite[Corollary \ref{1-corPartialBGGDual}]{Jiang2026Shla}. 
\begin{dfn}[Algebraic partial de Rham complexes]\label{dfnDRComplexNotation}--

(1) (over \( \FLT{T}^{\std} \))
Let \(\kw\) be a regular multiweight
and \(I\subset T^{c}\),
such that \(k_{\tau}\in\ZZ_{\le 0}\) for \(\tau\in I\).

By \cite[Corollary \ref{1-corPartialBGGDual}]{Jiang2026Shla},
we have a \( \ZZ^{T^{c}} \)-filtration \( \Fil_{\mrm{BGG}^{\vee}}^{\bullet} \)on \( \mcal{G}^{\std,+}_{\mu_{T}}(V^{(\Bbbk,-w)}) \).
For any \( J\subset T^{c} \), we denote by \( \Fil_{\mrm{BGG}^{\vee},J}^{\bullet} \) the corresponding \( \ZZ^{J} \)-filtration.
We further denote \(\unl{1}_{\kw,J}= ((\unl{1}_{\kw,J})_{\tau})_{\tau\in J}\in\ZZ^{J} \) with \( (\unl{1}_{\kw,J})_{j}=0 \)
if \( k_{\tau}\le 0 \)
and \( =1 \) if \( k_{\tau}\ge 2 \).

We define the \emph{algebraic partial de Rham complex} on \(\FLT{T}^{\std}\) as  
\begin{align*}
\dR_{I}^{\alg}(\omega^{\kw}_{\FLT{T}^{\std}}):=\gr^{\unl{1}_{\kw,T^{c}\backslash I}}_{\mrm{BGG}^{\vee},T^{c}\backslash I}(\mcal{G}^{\std,+}_{\mu_{T}}(V^{(\Bbbk,-w)}))\left\{-
\sum_{\tau\in\Sigma_{\infty}}\frac{w+k_{\tau}}{2}
\right\}\left[|\unl{1}_{\kw,T^{c}\backslash I}|\right]\\
\in\QCoh(G\backslash \FLT{T}^{\std,\pdR/\CC_{p},+}),
\end{align*}
which still carries a \( \ZZ^{I} \)-filtration. Concretely, it is isomorphic to the complex \[
(h^{+}_{\Fl^{\std},T})_{*}\omega^{\kw}_{\Fl^{\std}_{G,\mu}}\to \bigoplus_{\tau\in I}(h^{+}_{\Fl^{\std},T})_{*}\omega^{(s_{\tau}\cdot\Bbbk,w)}_{\Fl^{\std}_{G,\mu}}\{1-k_{\tau}\}\to \cdots \to (h^{+}_{\Fl^{\std},T})_{*}\omega^{(s_{I}\cdot\Bbbk,w)}_{\Fl^{\std}_{G,\mu}}\{\sum_{\tau\in I}(1-k_{\tau})\}.
\]


(2) (Canonical extension over \( \SGTKpTo{T}^{\sm} \)) Let \( \kw \) and \( I \) be as in (1). Then by \cite[Proposition \ref{1-propPartialBGGDual}]{Jiang2026Shla} and the construction of Proposition \ref{1-propBGGFilSVsm} loc. cit., we have a \( \ZZ^{T^{c}} \)-filtration \( \Fil^{\bullet}_{\mrm{BGG}^{\vee}} \) on \( \bar{\mcal{G}}_{\mu_{T},\dR,K}^{c,+}(V^{(\Bbbk,-w)}) \). 
We define the \emph{algebraic partial de Rham complex} on \(\SGTKTo{T}{K}\) as  
\begin{align*}
\dR_{I}^{\alg}(\omega^{\kw,\sm}_{\SGTKp{T}})
:=\gr^{\unl{1}_{\kw,T^{c}\backslash I}}_{\mrm{BGG}^{\vee},T^{c}\backslash I}(\bar{\mcal{G}}^{c,+}_{\mu_{T},\dR}(V^{(\Bbbk,-w)}))\left\{-
\sum_{\tau\in\Sigma_{\infty}}\frac{w+k_{\tau}}{2}
\right\}\left[|\unl{1}_{\kw,T^{c}\backslash I}|\right]
\\\in \QCoh(\SGTKpTo{T}^{\pdR/\CC_{p},+,\sm}_{\log}),
\end{align*} which still carries a \( \ZZ^{I} \)-filtration.

Thus concretely, 
it is
represented by the complex 
\begin{align}\label{alignPartialDRcomplexHilbSm}
(h^{\sm,+}_{T})_{*}\omega^{\kw,\sm}_{\SGTKp{T}}\to  \bigoplus_{I'\subset I,\#I'=i}(h^{\sm,+}_{T})_{*}\omega^{(s_{I'}\cdot\Bbbk,w),\sm}_{\SGTKp{T}}\{1-k_{\tau}\}\to\cdots\\\to (h^{\sm,+}_{T})_{*}\omega^{(s_{I}\cdot\Bbbk,w),\sm}_{\SGTKp{T}}\{\sum_{\tau\in I}(1-k_{\tau})\},
\end{align}
where for \(I'\subset I'\coprod\{\tau\}\subset I\), the differential from the \(I'\)-th term to the \(I'\coprod\{\tau\}\)-th term is given by \( \theta^{1-k_{\tau}}_{\tau} \), and the other differentials are zero.



(2)' (Subcanonical extension over \( \SGTKpTo{T}^{\sm} \)) We define the \emph{algebraic partial de Rham complex with support} on \(\SGTKTo{T}{K}\) by \[ \dR_{I}^{\alg}(\omega^{\kw,\sm}_{\SGTKp{T},c}):=\dR_{I}^{\alg}(\omega^{\kw,\sm}_{\SGTKp{T}})\otimes (\pi^{\pdR,\sm}_{K^{p}K_{p,0}})^{*}(\mscr{I}_{K^{p}K_{p,0}}), \] for \( \pi^{\pdR,\sm}_{K^{p}K_{p,0}}:\SGTKpTo{T}^{\pdR/\CC_{p},+,\sm}_{\log}\to \SGTKTo{T}{K}^{\pdR/\CC_{p},+}_{\log} \),
where \( \mscr{I}_{K^{p}K_{p,0}} \) is the line bundle in \cite[Remark \ref{1-rmkSubCanonicalExtension}]{Jiang2026Shla}. 

By \cite[Lemma \ref{1-lemProjectionFormuaForPerf}]{Jiang2026Shla}, \( \dR_{I}^{\alg}(\omega^{\kw,\sm}_{\SGTKp{T},c}) \) has a concrete description as a subcomplex of (\ref{alignPartialDRcomplexHilbSm})  
\begin{align}\label{alignPartialDRcomplexHilbSmCompact}
(h^{\sm,+}_{T})_{*}\omega^{\kw,\sm}_{\SGTKp{T},c}\to  \bigoplus_{I'\subset I,\#I'=i}(h^{\sm,+}_{T})_{*}\omega^{(s_{I'}\cdot\Bbbk,w),\sm}_{\SGTKp{T},c}\{1-k_{\tau}\}\to\cdots\\\to (h^{\sm,+}_{T})_{*}\omega^{(s_{I}\cdot\Bbbk,w),\sm}_{\SGTKp{T},c}\{\sum_{\tau\in I}(1-k_{\tau})\},
\end{align}
where \(\omega^{\kw,\sm}_{\SGTKp{T},c}:=\omega^{\kw,\sm}_{\SGTKp{T}}\otimes_{\mO_{\SGTKTo{T}{K}}}\mscr{I}_{\SGTKTo{T}{K}}\). 

(3) (over \( \FLT{T} \)) Let \((\Bbbk',w')\) be a \(T^{c}\)-multiweight, and \(J\subset T^{c}\), such that \(k'_{\tau}\in\ZZ_{\le 0}\) for \(\tau\in J\).
As in (1), \( \mcal{G}^{+}_{\mu_{T}}(V^{(\Bbbk',w')}_{T^{c}}) \) admits a \( \ZZ^{T^{c}} \)-filtration, and
we define the \emph{algebraic partial de Rham complex} on \(\FLT{T}\) as \[\bar{\dR}_{J}(\omega^{(\Bbbk',w')}_{\FLT{T}}):=\gr^{\unl{1}_{(\Bbbk',w'),T^{c}\backslash J}}_{\mrm{BGG}^{\vee},T^{c}\backslash J}(\mcal{G}^{+}_{\mu_{T}}(V^{(\Bbbk',w')}_{T^{c}}))
\left\{-
\sum_{\tau\in\Sigma_{\infty}}\frac{w'+k_{\tau}'}{2}
\right\}\left[|\unl{1}_{\kw,T^{c}\backslash J}|\right]
\in
\QCoh(\FLT{T}^{\pdR/\CC_{p},+}),
\] which is concretely
represented by the complex 
\begin{align*}
(h^{+}_{\Fl,T})_{*}\omega^{(\Bbbk',-w')}_{\FLT{T}}\to \cdots\to \bigoplus_{J'\subset J,\#J'=j}(h^{+}_{\Fl,T})_{*}\omega^{(s_{J'}\cdot\Bbbk',-w')}_{\FLT{T}}\{\sum_{j\in J'}(1-k_{\tau}')\}\to\cdots\\
\to (h^{+}_{\Fl,T})_{*}\omega^{(s_{J}\cdot\Bbbk',-w')}_{\FLT{T}}\{\sum_{j\in J}(1-k_{\tau}')\},
\end{align*}
with connecting morphisms given by \(\bar{\theta}^{1-k_{\tau}'}_{\tau}\) from the \(J'\)-th term to the \(J'\coprod\{\tau\}\)-th term (for any \(J'\coprod\{\tau\}\subset J\)), and zero on the other terms.

(4) Given \(\kw\) and \(I\) as in (1), and \((\Bbbk',w')\) and \(J\) as in (3), we define the \emph{algebraic partial de Rham complex} on \(\SGTKpTo{T}^{\la,0}\)
as 
\begin{align*}
\dR_{I,J}^{\alg}(\nabla_{\SGTKp{T}}^{\kw,(\Bbbk',w')}):=
(\pi_{\GM\HT,T}^{\pdR,\hpp})^{*}\left(
	\dR_{I}^{\alg}(\omega^{\kw,\sm}_{\SGTKp{T}})\boxtimes\bar{\dR}_{J}^{\alg}(\omega^{(\Bbbk',w')}_{\FLT{T}})
\right)\\
\in \QCoh((\SGTKpTo{T}^{\la,0})^{\pdR/\CC_{p},\hpp}_{\log}).
\end{align*}
By \cite[Lemma \ref{1-lemKunnFlSm}]{Jiang2026Shla} and the arguments in (2) and (3) above, \( 
\dR_{I,J}^{\alg}(\nabla_{\SGTKp{T}}^{\kw,(\Bbbk',w')}) \) has a concrete description as a complex, whose \( i \)-th term is
\begin{align*}
\bigoplus_{(I'\subset I,J'\subset J):|I'|+|J'|=i}(h^{\la,0,\hpp}_{T})_{*}(\nabla^{(s_{I'}\cdot\Bbbk,w),(s_{J'}\cdot \Bbbk',w')}_{\SGTKp{T}})\left\{
	\sum_{\tau\in I'}(1-k_{\tau}),\sum_{\tau\in J'}(1-k_{\tau}')
\right\}
\end{align*} and whose differentials are induced by \( d_{\tau}^{1-k_{\tau}} \)
for \( \tau\in I \)
and \( \bar{d}_{\tau}^{1-k_{\tau}'} \) for \( \tau\in J \).

Similarly, we define the \emph{algebraic partial de Rham complex with support} on \(\SGTKpTo{T}^{\la,0}\)
as 
\begin{align*}
\dR_{I,J}^{\alg}(\nabla_{\SGTKp{T},c}^{\kw,(\Bbbk',w')}):=
(\pi_{\GM\HT,T}^{\pdR,\hpp})^{*}\left(
	\dR_{I}^{\alg}(\omega^{\kw,\sm}_{\SGTKp{T},c})\boxtimes\bar{\dR}_{J}^{\alg}(\omega^{(\Bbbk',w')}_{\FLT{T}})
\right)\\
\in \QCoh((\SGTKpTo{T}^{\la,0})^{\pdR/\CC_{p},\hpp}_{\log}).
\end{align*}
\end{dfn}

We will show that the partial de Rham cohomology, i.e. the cohomology of \( 
\dR_{I,J}^{\alg}(\nabla_{\SGTKp{T}}^{\kw,(\Bbbk',w')}) \) and \( 
\dR_{I,J}^{\alg}(\nabla_{\SGTKp{T},c}^{\kw,(\Bbbk',w')}) \), is classical as Hecke modules if \( I\cup J=T^{c} \). Let us make this precise.
\begin{dfn}[Classical Hecke modules]\label{dfnClassicalHecke}
Let \(\mscr{C}\) be a \(D(\bar{\QQ})\)-linear stable \(\infty\)-category, \(\TT^{S}_{\QQ}:=\TT^{S}\otimes_{\ZZ}\QQ\), let \(M\in \Mod_{\TT^{S}_{\QQ}}(\mscr{C})\), and let \(\kw\) be a regular multiweight.
Let \(\mcal{S}(M_{\kw})\) be the set of characters \(\chi_{f}:\TT^{S}_{\QQ}\to\bar{\QQ}\) that correspond to classical Hilbert modular \( \TT^{S} \)-eigenforms \(f\) of weight \((\max(2-\Bbbk,\Bbbk),w)\) that are unramified away from the finite set \( S \), 
with \(\max(2-\Bbbk,\Bbbk):=(\max(2-k_{\tau},k_{\tau}))_{\tau\in\Sigma_{\infty}}\). 

Then we say that
\(M\in\Mod_{\TT^{S}_{\QQ}}(\mscr{C})\)
is \emph{classical of weight \(\kw\) as a Hecke module}, if 
for any prime ideal \( \mfk{a}\subset \TT^{S}_{\bar{\QQ}} \), if \( \mfk{a}\ne \Ker(\chi_{f}) \)
 for any \( \chi_{f}\in \mS(M_{\kw}) \), then
\( M_{\mfk{a}}\cong 0 \).
\end{dfn}
The main theorem of this section is the following:
\begin{thm}\label{thmClassicalityCohomoDR}
	Let \(I,J\subset T^{c}\), \(\kw\) be a regular multiweight, and 
\((\Bbbk',w')\) be a \( T^{c} \)-multiweight,
such that \(k_{\tau}\in\ZZ_{\le 0}\) for \(\tau\in I\), and \( k'_{\tau}\in\ZZ_{\le0} \) for \( \tau\in J \).
	Assume that \(I\cup J=T^{c}\).
	Then
	the complexes
	\begin{align*}
	R\Gamma\left((\SGTKpTo{T}^{\la,0})^{\pdR/\CC_{p},\hpp}_{\log},\dR_{I,J}^{\alg}(\nabla_{\SGTKp{T}}^{(\Bbbk,w),(\Bbbk',w')})\right)\\R\Gamma\left((\SGTKpTo{T}^{\la,0})^{\pdR/\CC_{p},\hpp}_{\log},\dR_{I,J}^{\alg}(\nabla_{\SGTKp{T},c}^{(\Bbbk,w),(\Bbbk',w')})\right)
	\end{align*}
	are classical of weight \(\kw\) as Hecke modules (Definition \ref{dfnClassicalHecke}). 
\end{thm}
The proof of this theorem will be given in \S
\ref{subsecProofOfClassicality}.
\subsubsection{Analytic partial de Rham complexes}
We will prove a locally analytic Jacquet-Langlands correspondence that relates \( \dR_{I,J}^{\alg}(\nabla^{\kw,(\Bbbk',w')}_{\SGTKp{T}}) \) for different \( T \). For this purpose, we actually have to work over the analytic de Rham stacks rather than the algebraic ones.
\begin{dfn}[Analytic partial de Rham complexes on flag varieties]---

(1)
For \( J \)
and \( (\Bbbk',w') \)
as in Definition \ref{dfnDRComplexNotation} (3), we define \[
\bar{\dR}_{J}(\omega^{(\Bbbk',w')}_{\FLT{T}}):=(g^{\hp}_{\Fl,T})_{*}(\bar{\dR}_{J}^{\alg}(\omega^{(\Bbbk',w')}_{\FLT{T}})),
\] where
\( g^{\hp}_{\Fl,T}:\FLT{T}^{\pdR/\CC_{p},\hp}\to \FLT{T}^{\dR/\CC_{p}} \) is as in \cite[Lemma \ref{1-lemDiffOpAlgvsAn}]{Jiang2026Shla}.

(2) For \( I \)
and \( \kw \) as in Definition \ref{dfnDRComplexNotation} (1), we define similarly \[
\dR_{I}(\omega^{(\Bbbk,w)}_{\FLT{T}^{\std}}):=(g^{\hp}_{\Fl^{\std},T})_{*}(\dR_{I}^{\alg}(\omega^{(\Bbbk,w)}_{\FLT{T}^{\std}}))\in \QCoh(G^{c,\an}\backslash \FLT{T}^{\std,\dR/\CC_{p}}),
\] where
\( g^{\hp}_{\Fl^{\std},T}:G^{c,\an}\backslash \FLT{T}^{\std,\pdR/\CC_{p},\hp}\to G^{c,\an}\backslash \FLT{T}^{\std,\dR/\CC_{p}} \) is as in \cite[Lemma \ref{1-lemDiffOpAlgvsAn}]{Jiang2026Shla}.
\end{dfn}
\begin{lem}\label{lemAnotherKunnethType}
Let \( (I,J,\kw,(\Bbbk',w')) \)
be as in Definition \ref{dfnDRComplexNotation} (4). Then the natural morphism \[
\dR_{I}^{\alg}(\omega^{\kw,\sm}_{\SGTKp{T}})\otimes (\pi_{\HT,T}^{\pdR})^{*}(\bar{\dR}_{J}(\omega^{(\Bbbk',w')}_{\FLT{T}}))\to
(p_{1,T})_{*}\dR_{I,J}^{\alg}(\nabla^{\kw,(\Bbbk',w')}_{\SGTKp{T}})
\] is an isomorphism, where \( p_{1,T} \) is as in the Cartesian diagram \[
\begin{tikzcd}
(\SGTKpTo{T}^{\la,0})^{\pdR/\CC_{p},\hpp}_{\log}\arrow[r,"p_{1,T}"]\arrow[d,"p_{2,T}"]
& \SGTKpTo{T}_{\log}^{\pdR/\CC_{p},\hp,\sm}\arrow[d,"\pi_{\HT,T}^{\pdR}"]\\
\FLT{T}^{\pdR/\CC_{p},\hp}
\arrow[r,"g_{\Fl,T}^{\hp}"]
&\FLT{T}^{\dR/\CC_{p}}.
\end{tikzcd}
\]
\end{lem}
\begin{proof}
By the complexes in Definition \ref{dfnDRComplexNotation}, it suffices to consider the case where \( I=J=\emptyset \). Concretely, since \( p_{1,T}\circ h^{\la,0,\hpp}_{T}\cong h^{\sm,+}_{T}\circ \pi^{\la,0}_{\sm,T} \), we are left to show that for \( \mcal{L}\in\Vect(\SGTKpTo{T}^{\sm}) \) and \( \mcal{L}'\in\Vect(\FLT{T}) \), 
 \[
(h^{\sm,+}_{T}\circ \pi^{\la,0}_{\sm,T})_{*}\left((\pi^{\la,0}_{\sm,T})^{*}(\mcal{L})\otimes (\pi^{\la,0}_{\HT,T})^{*}(\mcal{L}')\right)
\cong (h^{\sm,+}_{T})_{*}(\mcal{L})\otimes (\pi_{\HT,T}^{\pdR})^{*}(\beta_{\Fl,*}(\mcal{L}')).
\]
For this, the LHS is isomorphic to \( (h^{\sm,+}_{T})_{*}(\mcal{L}\otimes (\pi^{\la,0}_{\sm,T})_{*}\circ(\pi^{\la,0}_{\HT,T})^{*}(\mcal{L}')) \) by \cite[Lemma \ref{1-lemProjectionFormuaForPerf}]{Jiang2026Shla}. 
This is in turn isomorphic to \( (h^{\sm,+}_{T})_{*}(\mcal{L}\otimes (\pi_{\HT,T}^{\sm})^{*}\circ(\beta_{\Fl,T})_{*}(\mcal{L}')) \)
by proper base change (\cite[Lemma \ref{1-lemSmBaseChange} (2)]{Jiang2026Shla}), which is applicable by Theorem \ref{1-thmAnDRStackGeneral} (3) and Proposition \ref{1-propCartesianLA0} loc. cit. We then conclude by \cite[Lemma \ref{1-lemKunnFlSm} (3)]{Jiang2026Shla}.
\end{proof}
\begin{eg}\label{egdRITc}
Let \( (I,J=T^{c},\kw,(\Bbbk',w')) \)
be as in Definition \ref{dfnDRComplexNotation} (4).
By Corollary \ref{1-corFilterDualBGG} and Proposition \ref{1-propAnToAlgDRJuan} (2) of \cite{Jiang2026Shla}, we see that \(
\bar{\dR}_{T^{c}}(\omega^{(\Bbbk',w')}_{\FLT{T}})
\cong
V_{T^{c}}^{(\Bbbk',w')}\otimes \mO_{\FLT{T}^{\dR/\CC_{p}}}
\).
 Thus by Lemma \ref{lemAnotherKunnethType}, \[
(p_{1,T})_{*}
\dR_{I,T^{c}}^{\alg}(\nabla^{\kw,(\Bbbk',w')}_{\SGTKp{T}})
\cong \dR_{I}^{\alg}(\omega^{\kw,\sm}_{\SGTKp{T}})\otimes V^{(\Bbbk',w')}_{T^{c}}.
\] 
\end{eg}
Away from the toroidal boundaries, we can further push forward to \( \SGTKp{T}^{\dR} \).
Note that the boundaries only appear when \( T=\emptyset \). 
\begin{dfn}[Analytic partial de Rham complexes on \( \SGTKp{T}^{\dR} \)]
Let \( (I,J,\kw,(\Bbbk',w')) \) be
as in Definition \ref{dfnDRComplexNotation} (4). 
We define \[
\dR_{I,J}(\nabla^{\kw,(\Bbbk',w')}_{\SGTKp{T}}):=(g^{\sm,\hp}_{T})_{*}\left(((p_{1,T})_{*}
	\dR_{I,J}^{\alg}(\nabla^{\kw,(\Bbbk',w')}_{\SGTKp{T}}))|_{\SGTKp{T}^{\pdR/\CC_{p},\hp,\sm}}
\right),
\] where \( g^{\sm,\hp}_{T}:\SGTKp{T}^{\pdR/\CC_{p},\hp,\sm}\to \SGTKp{T}^{\dR} \). 
\end{dfn}

The partial de Rham complex has a reinterpretation in terms of Theorem \ref{1-thmLAstrucGMHTper} and Corollary \ref{1-corCartesianFlandSMGeneral} of \cite{Jiang2026Shla}:
\begin{dfn}
Consider the Cartesian diagram of solid stacks \[
\begin{tikzcd}
\SGTKp{T}^{\la,0}\arrow[r,"\pi_{\GM\HT,T}^{\la,0}"]\arrow[d,"\beta^{\la,0}_{T}"] & \Per_{G_{T},\mu_{T}}^{\dR/(\Fl_{G_{T},\mu_{T}}^{\std}\times_{\CC_{p}} \FLT{T})}/(G^{c,\an},*_{\GM,T})\arrow[d]\arrow[r,"\pi^{per}_{\GM,T}\times \pi^{per}_{\HT,T}"] & \Fl_{G_{T},\mu_{T}}^{\std}/G^{c,\an}\times_{\CC_{p}} \FLT{T}\arrow[d,"\beta_{\Fl^{\std},T}\times \beta_{\Fl,T}"] \\
\SGTKp{T}^{\dR}\arrow[r,"\pi_{\GM\HT,T}^{\dR}"] & \Per_{G_{T},\mu_{T}}^{\dR}/(G^{c,\an},*_{\GM,T})
\arrow[r,"\pi^{per,\dR}_{\GM,T}\times \pi^{per,\dR}_{\HT,T}"] & \Fl_{G_{T},\mu_{T}}^{\std,\dR}/G^{c,\an}\times_{\bar{\QQ}_{p}} \FLT{T}^{\dR}
\end{tikzcd}
\] 
in \cite[Corollary \ref{1-corCartesianFlandSMGeneral}]{Jiang2026Shla}.
Let \( (I,J,\kw,(\Bbbk',w')) \) be
as in Definition \ref{dfnDRComplexNotation} (4). We define \[
\dR_{I,J}(\nabla^{\kw,(\Bbbk',w')}_{\PerT{T}}):=(\pi^{per,\dR}_{\GM,T})^{*}(\dR_{I}(\omega^{\kw}_{\FLT{T}^{\std}}))\otimes (\pi^{per,\dR}_{\HT,T})^{*}(\bar{\dR}_{J}(\omega^{(\Bbbk',w')}_{\FLT{T}})).
\]
\end{dfn}
\begin{lem}\label{lemPartialDRasPullBack}
Let \( (I,J,\kw,(\Bbbk',w')) \) be
as in Definition \ref{dfnDRComplexNotation} (4).
Then we have a natural isomorphism \[
\dR_{I,J}(\nabla^{\kw,(\Bbbk',w')}_{\SGTKp{T}})\cong (\pi^{\dR}_{\GM\HT,T})^{*}\left(
	\dR_{I,J}(\nabla^{\kw,(\Bbbk',w')}_{\PerT{T}})
\right).
\]
\end{lem}
\begin{proof}
We will construct a natural isomorphism from the RHS to the LHS.
The composition \( (\pi_{\GM,T}^{per,\dR}\times \pi_{\HT,T}^{per,\dR})\circ \pi^{\dR}_{\GM\HT,T} \) factors as \[
\SGTKp{T}^{\dR}\xrightarrow{1\times \pi_{\HT,T}^{\dR}}
\SGTKp{T}^{\dR}\times \FLT{T}^{\dR}
\xrightarrow{\pi_{\GM,T}^{\dR}\times 1}\FLT{T}^{\std,\dR}/G^{c,\an}\times \FLT{T}^{\dR}.
\]
Consider the Cartesian diagram \[
\begin{tikzcd}
\SGTKp{T}^{\pdR/\CC_{p},\sm,+}
\arrow[r,"\pi_{\GM,T}^{\pdR}"]\arrow[d,"g^{\sm}_{T}"]
& \FLT{T}^{\std,\pdR/\CC_{p}}/G^{c,\an}\arrow[d,"g_{\Fl^{\std},T}"]
\\
\SGTKp{T}^{\dR}
\arrow[r,"\pi_{\GM,T}^{\dR}"]
& \FLT{T}^{\std,\dR}/G^{c,\an}.
\end{tikzcd}
\] obtained by combining Proposition \ref{1-propAlgGMtheory} and Proposition \ref{1-propAnalyticGMTheory} of \cite{Jiang2026Shla}. This gives a natural morphism \[ (\pi_{\GM,T}^{\dR})^{*}(\dR_{I}(\omega^{\kw}_{\FLT{T}}))\to (g^{\sm,\hp}_{T})_{*}(\dR^{\alg}_{I}(\omega^{\kw,\sm}_{\SGTKp{T}})), \]
which is an isomorphism by \cite[Proposition \ref{1-propAnToAlgDRJuan}, Theorem \ref{1-thmAnDRStackGeneral} (6)]{Jiang2026Shla} and the proper base change (\cite[Lemma \ref{1-lemSmBaseChange} (2)]{Jiang2026Shla}).

Thus we have natural morphisms 
\begin{align*}
 (\pi^{\dR}_{\GM\HT,T})^{*}\left(
	\dR_{I,J}(\nabla^{\kw,(\Bbbk',w')}_{\PerT{T}})
\right)&\cong (g^{\sm,\hp}_{T})_{*}(\dR^{\alg}_{I}(\omega^{\kw,\sm}_{\SGTKp{T}}))\otimes (\pi^{\dR}_{\HT,T})^{*}(\bar{\dR}_{J}(\omega^{(\Bbbk',w')}_{\FLT{T}}))
\\
&\cong (g^{\la,0,\hpp}_{T})_{*}(\dR_{I,J}^{\alg}(\nabla^{\kw,(\Bbbk',w')}_{\SGTKp{T}}))\cong 
\dR_{I,J}(\nabla^{\kw,(\Bbbk',w')}_{\SGTKp{T}}),
\end{align*}
where the second isomorphism is given by \cite[Lemma \ref{1-lemDiffOpAlgvsAn} (2)]{Jiang2026Shla}.
\end{proof}

\subsection{Locally analytic Jacquet-Langlands correspondence}
\label{subsecCompareUnderSpaces}
We fix \(T\subset \Sigma_{\infty}\)
and \(I\subset T^{c}\) non-empty. 
Let us denote \(T': =T\coprod I\subset \Sigma_{\infty}\). We will define a locally analytic Jacquet-Langlands correspondence relating the partial de Rham cohomology over \( \SGTKpTo{T} \)
and \( \SGTKpTo{T'} \) (Theorem \ref{thmJacquetLanglandsLA}).

\subsubsection{For analytic de Rham stacks}
Recall from
Definition \ref{dfnIbasicLocus} the notion of the \(I\)-basic locus \(\FLT{T}^{I-\bc}\subset \FLT{T}\). 
\begin{notation}
Denote \[\SGTKp{T}^{I-\bc}:=\pi_{\HT,T}^{-1}({\Fl_{G_{T},\mu_{T}}^{I-\bc}}),\] which is an open subspace of \(\SGTKpTo{T}\). 

Since for any \(I\ne \emptyset\),
\(I\)-basic locus does not meet the boundary by Lemma \ref{lemNoEmptyBasicLocus} and Proposition \ref{propGoodReducForNonOrd}, \( \SGTKp{T}^{I-\bc} \)
is an open subspace of \( \SGTKp{T} \), which is represented by a perfectoid space (\cite[Lemma \ref{1-lemTorCompacAsTateStack} (6)]{Jiang2026Shla}).
\end{notation}


Recall that by Theorem \ref{thmProductFormulaQuate}, there is a natural morphism \[\JL:=\JL_{T\to T'}:\SGTKp{T'}\times_{\CC_{p}}\mcal{M}_{T\to T'}\to \SGTKp{T}^{I-\bc},\]
which is a pro\'etale \(G_{T'}(\QQ_{p})\)-torsor, and which is equivariant for the action of \(G_{T}(\mA_{f})\).

We have the following results, translating the previous results into the setting of Tate stacks:
\begin{thm}\label{thmJLinTateSt}
We will use notation of Definition \ref{1-dfnIncarnationGroup} of \cite{Jiang2026Shla}.

(1)
We have a \( \underline{G_{T}(\QQ_{p})}_{\CC_{p}}\)-equivariant morphism of Tate stacks
\[\JL_{T\to T',K^{p}}:\SGTKp{T'}\times_{\CC_{p}}\mcal{M}_{T\to T'}
\to \SGTKp{T}^{I-\bc},
\] which is a \(\underline{G_{T'}(\QQ_{p})}_{\CC_{p}}\)-torsor. 

Moreover, after taking limit along \( K^{p}\subset G_{T}(\mA_{f}^{p})\cong G_{T'}(\mA_{f}^{p}) \), \( \varprojlim_{K^{p}}\JL_{T\to T',K^{p}} \) is \( G_{T}(\mA_{f}^{p})^{\sm}_{\CC_{p}} \)-equivariant.

(2) Taking 
\((-)^{\dR}\), we have a \(G_{T}(\QQ_{p})^{\sm}_{\bar{\QQ}_{p}}\)-equivariant morphism of Tate stacks over \( \bar{\QQ}_{p} \)
 \[\JL^{\dR}_{T\to T',K^{p}}:\SGTKp{T'}^{\dR}\times_{\bar{\QQ}_{p}}\mcal{M}_{T\to T'}^{\dR}\to \SGTKp{T}^{\dR,I-\bc}\]
which is a 
\(G_{T'}(\QQ_{p})^{\sm}\)-torsor. 

Moreover, after taking limit along \( K^{p}\subset G_{T}(\mA_{f}^{p})\cong G_{T'}(\mA_{f}^{p}) \), \( \varprojlim_{K^{p}}\JL_{T\to T',K^{p}}^{\dR} \) is \( G_{T}(\mA_{f}^{p})^{\sm}_{\bar{\QQ}_{p}} \)-equivariant.

(3) We have the following commutative diagram of Tate stacks: \begin{equation}
	\label{equaJLtransferTateSta}
	\begin{tikzcd}
	\SGTKp{T'}\times_{\CC_{p}}\mcal{M}_{T\to T'}\arrow[r]\arrow[rd,"r_{T'}\times r_{T,I}"']
	&
	(\SGTKp{T'}\times_{\CC_{p}}\mcal{M}_{T\to T'})/(1\subset \underline{G_{T'}(\QQ_{p})})^{\dagger}\arrow[r,"\JL_{T\to T',K^{p}}"]\arrow[d]
	& \SGTKp{T}^{I-\bc}\arrow[d,"r_{T}"]\\
	&
	\SGTKp{T'}^{\dR}\times_{\bar{\QQ}_{p}}\mcal{M}_{T\to T'}^{\dR}\arrow[r,"\JL_{T\to T',K^{p}}^{\dR}"]
	& \SGTKp{T}^{\dR,I-\bc},
\end{tikzcd}\end{equation}
where the square on the right is Cartesian, and 
the maps \(r_{T}\), \(r_{T'}\)
and \(r_{T,I}\)
are affine proper and descendable.
\end{thm}
\begin{rmk}
We will drop the subscripts sometimes when it does not cause confusion. 
\end{rmk}
\begin{proof}
By \cite[Theorem \ref{1-thmAnDRStackGeneral} (4)]{Jiang2026Shla} and Theorem \ref{thmProductFormulaQuate},
we know that \(\JL_{T\to T',K^{p}}\)
gives a \(\underline{G_{T'}(\QQ_{p})}\)-torsor of Tate stacks.
By \cite[Theorem \ref{1-thmAnDRStackGeneral} (5)]{Jiang2026Shla},
we know that \(\SGTKp{T}\to \SGTKp{T}^{\dR}\)
is a \(!\)-surjection. Now we can apply \cite[Corollary \ref{1-corTorsorDRSmTorsor}]{Jiang2026Shla},
which implies that \(\JL_{T\to T',K^{p}}^{\dR}\)
is a \(G_{T'}(\QQ_{p})^{\sm}_{\bar{\QQ}_{p}}\)-torsor. 
For (3), the big diagram commutes since there is a natural transformation \(\mrm{id}\to (-)^{\dR}\).
The induced map \[(\SGTKp{T'}\times_{\CC_{p}}\mcal{M}_{T\to T'})/(1\subset \underline{G_{T'}(\QQ_{p})})^{\dagger}\to \SGTKp{T}^{I-\bc}
\] is a \(G_{T'}(\QQ_{p})^{\sm}\)-torsor by \cite[Definition \ref{1-dfnIncarnationGroup}]{Jiang2026Shla}, and the right square is \(G_{T'}(\QQ_{p})^{\sm}\)-equivariant. 
The statements about the vertical maps are given by Theorem \ref{1-thmAnDRStackGeneral} (3) and (5) loc. cit.
\end{proof}
\begin{cor}\label{corIdeSMstructureSh}
We have an isomorphism in \( \QCoh(\SGTKp{T'}^{\dR}\times_{\bar{\QQ}_{p}}\mcal{M}_{T\to T'}^{\dR}) \) 
\[
\left(
	r_{T',*}(\mO_{\SGTKp{T'}})\boxtimes_{\CC_{p}} r_{T\to T',*}(\mO_{\mcal{M}_{T\to T'}})
\right)^{R-G_{T'}(\QQ_{p})-\sm}\cong 
(\JL^{\dR}_{T\to T',K^{p}})^{*}
r_{T,*}(\mO_{\SGTKp{T}}), 
\] where \( (-)^{R-G_{T'}(\QQ_{p})-\sm} \) is defined in the following sense: 

For any 
\( \mF\in \QCoh((\SGTKp{T'}^{\dR}\times_{\bar{\QQ}_{p}}\mcal{M}_{T\to T'}^{\dR})/\unl{G_{T'}(\QQ_{p})}) \), 
we define \[
\mF^{R-G_{T'}(\QQ_{p})-\sm}:=p_{12,*}\circ \beta^{*}(\mF)
\] for the morphisms as in 
\[
\begin{tikzcd}
\SGTKp{T'}^{\dR}\times_{\bar{\QQ}_{p}}\mcal{M}_{T\to T'}^{\dR}\times B(1\subset \unl{G_{T'}(\QQ_{p})})^{\dagger}\arrow[rd,"p_{12}"]
\arrow[d,"\beta"]
\\
(\SGTKp{T'}^{\dR}\times_{\bar{\QQ}_{p}}\mcal{M}_{T\to T'}^{\dR})/\unl{G_{T'}(\QQ_{p})} & 
\SGTKp{T'}^{\dR}\times_{\bar{\QQ}_{p}}\mcal{M}_{T\to T'}^{\dR}.
\end{tikzcd}
\]
\end{cor}
\begin{proof}
This follows from the Cartesian square in Theorem \ref{thmJLinTateSt} (3), 
and the proper base change (\cite[Lemma \ref{1-lemSmBaseChange} (2)]{Jiang2026Shla}).
\end{proof}
\begin{notation}
We denote by \( \bar{\mcal{G}}^{c}_{T,\dR,K^{p}} \) the \( G^{c,\an} \)-torsor over \( \SGTKp{T}^{\dR} \), obtained by pulling back \( \bar{\mcal{G}}^{c,\an}_{\dR,\Kpp} \) considered in \cite[Proposition \ref{1-propAnalyticGMTheory}]{Jiang2026Shla}.

Let \( \BdRX{T} \) and \( \BdRX{T}^{\la} \) be \( \BdRKp \) and \(\BdRKp^{\la}\in \QCoh(\SGTKp{T}^{\dR}) \) 
considered in \cite[Notation \ref{1-notationBdROnAnalyticDR}]{Jiang2026Shla}. Similarly, let \( \BdRX{T\to T'} \) and \( \BdRX{T\to T'}^{\la} \) be \( \BdRX{\mcal{M}} \) and \( \BdRX{\mcal{M}}^{\la}\in \QCoh(\mcal{M}_{T\to T'}^{\dR}) \). 
\end{notation}
\begin{thm}\label{propCompatibleGcTorsor}
Consider the diagram \[\begin{tikzcd}
		&
	\SGTKp{T'}^{\dR}\times_{\bar{\QQ}_{p}} \mcal{M}_{T\to T'}^{\dR}\arrow[rd,"\JL_{T\to T',K^{p}}^{\dR}"]\arrow[ld,"p_{1}"]
	\\
	\SGTKp{T'}^{\dR} & & \SGTKp{T}^{\dR}.
	\end{tikzcd}.\]
	Then we have a natural isomorphism of \(
    G^{c,\an}
    \)-torsors \[
    \JL_{T\to T',K^{p}}^{\dR,*}(\bar{\mcal{G}}^{c}_{T,\dR,K^{p}})
    \cong p_{1}^{*}(\bar{\mcal{G}}^{c}_{T',\dR,K^{p}}).
    \]
\end{thm}
\begin{proof}
By Corollary \ref{corIdeSMstructureSh}, we have a natural isomorphism \[
(p_{1}^{*}\BdRX{T'}\otimes_{B_{\dR}}
p_{2}^{*}\BdRX{T\to T'})^{R-G_{T'}(\QQ_{p})-\sm}\cong \JL_{T\to T',K^{p}}^{\dR,*}(\BdRX{T}),
\] where all the constructions are taken in the \( t^{\bullet} \)-complete sense. Therefore, 
by \cite[Theorems 1.5 \& 1.7]{JRC2021solid}, \[
R\Gamma((\fg,*_{T'}),(p_{1}^{*}(\BdRX{T'}^{\la}))\otimes_{B_{\dR}}
p_{2}^{*}(\BdRX{T\to T'}^{\la}))\cong \JL_{T\to T',K^{p}}^{\dR,*}(\BdRX{T}^{\la}),
\] where the notation is as in \cite[Lemma \ref{1-lemBdRLATwosides}]{Jiang2026Shla}.
Therefore, for any \( V\in\Rep(G^{c}) \), 
\begin{align*}
\JL_{T\to T',K^{p}}^{\dR,*}&(\bar{\mcal{G}}^{c}_{T,\dR,K^{p}}(V))\otimes B_{\dR}\cong 
\JL_{T\to T',K^{p}}^{\dR,*}(
\RHom_{(\fg,*_{T})}(V^{\vee},\BdRX{T}^{\la}))\\
&\cong 
R\Gamma((\fg,*_{T'}),(p_{1}^{*}\BdRX{T'}^{\la})\otimes_{B_{\dR}}
p_{2}^{*}(\RHom_{(\fg,*_{T})}(V^{\vee},(\BdRX{T\to T'}^{\la}))))\\
&\cong 
R\Gamma((\fg,*_{T'}),(p_{1}^{*}\BdRX{T'}^{\la})\otimes
(V,*_{T'}))\\
&\cong p_{1}^*(\bar{\mcal{G}}^{c}_{T',\dR,K^{p}}(V))\otimes B_{\dR},
\end{align*} where the first and the last isomorphisms are given by \cite[Corollary \ref{1-corReformulateRiemmanHilbPerfdCase}]{Jiang2026Shla}
and the third isomorphism is given by \cite[Proposition \ref{1-propRHCorrLocalSV}]{Jiang2026Shla}.
We conclude by taking \( E_{0}(D_{\sen}(-)) \) 
on both sides. 
\end{proof}

\subsubsection{For Grothendieck-Messing-Hodge-Tate period domains}
The map \( \JL_{T\to T',K^{p}} \) is compatible with the Grothendieck-Messing-Hodge-Tate period maps as we now explain.
\begin{notation}
For \( \mcal{M}_{T\to T'} \) (resp. \( \SGTKp{T} \)), we denote by \( \Per_{T\to T'} \) (resp. \( \PerT{T} \))
the associated Grothendieck-Messing-Hodge-Tate period domain period domain in Definition \ref{1-dfnGrothendieck-Messing-Hodge-TateperioddomainLocal} (resp. Definition \ref{1-dfnGMHTperDomain}) of \cite{Jiang2026Shla}.
We denote the period maps of \( \mcal{M}_{T\to T'} \)
as \[
\begin{tikzcd}[column sep=0.5in]
	&\arrow[d,"\pi_{\GM\HT,T\to T'}^{\la}"]
	\mcal{M}_{T\to T'}^{\la}
	\arrow[dl,"\pi_{\HT,T\to T'}"']
	\arrow[dr,"\pi_{\GM,T\to T'}"]
	\\
\FLTT{T}{I^{c}}\arrow[d,"\beta_{\Fl,I}"]
& \Per_{T\to T'}\arrow[d,"\beta_{\Per,I}"]
\arrow[r,"\pi_{\GM,T\to T'}^{per}"']
\arrow[l,"\pi_{\HT,T\to T'}^{per}"]
& \FLTT{T'}{I^{c}}^{\std}\arrow[d,"\beta_{\Fl^{\std},I}"]
\\
\FLTT{T}{I^{c}}^{\dR}
& \Per_{T\to T'}^{\dR}
\arrow[r,"\pi_{\GM,T\to T'}^{per,\dR}"']
\arrow[l,"\pi_{\HT,T\to T'}^{per,\dR}"]
& \FLTT{T'}{I^{c}}^{\std,\dR}
\end{tikzcd}
\]


To make the notation of \( \PerT{T} \) symmetric as in the definition of \( \Per_{T\to T'} \), we write \[
\PerT{T}:=(G_{T}^{\an}\times G^{c,\an}_{\Sigma_{\infty}})/(P_{\mu_{T}}^{\an}\times_{M^{c,\an}_{\mu_{T}}}
P^{\mrm{std},c,\an}_{\mu_{T}}). 
\] \( \PerT{T} \)  carries an action of \( G_{T}^{\an}\times G_{\Sigma_{\infty}}^{c,\an} \), which  
corresponds to the action of \( *_{\HT} \) 
and \( *_{\GM} \) in \cite[Definition \ref{1-dfnGMHTperDomain}]{Jiang2026Shla} respectively.

We denote the Grothendieck-Messing-Hodge-Tate period map for \( \SGTKp{T} \) (\cite[Definition \ref{1-dfnGMHTperMap}]{Jiang2026Shla})
and \( \mcal{M}_{T\to T'} \) (\cite[Theorem \ref{1-thmGMHTPeriMapLocal}]{Jiang2026Shla}) by \( \pi_{\GM\HT,T,K^{p}} \) 
and \( \pi_{\GM\HT,T\to T'} \) 
respectively.
\end{notation}
\begin{lem}\label{lemJLMapForPer}
We define a Jacquet-Langlands morphism \(\JL_{T\to T'}^{per}\) for the period domains  
\begin{align*}
    \JL_{T\to T'}^{per}:\PerT{T'}
\times \Per_{T\to T'}\to \PerT{T}
\end{align*}
as follows: we have an isomorphism \[i:
(G_{T}\times G_{\Sigma_{\infty}}^{c})/(P_{\mu_{T}}\times_{M^{c}_{\mu_{T}}}P^{\mrm{std},c}_{\mu_{T}})
\cong (
P_{\mu_{I}}\times_{M_{\mu_{I}}}
P^{\mrm{std}}_{\mu_{I}}\times 1)\backslash 
(G_{T}\times G_{T'}\times G_{\Sigma_{\infty}}^{c})/(1\times P_{\mu_{T'}}\times_{M^{c}_{\mu_{T'}}}
P^{\mrm{std},c}_{\mu_{T'}}),
\] induced by \(
(g,g')\mapsto (g,1,(g')^{-1})
\).
Consider the morphism \begin{align*}
f:
	(G_{T'}\times G^{c}_{\Sigma_{\infty}})&/(P_{\mu_{T'}}\times_{M^{c}_{\mu_{T'}}}
P^{\mrm{std},c}_{\mu_{T'}})
\times 
	(G_{T}\times G_{T'})/(P_{\mu_{I}}\times_{M_{\mu_{I}}}
P^{\mrm{std}}_{\mu_{I}})\\
&\to (
P_{\mu_{I}}\times_{M_{\mu_{I}}}
P^{\mrm{std}}_{\mu_{I}}\times 1)\backslash 
(G_{T}\times G_{T'}\times G_{\Sigma_{\infty}}^{c})/(1\times P_{\mu_{T'}}\times_{M^{c}_{\mu_{T'}}}
P^{\mrm{std},c}_{\mu_{T'}}),
\end{align*}
 sending \(
((g'_{T'},g_{\Sigma_{\infty}}),(g_{T},g_{T'}))
\) to \(
((g_{T})^{-1},(g_{T'})^{-1}g'_{T'},g_{\Sigma_{\infty}})
\). 

Then we define \[
\JL^{per}_{T\to T'}:=f\circ i^{-1},
\]
which is a \(G_{T'}^{\an}\)-torsor. 
\end{lem}
\begin{proof}
It suffices to show that \[
(G_{T}\times G_{\Sigma_{\infty}}^{c})/(P_{\mu_{T}}\times_{M^{c}_{\mu_{T}}}P^{\mrm{std},c}_{\mu_{T}})
\cong (G_{T}\times G_{T'}\times G_{\Sigma_{\infty}}^{c})/I.
\] Note that the LHS is isomorphic to \(
((G_{T}\times G_{\Sigma_{\infty}})/(P_{\mu_{T}}\times_{M^{c}_{\mu_{T}}}P^{\mrm{std}}_{\mu_{T}}))/1\times Z_{c}
\), and the RHS is isomorphic to \(((G_{T}\times G_{T'}\times G_{\Sigma_{\infty}})/\ti{I})/(1\times 1\times Z_{c})\), where \(\ti{I}\) denotes the subgroup generated by 
\(
1\times P_{\mu_{T'}}\times_{M_{\mu_{T'}}}
P^{\mrm{std}}_{\mu_{T'}}
\) and \(
P_{\mu_{I^{c}}}\times_{M_{\mu_{I^{c}}}}
P^{\mrm{std}}_{\mu_{I^{c}}}\times 1
\). 
So it suffices to show that \[
(G_{T}\times G_{\Sigma_{\infty}})/(P_{\mu_{T}}\times_{M^{c}_{\mu_{T}}}P^{\mrm{std}}_{\mu_{T}})\cong (G_{T}\times G_{T'}\times G_{\Sigma_{\infty}})/\ti{I}.
\] Now the desired isomorphism decomposes with respect to \(\tau\in\Sigma_{\infty}\), and one can verify the isomorphism on each \(\tau\)-factor without any difficulty. 
\end{proof}

\begin{thm}\label{thmJLCompatiblePeriodMaps}
In the setting of Theorem \ref{thmJLinTateSt}, the functor \(
\JL^{\dR}_{T\to T'}
\) is compatible with \(\JL^{per}_{T\to T'}\) in Lemma \ref{lemJLMapForPer}
along the period map \(\pi_{\GM\HT}^{\dR}\). 

More precisely, we have the following commutative diagrams 
\[
\begin{tikzcd}
    \SGTKp{T'}\times_{\CC_{p}} \mcal{M}_{T\to T'} \arrow[r,"\JL_{T\to T',K^{p}}"] \arrow[d, "\pi_{\GM\HT,T',K^{p}}\times \pi_{\GM\HT,T\to T'}"']
    & \SGTKp{T}\arrow[d,"\pi_{\GM\HT,T,K^{p}}"] \\
    (\PerT{T'}/G^{c,\an}_{\Sigma_{\infty}})\times_{\CC_{p}} \Per_{T\to T'} \arrow[r,"\JL_{T\to T'}^{per}"] & \PerT{T}/G^{c,\an}_{\Sigma_{\infty}},
\end{tikzcd}
\] and
\[
\begin{tikzcd}
    \SGTKp{T'}^{\dR}\times_{\bar{\QQ}_{p}} \mcal{M}_{T\to T'}^{\dR} \arrow[r,"\JL_{T\to T',K^{p}}^{\dR}"] \arrow[d, "\pi^{\dR}_{\GM\HT,T',K^{p}}\times \pi^{\dR}_{\GM\HT,T\to T'}"']
    & \SGTKp{T}^{\dR}\arrow[d,"\pi^{\dR}_{\GM\HT,T,K^{p}}"] \\
    (\PerT{T'}^{\dR}/G^{c,\an}_{\Sigma_{\infty}})\times_{\bar{\QQ}_{p}} \Per^{\dR}_{T\to T'} \arrow[r,"\JL_{T\to T'}^{per,\dR}"] & \PerT{T}^{\dR}/G^{c,\an}_{\Sigma_{\infty}},
\end{tikzcd}
\] Moreover, the two diagrams are compatible.
\end{thm}
\begin{proof}

By Theorem \ref{propCompatibleGcTorsor}, we have an isomorphism \[
    \JL^{\dR,*}(\bar{\mcal{G}}^{c}_{T,\dR,K^{p}})
    \cong p_{1}^*(\bar{\mcal{G}}^{c}_{T',\dR,K^{p}}).
\] Therefore, it suffices to show that we have the following commutative diagram:
\[
\begin{tikzcd}[column sep=0.15in]
\bar{\mcal{G}}^{c}_{T',\dR,K^{p}}\times \mcal{M}_{T\to T'}^{\dR}\arrow[d,"\JL^{\dR}"]
\arrow[r] & \left(
	(G_{T'}^{\an}\times G^{c,\an}_{\Sigma_{\infty}})/(P_{\mu_{T'}}^{\an}
	\times_{M^{c,\an}_{\mu_{T'}}}
P^{\mrm{std},c,\an}_{\mu_{T'}})
\right)^{\dR}\times 
\left((G_{T}^{\an}\times G_{T'}^{\an})/(P_{\mu_{I^{c}}}^{\an}\times_{M_{\mu_{I^{c}}}^{\an}}
P^{\mrm{std},\an}_{\mu_{I^{c}}})
\right)^{\dR}
\arrow[d,"\JL^{per,\dR}"]  \\
    \bar{\mcal{G}}^{c}_{T,\dR,K^{p}}
\arrow[r] & \left((G_{T}^{\an}\times G^{c,\an}_{\Sigma_{\infty}})/(P_{\mu_{T}}^{\an}\times_{M^{c,\an}_{\mu_{T}}}
P^{\mrm{std},c,\an}_{\mu_{T}})\right)^{\dR}.
\end{tikzcd}
\] For this, it suffices to show the following diagram of perfectoid spaces commutes
\[
\begin{tikzcd}[column sep=0.15in]
\bar{\mcal{G}}^{c}_{T',\dR,K^{p}}\times \mcal{M}_{T\to T'}\arrow[d,"\JL"]
\arrow[r] & \left(
	(G_{T'}^{\an}\times G^{c,\an}_{\Sigma_{\infty}})/(P_{\mu_{T'}}^{\an}
	\times_{M^{c,\an}_{\mu_{T'}}}
P^{\mrm{std},c,\an}_{\mu_{T'}})
\right)\times 
\left((G_{T}^{\an}\times G_{T'}^{\an})/(P_{\mu_{I^{c}}}^{\an}\times_{M_{\mu_{I^{c}}}^{\an}}
P^{\mrm{std},\an}_{\mu_{I^{c}}})
\right)
\arrow[d,"\JL^{per}"]  \\
    \bar{\mcal{G}}^{c}_{T,\dR,K^{p}}
\arrow[r] & \left((G_{T}^{\an}\times G^{c,\an}_{\Sigma_{\infty}})/(P_{\mu_{T}}^{\an}\times_{M^{c,\an}_{\mu_{T}}}
P^{\mrm{std},c,\an}_{\mu_{T}})\right),
\end{tikzcd}
\]
where \(\mcal{G}_{T,\dR,K^{p}}^{c}\) denotes the pull-back of \(\bar{\mcal{G}}_{T,\dR,K^{p}}^{c}\) along \(
\SGTKp{T}\to \SGTKp{T}^{\dR}
\), which is a smoothoid adic space. 

The diagram commutes as diamonds by considering the moduli interpretation and the construction of \( \JL \) 
in Theorem \ref{thmProductFormulaQuate}. 
Then we conclude that the diagram of adic spaces also commutes, since the diamondification functor is fully faithful on smoothoid adic spaces by \cite[Lemma 2.6 (4)]{Heuer2022moduli}.
\end{proof}
\begin{rmk}
Alternatively, the commutativity of the diagram of the diamonds is sufficient for implying the commutativity for their analytic de Rham stacks by \cite{AnschützBoscoLeBrasCamargoScholze2025analyticrhamstacksfarguesfontaine}.
\end{rmk}

\subsubsection{For locally analytic Shimura varieties}
We have the following neat corollary, which will not be used later.
\begin{cor}\label{corJLlaIsGlaTorsor}
\( \JL_{T\to T',K^{p}} \) has a canonical factorization \[
\JL^{\la}_{T\to T',K^{p}}:\SGTKp{T'}^{\la}\times \mcal{M}_{T\to T'}^{\la}\to \SGTKp{T}^{\la},
\] which is a \( G_{T'}(\QQ_{p})^{\la} \)-torsor. 
\end{cor}
\begin{proof}
This follows from \cite[Theorem \ref{1-thmLAstrucGMHTper} (2), Theorem \ref{1-thmGMHTPeriMapLocal} (2)]{Jiang2026Shla},
Theorem \ref{thmJLCompatiblePeriodMaps}, Lemma \ref{lemJLMapForPer},
and Theorem \ref{thmJLinTateSt}. 

More precisely, 
\( \JL^{\dR} \) is a \( G_{T'}(\QQ_{p})^{\sm} \)-torsor by  Theorem \ref{thmJLinTateSt}.
\( \JL^{per} \) is a \( G_{T'} \)-torsor by  Lemma \ref{lemJLMapForPer},
and thus \( \JL^{per,\dR} \) 
is a \( G_{T'}^{\dR} \)-torsor by \cite[Theorem \ref{1-thmAnDRStackGeneral} (1)]{Jiang2026Shla}. 
We can pulling back the diagram in Theorem \ref{thmJLCompatiblePeriodMaps} along \( \PerT{T}\to \PerT{T}^{\dR} \) and \( \Per_{T\to T'}\to \Per_{T\to T'}^{\dR} \), and  
using the Cartesian diagrams in \cite[Theorem \ref{1-thmLAstrucGMHTper} (2) and Theorem \ref{1-thmGMHTPeriMapLocal} (2)]{Jiang2026Shla}, we obtain the desired map \[
\JL^{\la}:\SGTKp{T'}^{\la}\times \mcal{M}_{T\to T'}^{\la}\to \SGTKp{T}^{\la},
\] which is a torsor for \( G_{T'}\times_{G_{T'}^{\dR}}G_{T'}(\QQ_{p})^{\sm} \), 
which is precisely \( G_{T'}(\QQ_{p})^{\la} \). 
\end{proof}
We conclude with the following locally analytic Jacquet-Langlands transfer, which is the main output of the lengthy preparation:
\begin{notation}\label{notationJsmMTT}
The space
\( \Per_{T\to T'}\cong \prod_{\tau\in\Sigma_{\infty}}\Per_{T\to T',\tau} \), where \[ \Per_{T\to T',\tau}:=(G_{T,\tau}\times G_{T',\tau})/(P_{\mu_{I^{c}},\tau}\times_{M_{\mu_{I^{c}},\tau}}P^{\std}_{\mu_{I^{c}},\tau}). \]
For any \( J\subset \Sigma_{\infty} \), 
we denote \(
\Per_{T\to T',J}
:=\prod_{\tau\in J}\Per_{T\to T',\tau}.
\) We denote the natural morphisms as \[
\begin{tikzcd}
\Per_{T\to T'}\arrow[r,"p_{T\to T',J}"]
\arrow[d,"\beta_{per,I}"]
& \Per_{T\to T',J} 
\arrow[r,"\pi_{\HT,T\to T'}^{per,J}"]
\arrow[d,"\beta_{per,I,J}"]
& 
\FLTT{T}{I^{c}}
\\
\Per_{T\to T'}^{\dR}\arrow[r,"p_{T\to T',J}^{\dR}"]
& \Per_{T\to T',J}^{\dR},
\end{tikzcd}
\] where \( \pi^{per,J}_{\HT,T\to T'} \)
is defined only when \( J\supset I\).

We define
\( \mcal{M}_{T\to T'}^{\la,J-*_{T'}-\sm} \) via the following pull-back diagram
\[
\begin{tikzcd}
\mcal{M}_{T\to T'}^{\la,J-*_{T'}-\sm}\arrow[d,"\beta^{J-*_{T'}-\sm}_{T\to T'}"]
\arrow[r] & \Per_{T\to T',J}^{\dR/\Fl_{G_{T},\mu_{I^{c}},I\cap J}} 
\times \Per_{T\to T',J^{c}}\arrow[d,"\beta^{per,J-*_{T'}-\sm}_{T\to T'}"]
\arrow[r,"\pi^{per,J-*_{T'}-\sm}_{\HT,T\to T'}"]
& \FLTT{T}{I^{c}}
\\
\mcal{M}_{T\to T'}^{\dR}
\arrow[r] & \Per_{T\to T'}^{\dR}.
\end{tikzcd}
\] In particular, 
we have a natural map \( \pi_{\HT,T\to T'}^{\la,J-*_{T'}-\sm}:\mcal{M}^{\la,J-*_{T'}-\sm}_{T\to T'} \to \FLTT{T}{I^{c}} \).

For any \( I \)-multiweight \( \kw \), we write \[\omega^{\kw,J-*_{T'}-\sm}_{T\to T'}:= (\beta_{T\to T'}^{J-*_{T'}-\sm})_{*}
	(\pi^{\la,J-*_{T'}-\sm}_{\HT,T\to T'})^{*}(\omega^{\kw}_{\FLTT{T}{I^{c}}}). \]
\end{notation}
\begin{thm}[Locally analytic Jacquet-Langlands transfer]\label{thmJacquetLanglandsLA}
	Let \(T\subset\Sigma_{\infty}\), \(I \subset J\subset T^{c}\), \(T':=T\coprod I\) and \(K\subset T^{\prime,c}\). Let \( \kw \) be a regular multiweight, and \( (\Bbbk',w') \) be a \( T^{c} \)-multiweight. Assume that \(k_{\tau}\in\ZZ_{\le 0}\)
	for \(\tau\in J\)
	and \(k'_{\tau}\in\ZZ_{\le0}\)
	for \(\tau\in K\).

	Consider the diagram \[\begin{tikzcd}
		&
	\SGTKp{T'}^{\dR}\times_{\bar{\QQ}_{p}} \mcal{M}_{T\to T'}^{\dR}\arrow[rd,"\JL^{\dR}_{T\to T',K^{p}}"]\arrow[ld,"p_{1}"]
	\\
	\SGTKp{T'}^{\dR} & & \SGTKp{T}^{\dR}.
	\end{tikzcd}\]

	Then
	there is a \( G_{T'}(\QQ_{p})^{\sm}\times G_{T}(\QQ_{p})^{\la}\)-equivariant isomorphism
	\begin{align*}
	(&\JL^{\dR}_{T\to T',K^{p}})^{*}\left(\dR_{J,K}(\nabla^{\kw,(\Bbbk',w')}_{\SGTKp{T}})\right)\\&\cong R\Gamma\left((\fg_{T^{\prime,c}},*_{T'}),\dR_{J\backslash I,K}(\nabla^{\kw,(\Bbbk'|_{T^{\prime,c}},w')}_{\SGTKp{T'}})
	\boxtimes 
	\omega^{(\Bbbk'|_{I},w'),T'-*_{T'}-\sm}_{T\to T'}
	\right),
	\end{align*}
	over \(\SGTKp{T'}^{\dR}\times\mcal{M}_{T\to T'}^{\dR}\).
	
	Moreover, after taking limit along \( K^{p}\subset G_{T}(\mA_{f}^{p})=G_{T'}(\mA_{f}^{p}) \), the induced isomorphism  is \( G_{T}(\mA_{f}^{p})^{\sm} \)-equivariant.
\end{thm}
\begin{proof}
By \cite[Theorem \ref{1-thmAnDRStackGeneral} (3)]{Jiang2026Shla},
we can apply proper base change (\cite[Lemma \ref{1-lemSmBaseChange} (2)]{Jiang2026Shla}) to the Cartesian diagram in Notation \ref{notationJsmMTT}. Then by Theorem \ref{thmJLCompatiblePeriodMaps}
and Lemma \ref{lemPartialDRasPullBack}, the proof 
boils down to the third isomorphism in Lemma \ref{lemJLtransferLAonPer}
below.
\end{proof}
\begin{lem}\label{lemJLtransferLAonPer}
Let \( T\subset \Sigma_{\infty} \), \( I\subset J\subset T^{c} \), \( K\subset T^{\prime,c} \), \( \kw \) and \( (\Bbbk',w') \)
as in Theorem \ref{thmJacquetLanglandsLA}.

Consider \[
\JL^{per,\dR}_{T\to T'}:
\PerT{T'}^{\dR}/G^{c,\an}_{\Sigma_{\infty}}
\times_{\bar{\QQ}_{p}}
\Per_{T\to T'}^{\dR}
\to 
\PerT{T}^{\dR}/G^{c,\an}_{\Sigma_{\infty}}
\] as in Theorem \ref{thmJLCompatiblePeriodMaps}. 

Then we have the following isomorphisms: \begin{align*}
(&\JL^{per,\dR}_{T\to T'})^{*}\left(
	(\pi_{\GM,T}^{per})^{*}(\dR_{J}(\omega_{\FLT{T}^{\std}}^{\kw}))
\right)
\cong p_{1}^{*}\left(
	(\pi^{per}_{\GM,T'})^{*}(\dR_{J\backslash I}(\omega_{\FLT{T'}^{\std}}^{\kw}))
\right),\\
(&\JL^{per,\dR}_{T\to T'})^{*}\left(
	(\pi_{\HT,T}^{per})^{*}(\bar{\dR}_{K}(\omega_{\FLT{T}}^{(\Bbbk',w')}))
\right)\\&\cong 
R\Gamma\left((\fg_{T^{\prime,c}},*_{T'}),
(\pi_{\HT,T}^{per})^{*}(\bar{\dR}_{K}(\omega_{\FLT{T}}^{(\Bbbk'|_{T^{\prime,c}},w')}))\boxtimes (\beta^{per,T'-*_{T'}-\sm}_{T\to T'})_{*} (\pi^{per,T'-*_{T'}-\sm}_{\HT,T\to T'})^{*}
(\omega^{(\Bbbk'|_{I},w')}_{\FLTT{T}{I}})
\right),\\
(&\JL^{per,\dR}_{T\to T'})^{*} \left(
	\dR_{J,K}(\nabla^{\kw,(\Bbbk',w')}_{\PerT{T}})\right)
\\&\cong 
R\Gamma\left((\fg_{T^{\prime,c}},*_{T'}),\dR_{J\backslash I,K}(\nabla_{\PerT{T'}}^{\kw,(\Bbbk'|_{T^{\prime,c}},w')})\boxtimes 
(\beta^{per,T'-*_{T'}-\sm}_{T\to T'})_{*} (\pi^{per,T'-*_{T'}-\sm}_{\HT,T\to T'})^{*}
(\omega^{(\Bbbk'|_{I},w')}_{\FLTT{T}{I}})
\right).
\end{align*}
\end{lem}
\begin{proof}
The third isomorphism follows from the first and the second. 


We start with the second isomorphism. We have a commutative diagram \[
\begin{tikzcd}
\PerT{T'}^{\dR}/G^{c,\an}_{\Sigma_{\infty}}
\times_{\bar{\QQ}_{p}}
\Per_{T\to T'}^{\dR}
\arrow[r]\arrow[d,"\pi_{\HT,T'}^{per}\times 1"]
&
\PerT{T}^{\dR}/G^{c,\an}_{\Sigma_{\infty}}\arrow[d,"\pi_{\HT,T}^{per}"] 
\\ 
\FLT{T'}^{\dR}
\times_{\bar{\QQ}_{p}} \Per_{T\to T'}^{\dR}
\arrow[r,"\JL_{\HT}^{\dR}"]
& \FLT{T}^{\dR}.
\end{tikzcd}
\] where \( \JL_{\HT}^{\dR} \)
is induced by 
\begin{align*}
\JL_{\HT}:\FLT{T'}\times_{\CC_{p}}
\Per_{T\to T'}&
\to (\FLT{T'}\times_{\CC_{p}}
\Per_{T\to T'})/G_{T'}\\
&\cong (G_{T'}\times G_{T})/((P_{\mu_{I^{c}}}^{\std}\cap P_{\mu_{T'}})\times_{M_{\mu_{I^{c}}}}P_{\mu_{I^{c}}})\\
&\cong (G_{T'}\times G_{T})/(P_{\mu_{I^{c}}}^{\std}\times_{M_{\mu_{I^{c}}}}P_{\mu_{T}})
\to \FLT{T}.
\end{align*}
Therefore, it suffices to show that 
we have a \( G_{T'}^{\dR}\times G_{T}  \)-equivariant isomorphism
\begin{align*}
&(\JL^{\dR}_{\HT})^{*}(\bar{\dR}_{K}(\omega^{(\Bbbk',w')}_{\FLT{T}}))\\
&\cong R\Gamma\left(
	(\fg_{T^{\prime,c}},*_{T'}),
	\bar{\dR}_{K}(\omega^{(\Bbbk'|_{T^{\prime,c}},w')}_{\FLT{T'}})\boxtimes (\beta^{per,T'-*_{T'}-\sm}_{T\to T'})_{*} (\pi^{per,T'-*_{T'}-\sm}_{\HT,T\to T'})^{*}(\omega^{(\Bbbk'|_{I},w')}_{\FLTT{T}{I}})
\right).
\end{align*}
Now the spaces decomposes into a product over \( \tau\in\Sigma_{\infty} \),
and the corresponding sheaf also decomposes by \cite[Corollary \ref{1-corPartialBGGDual}]{Jiang2026Shla}. We now verify the isomorphisms for each \( \tau \)-factor separately.

For \( \tau\in T \),
we are considering \[
\AnSpec(\bar{\QQ}_{p,\square})\times_{\bar{\QQ}_{p}}
\Per_{T\to T',\tau}^{\dR}
\xrightarrow{\JL_{\HT,\tau}^{\dR}}\AnSpec(\bar{\QQ}_{p,\square}),
\] and the sheaves on both sides are the structure sheaves.

For \( \tau\in I \),
we are considering \[
\AnSpec(\bar{\QQ}_{p,\square})\times_{\bar{\QQ}_{p}} 
((G_{T',\tau}\times G_{T,\tau})/(P_{\mu_{I}^{c},\tau}^{\std}\times_{M_{\mu_{I^{c}},\tau}}P_{\mu_{I^{c}},\tau}))^{\dR}
\xrightarrow{\JL_{\HT,\tau}^{\dR}}
\Fl_{G_{T},\mu_{T},\tau}^{\dR},
\] where the last maps is induced by \( \pi_{\HT,T\to T',\tau} \). Then the sheaves on both sides are given by \( (\JL_{\HT,\tau}^{\dR})^{*}(\beta_{\Fl,\tau})_{*}(\omega^{(\Bbbk'|_{\{\tau\}},w')}_{\FLTT{T}{\{\tau\}^{c}}}) \)
for \( \beta_{\Fl,\tau}:\FLTT{T}{\{\tau\}^{c}}\to \FLTT{T}{\{\tau\}^{c}}^{\dR} \).

For \( \tau\in T^{\prime,c} \), we are considering \[
\begin{tikzcd}
\Fl_{G_{T'},\mu_{T'},\tau}
\times_{\CC_{p}} (
(G_{T',\tau}\times G_{T,\tau})/G_{T,\tau})\arrow[d]
\arrow[r,"\JL_{\HT,\tau}"] &
\Fl_{G_{T},\mu_{T},\tau}\arrow[d] 
\\
\Fl_{G_{T'},\mu_{T'},\tau}^{\dR}
\times_{\bar{\QQ}_{p}} (
(G_{T',\tau}\times G_{T,\tau})/G_{T,\tau})^{\dR}
\arrow[r,"\JL_{\HT,\tau}^{\dR}"] &
\Fl_{G_{T},\mu_{T},\tau}^{\dR}, 
\end{tikzcd}
\] where \( \JL_{\HT,\tau} \) is a \( G_{T',\tau}  \)-torsor for the diagonal action of \( G_{T',\tau} \) on the LHS. Concretely, \( \JL_{\HT,\tau} \) is the same as \[
(B\backslash \GL_{2})
\times \GL_{2}\to (B\backslash \GL_{2})
\] induced by the action of \( \GL_{2} \) on \( B\backslash \GL_{2} \).
In particular, this map is compatible with the Hodge-Tate torsors via \[
(N\backslash \GL_{2})\times \GL_{2}\to (N\backslash \GL_{2}).
\] Here \( B \) and \( N \)
denote the subgroups of upper-triangular matrices and its unipotent radical respectively. 
Thus \[
\JL_{\HT,\tau}^{*}(\omega_{\Fl_{G_{T},\mu_{T},\tau}}^{(\Bbbk'|_{\tau},w')})\cong p_{1}^{*}(\omega_{\Fl_{G_{T'},\mu_{T'},\tau}}^{(\Bbbk'|_{\tau},w')})
\]
We can then push-forward the isomorphism to the analytic de Rham stack, and conclude by the argument of Corollary \ref{corIdeSMstructureSh}.

We now continue to prove the first isomorphism. Similarly, we have a a commutative diagram \[
\begin{tikzcd}
\PerT{T'}^{\dR}/G^{c,\an}_{\Sigma_{\infty}}
\times_{\bar{\QQ}_{p}}
\Per_{T\to T'}^{\dR}
\arrow[r]\arrow[d,"\pi_{\HT,T'}^{per}\times 1"]
&
\PerT{T}^{\dR}/G^{c,\an}_{\Sigma_{\infty}}\arrow[d,"\pi_{\HT,T}^{per}"] 
\\ 
\PerT{T'}^{\dR}/G^{c,\an}_{\Sigma_{\infty}}
\times_{\bar{\QQ}_{p}}
\FLTT{T'}{I^{c}}^{\std}
\arrow[r,"\JL_{\GM}^{\dR}"]
& \FLTT{T}{T}^{\std,\dR}/G^{c,\an}_{\Sigma_{\infty}}
\end{tikzcd}
\] where \( \JL_{\GM}^{\dR} \)
is induced by \[
\JL_{\GM}:\PerT{T'}
\times \FLTT{T'}{I^{c}}^{\std}
\to \FLT{T}^{\std},
\] which is defined in a similar way as \( \JL_{\HT} \).

Now it suffices to show that \[
p_{1}^{*}\left(
	(\pi^{per}_{\GM,T'})^{*}(\dR_{J\backslash I}(\omega^{\kw}_{\FLT{T'}^{\std}}))
\right)\cong (\JL^{\dR}_{\GM})^{*}(\dR_{J}(\omega^{\kw}_{\FLT{T}^{\std}})).
\]
Since \( J\supset I \), \( \dR_{J}(\omega^{\kw}_{\FLT{T}^{\std}}) \)
is the pull-back of \( \dR_{J\backslash I}(\omega^{\kw}_{\FLT{T'}^{\std}}) \)
along \( \FLT{T}^{\std,\dR}/G_{\Sigma_{\infty}}\to \FLTT{T}{T'}^{\std,\dR}/G_{\Sigma} \) by the proof of \cite[Corollary \ref{1-corPartialBGGDual}]{Jiang2026Shla}. This shows the desired isomorphism, since we have a commutative diagram \[
\begin{tikzcd}
\PerT{T'}^{\dR}/G^{c,\an}_{\Sigma_{\infty}}
\times_{\bar{\QQ}_{p}}
\FLTT{T'}{I^{c}}^{\std}
\arrow[rr,"\JL_{\GM}^{\dR}"]\arrow[rd,"\pi_{\GM,T'}^{per,\dR}\circ p_{1}"]
& & \FLTT{T}{T}^{\std,\dR}/G^{c,\an}_{\Sigma_{\infty}}\arrow[ld]\\
& \FLTT{T}{T'}^{\std,\dR}/G^{c,\an}_{\Sigma_{\infty}}.
\end{tikzcd}
\]
\end{proof}

\subsection{Classicality of partial de Rham cohomology}\label{subsecProofOfClassicality}
We now prove Theorem \ref{thmClassicalityCohomoDR}.
The key ingredients will be Theorem \ref{thmJacquetLanglandsLA}.
In fact, to realize an inductive proof, we will start by proving the following stronger version (Proposition \ref{propClassicalityCohoStronger}), from which Theorem \ref{thmClassicalityCohomoDR} follows by taking \(R\Gamma(\Fl^{\dR}_{G_{T},\mu_{T},I\backslash J},-)\). 
\begin{lem}\label{lemHodgeTatepdRIsPrim}
For any \( J\subset T^{c} \),
we consider the morphism given by the composition \[
	\pi_{\HT,T,J}^{\pdR}:\SGTKpTo{T}^{\pdR,\sm}_{\log}
	\to \FLT{T}^{\dR}
	\xrightarrow{p_{J}} \Fl_{G_{T},\mu_{T},J}^{\dR}
	\] given by \cite[Corollary \ref{1-corCartesianHilbertFlandSm}]{Jiang2026Shla}.
	Then \( \pi_{\HT,T,J}^{\pdR} \) is prim.
\end{lem}
\begin{proof}
The map \( \pi^{\pdR}_{\HT,T,J} \)
	is prim by \cite[Lemma \ref{1-lemSmDRStackProperty} (4) and Lemma \ref{1-lemFiniteLevelProper} (2)]{Jiang2026Shla}.
\end{proof}
For \(I\coprod J= T^{c}\), recall 
that by Lemma \ref{lemAnotherKunnethType}, we have an isomorphism over \( \SGTKp{T}^{\pdR/\CC_{p},\hp,\sm}_{\log} \)
\begin{align*}
	&(p_{1,T})_{*}\dR_{I,J}^{\alg}(\nabla^{\kw,(\Bbbk',w')}_{\SGTKp{T}})\\
	&\cong \dR_{I}^{\alg}(\omega^{\kw,\sm}_{T})\otimes \pi_{\HT,T}^{\pdR,*}(\bar{\dR}_{J}(\omega_{\FLT{T}}^{(\Bbbk',w')}))\\
	&\cong \dR_{I}^{\alg}(\omega^{\kw,\sm}_{T})\otimes 
	V^{(\Bbbk'|_{J},w')}_{J}
	\otimes \pi^{\pdR,*}_{\HT,T,I}((\beta_{\Fl,T,I})_{*}(\omega^{(\Bbbk'|_{I},w')}_{\FLTT{T}{I^{c}}})),
\end{align*}
for \(\beta_{\Fl,T,I}:\FLTT{T}{I^{c}}\to \FLTT{T}{I^{c}}^{\dR}\) and \(\pi^{\pdR}_{\HT,T,I}:\SGTKpTo{T}^{\pdR/\CC_{p},\hp,\sm}\to \Fl_{G_{T},\mu_{T},I}=\FLTT{T}{I^{c}}\), where the second isomorphism is given by   \cite[Corollary \ref{1-corFilterDualBGG}]{Jiang2026Shla}.
Then by adjuntion, we have a natural morphism in \( \QCoh(\FLTT{T}{I^{c}}^{\dR}/G_{T}(\QQ_{p})^{\la}) \)
\begin{align}\label{alignSeprateSomeAlso}
		(\pi_{\HT,T,I}^{\pdR})_{*}(\dR_{I}^{\alg}(\omega^{\kw,\sm}_{T}))
	\otimes V^{(\Bbbk'|_{J},w')}_{J}
	\otimes (\beta_{\Fl,T,I})_{*}(\omega^{(\Bbbk'|_{I},w')}_{\FLTT{T}{I^{c}}})
	\to 
	\\
	(\pi_{\HT,T,I}^{\dR})_{*}
	\left(\dR_{I,J}^{\alg}(\nabla^{\kw,(\Bbbk',w')}_{\SGTKp{T}})\right).
\end{align}
\begin{lem}\label{lemKunnthEasyPushForwToFl}
The natural morphism (\ref{alignSeprateSomeAlso}) is an isomorphism.
\end{lem}
\begin{proof}
It suffices to show that it induces isomorphisms on the graded pieces for the filtrations defined by the stupid truncations (\cite[Corollary \ref{1-corFilterDualBGG} and Corollary \ref{1-propBGGFilSVsm}]{Jiang2026Shla}). Then it follows from \cite[Proposition \ref{1-propCartesianLA0} and Lemma \ref{1-lemKunnethProper}]{Jiang2026Shla}.
\end{proof}

\begin{dfn}\label{dfnAbusePushToNoFilt}
Consider the morphism \( u^{\hpp}_{T} \) given by the composition \[
u^{\hpp}_{T}:(\SGTKp{T}^{\la,0})^{\pdR/\CC_{p},\hpp}_{\log}\xrightarrow{p_{1,T}}
\SGTKp{T}^{\pdR/\CC_{p},\hp,\sm}_{\log}\xrightarrow{f^{\sm}_{T}} \SGTKp{T}^{\pdR/\CC_{p},\sm}_{\log},
\] where \( f^{\sm}_{T} \) is induced by the natural map in \cite[Remark \ref{1-rmkFilDRtoDR}]{Jiang2026Shla}.

By abuse of notation, for  \(\kw\), \((\Bbbk',w')\)
	and \(I,J\subset T^{c}\) be as in Definition \ref{dfnDRComplexNotation} (4), we write 
	\begin{align*}
	\dR_{I,J}^{\alg}(\nabla_{\SGTKp{T}}^{\kw,(\Bbbk',w')})&:=(u^{\hpp}_{T})_{*}\left(\dR_{I,J}^{\alg}(\nabla_{\SGTKp{T}}^{\kw,(\Bbbk',w')})\right),\\
	\dR_{I,J}^{\alg}(\nabla_{\SGTKp{T},c}^{\kw,(\Bbbk',w')})&:=(u^{\hpp}_{T})_{*}\left(\dR_{I,J}^{\alg}(\nabla_{\SGTKp{T},c}^{\kw,(\Bbbk',w')})\right).
	\end{align*}
\end{dfn}
\begin{prop}\label{propClassicalityCohoStronger}
	Let \(\kw\), \((\Bbbk',w')\)
	and \(I,J\subset T^{c}\) be as in Definition \ref{dfnDRComplexNotation} (4). Assume that \(I\cup J=T^{c}\).
	Then 
	\[(\pi_{\HT,T,I\backslash J}^{\pdR})_{*}\left(\dR_{I,J}^{\alg}(\nabla_{\SGTKp{T}}^{\kw,(\Bbbk',w')})\right) \;\&\; (\pi_{\HT,T,I\backslash J}^{\pdR})_{*}\left(\dR_{I,J}^{\alg}(\nabla_{\SGTKp{T},c}^{\kw,(\Bbbk',w')})\right)\]
	are classical of weight \(\kw\) as Hecke modules in the sense of Definition \ref{dfnClassicalHecke} by taking \[\mscr{C}:=\QCoh\left(\Fl_{G_{T},\mu_{T},I\backslash J}^{\dR}/G_{T}(\QQ_{p})^{\la}\right).\]
\end{prop}
\begin{rmk}
Since the proof will be a bit long, let us first give a sketch of the proof: 

(1)
The case for \( (I,J) \) follows easily from the case for \( (I,J\backslash I) \); 

(2) For any \( I\coprod J=T^{c} \), and for any non-empty \( I_{0}\subset I \),
we apply the locally analytic Jacquet-Langlands transfer (Theorem \ref{thmJacquetLanglandsLA}) and the inductive hypothesis to conclude the classicality of \( \dR_{I,J} \) over the \( I_{0} \)-basic locus; 

(3) The union of all the \( I_{0} \)-basic loci for non-empty \( I_{0}\subset I \) 
covers everything but a cetain ``\( I \)-ordinary locus'', over which we can relate \( \dR_{I,J} \)
directly to \( \dR_{I,T^{c}} \); 

(4) To show the classicality of \( \dR_{I,T^{c}} \) over the \( I \)-ordinary locus, we reverse the procedure, and use the classicality over the \( I_{0} \)-basic locus, as well as the classicality of the global section of \( \dR_{I,T^{c}} \) (Example \ref{egdRITc}).
\end{rmk}
\begin{proof}
	We will prove the statement for \(\dR^{\alg}_{I,J}(\nabla_{\SGTKp{T}}^{\kw,(\Bbbk',w')})\). The proof works verbatim for \(\dR^{\alg}_{I,J}(\nabla_{\SGTKp{T},c}^{\kw,(\Bbbk',w')})\).
	We prove the result by induction on \(\#I\). 

	(1)
	First, we note that \(\dR^{\alg}_{I,J}(\nabla^{\kw,(\Bbbk',w')}_{\SGTKp{T}})\)
	is filtered by shifts of \( \dR_{I\backslash J,J}^{\alg}(\nabla^{(s_{K}\cdot \Bbbk,w),(\Bbbk',w')}_{\SGTKp{T}}) \) for \( K\subset I\cap J \).
	 One can see this immediately from the definition (Definition \ref{dfnDRComplexNotation}).

	So the case for \( (I,J) \) follows from that for \( (I\backslash J,J) \), and it suffices to consider the case where \(I\cap J=\emptyset\). This will be our ongoing assumption.

	(2)
	If \(I=\emptyset\), then \(J=T^{c}\) and \(\Fl^{\dR}_{G_{T},\mu_{T},\emptyset}=\AnSpec(\bar{\QQ}_{p,\square})\).
	By Example \ref{egdRITc}, \[\dR^{\alg}_{\emptyset,T^{c}}(\nabla^{\kw,(\Bbbk',w')}_{\SGTKp{T}})\cong \dR^{\alg}_{\emptyset}(\omega^{\kw,\sm}_{\SGTKp{T}})\otimes V_{T^{c}}^{(\Bbbk',w')}\in\QCoh(\SGTKp{T}^{\pdR/\CC_{p},\sm}_{\log}).\]
	Thus
	\begin{align*}
	(\pi_{\HT,T,\emptyset}^{\pdR})_{*}\left(\dR^{\alg}_{I,J}(\nabla_{\SGTKp{T}}^{\kw,(\Bbbk',w')})\right)&\cong 
	R\Gamma(\SGTKpTo{T}^{\sm},\omega^{\kw,\sm}_{\SGTKp{T}})
	\\
	&\cong \varinjlim_{K_{p}\subset G_{T}(\QQ_{p})}R\Gamma(\SGTKTo{T}{\Kpp},\omega^{\kw}_{\SGTKTo{T}{\Kpp}}),
	\end{align*} which is a classcial Hecke \(\TT^{S}\)-module of weight \(\kw\) by the description of the coherent cohomology of the automorphic vector bundles over the toroidal compactifications by \cite{Su2018coherent}, and the classical Jacquet-Langlands correspondence (\cite{JacquetLanglands2006automorphic}).
	Combined with (1),
	we know that the statement holds for \( (I,T^{c}) \) 
	for any \(I\subset T^{c}\). 


	(3)
	If \(I\ne \emptyset\), let us
	consider the following finite set: \[\mscr{P}_{I}:=\{(\p,I_{0}):\p\in\Sigma_{p},I_{0}\subset \Sigma_{\infty/\p}\cap I,1\le \#I_{0}\le 2\}.\]

	For any \((\p,I_{0})\in\mscr{P}_{I}
    \), \(\Fl_{G_{T},\mu_{T},I_{0}}=\Fl_{G_{T},\mu_{I_{0}^{c}}}\) (Notation \ref{notationFLGtMutI}) contains an open locus \(\Fl_{G_{T},\mu_{I_{0}^{c}}}^{\bc}\)
	(Definition \ref{dfnIbasicLocus}), and there is a natural map \[\Fl^{\dR}_{G_{T},\mu_{T},I}\xrightarrow{p_{I,I_{0}}} \Fl^{\dR}_{G_{T},\mu_{T},I_{0}}\supset \Fl^{\dR,\bc}_{G_{T},\mu_{I_{0}^{c}}}.
	\] 
	We denote \(\FLTT{T}{I^{c}}^{\dR,I_{0}-\bc}:=p_{I,I_{0}}^{-1}(\Fl^{\dR,\bc}_{G_{T},\mu_{I_{0}^{c}}})\).
	We claim that \[\left.(\pi_{\HT,T,I}^{\pdR})_{*}\left(\dR^{\alg}_{I,J}(\nabla_{\SGTKp{T}}^{\kw,(\Bbbk',w')})\right)\right|_{\FLTT{T}{I^{c}}^{\dR,I_{0}-\bc}}\in\QCoh(\FLTT{T}{I^{c}}^{\dR,I_{0}-\bc}/G_{T}(\QQ_{p})^{\la})\]
	is classical of weight \(\kw\)
	as a Hecke module.

	Note that \(I\)-basic locus does not meet the boundary by Lemma \ref{lemNoEmptyBasicLocus} and Proposition \ref{propGoodReducForNonOrd}. Therefore, it suffices to consider \( \pi^{\pdR}_{\HT,T,I}|_{\SGTKp{T}^{\pdR/\CC_{p},\sm}} \), which has a factorization \[
	\pi^{\pdR}_{\HT,T,I}:\SGTKp{T}^{\pdR/\CC_{p},\sm}
	\xrightarrow{g^{\sm}_{T}}
	\SGTKp{T}^{\dR}\xrightarrow{\pi^{\dR}_{\HT,T}} \FLT{T}^{\dR}\to \Fl_{G_{T},\mu_{T},I}.
	\] We denote by \( \pi^{dR}_{\HT,T,I} \) the composition of the last two maps. 
	Then 
	it suffices to show that 
	\[
	\left.(\pi_{\HT,T,I}^{\dR})_{*}\left(\dR_{I,J}(\nabla_{\SGTKp{T}}^{\kw,(\Bbbk',w')})\right)\right|_{\FLTT{T}{I^{c}}^{\dR,I_{0}-\bc}}\in\QCoh(\FLTT{T}{I^{c}}^{\dR,I_{0}-\bc}/G_{T}(\QQ_{p})^{\la})
	\]
	is classical of weight \(\kw\)
	as a Hecke module.

	We put \(T':=T\coprod I_{0}\).
	By Theorem \ref{thmJLinTateSt} and Theorem \ref{thmJacquetLanglandsLA},
	we have 
	\begin{itemize}
	\item a natural morphism \[\JL^{\dR}:\SGTKp{T'}^{\dR}\times\mcal{M}_{T\to T'}^{\dR}\to \SGTKp{T}^{\dR,I_{0}-\bc},\] 
	which
	is a \(G_{T'}(\QQ_{p})^{\sm}\)-torsor;
	\item a \( G_{T'}(\QQ_{p})^{\sm}\times G_{T}(\QQ_{p})^{\la} \)-equivariant isomorphism 
	\begin{align*}
	\JL&^{\dR,*}\left(\dR_{I,J}(\nabla^{\kw,(\Bbbk',w')}_{\SGTKp{T}})\right)\\&\cong R\Gamma\left((\fg_{T^{\prime,c}},*_{T'}),\dR_{I\backslash I_{0},J}(\nabla^{\kw,(\Bbbk'|_{T^{\prime,c}},w')}_{\SGTKp{T'}})
	\boxtimes \omega^{(\Bbbk'|_{I_{0}},w'),T'-*_{T'}-\sm}_{T\to T'}
	\right).
	\end{align*}
	\end{itemize}
	Moreover, when taking limits along \( K^{p} \), the tower of the isomorphisms is \( G_{T}(\mA_{f}^{p})^{\sm} \)-equivariant.

	Moreover,  by Theorem \ref{thmProductFormulaQuate} (2) and Lemma \ref{lemRZisAtorsorOverFl},
	we have a Cartesian diagram of diamonds 
	\[
	\begin{tikzcd}
	\SGTKp{T'}\times\mcal{M}_{T\to T'}\arrow[r,"\JL"]\arrow[d,"{(\pi_{\HT,T,I\backslash I_{0}},1)}"]
	& \SGTKp{T}^{\dR,I_{0}-\bc}\arrow[d,"\pi_{\HT,T,I}"]
	\\
	\FLTT{T'}{(I\backslash I_{0})^{c}}\times
	\mcal{M}_{T\to T'}\arrow[r,"\JL_{I_{0}}"]
	& \FLTT{T}{I^{c}}^{\dR,I_{0}-\bc},
	\end{tikzcd}
	\] and
	 we have a 
	Cartesian 
	commutative diagram \[
	\begin{tikzcd}
	\SGTKp{T'}^{\dR}\times\mcal{M}_{T\to T'}^{\dR}\arrow[r,"\JL^{\dR}"]\arrow[d,"{(\pi_{\HT,T,I\backslash I_{0}}^{\dR},1)}"]
	& \SGTKp{T}^{\dR,I_{0}-\bc}\arrow[d,"\pi_{\HT,T,I}^{\dR}"]
	\\
	\FLTT{T'}{(I\backslash I_{0})^{c}}^{\dR}\times
	\mcal{M}_{T\to T'}^{\dR}\arrow[r,"\JL^{\dR}_{I_{0}}"]
	& \FLTT{T}{I^{c}}^{\dR,I_{0}-\bc}. 
	\end{tikzcd}
	\] by \cite[Definition 4.5.3 \& Proposition 4.7.11 (3)]{AnschützBoscoLeBrasCamargoScholze2025analyticrhamstacksfarguesfontaine}.
	Note that the functor \( F \)
	in \cite[Proposition 4.7.11 (3)]{AnschützBoscoLeBrasCamargoScholze2025analyticrhamstacksfarguesfontaine} commutes with finite limits by the same argument as in \cite[Lemma \ref{1-lemSolidToTateAnalyticficationFunctor}]{Jiang2026Shla}.

	By \cite[Proposition \ref{1-propCartesianLA0}, Lemma \ref{1-lemFiniteLevelProper} (1), Theorem \ref{1-thmAnDRStackGeneral} (3), Theorem \ref{1-thmVariousDescent}]{Jiang2026Shla}, \( \pi^{\dR}_{\HT,T,I} \)
	is cohomogically co-smooth.
	Thus by \cite[Lemma \ref{1-lemSmBaseChange} (2)]{Jiang2026Shla}, we have a \( \TT^{S}\times G_{T}(\QQ_{p})^{\la} \)-equivariant isomorphism
	\begin{align*}
	&(\pi^{\dR}_{\HT,T,I})_{*}
	\dR_{I,J}(\nabla^{\kw,(\Bbbk',w')}_{\SGTKp{T}})|_{\FLTT{T}{I^{c}}^{\dR,I_{0}-\bc}}\\
	&\cong
	F_{T\to T'}\left( (\pi^{\dR}_{\HT,T,I})_{*}(\JL^{\dR})_{!}\left(
		\dR_{I\backslash I_{0},J}(\nabla^{\kw,(\Bbbk'|_{T^{\prime,c}},w')}_{\SGTKp{T'}})
	\boxtimes \omega^{(\Bbbk'|_{I_{0}},w'),T'-*_{T'}-\sm}_{T\to T'}
	\right)\right)\\
	&\cong F_{T\to T'}\left(
( \JL^{\dR}_{I_{0}})_{*}\left((\pi^{\dR}_{\HT,T,I\backslash I_{0}})_{!}
		\dR_{I\backslash I_{0},J}(\nabla^{\kw,(\Bbbk'|_{T^{\prime,c}},w')}_{\SGTKp{T'}})
	\boxtimes \omega^{(\Bbbk'|_{I_{0}},w'),T'-*_{T'}-\sm}_{T\to T'}
	\right)
	\right),
	\end{align*}
	where \( F_{T\to T'}:=C_{*}\left(G_{T'}(\QQ_{p})^{\sm},R\Gamma((\fg_{T^{\prime,c}},*_{T'}),-)\right) \) and \( C_{*}(G_{T'}(\QQ_{p})^{\sm},-):=v_{G_{T'}(\QQ_{p})^{\sm},!} \)
	for\[
	v_{G_{T'}(\QQ_{p})^{\sm}}:\AnSpec(\CC_{p})/G_{T'}(\QQ_{p})^{\sm}\to \AnSpec(\CC_{p}),
	\] and 
	the second isomorphism is given by \cite[Lemma \ref{1-lemKunnethProper}]{Jiang2026Shla}.
	By inductive hypothesis, \[ (\pi^{\dR}_{\HT,T,I\backslash I_{0}})_{*}
		\dR_{I\backslash I_{0},J}(\nabla^{\kw,(\Bbbk'|_{T^{\prime,c}},w')}_{\SGTKp{T'}}) \] is a classical \( \TT^{S} \)-module of weight \( \kw \), which implies the same for 
		\[ (\pi^{\dR}_{\HT,T,I})_{*}
	\dR_{I,J}(\nabla^{\kw,(\Bbbk',w')}_{\SGTKp{T}})|_{\FLTT{T}{I^{c}}^{\dR,I_{0}-\bc}}. \]

(3)' Let us denote by \(U_{T,I}\) the following open subspace of \(\FLTT{T}{I^{c}}=\Fl_{G_{T},\mu_{T},I}\) \[
U_{T,I}:= \bigcup_{I_{0}\in\mscr{P}_{I}} \FLTT{T}{I^{c}}^{I_{0}-\bc}\subset \Fl_{G_{T},\mu_{T},I}.
\]
Then by the results in the step (3), \begin{align}\label{alignVanishI}
\left.(\pi_{\HT,T,I}^{\pdR})_{*}\left(\dR_{I,J}^{\alg}(\nabla_{\SGTKp{T}}^{\kw,(\Bbbk',w')})\right)\right|_{U_{T,I}^{\dR}}
\end{align} is a classical \( \TT^{S} \)-module of weight \( \kw \).

(4)
Note that
by Lemma \ref{lemNoEmptyBasicLocus},
we know that \[U_{T,I}
=\begin{cases}
\Fl_{G_{T},\mu_{T},I}\backslash \im(\mP^{1}(F_{p})\to \Fl_{G_{T},\mu_{T},I}),\;&2\mid r_{\p,T},\;\forall \p\in\{\p\in\Sigma_{p}:\Sigma_{\infty/\p}\cap I\ne \emptyset\};\\
\Fl_{G_{T},\mu_{T},I},\;&2\nmid r_{\p,T},\;\exists \p\in\{\p\in\Sigma_{p}:\Sigma_{\infty/\p}\cap I\ne \emptyset\}.
\end{cases}\]

Therefore, for the rest of the proof, we assume that \(2|r_{\p,T}\) for all \(\p\) with \(\Sigma_{\infty/\p}\cap I\ne \emptyset\).
We
denote the complement of \(U_{T,I}\) in \(\Fl_{G_{T},\mu_{T},I}\) (\cite[Example \ref{1-egDaggerNBHD}]{Jiang2026Shla}) by 
\[V_{T,I}^{\dagger}:=\im(\mP^{1}(F_{p})\to \Fl_{G_{T},\mu_{T},I})^{\dagger}
\cong \prod_{\Sigma_{\infty/\p}\cap I\ne \emptyset }\left(\mP^{1}(F_{\p})\subset \Fl_{G_{T},\mu_{T},\Sigma_{\infty/\p}\cap I}\right)^{\dagger}.
\]
Consider the partition \[
j_{T,I}:U_{T,I}^{\dR}/G_{T}(\QQ_{p})^{\la}\hookrightarrow \Fl_{G_{T},\mu_{T},I}^{\dR}/G_{T}(\QQ_{p})^{\la}
\hookleftarrow (V_{T,I}^{\dagger})^{\dR}/G_{T}(\QQ_{p})^{\la}:i_{T,I}.
\]
Using the excision sequence
associated to \( (j_{T,I},i_{T,I}) \),
it remains to show that \[
(\pi^{\pdR}_{\HT,T,I})_{*}\dR^{\alg}_{I,J}(\nabla^{\kw,(\Bbbk',w')}_{\SGTKp{T}})|_{(V^{\dagger}_{T,I})^{\dR}/G_{T}(\QQ_{p})^{\la}}
\] is a classical \( \TT^{S} \)-module of weight \( \kw \).

We also write \(V_{T,I}:=\prod_{\Sigma_{\infty/\p}\cap I\ne \emptyset }\mP^{1}(F_{\p})
\), which
is a homogeneous space under the action of \(
G_{T}(\QQ_{p})
\), and by \cite[Definition \ref{1-dfnAnDRStack}]{Jiang2026Shla}, \(V_{T,I}^{\dR}\cong ((V_{T,I})^{\dagger})^{\dR}\).
For any point \(x\in V_{T,I}\), we write \(x^{\dagger}:=(\{x\}\subset \Fl_{G_{T},\mu_{T},I})^{\dagger}\) (\cite[Example \ref{1-egDaggerNBHD}]{Jiang2026Shla}). 
 Let we fix a point \(
\infty_{I}\in V_{T,I}
\), and denote its stabilizer by \(Q_{I}\subset G_{T}(\QQ_{p})\). 
Then \[
V_{T,I}^{\dagger}/G_{T}(\QQ_{p})^{\la}
\cong \infty_{I}^{\dagger}/(Q_{I}\subset G_{T})^{\dagger},
\] where \( (Q_{I}\subset G_{T})^{\dagger} \)
is as in \cite[Example \ref{1-egDaggerNBHD}]{Jiang2026Shla}.



By Lemma \ref{lemKunnthEasyPushForwToFl}, we have a natural isomorphism in \(\QCoh((V_{T,I}^{\dagger})^{\dR}/G_{T}(\QQ_{p})^{\la}) \cong \QCoh((\infty_{I}^{\dagger})^{\dR}/(Q_{I}\subset G_{T})^{\dagger}) \)
\begin{align}\label{alignSeprateSomeAlso2}
	&\left.(\pi_{\HT,T,I}^{\pdR})_{*}
	\left(\dR_{I,J}^{\alg}(\nabla^{\kw,(\Bbbk',w')}_{\SGTKp{T}})\right)
	\right|_{(\infty^{\dagger}_{I})^{\dR}}\\&\cong \left.
		(\pi_{\HT,T,I}^{\pdR})_{*}(\dR_{I}^{\alg}(\omega^{\kw,\sm}_{\SGTKp{T}}))\right
	|_{(\infty_{I}^{\dagger})^{\dR}}
	\otimes_{\CC_{p}} V^{(\Bbbk'|_{J},w')}_{J}
	\otimes_{\CC_{p}} 
	W^{(\Bbbk'|_{I},w')},
\end{align}
where \(\beta_{\Fl,T,I}|_{\infty_{I}^{\dagger}}:\infty_{I}^{\dagger}\to \infty_{I}^{\dagger,\dR}\cong \AnSpec(\bar{\QQ}_{p})\) and \[W^{(\Bbbk'|_{I},w')}:=(\beta_{\Fl,T,I}|_{\infty^{\dagger}_{I}})_{*}(\omega^{(\Bbbk'|_{I},w')}_{\FLTT{T}{I^{c}}}|_{\infty_{I}^{\dagger}})\in\QCoh(\bar{\QQ}_{p}).\]

Similarly, 
\begin{align}\label{alignSeparteSomeHH}
	&(\pi_{\HT,T,I}^{\pdR})_{*}\left(\dR^{\alg}_{I,T^{c}}(\nabla^{\kw,(\Bbbk',w')}_{\SGTKp{T}})\right)|_{(\infty^{\dagger}_{I})^{\dR}}
	\\&
	\cong \left.(\pi_{\HT,T,I}^{\pdR})_{*}\left(\dR_{I}^{\alg}(\omega^{\kw,\sm}_{\SGTKp{T}})
	\right)\right|_{(\infty_{I}^{\dagger})^{\dR}}
	\otimes_{\CC_{p}} V_{T'}^{(\Bbbk'|_{T'},w')}.
\end{align}

Let \[\mF_{I,T}:=(\pi_{\HT,T,I}^{\pdR})_{*}\left(\dR^{\alg}_{I,T^{c}}(\nabla^{\kw,(\Bbbk',w')}_{\SGTKp{T}})\right)\in\QCoh_{G_{T}(\QQ_{p})^{\la}}(\Fl_{G_{T},\mu_{T},I}^{\dR}).
\] and \(\mF_{I,T}|_{(\infty_{I}^{\dagger})^{\dR}}\in \QCoh((\infty_{I}^{\dagger})^{\dR}/(Q_{I}\subset G_{T}(\QQ_{p})^{\la})^{\dagger})\). 
By (\ref{alignSeprateSomeAlso2}) and the step (3)',
it remains to show that \( \mF_{I,T}|_{(\infty_{I}^{\dagger})^{\dR}} \) is classical of weight \( \kw \) as a \( \TT^{S} \)-module.
(5) 
Using the excision, we have a fiber sequence
\begin{align*}
	R\Gamma_{c}\left(U_{T,I}^{\dR},\mF_{I,T}\right)\to 
R\Gamma\left(\Fl_{G_{T},\mu_{T},I}^{\dR},\mF_{I,T}\right)\to 
R\Gamma\left((V_{T,I}^{\dagger})^{\dR},\mF_{I,T}\right).
\end{align*}

Now using (\ref{alignSeparteSomeHH}) and the filtration in the step (1), 
we know by the step (3)'
that \[R\Gamma_{c}(U^{\dR}_{T,I},\mF_{I,T})
\] is classical of weight \( \kw \) as a \( \TT^{S} \)-module.
On the other hand, \[R\Gamma(\Fl^{\dR}_{G_{T},\mu_{T},I},\mF_{I,T})
\] is classical of weight \( \kw \) as a \( \TT^{S} \)-module
by the step (2). Therefore, \[
R\Gamma\left((V_{T,I}^{\dagger})^{\dR},\mF_{I,T}\right)\cong q_{*}(\mF_{I,T}|_{(\infty_{T,I}^{\dagger})^{\dR}})
\] is also classical of weight \( \kw \) as a \( \TT^{S} \)-module,
where \[
q:(V_{T,I}^{\dagger})^{\dR}/G_{T}(\QQ_{p})^{\la}\cong (\infty_{T,I}^{\dagger})^{\dR}/(Q_{I}\subset G_{T})^{\dagger}
\to \AnSpec(\bar{\QQ}_{p,\square})/G_{T}(\QQ_{p})^{\la}.
\]
Note that the fiber of \( q \)
is \( (V_{T,I}^{\dagger})^{\dR} \), which is \emph{affine} and proper by \cite[Theorem \ref{1-thmAnDRStackGeneral} (5)]{Jiang2026Shla}. 
Thus \( q \) is prim by \cite[Theorem \ref{1-thmVariousDescent} (4)]{Jiang2026Shla}.
Moreover, \( q_{*} \) is conservative by the affine-ness of \( (V_{T,I}^{\dagger})^{\dR} \)
and proper base change (\cite[Lemma \ref{1-lemSmBaseChange} (2)]{Jiang2026Shla}).
As a result, any \( \mfk{a}\subset \TT^{S}_{\QQ} \), \[
q_{*}((\mF_{I,T}|_{(\infty_{T,I}^{\dagger})^{\dR}})_{\mfk{a}})\cong q_{*}((\mF_{I,T}|_{(\infty_{T,I}^{\dagger})^{\dR}}))_{\mfk{a}},
\] so we can conclude that \( (\mF_{I,T}|_{(\infty_{T,I}^{\dagger})^{\dR}})_{\mfk{a}}\cong 0 \) if \( q_{*}((\mF_{I,T}|_{(\infty_{T,I}^{\dagger})^{\dR}}))_{\mfk{a}}\cong 0 \). Hence \( \mF_{I,T}|_{(\infty_{T,I}^{\dagger})^{\dR}} \) is classical of weight \( \kw \), as desired.
\end{proof}

\section{Fontaine operator}\label{sectionFontaineOperator}

In this section, we define the geometric Fontaine operators, and relate them 
to the differential operators \(d_{\tau}^{-k_{\tau}+1}\)
and \(\bar{d}_{\tau}^{-k_{\tau}+1}\)
that we have defined over \(\SGTTo{T}^{\dR}\) (Theorem \ref{thmFontainEqDDbar}). 



In \S \ref{subsecFontaineOpe}, we will give a general discussion of Fontaine operators. 

In \S \ref{subsecHodgeTateNess}, we compute the graded pieces of \(\BdR^{\la}\) after fixing infinitesimal character. 

In \S \ref{subsecGeomeFontaine}, we define the geometric Fontaine operators, and state the main theorem relating the geometric Fontaine operators to differential operators introduced in \S \ref{subsectionConstrDiffeOper} (Theorem \ref{thmFontainEqDDbar}). 
The proof of Theorem \ref{thmFontainEqDDbar} is quite long and technical, and will be finished in \S \ref{subsecProofGeoFontaineUnique} and \S \ref{subsecProofGeoFontaineNonzero}. The key idea is to explore the extra structure (see Proposition \ref{propFontaineIsdRlinear} for details), which then pins down the operator up to a scalar. One key ingredient of the argument is the work of \cite{GuoLi2021period}.

Finally in \S \ref{subsecFontaineComplex}, we combine all the Fontaine operators into a complex, which we refer to \emph{Bernstein-Gelfand-Gelfand-Fontaine complex}. We will establish the classicality of the BGGF complex utilizing Theorems \ref{thmClassicalityCohomoDR} and \ref{thmFontainEqDDbar}. This will be the only output needed for proving the classicality theorem for Hilbert modular forms (Theorem \ref{thmClassicalityCompleteCoho}).

\begin{rmk}
Before we start, let us remark that with Theorem \ref{thmClassicalityCohomoDR} established, we will not use the formalism of analytic stacks from now on. Instead, we will work on
the analytic site of \( \SGTKpTo{T} \). The translation is done as follows: we have natural maps \[
\SGTKpTo{T}^{\la}\to \SGTKpTo{T}^{\la,0}\to \SGTKpTo{T}^{\sm}\to \SGTKpTo{T}_{\log}^{\pdR/\CC_{p},\sm}\to |\SGTKpTo{T}|,
\] and we will regard any quasicoherent sheaf on any of the four stacks on the left-hand side as a sheaf on the analytic site of \( \SGTKpTo{T} \) by \( * \)-pushing forward. 

More concretely, for any open subspace \( U\subset |\SGTKpTo{T}| \), by construction \( U \) defines an open subspace \( U^{\pdR/\CC_{p},\sm}_{\log}\subset \SGTKpTo{T}^{\pdR/\CC_{p},\sm}_{\log} \). For any \( \mF\in \QCoh(\SGTKpTo{T}^{\pdR/\CC_{p},\sm}_{\log}) \), we define a sheaf \( \mF\in D(\SGTKpTo{T}_{\an}) \)
by putting \[
\mF(U):=R\Gamma(U^{\pdR/\CC_{p},\sm}_{\log},\mF).
\] This convention also applies to \( \mF \) on the other stacks.
\end{rmk}

\subsection{Fontaine operator}\label{subsecFontaineOpe}

The goal of this subsection is to give a general discussion about the Fontaine operator. 
We start by introducing some notation.

\begin{notation}\label{notationLaforProObj}
	Let \(\mcal{K}\) be a \(p\)-adic Lie group.
	
    (1) For any solid \(\mcal{K}\)-representation \(M\) over \(\QQ_{p}\) equipped
	with a derived complete \(\mcal{K}\)-stable filtration \(\Fil^{\bullet}M\),
	we denote \[M^{R-\mcal{K}-\la}:=\Rlim_{i}(M/\Fil^{i}M)^{R-\mcal{K}-\la}\]
	endowed with the filtration \(\Fil^{\bullet}M^{R-\mcal{K}-\la}:=\Rlim_{i}(\Fil^{\bullet}M/\Fil^{i}M)^{R-\mcal{K}-\la}\).

(2)
	Let \(R\) be a solid ring over \(\QQ_{p}\), and let \(f\) be a non-zero divisor in \(R(*)\). Assume that we have an action of \(\mcal{K}\) on \(R\), 
	such that \(f\in R^{\mcal{K}-\la}(*)\).
	Let \(M\) be a solid semi-linear representation of \(\mcal{K}\) over \(R[1/f]\),
	equipped with a (derived) \(f\)-complete \(\mcal{K}\)-stable lattice \(M^{\circ}\subset M\) over \(R\),
	such that \(M=M^{\circ}[1/f]\). 
	Then we denote \[M^{R-\mcal{K}-\la}:=M^{\circ,R-\mcal{K}-\la}[1/f],\]
    where we equip \(M^{\circ}\) with the \(f\)-adic filtration, and \(
    M^{\circ,R-\mcal{K}-\la} 
    \) is defined as in (1).
\end{notation}
\begin{notation}
Let \(K\) be a finite extension of \(\QQ_{p}\).
Denote \(
    H_{K}:=\Gal(\bar{K}/K(\zeta_{p^{\infty}}))
    \) and \(\Gamma_{K}:=\Gal(K(\zeta_{p^{\infty}})/K)
    \). We will fix a generator \(\Theta_{K}\in \Lie(\Gamma_{K})\). We will omit the subscript if \(K=\QQ_{p}\).
\end{notation}
\begin{eg}
We have that \(\CC_{p}^{hH,R-\Gamma-\la}\cong \QQ_{p}(\zeta_{p^{\infty}})
\). We endow \(B_{\dR}^{+}\) with the \(t\)-adic filtration, and then we have by definition that \[B_{\dR}^{+,hH,R-\Gamma-\la}\cong \QQ_{p}(\zeta_{p^{\infty}})\llbracket t\rrbracket,\]
	and \[B_{\dR}^{hH,R-\Gamma-\la}\cong \QQ_{p}(\zeta_{p^{\infty}})((t)).\]   
\end{eg}
\begin{notation}\label{notationArithSenHT}
	Let \(K\) be a finite extension of \(\QQ_{p}\).
	Fix an embedding \(\bar{K}\subset \CC_{p}\).

	(1) Let \(V_{0}\) be a solid semi-linear representation of \(\Gal_{K}\) over \(\CC_{p}\). We say that \(V_{0}\) \emph{admits an arithmetic Sen operator}, if \[V_{0}\cong V_{0}^{hH_{K},R-\Gamma_{K}-\la}\hatotimes_{K(\zeta_{p^{\infty}})}\CC_{p}.\]

    We say that \(V_{0}\) is \emph{Hodge-Tate} 
	if \(V_{0}
    \) admits an arithmetic Sen operator, and
    \(V_{0}^{hH_{K},R-\Gamma_{K}-\la}=\bigoplus_{i=1}^{n}(V_{0}^{hH_{K},R-\Gamma_{K}-\la,\Theta_{K}=-a_{i}})\)
	for some finite set \(\{a_{i}:i=1,\ldots,n\}\subset \ZZ\). 
	In that case, those \(a_{i}\)
	with \(V_{0}^{hH_{K},R-\Gamma_{K}-\la,\Theta_{K}=-a_{i}}\ne 0\)
	are called \emph{Hodge-Tate weights} of \(V_{0}\). 

    Note that if \(K'/K\) is a finite extension, and \(V_{0}\) is a representation of \(\Gal_{K}\) as above that admits an arithmetic Sen operator, then \(V_{0}|_{\Gal_{K'}}\) also admits an arithmetic Sen operator.
	
	(2) Let \(V\) be a solid semi-linear representation of \(\Gal_{K}\) over \(B_{\dR}\), equipped with a derived \(t\)-adically complete \(\Gal_{K}\)-stable \(B_{\dR}^{+}\)-lattice \(V^{+}\).
	Assume that \(V^{+}/t\) is a semi-linear Hodge-Tate representation of \(\Gal_{K}\) over \(\CC_{p}\) of Hodge-Tate weights \(\{a_{i}:i=1,\ldots,n\}\).
    
	We consider the action of \(\Theta_{K}\)
	on \(V^{hH_{K},R-\Gamma_{K}-\la}\). 
	We denote by \(\Fil^{t}_{-m}V\) the \(t\)-adic filtration,
	given by \(\Fil^{t}_{-m}V:=t^{m}V^{+}\). Then the action of \(\Theta_{K}\) on \(\gr^{t}_{-m}V^{hH_{K},R-\Gamma_{K}-\la}\) is semi-simple with the set of eigenvalues being
	\(\{m-a_{i}:i=1,\ldots,n\}\). 

	In particular, we can take the part where the action of \(\Theta_{K}\) is nilpotent, which we denote as \(E_{0}(V^{hH_{K},R-\Gamma_{K}-\la})\). More precisely, 
	we can consider \(E_{0}((\Fil^{-m}V/\Fil^{n}V)^{hH_{K},R-\Gamma_{K}-\la})\), which stabilizes when \(m,n\) are sufficiently large,
	which we denote as \(E_{0}(V^{hH_{K},R-\Gamma_{K}-\la})\). 
    
    We will write \[\Darith(V):=\varinjlim_{K'/K}E_{0}(V^{hH_{K'},R-\Gamma_{K'}-\la}),\] where the colimit is taken over all the finite extensions \(K\subset K'\subset \CC_{p}\).
    
    Then on \(
    \Darith(V)
    \), there is a nilpotent action of \(\Theta_{K}\),
	which we refer to as the \emph{Fontaine operator}, and denote as \(\nabla_{K}
    \). Note that \(\Fil^{t}_{\bullet}V\) induces a finite filtration on \(
    \Darith(V)
    \), where the non-trivial graded pieces are \(
    \gr^{t}_{-a_{i}}(\Darith(V))\cong (V^{+}/t)^{\Theta_{K}=a_{i}}(-a_{i})
    \). Note that \(
    \nabla_{K}
    \) preserves the filtration, and \(\nabla_{K}|_{\gr^{t}_{-a_{i}}(\Darith(V))}=0
    \) for \(
    i=1,\ldots,n
    \).
\end{notation}
\begin{lem}[{\cite[Théorème 4.1]{Fontaine2004arithmetique}}]\label{lem:Fontaine0de Rham}
	If \(V\) is a continuous Hodge-Tate representation of \(\Gal_{K}\)
	over \(\QQ_{p}\) of finite dimension, 
	then \(V\) is de Rham
	if and only if the action of Fontaine operator \(\nabla_{K}\)
	on \(\Darith(V\otimes_{\QQ_{p}}B_{\dR})\)
	is zero.
\end{lem} 
\begin{proof}
    Note that \(\Darith(-)\) coincides with \(
    D_{\mrm{pdR}}(-)
    \) loc. cit., and the statement is the content of \cite[Théorème 4.1]{Fontaine2004arithmetique}. Let us give a proof here. 
	By Sen theory, \(\dim_{K(\zeta_{p^{\infty}})}((V\otimes \CC_{p})^{hH,R-\Gamma_{K}-\la})=\dim_{\QQ_{p}}(V)\), and by the argument in Notation \ref{notationArithSenHT}, \[
	\dim_{K(\zeta_{p^{\infty}})}(\Darith(V\otimes B_{\dR}))=
	\dim_{K(\zeta_{p^{\infty}})}((V\otimes \CC_{p})^{hH,R-\Gamma_{K}-\la})=\dim_{\QQ_{p}}(V).
	\]
	Note that
	\[
	D_{\dR}(V)\cong (V\otimes B_{\dR})^{\Gal_{K}}
	\cong ((V\otimes B_{\dR})^{hH,R-\Gamma_{K}-\la})^{\Gamma_{K}}\cong \Darith(V\otimes B_{\dR})^{\Gamma_{K}},
	\] 
	where \(\Gamma_{K}\)
	acts on \(\Darith(V\otimes B_{\dR})\)
	locally analytically, and the element \(\nabla_{K}\) in the Lie algebra acts nilpotently.
	In particular, we have an injection \[
	D_{\dR}(V)\otimes_{K}K(\zeta_{p^{\infty}})
	\hookrightarrow \Darith(V\otimes B_{\dR}),
	\] where \(\nabla_{K}\) acts on the LHS trivially.
	If \(V\) is de Rham, then by comparing the dimension, this injection is forced to be an isomorphism, and thus \(\nabla_{K}\) acts on the RHS also trivially.
	On the other hand, if \(\nabla_{K}\) acts on \(\Darith(V\otimes B_{\dR})\)
	trivially,
	then the action of \(\Gamma_{K}\) is smooth, and then by Galois descent \[
	\dim_{K}(D_{\dR}(V))=\dim_{K}(\Darith(V\otimes B_{\dR})^{\Gamma_{K}})=\dim_{K(\zeta_{p^{\infty}})}(\Darith(V\otimes  B_{\dR}))=\dim_{\QQ_{p}}(V)
	\] and \(V\) is de Rham.
\end{proof}

\subsection{Hodge-Tate decomposition}\label{subsecHodgeTateNess}
We continue working with the Shimura datum \((G_{T},\mu_{T})\)
for \(T\subset \Sigma_{\infty}\) as in Section \ref{sectionGeoSV}. 
Recall that \(\SGTKp{T}\) denotes the associated perfectoid Shimura variety over \( \CC_{p} \),
and we denote by \(\SGTKpTo{T}\) the toroidal compactification associated to a fixed cone decomposition. Note that \(
\SGTKp{T}=\SGTKpTo{T}
\) unless \(T=\emptyset\). 

Let \(\kw\) be a regular multiweight. We will use the notation from \S \ref{subsectionConstrDiffeOper}, especially Notation \ref{notationSomeAlgRep}.
\begin{notation}
	For any solid \(\fg\)-representation \(V\) over \(\bar{\QQ}_{p}\), and for \(I\coprod J\subset \Sigma_{\infty}\),
	we denote \(V^{\lambda^{\kw}_{I},V^{\kw}_{J}}:=\RHom_{Z(\fg_{I})\times \fg_{J}}(\lambda_{I}^{(\Bbbk|_{I},w)}\boxtimes V_{J}^{(\Bbbk|_{J},w)}, V)\),
	where \(\lambda_{I}^{(\Bbbk|_{I},w)}:=\bigotimes_{\tau\in I}(\lambda^{(k_{\tau},w)}_{\tau}\circ\pr_{\tau})\) and \(V_{J}^{(\Bbbk|_{J},w)}:=\bigotimes_{\tau\in J}(V^{(k_{\tau},w)}_{\tau}\circ\pr_{\tau})\).

	We will simply write \( V^{\lambda^{\kw}} \)
	if \( I=\Sigma_{\infty} \)
	and \( J=\emptyset \). 
\end{notation}
Recall the sheaf \(\mO^{\la,\kw}_{\SGTKp{T}}\) from Notation \ref{notationOlakw}. 
\begin{prop} 
The sheaf \(\mO^{\la,\kw}_{\SGTKp{T}}\) admits an arithmetic Sen operator, and
	is Hodge-Tate, with Hodge-Tate weights \(-\chi^{\kw}(d\mu_{T})=\sum_{\tau\in T^{c}}\frac{k_{\tau}-w}{2}\), that is,
	the arithmetic Sen operator acts via \(\chi^{\kw}(d\mu_{T})\). 

Here we are using a sheaf version of Notation \ref{notationArithSenHT}. More precisely, for any open sub-diamond \(U\subset \SGTKpTo{T}\) that is stable under the action of some open subgroup \(\Gal_{K'}\subset \Gal_{K}\), \(
R\Gamma(U,\mO^{\la,\kw}_{\SGTKp{T}})
\) as a semi-linear \(\Gal_{K'}\)-representation admits an arithmetic Sen operator, and is Hodge-Tate of weight \(-\chi^{\kw}(d\mu_{T})\).
\end{prop}
\begin{proof}
This follows from \cite[Theorem 6.3.5]{Juan2022.09locallyShi}.
\end{proof}
\begin{prop}\label{propInfiniteCharToHorizontal} Let \(\kw\) be a regular multiweight. We have that
	\[\mO^{\la,\lambda^{\kw}_{T^{c}},V_{T}^{\kw}}_{\SGTKp{T}}\cong \bigoplus_{s\in W^{M_{\mu_{T}}}}\mO_{\SGTKp{T}}^{\la,(s\cdot \Bbbk,w)},\] 
	where \(W^{M_{\mu_{T}}}:=\{s_{I}=\prod_{\tau\in I}s_{\tau}:I\subset T^{c}\}\).

	In particular, 
	\(
    \mO^{\la,\lambda^{\kw}_{T^{c}},V^{\kw}_{T}}_{\SGTKp{T}}
    \) 
	concentrates in degree \(0\),
	and is Hodge-Tate
	of Hodge-Tate weights \(-\chi^{(s\cdot\Bbbk,w)}(d\mu_{T})\)
	for \(s\in W^{M_{\mu_{T}}}\).
\end{prop}
\begin{proof}
	This follows from the fact that the Casmir element for \(\fg_{\tau}\cong\mfk{gl}_{2}\), \(\Omega_{\tau}=\frac{1}{2}\theta(\mrm{H}_{\tau})^{2}+\theta(\mrm{H}_{\tau})\), by combining \cite[Lemma \ref{1-lemCompareHorPilloniAndPan}]{Jiang2026Shla}and \cite[Corollary 4.2.8]{Pan22}. The concentration in degree \(0\) follows from Remark \ref{rmkOlaKWConcentrate}.
\end{proof}

\subsection{Geometric Fontaine operator}\label{subsecGeomeFontaine}
We denote by \(\BdR^{+}\) and \( \OBdRlog^{+} \) the filtered period sheaves on \(\SGTKTo{T}{K}_{\proket}\) as defined in \cite[Definitions 2.2.3 \& 2.2.10]{DLLZ2022logarithmicJAMS}. 
By abuse of notation, we denote by the same notation their restrictions to \( \SGTKpTo{T}_{\an} \).
Let \( \BdR:=\BdR^{+}[1/t] \) and \( \OBdRlog:=\OBdRlog^{+}[1/t] \). 
\begin{eg}
We apply Notation \ref{notationLaforProObj} (1) to \(\BdR^{+}\), and denote \[\BdR^{+,\la}:=\Rlim_{n}(\BdR^{+}/t^{n}\BdR^{+})^{\la}.\] 
Here \( (-)^{\la} \) is taken in the sense of \cite[Notation \ref{1-notationRlaConvention}]{Jiang2026Shla}.
We apply Notation \ref{notationLaforProObj} (2) to \(\BdR\), and thus \( \BdR^{\la}=\BdR^{+,\la}[1/t] \). 
In particular, \( \BdR^{\la} \) is filtered by \( \mO^{\la}_{\SGTKp{T}}(i) \). 
\end{eg}
\begin{prop}\label{corCompDartihBdR}
	Let \(\kw\) be a multiweight with \(k_{\tau}\in \ZZ_{\le 0}\).
	Then \( \Fil^{t}_{\bullet} \) on \(\BdR\) induces a finite filtration on
	\(\Darith(\BdR^{\la,\lambda^{\kw}_{T^{c}},V_{T}^{\kw}})\), 
	with the graded pieces
	being \[\gr^{m}\Darith(\BdR^{\la,\lambda^{\kw}_{T^{c}},V_{T}^{\kw}})\cong \bigoplus_{m=-\chi^{(s\cdot \Bbbk,w)}(d\mu_{T}),s\in W^{M_{\mu_{T}}}}
	\mO^{\la,(s\cdot\Bbbk,w),R-H-\sm,R-\Gamma-\la}_{\SGTKp{T}}\left(-\chi^{(s\cdot \Bbbk,w)}(d\mu_{T})\right).\]
	where \(\Darith(-)\) is as in Notation \ref{notationArithSenHT}.
\end{prop}
Now we consider the action of \(\nabla\) on \(
\Darith(\BdR^{\la,\lambda^{\kw}_{T^{c}},V^{\kw}_{T}})
\),
which is zero on \(\gr^{m}\Darith(\BdR^{\la,\lambda_{T^{c}}^{\kw},V_{T}^{\kw}})\). We fix a regular multiweight \(\kw\).
\begin{dfn}
	For any subset \(I_{0},J_{0}\subset T^{c}\), 
	we say that
	\(I_{0}\) and \(J_{0}\)
	are \emph{adjacent} in \(T^{c}\) for the weight \(\kw\)
	if \(\chi^{s_{I_{0}}\cdot\kw}(d\mu_{T})\)
	and \(\chi^{s_{I_{0}\cup\{\tau\}}\cdot\kw}(d\mu_{T})\)
	are adjacent in the set \(\{\chi^{s_{I}\cdot\kw}(d\mu_{T}):I\subset T^{c}\}\), and \(J_{0}=I_{0}\coprod \{\tau\}\) for some \(\tau\in T^{c}\backslash I_{0}\).
\end{dfn}
\begin{eg}
	If \(\kw\) is parallel, \(I_{0}\)
and \(I_{0}\cup\{\tau\}\)
are adjacent whenever \(\tau\notin I_{0}\).
\end{eg}
\begin{dfn}[Geometric Fontaine operators]\label{dfnFontaineOpe} Let \(\kw\) be a regular multiweight, such that \(k_{\tau}\in\ZZ_{\le 0}\)
for all \(\tau\in \Sigma_{\infty}\). 

	Assume that \(m_{0}<m_{1}\) are adjacent in the set \(\{\chi^{s_{I}\cdot\kw}(d\mu_{T}):I\subset T^{c}\}\). Then the action of \(\nabla\) on \(\Darith(\BdR^{\la,\lambda_{T^{c}}^{\kw},V_{T}^{\kw}})\)
	induces a unique monodromy map, which we refer to as \emph{geometric Fontaine operator}, and denote by \[N_{m_{0}\to m_{1}}:\gr^{m_{0}}\Darith(\BdR^{\la,\lambda_{T^{c}}^{\kw},V_{T}^{\kw}})\to \gr^{m_{1}}\Darith(\BdR^{\la,\lambda_{T^{c}}^{\kw},V_{T}^{\kw}}).
	\] 

	Now using the decompositions in Proposition \ref{corCompDartihBdR},
	for any \(I_{i}\subset T^{c}\) with \(m_{i}=-\chi^{(s_{I_{i}}\cdot\Bbbk,w)}(d\mu_{T})\), for \(i=0,1\), we have an induced morphism \[
	N_{I_{0}\to I_{1}}:\mO^{\la,(s_{I_{0}}\cdot\Bbbk,w)}_{\SGTKp{T}}\left(m_{0}\right)\to \mO^{\la,(s_{I_{1}}\cdot\Bbbk,w)}_{\SGTKp{T}}\left(m_{1}\right).
	\]

	If in addition \(I_{1}=I_{0}\cup \{\tau\}\), that is, \(I_{0}\) and \(I_{1}\) are adjacent, we will write \(N_{I_{0},\tau}:=N_{I_{0}\to I_{1}}\).
\end{dfn}

The main goal of this section is to calculate \(N_{I_{0}\to I_{1}}\),
and to relate it to the differential operators in \S \ref{subsectionConstrDiffeOper}.
Let us denote by \(F_{T}\) the reflex field of \((G_{T},\mu_{T})\) and \(\tau_{T,0}:F_{T}\hookrightarrow \CC\) the prefixed embedding.
Recall that we have fixed an isomorphism \(\bar{\QQ}_{p}\cong \CC\),
which induces an embedding \(\tau_{T,0}:F_{T}\hookrightarrow F_{T,\p_{T}}\hookrightarrow \bar{\QQ}_{p}\)
for a unique place \(\p_{T}\) of \(F_{T}\).
\begin{thm}\label{thmFontainEqDDbar}
	Let \(\kw\) be a regular multiweight. Let \(m_{0}<m_{1}\), and \(I_{0},I_{1}\)
	be as in Definition \ref{dfnFontaineOpe}. 
	
	(1) Then \(N_{I_{0}\to I_{1}}= 0\) unless \(I_{0}\) and \(I_{1}\) are adjacent for the weight \(\kw\). 

	(2) If \(I_{0},I_{1}\)
	are adjacent in \(T^{c}\)
	for the weight 
	\(\kw\),
	 and \(k_{\tau}\in\ZZ_{\le0}\), then \(I_{1}=I_{0}\coprod \{\tau\}\) for some \(\tau\in T^{c}\backslash I_{0}\), 
	 \[N_{I_{0},\tau}=N_{I_{0}\to I_{1}}=c\cdot d^{-k_{\tau}+1}_{\tau}\circ \bar{d}^{-k_{\tau}+1}_{\tau},\]
	where \(c\in F_{T,\p_{T}}^{\times}\), and \(d^{-k_{\tau}+1}_{\tau}\)
	and \(\bar{d}^{-k_{\tau}+1}_{\tau}\)
	are defined in \S \ref{subsectionConstrDiffeOper}. 
\end{thm}

The proof of the theorem will be finished in two steps. We will first show in \S \ref{subsecProofGeoFontaineUnique} that \(N_{I_{0},\tau}\) is unique up to a scalar \(c\in \QQ_{p}\), and then we will in \S \ref{subsecProofGeoFontaineNonzero} prove that \(N_{I_{0},\tau}\) is not zero, which will then imply that \(c\ne 0\).

\subsection{Proof of Theorem \ref{thmFontainEqDDbar}: Uniqueness}
\label{subsecProofGeoFontaineUnique}
Let us start from the first step. We will work in the category of filtered modules over some filtered \(\mbb{E}_{\infty}\)-rings. 
Let us start by recalling some basic definitions. The readers can find details in \cite[\S 5.1]{BhattMorrowScholze2019topological}. 

\begin{dfn}[Trivial filtration functor]
Assume that \(\mscr{C}\) is presentable.
We have a functor \(\mrm{ev}_{\underline{i}}:\Fil^{\ZZ^{N}}(\mscr{C})\to \mscr{C},\;C\mapsto \Fil^{\underline{i}}C\),
which commutes with small limits and colimits. Thus by \cite[Corollary 5.5.2.9]{Lurie2009HTT}, \(\mrm{ev}_{\underline{i}}\) admits a left adjoint, which we denote as \(\mrm{triv}^{\underline{i}}:\mscr{C}\to \Fil^{\ZZ^{N}}(\mscr{C})\).
Concretely, given \(F\in \mscr{C}\), 
\[\Fil^{\underline{j}}\mrm{triv}^{\underline{i}}(F)=\begin{cases}
F,& \underline{j}\le \underline{i};\\
0,& \text{else}.
\end{cases}\] 

If in addition \(\mscr{C}\) is equipped with a symmetric monoidal structure, and if we are given \(R\in \mrm{CAlg}(\Fil^{\NN}(\mscr{C}))\), then since \(\gr^{0}(-)|_{\Fil^{\NN}(\mscr{C})}\) is symmetric monoidal, \(\gr^{0}(R)\) has an induced structure of \(\mbb{E}_{\infty}\)-algebra in \(\mscr{C}\) such that \(R\to \mrm{triv}^{0}(\gr^{0}(R))\) is a morphism in \(\mrm{CAlg}(\Fil^{\ZZ}(\mscr{C}))\). Then \(\mrm{triv}^{i}:\mscr{C}\to \Fil^{\ZZ}(\mscr{C})\) induces a functor \[\mrm{triv}^{i}:\Mod_{\gr^{0}(R)}(\mscr{C})\to \Mod_{\mrm{triv}^{0}(\gr^{0}(R))}(\Fil^{\ZZ}(\mscr{C}))\to  \Mod_{R}(\Fil^{\ZZ}(\mscr{C})).\]
\end{dfn}
\begin{prop}\label{propFontaineIsdRlinear}
	Let \(R\in \mrm{CAlg}(\Fil^{\NN}(D_{\an}(\SGTKpTo{T},\CC_{p})))\) be either 
	\(\dR(\mO^{\sm}_{\SGTKp{T}})\) or \(\pi_{\HT}^{-1}(\dR(\mO_{\FLT{T}}))\) equipped with the filtration induced by the stupid truncation of the de Rham complex. In particular, \(\gr^{0}(R)\cong \mO^{\sm}_{\SGTKp{T}}\) or \(\pi_{\HT}^{-1}(\mO_{\FLT{T}})\).

	Let \(m_{0}<m_{1}\) and \(I_{0},I_{1}\) be as in Definition \ref{dfnFontaineOpe}.
	Then the geometric Fontaine operator in Definition \ref{dfnFontaineOpe} can be naturally upgraded to a morphism of filtered \(R\)-modules 
	\[N_{I_{0}\to I_{1}}:\mrm{triv}^{m_{0}}\left(\mO^{\la,(s_{I_{0}}\cdot\Bbbk,w)}_{\SGTKp{T}}\left(m_{0}\right)\right)\to \mrm{triv}^{m_{1}}\left(
	\mO^{\la,(s_{I_{1}}\cdot\Bbbk,w)}_{\SGTKp{T}}\left(m_{1}\right)\right).\] The same statement also holds for \(N_{c,I_{0}\to I_{1}}\).
\end{prop}
\begin{rmk}\label{rmkFontaineIsdRbilinear}
	We can consider \(R':=\dR(\mO^{\sm}_{\SGTKp{T}})\hatotimes \dR(\mO_{\FLT{T}})\) as a \(\ZZ^{\NN^{2}}\)-filtered algebra in \(D_{\an}(\SGTKpTo{T},\CC_{p})\). Then Proposition \ref{propFontaineIsdRlinear} is equivalent to the following statement:
	the geometric Fontaine operator in Definition \ref{dfnFontaineOpe} can be naturally upgraded to a morphism of \(\ZZ^{2}\)-filtered \(R'\)-modules 
	\[N_{I_{0}\to I_{1}}:\mrm{triv}^{(m_{0},m_{0})}\left(\mO^{\la,(s_{I_{0}}\cdot\Bbbk,w)}_{\SGTKp{T}}\left(m_{0}\right)\right)\to \mrm{triv}^{(m_{1},m_{1})}\left(
	\mO^{\la,(s_{I_{1}}\cdot\Bbbk,w)}_{\SGTKp{T}}\left(m_{1}\right)\right)
	\]
	The natural appearance of the \(\ZZ^{2}\)-filtration partially explains why the Fontaine operator is a differential operator of order 2
	as in \cite{Pan2209.06II}.
\end{rmk}
\begin{rmk}\label{rmkFontainDRLinNoTameLevel}
	Since the transition maps along the prime-to-\(p\) levels are finite \'etale, the geometric Fontaine operator above is clearly compatible when \(K^{p}\) varies, and we obtain a \(G_{T}(\mA_{f})\)-equivariant morphism of \(\ZZ^{2}\)-filtered \(\dR(\mO^{\sm}_{\SGT{T}})\hatotimes\dR(\mO_{\FLT{T}})\)-modules in \(D_{\an}(\SGTTo{T},\CC_{p})\) \[
	N_{I_{0}\to I_{1}}:\mrm{triv}^{(m_{0},m_{0})}\left(\mO^{\la,(s_{I_{0}}\cdot\Bbbk,w)}_{\SGT{T}}\left(m_{0}\right)\right)\to \mrm{triv}^{(m_{1},m_{1})}\left(
	\mO^{\la,(s_{I_{1}}\cdot\Bbbk,w)}_{\SGT{T}}\left(m_{1}\right)\right),
	\] where in this section, 
	\(\SGTTo{T}:=\varprojlim_{K^{p}}\SGTKpTo{T}\), where the limit is taken in the category of diamonds, and \( \mO^{\la,\kw}_{\SGT{T}}:=\varinjlim_{K^{p}}\pi_{K^{p}}^{-1}(\mO^{\la,\kw}_{\SGTKp{T}}) \) for \( \pi_{K^{p}}:\SGTTo{T}\to \SGTKpTo{T} \).
\end{rmk}
\begin{proof}
	We start by considering the case where \(R=\dR(\mO^{\sm}_{\SGTKp{T}})\). 
	Note that \(\SGTKpTo{T}\) is base changed from a finite extension of \(\QQ_{p}\) to \(\CC_{p}\). Let us define \(\SGTKTo{T}{K}^{\mrm{rat}}\) as the model of \(\SGTKTo{T}{K}\) defined over \(\bar{\QQ}_{p}\), and
	\( \SGTKpTo{T}^{\mrm{rat}}:=\varprojlim_{K_{p}}\SGTKTo{T}{K}^{\mrm{rat}} \) with \( \pi_{K_{p}}:\SGTKpTo{T}^{\mrm{rat}}\to \SGTKTo{T}{K}^{\mrm{rat}} \).
	Let \( \mO^{\sm,\mrm{rat}}_{\SGTKp{T}}:=\varinjlim_{K_{p}}\mO_{\SGTKTo{T}{K}^{\mrm{rat}}} \), which still carries a differential,
	and we write \(R^{\mrm{rat}}:=\dR(\mO^{\sm,\mrm{rat}}_{\SGTKp{T}})\in \mrm{CAlg}(\Fil^{\ZZ}(D_{\an}(\SGTKpTo{T}^{\mrm{rat}},\bar{\QQ}_{p})))\) such that \(R\cong R^{\mrm{rat}}\hatotimes_{\bar{\QQ}_{p}}\CC_{p}\). In particular, \(R^{R-H-\sm,R-\Gamma-\la}\cong R^{\mrm{rat}}\).

	By \cite{GuoLi2021period}, there is a morphism between filtered 
	algebras in \(D_{\proet}(\SGTKTo{T}{K}^{\mrm{rat}})\) \[\nu^{-1}\dR_{{\SGTKTo{T}{K}^{\mrm{rat}}}/\bar{\QQ}_{p}}\to \mbb{B}_{\dR,\SGTKTo{T}{K}^{\mrm{rat}}}^{+}=\widehat{\dR}^{\an}_{(\SGTKTo{T}{K}^{\mrm{rat}})_{\proet}/\bar{\QQ}_{p}},\]
	where \(\nu^{-1}\) denotes the pull-back from the \'etale site to the pro\'etale site. 
	We take the global section at \(\SGTKpTo{T}\), and then take colimit along the level \(K_{p}\). Then we obtain a morphism in \(\mrm{CAlg}(\Fil^{\ZZ}(D_{\an}(\SGTKpTo{T},\bar{\QQ}_{p})))\) \[
	R^{\mrm{rat}}\to \BdR^{+}.\]

	Since the action of \(G_{T}(\QQ_{p})\) on \(R^{\mrm{rat}}\)
	is smooth, we have a factorization \[R^{\mrm{rat}}\to \BdR^{\la}.\]
	We can further take \((-)^{R-H-\sm,R-\Gamma-\la}\) to obtain a morphism \[R^{\mrm{rat}}\to \BdR^{\la,R-H-\sm,R-\Gamma-\la}.\] 
	Note that the action of \(\Theta\in \Lie(\Gamma)\) on \(R^{\mrm{rat}}\) is zero. 
	Moreover, since the action of \(\Gamma\)
	preserves the algebra structure and the filtration,
	\(\Theta\)  acts on \(\BdR^{\la,R-H-\sm,\Gamma-\la}\) as a derivative, and preserves the filtration.
	As a result, 
	we know that \(\Theta\) can be upgraded to \[\Theta \in \End_{\Mod_{R^{\mrm{rat}}}(\Fil^{\ZZ}(D_{\an}(\SGTKpTo{T},\bar{\QQ}_{p})))}(\BdR^{\la,R-H-\sm,\Gamma-\la}).
	\] 
Therefore, the action of \(\nabla\) on \(\Darith(\BdR^{\la,\lambda^{\kw}_{T^{c}},V^{\kw}_{T}}) \) can be upgraded to \[\nabla \in \End_{\Mod_{R^{\mrm{rat}}}(\Fil^{\ZZ}(D_{\an}(\SGTKpTo{T},\bar{\QQ}_{p})))}(\Darith(\BdR^{\la,\lambda^{\kw}_{T^{c}},V^{\kw}_{T}})).
\]

We consider the action of \(\Theta\) on the following fiber sequence of filtered modules 
\begin{align*}
\mrm{triv}^{m_{1}}\Darith(\gr^{m_{1}}\BdR^{\la,\lambda^{\kw}_{T^{c}},V_{T}^{\kw}})\to 
\Darith(\Fil^{m_{0}}\BdR^{\la,\lambda^{\kw}_{T^{c}},V_{T}^{\kw}}/\Fil^{m_{1}+1})\\
\to \mrm{triv}^{m_{0}}
\Darith(\gr^{m_{0}}\BdR^{\la,\lambda^{\kw}_{T^{c}},V_{T}^{\kw}})\xrightarrow{+1}.
\end{align*}
Let us omit \(\BdR^{\la,\lambda^{\kw}_{T^{c}},V_{T}^{\kw}}\) from the notation for a moment.
Then taking \(\Hom(\mrm{triv}^{m_{0}}\Darith(\gr^{m_{0}}))\) in the category of filtered \(R^{\mrm{rat}}\)-modules, we have a fiber sequence 
\begin{align*}
	\Hom(\mrm{triv}^{m_{0}}\Darith(\gr^{m_{0}}),\mrm{triv}^{m_{1}}\Darith(\gr^{m_{1}}))\to \Hom(\mrm{triv}^{m_{0}}\Darith(\gr^{m_{0}}),\Darith(\Fil^{m_{0}}/\Fil^{m_{1}+1}))\\
	\to \Hom(\mrm{triv}^{m_{0}}\Darith(\gr^{m_{0}}),\mrm{triv}^{m_{0}}\Darith(\gr^{m_{0}}))\xrightarrow{+1}.
\end{align*}
Note that by adjunction, \begin{align*}
\Hom_{\Mod_{R^{\mrm{rat}}}(\Fil^{\ZZ}(D_{\an}(\SGTKpTo{T})))}(\mrm{triv}^{m_{0}}\Darith(\gr^{m_{0}}),\mrm{triv}^{m_{0}}\Darith(\gr^{m_{0}}))\\
\cong \Hom_{\Mod_{\gr^{0}(R)}(D_{\an}(\SGTKpTo{T}))}(\gr^{0}(R^{\mrm{rat}})\otimes_{R^{\mrm{rat}}}\Darith(\gr^{m_{0}}),\Darith(\gr^{m_{0}})).
\end{align*}
By Proposition \ref{corCompDartihBdR}, \(\Darith(\gr^{m_{0}})\) is concentrated in degree \(0\),
and by \cite[Lemma \ref{1-lemDiagonalLogProperty}]{Jiang2026Shla}, \(\gr^{0}(R^{\mrm{rat}})\otimes_{R^{\mrm{rat}}}\Darith(\gr^{m_{0}})\) is concentrated in degree \(0\),
so \[\pi_{>0}(\Hom(\mrm{triv}^{m_{0}}\Darith(\gr^{m_{0}}),\mrm{triv}^{m_{0}}\Darith(\gr^{m_{0}})))\cong 0.\]
Since \[\nabla\in \pi_{0}\Hom(\mrm{triv}^{m_{0}}\Darith(\gr^{m_{0}}),\Darith(\Fil^{m_{0}}/\Fil^{m_{1}+1})) \] is mapped to \[0 \in \pi_{0}\Hom(\mrm{triv}^{m_{0}}\Darith(\gr^{m_{0}}),\mrm{triv}^{m_{0}}\Darith(\gr^{m_{0}})),\]
we know that \(\nabla\) induces a unique element \[N_{m_{0}\to m_{1}}\in \pi_{0}\Hom_{\Mod_{R^{\mrm{rat}}}(\Fil^{\ZZ}D_{\an}(\SGTKpTo{T}))}(\mrm{triv}^{m_{0}}\Darith(\gr^{m_{0}}),\mrm{triv}^{m_{1}}\Darith(\gr^{m_{1}})),\]
which then by Proposition \ref{corCompDartihBdR} induces \[
N_{I_{0}\to I_{1}}:\mrm{triv}^{m_{0}}\left(\mO^{\la,(s_{I_{0}}\cdot\Bbbk,w)}_{\SGTKp{T}}\left(m_{0}\right)\right)\to
\mrm{triv}^{m_{1}}\left(\mO^{\la,(s_{I_{1}}\cdot\Bbbk,w)}_{\SGTKp{T}}\left(m_{1}\right)\right),\]
in \(\Mod_{R^{\mrm{rat}}}(\Fil^{\ZZ}D_{\an}(\SGTKpTo{T}))\)
as desired.

The statement for \(R=\pi_{\HT}^{-1}(\dR(\mO_{\FLT{T}}))\)
follows from the same proof. It suffices to construct a map of filtered algebras
in \(D_{\an}(\SGTKpTo{T},\bar{\QQ}_{p})\)
\[R^{\mrm{rat}}\to \BdR^{+},\] 
where \(R^{\mrm{rat}}:=\pi_{\HT}^{-1}\dR(\mO_{\FLT{T}^{\mrm{rat}}})\) where \(\FLT{T}^{\mrm{rat}}\) is the model of \(\FLT{T}\) over \(\bar{\QQ}_{p}\).

By \cite{GuoLi2021period},
we have \(\nu^{-1}(\dR(\mO_{\FLT{T}^{\mrm{rat}}}))\to \mbb{B}_{\dR,\FLT{T}^{\mrm{rat}}}^{+}\)
on the pro\'etale site of \(\FLT{T}^{\mrm{rat}}\), where \(\nu^{-1}\) denotes the pull-back from the \'etale site of \(\FLT{T}^{\mrm{rat}}\) to its pro\'etale site. 
Now for any perfectoid \(X\)
in the pro\'etale site of \(\FLT{T}^{\mrm{rat}}\),
we can consider \(X':=X\times_{\FLT{T}}\SGTKp{T}\) as a diamond,
then \(X'\) is pro\'etale over \(\SGTKpTo{T}\).
Hence there is a map of filtered algebras \[\nu^{\prime,-1}(\pi_{\HT}^{-1}\dR(\mO_{\FLT{T}^{\mrm{rat}}}))(X')
\cong \nu^{-1}(\dR(\mO_{\FLT{T}^{\mrm{rat}}}))(X')
\to  \mbb{B}_{\dR,\FLT{T}^{\mrm{rat}}}^{+}({X'})\cong \mbb{B}_{\dR,\SGTKpTo{T}}^{+}(X'),\]
where \(\nu^{\prime,-1}\) denotes the pull-back from the \'etale site of \(\SGTKpTo{T}\) to its pro\'etale site. 
Since such \(X'\) for varying \(X\) covers \(\SGTKp{T}\), and is closed under fiber product,
we obtain a morphism of filtered algebras in the pro\'etale site of \(\SGTKpTo{T}\) \[\nu^{\prime,-1}(\pi_{\HT}^{-1}\dR(\mO_{\FLT{T}^{\mrm{rat}}}))\to \mbb{B}_{\dR,\SGTKpTo{T}}^{+}\] 
by descent. Then we can push forward both sides to \((\SGTKpTo{T})_{\an}\), and by \cite[Corollary 3.17]{Scholze13},
we obtain a map of filtered algebras in \(D_{\an}(\SGTKpTo{T},\bar{\QQ}_{p})\)
\[R^{\mrm{rat}}=\pi_{\HT}^{-1}\dR(\mO_{\FLT{T}^{\mrm{rat}}})\to \BdR^{+},
\] as desired.
\end{proof}
Now we can use the additional structure to obtain the desired uniqueness.
\begin{lem}\label{lemRealizeToAnalyticSite}
We have a functor \[
\mcal{H}:\QCoh((\SGTKpTo{T}^{\la,0})^{\pdR/\CC_{p},\hpp}_{\log})\to \Mod_{\dR(\mO^{\sm}_{\SGTKp{T}})\hatotimes \dR(\mO_{\FLT{T}})}\Fil^{\ZZ^{2}}(D_{\an}(\SGTKpTo{T},\CC_{p})),
\] such that \( \mcal{H}:\mF\mapsto (U\mapsto R\Gamma((U^{\la,0})^{\pdR/\CC_{p},\hpp}_{\log},\mF)) \), where \( U\subset \SGTKpTo{T} \) is an open subspace, \[
(U^{\la,0})^{\pdR/\CC_{p},\hpp}_{\log}:=U^{\pdR/\CC_{p},\hp,\sm}_{\log}\times_{\FLT{T}^{\dR}}\FLT{T}^{\pdR/\CC_{p},\hp},
\] which is an open subspace of \( (\SGTKpTo{T}^{\la,0})^{\pdR/\CC_{p},\hpp}_{\log} \). In particular, \[
R\Gamma((\SGTKpTo{T}^{\la,0})^{\pdR/\CC_{p},\hpp}_{\log},\mF)
\cong R\Gamma_{\an}(\SGTKp{T},\mcal{H}(\mF)).
\]

Moreover, \[ \gr^{i,j}\left(\mO^{\sm}_{\SGTKp{T}}\hatotimes_{\dR_{\log}(\mO^{\sm}_{\SGTKp{T}})}\mcal{H}(\mF)\hatotimes_{\dR(\mO_{\FLT{T}})}\mO_{\FLT{T}}\right)\cong \gr^{i,j}((h_{T}^{\la,0,\hpp})^{*}\mF). \]
\end{lem}
\begin{proof}
This follows from \cite[Lemma \ref{1-lemLogDRStackIsAffine} and Lemma \ref{1-lemQCohSMsv} (3)]{Jiang2026Shla}.
\end{proof}
\begin{prop}\label{propUniquenessDRLinear}
	Let \(\kw\),
	\(m_{0}<m_{1}\) and \(I_{0},I_{1}\) be as in Definition \ref{dfnFontaineOpe}.
	Let \[\mscr{H}:=\pi_{0}\Hom
	\left(\mrm{triv}^{(m_{0},m_{0})}\left(\mO^{\la,(s_{I_{0}}\cdot\Bbbk,w)}_{\SGT{T}}
	\left(m_{0}\right)
	\right),
	\mrm{triv}^{(m_{1},m_{1})}\left(\mO^{\la,(s_{I_{1}}\cdot\Bbbk,w)}_{\SGT{T}}
	\left(m_{1}\right)\right)
	\right)^{G_{T}(\mA_{f})},
	\]
	where \(\Hom(-)\) is taken in the category of 
	\(\ZZ^{2}\)-filtered \(\dR(\mO^{\sm}_{\SGT{T}})\hatotimes \dR(\mO_{\FLT{T}})\)-modules in \(D_{\an}(\SGTTo{T},\CC_{p})\).

	Then \[
	\mscr{H}=\begin{cases}
	0,&\text{\(I_{0}\) and \(I_{1}\) are not adjacent};\\
	\CC_{p}\cdot \mcal{H}(d^{-k_{\tau}+1}_{\tau}\circ \bar{d}^{-k_{\tau}+1}_{\tau})\ne 0,&\text{\(I_{0}\) and \(I_{1}\) are adjacent, i.e. }I_{1}=I_{0}\coprod \{\tau\},
	\end{cases}
	\] where \( d^{-k_{\tau}+1}_{\tau}\) and \( \bar{d}^{-k_{\tau}+1}_{\tau} \) 
	are defined in Notation \ref{notationDiffHilber}.
\end{prop}
\begin{proof}
	By adjunction, \begin{align*}
		f\in \pi_{0}\Hom\left(
			\mrm{triv}^{(m_{0},m_{0})}\left(\mO^{\la,(s_{I_{0}}\cdot\Bbbk,w)}_{\SGT{T}}\left(m_{0}\right)\right),
	\mrm{triv}^{(m_{1},m_{1})}\left(\mO^{\la,(s_{I_{1}}\cdot\Bbbk,w)}_{\SGT{T}}\left(m_{1}\right)\right)
		\right)
	\end{align*}
	is equivalent to a morphism of \(\ZZ^{2}\)-filtered \(\mrm{triv}^{(0,0)}(\mO^{\sm}_{\SGT{T}}\otimes\mO_{\FLT{T}})\)-modules 
		\begin{align*}
	f:
			\mO^{\sm}_{\SGT{T}}\hatotimes_{\dR(\mO^{\sm}_{\SGT{T}})}
			\mrm{triv}^{(m_{0},m_{0})}\left(\mO^{\la,(s_{I_{0}}\cdot\Bbbk,w)}_{\SGT{T}}\left(m_{0}\right)\right)
			\hatotimes_{\dR(\mO_{\FLT{T}})}\mO_{\FLT{T}}\\
			\to 
			\mrm{triv}^{(m_{1},m_{1})}\left(\mO^{\la,(s_{I_{1}}\cdot\Bbbk,w)}_{\SGT{T}}\left(m_{1}\right)\right).
		\end{align*}

	We will use Notation \ref{notationNablaKW}. In particular, we have that \[\nabla^{\kw,(\Bbbk,-w)_{T^{c}}}_{\SGT{T}}\cong \mO^{\la,\kw}_{\SGT{T}}\left(-\chi^{\kw}(d\mu_{T})
		\right).\]

	By \cite[Lemma \ref{1-lemDiagonalLogProperty}]{Jiang2026Shla},
	the non-trivial graded pieces of 
	\[\mO^{\sm}_{\SGT{T}}\hatotimes_{\dR(\mO^{\sm}_{\SGT{T}})}
			\mrm{triv}^{(m_{0},m_{0})}\left(\mO^{\la,(s_{I_{0}}\cdot\Bbbk,w)}_{\SGT{T}}\left(m_{0}\right)\right)
			\hatotimes_{\dR(\mO_{\FLT{T}})}\mO_{\FLT{T}}\]
	are precisely  
	\begin{align}\label{alignCompuGradeNabla}
		\gr^{(m_{0}+i,m_{0}+j)}&\cong \Sym^{i}\Omega^{1,\sm}_{\SGT{T}}\otimes_{\mO^{\sm}_{\SGT{T}}}
	\mO^{\la,(s_{I_{0}}\cdot\Bbbk,w)}_{\SGT{T}}\left(m_{0}\right)
	\otimes_{\mO_{\FLT{T}}}\Sym^{i}\Omega^{1}_{\FLT{T}} 
	\\
	&\cong \bigoplus_{|\underline{i}|=i,\underline{i}\in(\ZZ_{\ge 0})^{T^{c}},|\underline{j}|=j,\underline{j}\in(\ZZ_{\ge 0})^{T^{c}}}
	\nabla_{\SGT{T}}^{(s_{I_{0}}\cdot\Bbbk+2\underline{i},w),(s_{I_{0}}\cdot\Bbbk+2\underline{j},-w)_{T^{c}}},
	\end{align}
	for \(i,j\in \NN\).

\begin{lem}\label{lemCompuHomNabla}
	Let \(\kw\) and \((\mbb{K},W)\) be multiweights, and let \((\Bbbk',w')\) and \((\mbb{K}',W')\) be \(T^{c}\)-multiweights. Then 
	we have 
	\[\pi_{0}\Hom_{\mO^{\sm}_{\SGT{T}}\hatotimes\mO_{\FLT{T}}}\left(
		\nabla^{\kw,(\Bbbk',w')}_{\SGT{T}},\nabla^{(\mathbb{K},W),(\mathbb{K}',W')}_{\SGT{T}}
	\right)^{G_{T}(\mA_{f})}=0\]
	unless \(\kw=(\Bbbk,W)\)
	and \((\Bbbk',w')=(\mbb{K}',W')\), 
	in which case the space is of dimension \(1\). Here the \(\Hom(-)\) is taken in the category of \(\mO^{\sm}_{\SGT{T}}\hatotimes\mO_{\FLT{T}}\)-modules.
\end{lem}
\begin{proof}[Proof of the lemma]
Assume that there exists such a \(G_{T}(\mA_{f})\)-equivariant map. By considering the horizontal action, we show easily that \((\Bbbk',w')=(\mbb{K}',W')\). So by twisting, 
	we can assume that \((\Bbbk',w')=(\Bbbk',W')=(\underline{0},0)\). 
	
	Then we have
	\begin{align*}
		\pi_{0}\Hom_{\mO^{\sm}_{\SGT{T}}\otimes\mO_{\FLT{T}}}\left(
		\nabla^{\kw,(\underline{0},0)}_{\SGT{T}},\nabla^{(\mathbb{K},W),(\underline{0},0)}_{\SGT{T}}
	\right)^{G_{T}(\mA_{f})}&\cong \pi_{0}\Hom_{\mO^{\sm}_{\SGT{T}}}(\omega^{\kw,\sm}_{\SGT{T}},\omega^{(\mbb{K},W),\sm}_{\SGT{T}})^{G_{T}(\mA_{f})}
	\\&\cong 
	H^{0}(\omega^{(\mbb{K},W),\sm}\otimes \omega^{(\Bbbk,w),\sm,\vee})^{G_{T}(\mA_{f})}.
	\end{align*}
	Now it suffices to show that \[H^{0}(\omega^{\kw,\sm}_{\SGT{T}})^{G_{T}(\mA_{f})}\cong \begin{cases}
		0, & \kw\ne (\underline{0},0)\\
		\CC_{p}, & \kw= (\underline{0},0).
	\end{cases}\]
	The case for \(\kw= (\underline{0},0)\)
	follows from the fact that the action of \(G_{T}(\mA_{f})\)
	on \(\pi_{0}(\SGT{T})\) is transitive. 
    
	If \(\kw\ne (\underline{0},0)\),
	the space is \(0\) by the following lemma.
    \begin{lem}
        \(H^{0}(\SGT{T}^{\sm},\omega^{\kw,\sm}_{\SGT{T}})^{G_{T}(\mA_{f})}\cong 0\)
	if \(\kw\ne (\underline{0},0)\).
    \end{lem}
    \begin{proof}
        Such an element will give rise to a quaternionic automorphic form
	of weight \(\kw\).
	We use the archimedean uniformization of the Shimura variety. Then such a form corresponds to a section \(f_{K}\) of the vector bundle \(\omega^{\kw}
    \) over \(G_{T}(\QQ)\backslash (X_{T}\times G_{T}(\mA_{f})/K)\) for all \(K\subset G_{T}(\mA_{f})\)
	that is compatible along pull-backs between different levels. So we know that it defines a section over \(X_{T}\times G_{T}(\mA_{f})\)  that is invariant for the action of the group 
	\(G_{T}(\QQ)\times G_{T}
   (\mA_{f}) \), where \(G_T(\mA_{f})\) only acts on the second factor. Equivalently, it defines a \(G_{T}(\QQ)\)-equivariant section over \(X_T\). By density, we know that it is also equivariant for the action of \(G_{T}(\RR)\).
	Looking at the fiber at a fixed point in \(X_{T}\), this corresponds 
	to a \(P_{\mu_{T}}(\RR)\)-invariant vector in \(V^{\kw}\),
	which must be zero unless \(\kw=(\underline{0},0)\).
    \end{proof}
\end{proof}
Assume that we are given a \emph{non-zero} \(G_{T}(\mA_{f})\)-equivariant morphism of \(\ZZ^{2}\)-filtered \(\mrm{triv}^{(0,0)}(\mO^{\sm}_{\SGT{T}}\hatotimes
\mO_{\FLT{T}})\)-modules 
		\begin{align*}
	f:
			\mO^{\sm}_{\SGT{T}}\hatotimes_{\dR(\mO^{\sm}_{\SGT{T}})}
			\mrm{triv}^{(m_{0},m_{0})}\left(\mO^{\la,(s_{I_{0}}\cdot\Bbbk,w)}_{\SGT{T}}\left(m_{0}\right)\right)
			\hatotimes_{\dR(\mO_{\FLT{T}})}\mO_{\FLT{T}}\\
			\to 
			\mrm{triv}^{(m_{1},m_{1})}\left(\mO^{\la,(s_{I_{1}}\cdot\Bbbk,w)}_{\SGT{T}}\left(m_{1}\right)\right).
		\end{align*}

Let \(\underline{t}=(t_{0},t_{1})\) be a maximal element in \(\ZZ^{2}\) such that \(f|_{\Fil^{\underline{j}}}=0\) for any \(\underline{j}>\underline{t}\), and \(f|_{\Fil^{\underline{t}}}\ne 0\). 
Such \(\underline{t}\) exists, because \(f\) is \(\ZZ^{2}\)-filtered, and we know moreover that \((m_{0},m_{0})\le \underline{t}\le (m_{1},m_{1})\). Then by (\ref{alignCompuGradeNabla}), we have a \emph{non-zero} \(G_{T}(\mA_{f})\)-equivariant morphism of \(\mO^{\sm}_{\SGT{T}}\hatotimes\mO_{\FLT{T}}\)-modules \[
f|_{\gr^{\underline{t}}}:\bigoplus_{|\underline{i}|=t_{0}-m_{0},\underline{i}\in(\NN)^{T^{c}},|\underline{j}|=t_{1}-m_{0},\underline{j}\in(\NN)^{T^{c}}}
	\nabla_{\SGT{T}}^{(s_{I_{0}}\cdot\Bbbk+2\underline{i},w),(s_{I_{0}}\cdot\Bbbk+2\underline{j},-w)_{T^{c}}}\to \nabla_{\SGT{T}}^{(s_{J_{0}}\cdot\Bbbk,w),(s_{J_{0}}\cdot \Bbbk,-w)_{T^{c}}}
	\]

Using Lemma \ref{lemCompuHomNabla}, 
this is impossible unless \(I_{0}\subset J_{0}\), that is, unless \(I_{0}\) and \(J_{0}\) are adjacent for the weight \(\kw\). 

Now we assume that \(I_{0}\) and \(J_{0}\) are adjacent, and \(J_{0}=I_{0}\coprod\{\tau\}\). Then by Lemma \ref{lemCompuHomNabla},
we know that \(\underline{t}=(m_{0}+1-k_{\tau},m_{0}+1-k_{\tau})=(m_{1},m_{1})\),
and the restriction of \(f|_{\gr^{(m_{1},m_{1})}}\) to the factor corresponding to \((\underline{i},\underline{j})=((1-k_{\tau})\cdot \delta_{\tau},(1-k_{\tau})\cdot \delta_{\tau})\) is given by multiplication by a scalar \(c_{f}\in\CC_{p}\). If \(f'\) is another such morphism, with the corresponding scalar being \(c_{f'}\in\CC_{p}\), then 
\(f-(c_{f}/c_{f'})\cdot f'\) is zero on the very same factor, and is thus forced to be zero by Lemma \ref{lemCompuHomNabla} and (\ref{alignCompuGradeNabla}). Therefore, \(\mscr{H}\) is of dimension at most \(1\).
On the other hand, \(\mcal{H}(d_{\tau}^{-k_{\tau}+1}\circ \bar{d}_{\tau}^{-k'_{\tau}+1})\)
gives a non-zero element by Lemma \ref{lemBasicNablaDiffe} and Lemma \ref{lemRealizeToAnalyticSite}. 
Therefore,
\(\mscr{H}\) is non-zero, and it is generated by \(\mcal{H}(d_{\tau}^{-k_{\tau}+1}\circ \bar{d}_{\tau}^{-k'_{\tau}+1})\).
\end{proof}

\subsection{Proof of Theorem \ref{thmFontainEqDDbar}: Non-vanishing}\label{subsecProofGeoFontaineNonzero}
We now conclude the proof of Theorem \ref{thmFontainEqDDbar}.
Given Proposition \ref{propUniquenessDRLinear} and Proposition \ref{propFontaineIsdRlinear} (as well as Remark \ref{rmkFontainDRLinNoTameLevel}),
it suffices to show that \(N_{I_{0},\tau}\) is not zero.
\begin{prop}\label{propFontaineNonZero}
	Let the notation be as in Notation \ref{dfnFontaineOpe}. Then
	\(N_{I_{0},\tau}\) is not zero.
\end{prop}
\begin{notation}
	Define the \(t\)-adic filtration on \( \BdR \) by \(\Fil^{t}_{\bullet}\BdR:=t^{-\bullet}\BdR^{+}\), and denote the graded pieces as \( \gr^{t}_{\bullet} \).
	We define similarly \( (\OBdRlog,\Fil^{t}_{\bullet}) \) using \( \OBdRlog^{+} \).

	To distinguish, we will refer to the filtrations on \(\OBdRlog\) in \cite{Scholze13} and \cite{DLLZ2022logarithmicJAMS} as \emph{Hodge filtrations}. Note that the \(t\)-adic filtration and the Hodge filtration coincide on \(\BdR\). 
\end{notation}
\begin{proof}
	Recall that \[N_{I_{0},\tau}:\mO^{\la,(s_{I_{0}}\cdot\Bbbk,w)}_{\SGTKp{T}}\left(m_{0}\right)\to 
	\mO^{\la,(s_{I_{0}\cup\{\tau\}}\cdot\Bbbk,w)}_{\SGTKp{T}}\left(m_{1}\right)\]
	is induced by the action of \(\nabla\)
	on \(\Darith(\BdR^{\la,\lambda^{\kw}_{T^{c}},V_{T}^{\kw}})\).
	Note that
	\(N_{I_{0},\tau}\) is \(\dR(\mO^{\sm}_{\SGTKp{T}})\)-linear,
	and can be extended to \[N_{I_{0},\tau}:\mO^{\la,(s_{I_{0}}\cdot\Bbbk,w)}_{\SGTKp{T}}\left(m_{0}\right)\hatotimes_{\dR(\mO^{\sm}_{\SGTKp{T}})}\mO^{\sm}_{\SGTKp{T}}\to 
	\mO^{\la,(s_{I_{0}\cup\{\tau\}}\cdot\Bbbk,w)}_{\SGTKp{T}}\left(m_{1}\right)\hatotimes_{\dR(\mO^{\sm}_{\SGTKp{T}})}\mO^{\sm}_{\SGTKp{T}},\]
	which is filtered, with the graded pieces given by \[\gr^{i}(N_{I_{0},\tau}):\mO^{\la,(s_{I_{0}}\cdot\Bbbk,w)}_{\SGTKp{T}}\left(m_{0}\right)\otimes\Sym^{i-m_{0}}\Omega^{1,\sm}\to
	\mO^{\la,(s_{I_{0}\cup\{\tau\}}\cdot\Bbbk,w)}_{\SGTKp{T}}\left(m_{1}\right)\otimes \Sym^{i-m_{1}}\Omega^{1,\sm}
	 \]
	for \(i\in\NN_{\ge m_{1}}\).
	Clearly, for our purpose, it suffices to show that \(\gr^{m_{1}}(N_{I_{0},\tau})\ne 0\).

Now we note that \(\gr^{i}(N_{I_{0},\tau})\) can be reinterpreted as follows: the action of \(\nabla\) on \(\Darith(\BdR^{\la,\lambda^{\kw}_{T^{c}},V^{\kw}_{T}})\)
	induces an action on \[\Darith(\OBdRlog^{\la,\lambda^{\kw}_{T^{c}},V_{T}^{\kw}})
	\cong
	\Darith(\BdR^{\la,\lambda^{\kw}_{T^{c}},V_{T}^{\kw}})\hatotimes_{\dR(\mO^{\sm}_{\SGTKp{T}})}\mO^{\sm}_{\SGTKp{T}},
	\] and thus also on \(\gr^{i}\Darith(\OBdRlog^{\la,\lambda^{\kw}_{T^{c}},V^{\kw}_{T}})\).
Note that \(\gr^{i}\OBdRlog^{\la}\)
admits an ascending filtration induced by the \(t\)-adic filtration
of \(\OBdRlog^{+,\la}\),
such that \[\gr^{t}_{j}\gr^{i}\OBdRlog^{\la}\cong \mO^{\la}(-j)\otimes_{\mO^{\sm}} \Sym^{j+i}\Omega^{1,\sm},\] for \(j\ge -i\).
Then \[\gr^{t}_{j}\Darith(\gr^{i}\OBdRlog^{\la,\lambda^{\kw}_{T^{c}},V^{\kw}_{T}})\otimes_{\QQ_{p}(\zeta_{p^{\infty}})}\CC_{p}\cong \bigoplus_{j=\chi^{(s\cdot \Bbbk,w)}(d\mu_{T}),s\in W^{M_{\mu_{T}}}}
\mO^{\la,(s\cdot\Bbbk,w)}_{\SGTKp{T}}(-j)\otimes_{\mO^{\sm}} \Sym^{j+i}\Omega^{1,\sm},\]
and the action of \(\Theta\) is zero on \(\gr^{t}_{j}\), and 
then the action on the extension between \(\gr^{t}_{-m_{0}}\)
and \(\gr^{t}_{-m_{1}}\)
induces 
\begin{align}
	\gr^{i}(N_{I_{0},\tau}):\mO^{\la,(s_{I_{0}}\cdot\Bbbk,w)}_{\SGTKp{T}}\left(m_{0}\right)\otimes\Sym^{i-m_{0}}\Omega^{1,\sm}\to
	\mO^{\la,(s_{I_{0}\cup\{\tau\}}\cdot\Bbbk,w)}_{\SGTKp{T}}\left(m_{1}\right)\otimes \Sym^{i-m_{1}}\Omega^{1,\sm}.
\end{align}
We can further decompose the target and the source.
\(\Sym^{i-m_{0}}\Omega^{1,\sm}\)
has a decomposition into \(\Omega^{1,\sm,\underline{l}}:=\bigotimes_{\tau\in T^{c}}\Omega^{1,\sm}_{\tau}\),
with \(\underline{l}=(l_{\tau})_{\tau\in T^{c}}\in \NN^{T^{c}}\)
with \(|\underline{l}|=i-m_{0}\). 

Let us fix one such \(\underline{l}\) such that \(l_{\tau}\ge m_{1}-m_{0}\), and then we claim that
the induced map \begin{align}\label{alignMapInFaltingsExt}
	\gr^{i}(N_{I_{0},\tau}):\mO^{\la,(s_{I_{0}}\cdot\Bbbk,w)}_{\SGTKp{T}}\left(m_{0}\right)\otimes \Omega^{1,\sm,\underline{l}}\to
	\mO^{\la,(s_{I_{0}\cup\{\tau\}}\cdot\Bbbk,w)}_{\SGTKp{T}}\left(m_{1}\right)\otimes \Omega^{1,\sm,\underline{l}-(m_{1}-m_{0})1_{\tau}}.
\end{align}
is not zero.

Now there are many ways to conclude the proof. For example, one can simply mimic the proof in \cite[\S 6.5]{Pan2209.06II}. For simplicity, we take a different approach. We will take \(\fb\)-cohomology, and reduce the problem to a calculation over the flag variety, and then reduce to the curve case that is treated in \cite{Jiang2023thetaModCur}.

Recall that by \cite{Juan2022.09locallyShi},
we know that \[\gr^{0}\OBdRlog\cong \pi_{\HT}^{*}(\underline{\mO(M_{\mu_{T}}
\backslash P_{\mu_{T}})}),\]
where \(\mO(M_{\mu_{T}}
\backslash P_{\mu_{T}})\)
is a \(P_{\mu_{T}}\)-representation by letting \(P_{\mu_{T}}\) act by right multiplication.
Note that \(M_{\mu_{T}}\backslash P_{\mu_{T}}\cong N_{\mu_{T}}\cong \prod_{\tau\in T^{c}}N_{\mu_{T},\tau}\).
We will denote by 
\(\mO(M_{\mu_{T}}\backslash P_{\mu_{T}})^{\mrm{deg}=\underline{l}}\)
as the subspace generated by \(\underline{x}^{\underline{l}}\),
with \(\underline{x}:=(x_{\tau})\) with \(N_{\mu_{T},\tau}=\Spec \CC_{p}[x_{\tau}]\).

So \[\gr^{0}\OBdRlog^{\la}\cong \tVBfu\left((\underline{\mO(M_{\mu_{T}}\backslash P_{\mu_{T}})}\otimes \mO_{G,1})^{\fn^{0}}\right),\]
and \[\gr^{0}\OBdRlog^{\la,\lambda_{T^{c}}^{\kw},V_{T}^{\kw}}\cong \tVBfu\left((\underline{\mO(M_{\mu_{T}}\backslash P_{\mu_{T}})}\otimes \mO_{G,1}^{(*_{2},\lambda_{T^{c}}^{\kw},V_{T}^{\kw})})^{\fn^{0}}\right).\]
Then by \cite{Juan2022.09locallyShi}, the action of \(\Theta\) in induced by \(\theta(d\mu_{T})\) on the right-hand side, and \[\Darith(\gr^{i}\OBdRlog^{\la,\lambda_{T^{c}}^{\kw},V_{T}^{\kw}})\hatotimes_{\QQ_{p}(\zeta_{p^{\infty}})}\CC_{p}\cong \tVBfu\left(E_{-i}((\underline{\mO(M_{\mu_{T}}\backslash P_{\mu_{T}})}\otimes \mO_{G,1}^{(*_{2},\lambda_{T^{c}}^{\kw},V_{T}^{\kw})})^{\fn^{0}})\right),\]
where on the right-hand side,
\(E_{-i}\) denotes the generalized eigenspace for \(\theta(d\mu_{T})\),
where \(\theta\) is the horizontal action induced by \(*_{1,3}\),
with \(*_{3}\) denotes the action of \(\fg\) on \(\underline{\mO(M_{\mu_{T}}\backslash P_{\mu_{T}})}\).

Let us denote \[\mF_{\kw}:=(\underline{\mO(M_{\mu_{T}}\backslash P_{\mu_{T}})}\otimes \mO_{G,1}^{(*_{2},\lambda_{T^{c}}^{\kw},V_{T}^{\kw})})^{\fn^{0}}\]
on \(\FLT{T}\)
and \[\mF_{\kw,\tau}:=(\underline{\mO(M_{\mu_{T},\tau}\backslash P_{\mu_{T},\tau})}\otimes \mO_{G_{\tau},1}^{(*_{2},\lambda_{\tau}^{\kw})})^{\fn_{\tau}^{0}}\]
on \(\Fl_{G_{T},\mu_{T},\tau}\) for \(\tau\in T^{c}\).
Then we know \[\mF_{\kw}\cong (\bigotimes_{\tau\in T^{c}}\pr_{\tau}^{*}\mF_{\kw,\tau})\otimes V_{T}^{\kw,\vee},\]
and we have the action of \(\theta(d\mu_{\tau})\) on \(\mF_{\kw,\tau}\),
such that the action of \(\theta(d\mu_{T})\) on \(\mF_{\kw}\)
is given by \[\theta(d\mu_{T})=\sum_{\tau\in T^{c}}\theta(d\mu_{\tau}).\]

We can consider the degree filtration on \(\mO(M_{\mu_{T}}\backslash P_{\mu_{T}})\), which we also denote as \(\gr_{\bullet}^{t}\), then \[\gr_{-j}^{t}\mF_{\kw}\cong \underline{\mO(M_{\mu_{T}}\backslash P_{\mu_{T}})^{\mrm{deg}=j}}\otimes C^{\la,\lambda^{\kw}_{T^{c}},V_{T}^{\kw}}\cong \bigoplus_{j=\chi^{(s\cdot \Bbbk,w)}(d\mu_{T}),s\in W^{M_{\mu_{T}}}} \underline{\mO(M_{\mu_{T}}\backslash P_{\mu_{T}})^{\mrm{deg}=j}}\otimes C^{\la,(s_{I_{0}}\cdot\Bbbk,w)}\]
Then the isomorphism \[\gr^{i}\OBdRlog^{\la,\lambda_{T^{c}}^{\kw},V_{T}^{\kw}}\cong \tVBfu\left(\mF_{\kw}\right) \] is compatible with the \(t\)-adic filtration on both sides, inducing 
\[\gr^{t}_{-j}\gr^{i}\OBdRlog^{\la,\lambda_{T^{c}}^{\kw},V_{T}^{\kw}}\cong \tVBfu\left(\gr^{t}_{-j}\mF_{\kw}\right)(i)\]
 and  \[\mO^{\la,(s_{I_{0}}\cdot\Bbbk,w)}_{\SGTKp{T}}(m_{0})\otimes \Sym^{i-m_{0}}\Omega^{1,\sm}\cong \tVBfu(\underline{\mO(M_{\mu_{T}}\backslash P_{\mu_{T}})^{\mrm{deg}=i-m_{0}}}\otimes C^{\la,(s_{I_{0}}\cdot \Bbbk,w)})(i),\]
 and \[\mO^{\la,(s_{I_{0}}\cdot\Bbbk,w)}_{\SGTKp{T}}(m_{0})\otimes \Omega^{1,\sm,\underline{l}}\cong \tVBfu(\underline{\mO(M_{\mu_{T}}\backslash P_{\mu_{T}})^{\mrm{deg}=\underline{l}}}\otimes C^{\la,(s_{I_{0}}\cdot \Bbbk,w)})(i).\]
 Therefore, the map (\ref{alignMapInFaltingsExt})
 is induced by the monodromy map \[\underline{\mO(M_{\mu_{T}}\backslash P_{\mu_{T}})^{\mrm{deg}=\underline{l}}}\otimes C^{\la,(s_{I_{0}}\cdot \Bbbk,w)}\to 
 \underline{\mO(M_{\mu_{T}}\backslash P_{\mu_{T}})^{\mrm{deg}=\underline{l}-(m_{1}-m_{0})1_{\tau}}}\otimes C^{\la,(s_{I_{0}\cup \{\tau\}}\cdot \Bbbk,w)}\]
 induced by the action of 
 \(\theta(d\mu_{T})+i\) acting on 
 \(E_{-i}(\mF_{\kw})\). 

 Now consider the following  subquotient of \(\mF_{\kw}\) \[\mF_{\kw}^{\underline{l}^{\tau}}:= 
 V_{T}^{\kw}\otimes \bigotimes_{\tau'\in T^{c}\backslash \tau}\pr_{\tau'}^{*}\left(\underline{\mO(M_{\mu_{T},\tau'}\backslash P_{\mu_{T},\tau'})^{\mrm{deg}=l_{\tau'}}}
 \otimes C^{\la,((s_{I_{0}}\cdot\Bbbk)_{\tau'},w)}_{\tau'}\right)\otimes \pr_{\tau}^{*}(\mF_{\kw,\tau}), \]
 then the concerned graded pieces are both realized in the subquotient.
 Moreover, for \(\tau'\in T^{c}\backslash \{\tau\}\),
 the action of \(\theta(d\mu_{\tau'})\) on \(\mF^{\underline{l}^{\tau}}_{\kw}\)
 is given by the constant \(-l_{\tau'}+\chi^{(s_{I_{0}}\cdot\Bbbk,w)}(d\mu_{\tau'})\). So we have \[\theta(d\mu_{T})+i=\theta(d\mu_{\tau})+l_{\tau}-\chi^{(s_{I_{0}}\cdot\Bbbk,w)}(d\mu_{\tau}),\]
 where the right-hand side only acts on the \(\tau\)-th term \(\mF_{\kw,\tau}\), so we are reduced to the case of dimension \(1\) case.
 Also note that we have a fiber sequence \begin{align*} 0\to
\underline{\mO(M_{\mu_{T},\tau}\backslash P_{\mu_{T},\tau})^{\mrm{deg}=l_{\tau}-k_{\tau}+1}}
 \otimes C^{\la,((s_{I_{0}\cup\{\tau\}}\cdot\Bbbk)_{\tau},w)}_{\tau}
 \to 
 E_{\chi^{(s_{I_{0}}\cdot \Bbbk,w)}(d\mu_{\tau})-l_{\tau}}(\mF_{\kw,\tau})
 \\\to\underline{\mO(M_{\mu_{T},\tau}\backslash P_{\mu_{T},\tau})^{\mrm{deg}=l_{\tau}}}
 \otimes C^{\la,((s_{I_{0}}\cdot\Bbbk)_{\tau},w)}_{\tau}
\to 0,
 \end{align*}
 and we have \[E_{-i}(\mF^{\underline{l}^{\tau}}_{\kw})\cong 
 V_{T}^{\kw}\otimes \bigotimes_{\tau'\in T^{c}\backslash \tau}\pr_{\tau'}^{*}\left(\underline{\mO(M_{\mu_{T},\tau'}\backslash P_{\mu_{T},\tau'})^{\mrm{deg}=l_{\tau'}}}
 \otimes C^{\la,((s_{I_{0}}\cdot\Bbbk)_{\tau'},w)}_{\tau'}\right)\otimes \pr_{\tau}^{*}(
	E_{\chi^{(s_{I_{0}}\cdot \Bbbk)}(d\mu_{\tau})-l_{\tau}}(\mF_{\kw,\tau})),\]
 which after taking \(\tVBfu(-)\)
gives the extension between \(\mO^{\la,(s_{I_{0}}\cdot\Bbbk,w)}_{\SGTKp{T}}\left(m_{0}\right)\otimes \Omega^{1,\sm,\underline{l}}\)
and \(\mO^{\la,(s_{I_{0}\cup\{\tau\}}\cdot\Bbbk,w)}_{\SGTKp{T}}\left(m_{1}\right)\otimes \Omega^{1,\sm,\underline{l}-(m_{1}-m_{0})1_{\tau}}\). 

Now we need to study the action of \(\theta(d\mu_{\tau})+l_{\tau}-\chi^{(s_{I_{0}}\cdot\Bbbk,w)}(d\mu_{\tau})\)
on \(E_{\chi^{(s_{I_{0}}\cdot \Bbbk,w)}(d\mu_{\tau})-l_{\tau}}(\mF_{\kw,\tau})\),
and show that the corresponding monodromy operator is not zero even after applying \(\tVBfu\).
This question can be simplified after we take \(\fb\)-cohomology
as in \cite{Jiang2023thetaModCur}. Recall the following lemma:
\begin{lem}[{\cite{Pilloni22}}, {\cite[Lemma 4.15]{Jiang2023thetaModCur}}]
	In this lemma, we denote \(G=\GL_{2,\QQ}\supset B=H N\),
	where \(B\) is the subgroup of upper triangular matrices,
	and \(N\) is its nilpotent radical and \(H\) is the diagonal subgroup. Denote by \(\fg,\fb,\fn\) their Lie algebras. Let \(\Fl=B\backslash G\cong\mP^{1}\), and we fix the Bruhat stratification 
	\[\Fl\cong \mA^{1}\coprod \{\infty\}.\] Denote by \(i:\infty^{\dagger}\hookrightarrow \Fl\).
	We also denote \(C^{\la}:=(\mO_{\Fl}\otimes\mO_{G,1})^{(\fn^{0},*_{1,3})}\) and denote \(*_{2}\) (resp. \(*_{1}\), resp. \(*_{3}\)) the action of \(\fg\) induced by right multiplication (resp. left multiplication, resp. on \(\mO_{Fl}\)).
	Fix \((a,b)\in\ZZ^{\oplus}\)
	with \(a\ge b\).
	
	(1) We have \[\RHom_{\fb}((a,b),(C^{\la},*_{2}))\cong \mcal{N}_{b}\oplus \mcal{N}_{1+a},\]
	with \(\mcal{N}_{1+a}\cong i_{*}i^{*}\omega^{(1-b,-1-a)}_{\Fl}[-1]\)
	and \(\mcal{N}_{b}\)
	lies in a distinguished triangle \[\omega^{(-a,-b)}_{\Fl}\to \mcal{N}_{b}\to i_{*}i^{*}\omega^{(-a,-b)}_{\Fl}[-1].\]
	Here \(\omega^{(a,b)}_{\Fl}\) refers the line bundle over \(\Fl\)
	associated to the character \(\chi^{(a,b)}\). This differs from the notation in \cite{Jiang2023thetaModCur} by \(w_{0}\). The convention here makes \(\Omega^{1}_{\Fl}\cong \omega^{(1,-1)}_{\Fl}\).

	(2) Consider \(\mF:=(\underline{\mO(H\backslash B)}\otimes \mO_{G,1})^{\fn^{0}}\),
	then \(\mF\) is filtered by the degree in \(\mO(H\backslash B)\),
	with \(\mF^{\mrm{deg}=i}\cong \omega^{(-i,i)}_{\Fl}\otimes C^{\la}\),
	and then \(\RHom_{\fb}((a,b),(\mF,*_{2}))\)
	is filtered by \(\omega^{(-i,i)}_{\Fl}\otimes \RHom((a,b),C^{\la})\),
	and we can consider the action of \(\theta(d\mu)\),
	with \(d\mu=(1,0)\in\fh\),
	and \(\theta\) the horizontal action induced by \(*_{1,3}\). 
	Then for any \(k\ge 1+a\), there is a fiber sequence \[\RHom_{\fb}((a,b),\mF^{\mrm{deg}=k-a-1})^{\theta(d\mu)=-k}\to E_{-k}(\RHom_{\fb}((a,b),\mF))\to \RHom_{\fb}((a,b),\mF^{\mrm{deg}=k-b})^{\theta(d\mu)=-k}\xrightarrow{+1},\]
	with \[\RHom_{\fb}((a,b),\mF^{\mrm{deg}=k-a-1})^{\theta(d\mu)=-k}\cong i_{*}i^{*}\omega^{(-k,k-a-b)}_{\Fl}[-1],\]
	and a fiber sequence \[\omega^{(-k,k-a-b)}_{\Fl}\to \RHom_{\fb}((a,b),\mF^{\mrm{deg}=k-b})^{\theta(d\mu)=-k}
	\to i_{*}i^{*}\omega^{(-k,k-a-b)}_{\Fl}[-1]\xrightarrow{+1}\]
	such that the action of \(\theta(d\mu)+k\)
	on \(E_{-k}(\RHom_{\fb}((a,b),\mF))\)
	is given by a unique monodromy morphism \[\RHom_{\fb}((a,b),\mF^{\mrm{deg}=k-b})^{\theta(d\mu)=-k}
	\to i_{*}i^{*}\omega^{(-k,k-a-b)}_{\Fl}[-1]\cong  \RHom_{\fb}((a,b),\mF^{\mrm{deg}=k-a-1})^{\theta(d\mu)=-k}. \]
\end{lem}
\begin{proof}[Proof of the lemma]
	Everything is essentially \cite[Lemma 4.15]{Jiang2023thetaModCur}.
	Twisting the action by an element in the center, we can reduce to the case where \(b=0\). Finally, 
	It suffices to note that over \(\Fl\), there is an isomorphism \[\Sym^{k}V\otimes\mO_{\Fl}\cong \underline{\mO(H\backslash B)^{\mrm{deg}\le k}}\otimes \omega^{(k,0)}_{\Fl},\]
	as one can verify the isomorphism for the corresponding 
	\(B\)-representations.
\end{proof}
Given the lemma, we get back to the original proof.
We can apply \(\RHom_{U\fg_{\tau}/Z\fg_{\tau}}(M_{\kw,\tau},-)\)
with \(M_{\kw,\tau}:=U\fg_{\tau}\otimes_{U\fb_{\tau}}\chi^{\left(\frac{-w-k_{\tau}}{2},\frac{-w+k_{\tau}}{2}\right)}\) to \(\mF^{\underline{l}^{\tau}}_{\kw}\). 
Then we have \[\RHom_{U\fg_{\tau}/Z\fg_{\tau}}(M_{\kw,\tau},\mF_{\kw})\cong 
\RHom_{U\fb_{\tau}}(\chi^{\left(\frac{-w-k_{\tau}}{2},\frac{-w+k_{\tau}}{2}\right)},(\underline{\mO(M_{\mu_{T}}\backslash P_{\mu_{T}})}\otimes \mO_{G,1}^{(*_{2},\lambda_{T^{c}\backslash \{\tau\}}^{\kw},V_{T}^{\kw})})^{\fn^{0}}). \]
We can track the proof above,
and then we have a fiber sequence 
\[\mcal{X}_{1}\to \RHom_{U\fg_{\tau}/Z\fg_{\tau}}(M_{\kw,\tau},E_{\chi^{(s_{I_{0}}\cdot \Bbbk,w)}(d\mu_{\tau})-l_{\tau}}(\mF_{\kw,\tau}))\to \mcal{X}_{2}\xrightarrow{+1}\]
such that \(\mcal{X}_{1}\cong i_{\tau,*}i^{*}_{\tau}\omega^{\left(\frac{w+k_{\tau}}{2}-l_{\tau},\frac{w-k_{\tau}}{2}+l_{\tau}\right)}_{\Fl_{\tau}}[-1]\)
and a fiber sequence \[\omega^{\left(\frac{w+k_{\tau}}{2}-l_{\tau},\frac{w-k_{\tau}}{2}+l_{\tau}\right)}_{\Fl_{\tau}}\to \mcal{X}_{2}\to i_{\tau,*}i^{*}_{\tau}\omega^{\left(\frac{w+k_{\tau}}{2}-l_{\tau},\frac{w-k_{\tau}}{2}+l_{\tau}\right)}_{\Fl_{\tau}}[-1],\]
with \(i_{\tau}:\infty_{\tau}^{\dagger}\hookrightarrow \Fl_{\tau}\).
This follows from the lemma as \(k_{\tau}\in\ZZ_{\le 0}\). Also note that \(\chi^{(s_{I_{0}}\cdot\Bbbk,w)}(d\mu_{\tau})=\frac{w+k_{\tau}}{2}\). More importantly, the action of \(\theta(d\mu_{\tau})-\chi^{(s_{I_{0}}\cdot \Bbbk,w)(d\mu_{\tau})}+l_{\tau}\) is induced by the map \[\mcal{X}_{2}\to i_{\tau,*}i^{*}_{\tau}\omega^{\left(\frac{w+k_{\tau}}{2}-l_{\tau},\frac{w-k_{\tau}}{2}+l_{\tau}\right)}_{\Fl_{\tau}}[-1]\cong \mcal{X}_{1}.\]

Now we can consider \(\pr_{\tau}^{*}(-)\)
and tensor with \[V_{T}^{\kw}\otimes \bigotimes_{\tau'\in T^{c}\backslash \tau}\pr_{\tau'}^{*}\left(\underline{\mO(M_{\mu_{T},\tau'}\backslash P_{\mu_{T},\tau'})^{\mrm{deg}=l_{\tau'}}}
 \otimes C^{\la,((s_{I_{0}}\cdot\Bbbk)_{\tau'},w)}_{\tau'}\right)\]
 and then apply \(\tVBfu\). Let 
 \[\mcal{Y}:=V_{T}^{\kw}\otimes \bigotimes_{\tau'\in T^{c}\backslash \tau}\pr_{\tau'}^{*}\left(\underline{\mO(M_{\mu_{T},\tau'}\backslash P_{\mu_{T},\tau'})^{\mrm{deg}=l_{\tau'}}}
 \otimes C^{\la,((s_{I_{0}}\cdot\Bbbk)_{\tau'},w)}_{\tau'}\right)\otimes \pr_{\tau}^{*}\omega_{\Fl_{\tau}}^{\left(\frac{w+k_{\tau}}{2}-l_{\tau},\frac{w-k_{\tau}}{2}+l_{\tau}\right)},
 \]
 and then
 we know that in (\ref{alignMapInFaltingsExt}),
 \[\RHom_{U\fg_{\tau}/Z\fg_{\tau}}\left(M_{\kw,\tau},\mO^{\la,(s_{I_{0}}\cdot\Bbbk,w)}_{\SGTKp{T}}\left(m_{0}\right)\otimes \Omega^{1,\sm,\underline{l}}\right)\cong \left[\tVBfu(\mcal{Y})-i_{\tau,*}i_{\tau}^{*}\tVBfu(\mcal{Y})[-1]\right](i),\]
 and \[\RHom_{U\fg_{\tau}/Z\fg_{\tau}}\left(M_{\kw,\tau},
	\mO^{\la,(s_{I_{0}\cup\{\tau\}}\cdot\Bbbk,w)}_{\SGTKp{T}}\left(m_{1}\right)\otimes \Omega^{1,\sm,\underline{l}-(m_{1}-m_{0})1_{\tau}}\right)\cong i_{\tau,*}i_{\tau}^{*}\tVBfu(\mcal{Y})[-1](i)\]
 where we write \([-]\)
 for an extension of two objects, and 
 \(i_{\tau}\) denotes \(i_{\tau}:\pi_{\HT}^{*}(\pr_{\tau}^{*}(\infty_{\tau}^{\dagger}))\hookrightarrow \SGTKpTo{T}\). 
 Then by the argument above, the morphism \(\gr^{i}(N_{I_{0},\tau})\) (\ref{alignMapInFaltingsExt})
 coincides with the projection map \[\left[\tVBfu(\mcal{Y})-i_{\tau,*}i_{\tau}^{*}\tVBfu(\mcal{Y})[-1]\right](i)\to i_{\tau,*}i_{\tau}^{*}\tVBfu(\mcal{Y})[-1](i).\]
 Note that by \cite[Theorem \ref{1-thmMainThmGeomSenVB}]{Jiang2026Shla}, \(\tVBfu(\mcal{Y})\) is an ON Banach module over \(\SGTKpTo{T}^{\sm}\), satisfying \[\tVBfu(\mcal{Y})\otimes_{\mO^{\sm}_{\SGTKp{T}}}\mO_{\SGTKp{T}}\cong \mcal{Y}\otimes_{\mO_{\Fl}}\mO_{\SGTKp{T}},\] 
 and \[i_{*}i^{*}\tVBfu(\mcal{Y})\otimes_{\mO^{\sm}_{\SGTKp{T}}}\mO_{\SGTKp{T}}\cong i_{*}i^{*}\mcal{Y}\otimes_{\mO_{\Fl}}\mO_{\SGTKp{T}},\] Note that \(\mcal{Y}\) is finite projective over \(\bigotimes_{\tau'\in T^{c}\backslash \tau}\pr_{\tau'}^{*}C^{\la}_{\tau'}\). Moreover, since \(C^{\la}_{\tau'}\) is locally over \(\Fl_{G_{T},\mu_{T},\tau'}\) of the form \(\varinjlim_{n}\mO_{\Fl_{G_{T},\mu_{T},\tau}}\langle t_{\tau}/p^{n} \rangle\),
 the natural map \(\mO_{\Fl}\to \bigotimes_{\tau'\in T^{c}\backslash \tau}\pr_{\tau'}^{*}C^{\la}_{\tau'}\) locally over \(\Fl\) admits a splitting. So
 \(i_{*}i^{*}\mcal{Y}\otimes_{\mO_{\Fl}}\mO_{\SGTKp{T}}\)
 is non-zero,
 and thus \(i_{*}i^{*}\tVBfu(\mcal{Y})\) is not zero. So we can conclude that the map \(\gr^{i}(N_{I_{0},\tau})\)
 is not zero.
\end{proof}
\begin{proof}[Proof of Theorem \ref{thmFontainEqDDbar}]
	Combining Proposition \ref{propUniquenessDRLinear}, Proposition \ref{propFontaineIsdRlinear}, 
	Remark \ref{rmkFontainDRLinNoTameLevel} and Proposition \ref{propFontaineNonZero}, we know that \(N_{I_{0},\tau}=c\cdot d_{\tau}^{-k_{\tau}+1}\circ \bar{d}_{\tau}^{-k_{\tau}+1}\) for \(c\in\CC_{p}^{\times}\). Since both morphisms are \(\Gal_{F_{T},\p_{T}}\)-equivariant, we know that \(c\in F_{T,\p_{T}}^{\times}\).
\end{proof}

\subsection{Bernstein-Gelfand-Gelfand-Fontaine complex}
\label{subsecFontaineComplex}
The construction of \cite[Definition \ref{1-dfnBGGFcomplex}]{Jiang2026Shla} specializes to the following complex in the quaternionic setting: 
\begin{dfn}\label{dfnFontainComplex}
	We fix a regular \(T^{c}\)-parallel multiweight \(\kw\). 
	Using \(N_{I_{0},\tau}\), 
	we can define the \emph{Bernstein-Gelfand-Gelfand-Fontaine complex} by 
	\begin{align*}
	&\mrm{Font}(\BdR^{\la,\lambda^{\kw}_{T^{c}},V_{T}^{\kw}}):=\\&\left[
	\mO^{\la,\kw}_{\SGTKp{T}}\to\bigoplus_{\tau\in T^{c}}\mO^{\la,(s_{\tau}\cdot\Bbbk,w)}_{\SGTKp{T}}\to \cdots \bigoplus_{|I_{0}|=i,I_{0}\subset T^{c}}\mO^{\la,(s_{I_{0}}\cdot\Bbbk,w)}_{\SGTKp{T}}\to \cdots \to \mO^{\la,(s_{\Sigma_{\infty}}\cdot\Bbbk,w)}_{\SGTKp{T}}
\right],
	\end{align*}
with transition maps given by \(N_{I_{0},\tau}\) 
for any \(I_{0}\subsetneq I_{0}\cup\{\tau\} \subset T^{c}\) and zero otherwise. 
By Theorem \ref{thmFontainEqDDbar}
	and the commutativity in Lemma \ref{lemBasicNablaDiffe}, we can and we do modify the transition maps by a scalar in \(F_{T,\p_{T}}^{\times}\) such that \(\mrm{Font}(\BdR^{\la,\lambda^{\kw}_{T^{c}},V_{T}^{\kw}})\) indeed forms a complex.
\end{dfn}
The following theorem gives an affirmative answer to \cite[Question \ref{1-conjFontaineComplexIsAutomorphic} (1)]{Jiang2026Shla} in the quaternionic case.
\begin{thm}\label{thmFontainComplexClassical}
	Assume that \(\kw\) is a regular multiweight, such that \(k_{\tau}\le 0\) for \(\tau\in\Sigma_{\infty}\). Then
	the complex
	 \(\mrm{Font}(\BdR^{\la,\lambda^{\kw}_{T^{c}},V_{T}^{\kw}})\) 
	is classical of weight \(\kw\) as Hecke modules (Definition \ref{dfnClassicalHecke}). 
\end{thm}
\begin{proof}
	This follows from Theorem \ref{thmClassicalityCohomoDR} and Theorem \ref{thmFontainEqDDbar} as follows: there is a distinguished triangle \[\Fib(d_{\tau}^{1-k_{\tau}})\to \Fib(\bar{d}_{\tau}^{1-k_{\tau}}\circ d_{\tau}^{1-k_{\tau}})\to \Fib(\bar{d}_{\tau}^{1-k_{\tau}}).\]
	Using this triangle repeatedly, and Theorem \ref{thmFontainEqDDbar},
	we know that \(\mrm{Font}(\BdR^{\la,\lambda_{T^{c}}^{\kw},V_{T}^{\kw}})\)
	is filtered by \(\dR(\nabla_{\SGTKp{T}}^{(\Bbbk',w'),(\Bbbk'',w'')})\)
	and thus has classical 
	global cohomology by Theorem \ref{thmClassicalityCohomoDR}.
\end{proof}

\section{Application to classicality}\label{sectionAppplicationClassicality}

In this section, as a global application, we show a classicality theorem for Hecke eigenclasses in the completed cohomology of Hilbert modular varieties (Theorem \ref{thmClassicalityCompleteCoho}). As a corollary, we also prove some cases of the Langlands-Clozel-Fontaine-Mazur conjecture for \( \GL_{2} \) over totally real fields (Theorem \ref{thmConjecFontaineMazur}).
In fact, it is not difficult to also prove some classicality results for other quaternionic Shimura varieties using Theorem \ref{thmFontainComplexClassical}, which we omit here for brevity.
\subsection{Classicality theorem}\label{subsec:ClassiThm}
For this section, we fix \( G:=\mrm{Res}_{F/\QQ}(\GL_{2}) \), and we consider the tower of Hilbert modular varieties \( \mrm{Sh}_{K}(G) \) as in \S \ref{subsectionQuaternUniSV} for \( T=\emptyset \). Let \( d:=[F:\QQ] \).
	We fix a neat tame level \(K^{p}\subset G(\mA_{f}^{p})\).
	Let \(S\) be a finite set of places of \(\QQ\),
	including \(p\) and \(\infty\)
	and all the ramified places of \(K^{p}\) or \(F \).
	Denote by \(\TT^{S}\) the spherical Hecke algebra over \(\ZZ_{p}\) generated by the spherical Hecke operators of \(G\) at finite places away from \(S\). 
\begin{dfn}\label{dfn:CompletedCohomo}
	We define the completed cohomology (resp. with compact support) defined by \cite{Emerton06} for the tower of the Shimura varieties \(\{\mS_{K^{p}K_{p}}(G)\}\) as 
	\begin{align*}
	R\Gamma(K^{p},L)&:=\Rlim_{n}\varinjlim_{K_{p}}R\Gamma_{\et}(\mS_{\Kpp}(G),\ZZ/p^{n}),
	\\
	(\text{resp. }R\Gamma_{c}(K^{p},L)&:=\Rlim_{n}\varinjlim_{K_{p}}R\Gamma_{c,\et}(\mS_{\Kpp}(G),\ZZ/p^{n})).
	\end{align*}
	We denote its \( i \)-th cohomology group by \( \ti{H}^{i}(K^{p},L) \) (resp. \( \ti{H}^{i}_{c}(K^{p},L) \)).
	It carries an action of \( \GL_{2}(F_{p})\times \TT^{S}\times \Gal_{\QQ} \). 
\end{dfn}
\begin{dfn}\label{dfnAssociatedGaloisRep}
Let \(\rho:\Gal_{F,S}\to \GL_{2}(L)\) be a continuous representation.
 We denote by \(\chi=\chi_{\rho}:\TT^{S}[1/p]\to L\)  the associated character determined by the Eichler-Shimura relation (\ref{alignEicherShimuraDetermin}), that is, \(\chi(S_{\mfk{l}}):=\mrm{Nm}(\mfk{l})^{-1}\det(\rho(\Fr_{\mfk{l}}))\) and \(\chi(T_{\mfk{l}}):=\mrm{Tr}(\Fr_{\mfk{l}})\).
\end{dfn}
\begin{thm}\label{thmClassicalityCompleteCoho}
	Let \(\rho:\Gal_{F,S}\to \GL_{2}(L)\) be a continuous representation,
	with the associated character \( \chi \) determined by the Eichler-Shimura relation (\ref{alignEicherShimuraDetermin}).

	We assume that 
	
	(1) \(\ti{H}^{d}(K^{p},L)[\chi]:=\Hom_{\TT^{S}}(\chi,\ti{H}^{d}(K^{p},L))\ne 0\);

	(2) \(\rho\) is absolutely irreducible;
	
	(3) There exist \(k\in \ZZ_{\le 0},w\in\ZZ,k\equiv w\pmod{2}\), such that for each \(\tau\in \Sigma_{\infty}\), \(\rho\) is \(\tau\)-de Rham, and has 
	\(\tau\)-Hodge-Tate weights \(\frac{k-w}{2}\) and \(\frac{2-k-w}{2}\).

	Then the character \(\chi\) is classical,
	i.e. there exists a Hilbert modular eigenform \(f\) of parallel weight \( (2-k,w) \),
	such that \(\TT^{S}\) acts on \(f\) via \(\chi\). In other words, 
	\(\rho\cong \rho_{f}\).
\end{thm}

Let us sketch here the idea of the proof of Theorem \ref{thmClassicalityCompleteCoho}, which follows closely the idea of \cite{Pan2209.06II}. The key ingredient is Theorem \ref{thmFontainComplexClassical}. To apply the theorem, we need to show that the de Rhamness condition
will force the Hecke eigenvalue to appear in the cohomology of the Bernstein-Gelfand-Gelfand-Fontaine complex. 

There are some subtleties in the implementation of the strategy:

(1)
Before we can reach the Bernstein-Gelfand-Gelfand-Fontaine complex, we need to determine the infinitesimal character of the Hecke eigenspace. 

(2) We expect to see the tensor induction of the Galois representation rather than the Galois representation itself in the completed cohomology. Hence an additional semisimplicity result is required in the case where \( \oIndFQ\rho \) is reducible.

Both (1) and (2) (in the strongly irreducible case) will be solved using the plectic action constructed in a joint work with Lue Pan (\cite{JiangPan2025Plectic}). We will summarize the relevant result in \S \ref{subsectionPlecticGaloisAction}. We finish the proof of the main theorem (Theorem \ref{thmClassicalityCompleteCoho}) in \S \ref{subsecProofMainThm}. 

\begin{rmk}
For (1), a result about infinitesimal character has been obtained by 
\cite{DospinescuPaskunasSchaen2020infinitesimalFamily} under weakly non-Eisenstein assumption. For (2), the semisimplicity could also be proven by adapting the argument of \cite{Nekovar2018eichler}. On the other hand, our method in \cite{JiangPan2025Plectic} has no assumption on the reduction \( \bar{\rho} \), and establishes an isomorphism \[
\ti{H}^{d}(K^{p},L)^{\la}[\chi]\cong \oIndFQ\rho\otimes \Pi^{\la}(\rho)
\] for some locally analytic representation \( \Pi^{\la}(\rho) \). 
\end{rmk}




\subsection{Plectic Galois action}\label{subsectionPlecticGaloisAction}
We will need the following result in a joint work with Lue Pan \cite{JiangPan2025Plectic} on a partial construction of the plectic Galois action conjectured in \cite{NekovarScholl2017Plectic}. See the introduction of \cite{JiangPan2025Plectic} for a more detailed exposition on the plectic conjecture. 

Recall that we have fixed \(\tau_{0}:F\hookrightarrow \bar{\QQ}\subset \CC\cong^{\iota} \bar{\QQ}_{p}\), and \( \Sigma_{\infty} \)
denotes the set of archimedean places of \( F \). 
Using \( \tau_{0} \),
we have a bijection \( \Sigma_{\infty}\cong \Gal_{\QQ}/\Gal_{F} \). Here we identify \( \Gal_{F}:=\mrm{Aut}(\QQ/\tau_{0}(F)) \).
\begin{dfn}\label{dfnPlecGalois}
(1)
We define the \emph{plectic Galois group} of \( F \)
over \( \QQ \) as \[
\Gal^{plec}_{F}:=\mrm{Aut}(F\otimes_{\QQ}\bar{\QQ}/F).
\]
Note that \( \Gal_{\QQ} \) is canonically a subgroup of \( \Gal^{plec}_{F} \) by acting on the factor \( \bar{\QQ} \).
Using the isomorphism \( F\otimes_{\QQ}\bar{\QQ}\cong \prod_{\tau\in\Sigma_{\infty}}\bar{\QQ} \),
we have a canonical exact sequence \[
1\to \prod_{\tau\in\Sigma}\mrm{Aut}(\bar{\QQ}/\tau(F))\to \Gal^{plec}_{F}\to \mrm{Aut}(\Sigma_{\infty})\to 1.
\]
We denote by \( \Gal_{F}^{plec,\circ} \)
the kernel of \( \Gal_{F}^{plec}\to \mrm{Aut}(\Sigma_{\infty}) \).


(2)
Let \( \bar{\QQ}^{S-\ur} \) be the maximal subfield of \( \bar{\QQ} \) that is unramified over \( \QQ \) away from \( S \), and then \( \tau \) factors through \( \bar{\QQ}^{S-\ur} \) for any \( \tau\in\Sigma_{\infty} \). We define \( \Gal_{F,S}:=\mrm{Aut}(\bar{\QQ}^{S-\ur}/\tau_{0}(F))\subset \Gal_{\QQ,S} \), and \(
\Gal_{F,S}^{plec}:=\mrm{Aut}(F\otimes_{\QQ}\bar{\QQ}^{S-\ur}/F).
\) Then \( \Gal_{\QQ,S}\subset \Gal_{F,S}^{plec} \), and we have a canonical exact sequence
\[
1\to \prod_{\tau\in\Sigma_{\infty}}\Gal(\bar{\QQ}^{S-\ur}/\tau(F))\to \Gal_{F,S}^{plec}\to \mrm{Aut}(\Sigma_{\infty})\to 1.
\] After fixing a set of representative of \( \Gal_{\QQ,S}/\Gal_{F,S} \) in \( \Gal_{\QQ,S} \), we have an isomorphism 
\( \Gal_{F,S}^{plec}\cong \Gal_{F,S}^{\Sigma_{\infty}}\rtimes \mrm{Aut}(\Sigma_{\infty}) \).
\end{dfn}


\begin{conj}[{\cite[Conjecture 6.1]{NekovarScholl2017Plectic}}]
There exists a ``natural'' action of \( \Gal_{F,S}^{plec} \) extending that of \( \Gal_{\QQ,S} \) to an  on \( 
R\Gamma_{\et}(\SGTK{}{\Kpp}_{\bar{\QQ}},\ZZ/p^{n}) \).
\end{conj}
Let us elaborate more on the conjectured property of being ``natural''. We fix \( S \) and \( K^{p} \) as in \S \ref{subsec:ClassiThm}. 
In Definition \ref{dfnPadcHeckeAlg} and Theorem \ref{thmConstrDeterminantHilbert}, we construct a \( p \)-adic Hecke algebra \( \hat{\TT}^{S} \) acting on \( R\Gamma_{\et}(\SGTK{}{K}_{\bar{\QQ}},\ZZ/p^{n}) \)
and a
universal determinant \( D:\ZZ_{p,\square}[\Gal_{F,S}]\to \hat{\TT}^{S} \) (in the sense of \cite{Chenevier2014determinants}),
such that for any \( \ell\notin S \)
and \( \mfk{l}|\ell \),
\[
D(t-\Fr_{\mfk{l}})=t^{2}-T_{\mfk{l}}\cdot t+\Nm(\mfk{l})S_{\mfk{l}}.
\]

Recall from Definition \ref{dfnPadcHeckeAlg} that \( \hat{\TT}^{S} \)
has only finitely many maximal open ideals \( \{\mm_{i}\}_{i=1}^{m} \), and we have \( D/\mm_{i}:\FF_{i,\square}[\Gal_{F,S}]\to \FF_{i} \) for \( \FF_{i}:=\hat{\TT}^{S}/\mm_{i} \).
We enlarge the field of coefficient \( L \) such that \( W(\FF_{i})\subset L \) for all \( i \). 
\begin{notation}\label{notationPseudoDefor}
We define \( R_{i}^{ps} \) to be the universal deformation ring of \( D/\mm_{i} \) over \( W(\FF_{i}) \), which is a Noetherian local ring with a maximal ideal \( \mm^{ps}_{i} \).
 We can find a surjection \( W(\FF_{i})[[x_{1},\ldots,x_{m}]]\twoheadrightarrow R_{i}^{ps} \).

We define an analytic ring structure \( R_{i,\square}^{ps} \) on \( R_{i}^{ps} \)
to be the one induced from \( \ZZ_{\square} \). 
Concretely, by \cite[Proposition 6.3]{CS19},
for any profinite set \( S=\varprojlim_{j}S_{j} \) for \( S_{j} \) finite sets,
\[
R_{i,\square}^{ps}[S]\cong \varprojlim_{n,j}R_{i}^{ps}/(\mm^{ps}_{i})^{n}[S_{j}].
\]

We further define \( R^{ps}:=\prod_{i=1}^{m}R_{i}^{ps} \), and \( R^{ps}_{\square}:=\prod_{i=1}^{m}R^{ps}_{i,\square} \). 
\end{notation}
\begin{notation}
Denote by \( D^{u} \) the universal determinant \( D^{u}:R^{ps}_{\square}[\Gal_{F,S}]\to R^{ps} \). 
In the joint work with Lue Pan (\cite[Propsition 3.9]{JiangPan2025Plectic}), we constructed from \( D^{u} \) a \( 2^{d} \)-dimensional determinant \[
D^{plec}:R^{ps}_{\square}[\Gal_{F,S}^{plec}][1/p]\to R^{ps}[1/p],
\] which is the family version of the plectic induction \( \rho\mapsto \Ind_{F}^{plec}\rho \) (\cite[Definition 3.2]{JiangPan2025Plectic}), whose restriction to \( \Gal_{\QQ,S}\subset \Gal_{F,S}^{plec} \) is the usual tensor induction \( \oIndFQ\rho \). 
\end{notation}

\begin{dfn}[{\cite[\S 1.17]{Chenevier2014determinants}}]
Let \( D:R\to A \) be a continuous \( d \)-dimensional determinant, where both \( R \)
and \( A \) are equipped with profinite topology. 
For any \( B\in\mrm{CAlg}_{A} \) and \( r\in R\otimes_{A}B \), we denote its \emph{characteristic polynomial} (in \(x\)) as \(D(x\cdot 1-r)=:\sum_{i=0}^{d(D)}(-1)^{i}\wedge_{i}(r)\cdot x^{d-i},\;\wedge_{i}(r)\in B.\)

We define the \emph{Cayley-Hamilton ideal} \( \CH(D) \)
of \( D \) to be the closed 2-sided ideal of \( R \)
generated by the coefficients of the polynomials
 in $t_{1},\ldots,t_{n}$ 
of the form 
 $$\sum_{i=0}^{d(D)}(-1)^{i}\wedge_{i}(r)\cdot r^{d-i}$$ for $r=\sum_{j=1}^{n}r_{j}t_{j}$,
  for varying \( n\in\NN \) and $r_{j}\in R$ for \( j=1,\ldots,n \). 
\end{dfn}

Given the existence of \( D \) (Theorem \ref{thmConstrDeterminantHilbert}),
we have a natural morphism \( R_{i}^{ps}\to \hat{\TT}^{S}_{\mm_{i}} \), and \( R^{ps}\to \hat{\TT}^{S} \). Then we have an action of \( R^{ps}_{\square}[\Gal_{\QQ,S}] \)
on \( R\Gamma_{\et}(\SGTK{}{\Kpp},\ZZ/p^{n}) \), and thus an action of \( R^{ps}_{\square}[\Gal_{\QQ,S}][1/p] \) on \( R\Gamma(K^{p},\QQ_{p}) \).
By being ``natural'', we expect to extend the action of \( R^{ps}_{\square}[\Gal_{\QQ,S}][1/p] \) to that of \( R^{ps}_{\square}[\Gal_{F,S}^{plec}][1/p] \), which further factors through \( R^{ps}_{\square}[\Gal_{F,S}^{plec}][1/p]/\CH(D^{plec}) \).

\begin{dfn}[Partial Sen operator, {\cite[Definition 3.14]{JiangPan2025Plectic}}]\label{dfnPartialSen}


For \( \tau\in\Sigma_{\infty} \), the composition \( \iota\circ \tau \) determines a place \( \tau_{p} \) of \( F \) over \( p \) with an embedding \( \tau:\tau(F)_{\tau_{p}}\hookrightarrow \bar{\QQ}_{p} \). 

Given a continuous representation \( (\rho^{plec},V^{plec}):\Gal_{F}^{plec}\to \GL_{n}(L) \), we consider the restriction to \( \Gal(\bar{\QQ}_{p}/\tau(F)_{\tau_{p}})\subset \Gal(\bar{\QQ}/\tau(F)) \), 
which define by Sen theory a Sen operator \( \Theta_{\tau}\in \End(V^{plec})\otimes_{\QQ_{p}}\CC_{p} \), 
which we refer to as the \emph{\( \tau \)-partial Sen operator} of \( (\rho^{plec},V^{plec}) \).
\end{dfn}
By the proof of \cite[Proposition 3.18]{JiangPan2025Plectic}, if \( \rho:\Gal_{F,S}\to \GL_{2}(L) \) is \( \tau \)-Hodge-Tate of weights \( a_{\tau},b_{\tau} \), then the \( \tau \)-partial Sen operator \( \Theta_{\tau} \)
of \( \Ind^{plec}\rho \) is semisimple with eigenvalues \( -a_{\tau},-b_{\tau} \). 

\begin{dfn}
Let \( \Gamma \) be a profinite group.
We say that a continuous representation \( \rho:\Gamma\to \GL_{n}(L) \) is \emph{strongly irreducible} if for any open subgroup \( U\subset \Gamma \), \( \rho|_{U} \)
is absolutely irreducible.
\end{dfn}


\begin{thm}[\cite{JiangPan2025Plectic}]\label{thmPadicPlectic}
	Assume that we are given \(\chi:R^{ps}[1/p]\to L \) that corresponds to a \emph{strongly irreducible} representation \( \rho:\Gal_{F,S}\to \GL_{2}(L) \),
	such that 
	\( \rho \)
	has distinguished \( \tau \)-Hodge-Tate weights for some \( \tau\in\Sigma_{\infty} \).
	Let \(\p_{\chi}:=\Ker(\chi)\).

	(1) We have an isomorphism \[
	R\Gamma(K^{p},\QQ_{p})_{\p_{\chi}}\cong \ti{H}^{d}(K^{p},\QQ_{p})_{\p_{\chi}}[-d].
	\] and
	there exists an action of \[ (R^{ps}_{\square}[\Gal_{F,S}^{plec}][1/p]/\CH(D^{plec}))_{\p_{\chi}}\cong M_{2^{d}}(R^{ps}[1/p]_{\p_{\chi}}) \] on
	\( \ti{H}^{d}(K^{p},\QQ_{p})^{\la}_{\p_{\chi}} \) extending that of \( R^{ps}_{\square}[\Gal_{\QQ,S}] \).


	(2) There is a \( L_{\square}[\Gal_{F,S}^{plec}]/\CH(\Ind^{plec}_{F}\rho)\times \GL_{2}(F_{p}) \)-equivariant isomorphism \[
	\RHom_{R^{ps}}(\chi,\ti{H}^{d}(K^{p},L)^{\la})\cong \Ind_{F}^{plec}\rho\otimes \Pi^{\la}(\rho)
	\] for some (possibly derived) locally analytic representation \( \Pi^{\la}(\rho) \) of \( \GL_{2}(F_{p}) \).
	
	Moreover,
	along the isomorphism, after base change to \( \CC_{p} \), the action of the \( \tau \)-partial Sen operator \( \Theta_{\tau} \) on the RHS is induced by \( -\theta^{Pan}(d\mu_{\tau})=\theta(d\mu_{\tau}) \) (\cite[Definition \ref{1-dfnHorActionPan} \& \ref{1-dfnHorActionPilloni}]{Jiang2026Shla}) on the LHS. Here \( d\mu_{\tau}\in\mrm{Lie}(\mm_{\mu_{\emptyset}})\cong \bigoplus_{\tau\in\Sigma_{\infty}}(\CC_{p}^{\oplus 2}) \) is \( (1,0) \)
	at the \( \tau \)-component, and \( (0,0) \) at other components.
\end{thm}
\begin{proof}
This is the content of \cite[Theorem 1.2]{JiangPan2025Plectic}. The compatibility with partial Sen operators is established in \cite[Theorem 4.35]{JiangPan2025Plectic}. Note that in terms of \( \theta^{Pan} \), we are considering \( (-1,0) \) here, instead of \( (0,1) \) in \cite{JiangPan2025Plectic}, due to the homological convention in this paper.
\end{proof}
Conjecturally, the statement should hold assuming only the absolute irreducibility of \( \rho \) rather than the strong irreducibility. However, the strong irreducibility assumption does not cause any trouble for our proof of classicality, thanks to the following dichotomy between the strongly irreducible case and the CM case:
\begin{lem}\label{lemCMorStrongIrr}
Let 
\(\rho:\Gal_{F,S}\to \GL_{2}(L)\) be an absolutely irreducible continuous Galois representation. Assume that 
there exists \( \tau \) such that the \( \tau \)-Hodge-Tate weights are distinct. Then
\begin{itemize}
\item either \( \rho \) is strongly irreducible;
\item or \( \rho \) has multiplication, that is, there exists a quadratic extension \( E/F \) that is unramified away from \( S \), and a continuous character \( \eta:\Gal_{E,S}\to \GL_{1}(L) \), such that \( \rho\cong \Ind_{\Gal_{F,S}}^{\Gal_{E,S}}(\eta) \).
\end{itemize}
If in addition \( \rho \) is de Rham of regular \( \tau \)-Hodge-Tate weights for \emph{all} \( \tau \), 
then in the second case, \( E \) has to be a CM extension, that is, \( \rho \) has complex multiplication.
\end{lem}
\begin{rmk}
By class field theory, for fixed \( S \), there are only finitely many such quadratic extensions.
\end{rmk}
\begin{proof}
Assume that \( \rho \) is not strongly irreducible. Then there exists an open normal subgroup \( U\subset \Gal_{F,S} \), such that \( \rho|_{U} \) is reducible. Consider the Lie algebra \( \Lie(\rho(U))\subset \mfk{gl}_{2,L} \). Since \( \rho \) is reducible as a representation of \( \Lie(\rho(U)) \), \( \Lie(\rho(U)) \) must be contained in a Borel, which we conjugate (up to extending \( L \)) to be the sub-Lie algebra of upper-triangular matrices \( \fb_{L}\subset \mfk{gl}_{2,L} \). 

We claim that \( \Lie(\rho(U)) \) actually factors through the torus; otherwise, the normalizer of \( \Lie(\rho(U)) \) in \( \GL_{2}(L) \) will also be contained in the subgroup of upper-triangular matrices, and thus \( \rho \) cannot be irreducible. 
Since there exists \( \tau \) such that the \( \tau \)-Hodge-Tate weights are regular, by \cite{Sen1973lie}, we know that the map from \( \Lie(\rho(U')) \) to the maximal torus of \( \Lie(\PGL_{2}) \) is surjective. Therefore,
\( \rho(\Gal_{F,S}) \) is contained in the normalizer of the Cartan subalgebra of \( \mfk{sl}_{2,L} \), that is, \[
\rho(\Gal_{F,S})\subset T(L)\cup w_{0}T(L),
\] where \( T \) is the maximal torus of \( \GL_{2} \) of diagonal matrices, and \( w_{0}:=\begin{pmatrix}
0&1\\1&0
\end{pmatrix} \).
Let \( U':=\rho^{-1}(T)\subset \Gal_{F,S} \) denote the pre-image of the torus of diagonal matrices along \( \rho \). 

Since \( \rho \) is irreducible, \( U'\ne \Gal_{F,S} \). So
\( U' \) is an open subgroup of index \( 2 \), and thus corresponds to a quadratic extension \( E/F \) unramified away from \( S \). The restriction  \( \rho|_{U'} \) decomposes as \( \eta_{1}\oplus \eta_{2} \), 
for \( \eta_{i}:U'=\Gal_{E,S}\to \GL_{1}(L) \), and by adjunction, we have a non-zero map \( \Ind_{\Gal_{E,S}}^{\Gal_{F,S}}(\eta_{1})\to \rho \), which has to be an isomorphism since \( \rho \) is irreducible.

Now if we further assume that \( \rho \) is de Rham of regular \( \tau \)-Hodge-Tate weights for \emph{all} \( \tau \), then 
we conclude by Lemma \ref{lem:characterCMfield} below.
\end{proof}
\begin{lem}\label{lem:characterCMfield}
Let \( E \) be a number field that is unramified away from \( S \), and let \( \eta:\Gal_{E,S}\to \GL_{1}(L) \) be a continuous de Rham representation. 
For any \( \tau:E\hookrightarrow \CC\cong^{\iota}\bar{\QQ}_{p} \), denote by \( -a_{\tau}\in\ZZ \) the \( \tau \)-Hodge-Tate weight of \( \eta \).
Let \( \Sigma_{\infty,E} \)
denote the set of archimedean places of \( E \).
For \( \nu\in \Sigma_{\infty,E} \), we choose an embedding \( \tau_{\nu}:E_{\nu}\hookrightarrow \CC \).

Then there exists \( w\in \ZZ \), such that for any \( \nu\in\Sigma_{\infty,E} \), \( a_{\tau_{v}}=w \) if \( E_{v}\cong \RR \), and \( a_{\tau_{v}}+a_{\tau_{v}^{c}}=w \) if \( E_{v}\cong \CC \).
\end{lem}
\begin{proof}
We enlarge \( L \) such that it contains all the embedding \( E\hookrightarrow \bar{\QQ}_{p} \). 
By class field theory, \( \eta \)
corresponds to a continuous character \[
\eta:E^{\times}\backslash \mA_{E}^{\times}/
\left(\prod_{\ell\ne p}U_{\ell}\times (E\otimes_{\QQ} \RR)^{\times,\circ}\right)\to L^{\times},
\]  where for \( \ell\ne p \), \( U_{\ell} \)
is an open subgroup of \((\ZZ_{\ell}\otimes \mO_{E})^{\times} \), and \( U_{\ell}=(\ZZ_{\ell}\otimes \mO_{E})^{\times} \) if \( \ell\notin S \).
Using the finiteness of the class number, the LHS contains \( (\ZZ_{p}\otimes \mO_{E})^{\times}/\mO_{E,U^{p}}^{\times} \) as a finite index open subgroup, where \(
\mO_{E,U^{p}}^{\times}:=\{x\in \mO_{E}^{\times}:x\in U_{\ell},\forall \ell\ne p\}\subset \mO_{E}^{\times}.
\)
Since \( \eta \) is de Rham, the character \( \eta \) is locally algebraic, that is, there exists an open subgroup \( U_{p}\subset (\ZZ_{p}\otimes \mO_{E})^{\times} \), 
such that \[
\eta(x)=\prod_{\tau:E\hookrightarrow L}\tau(x)^{a_{\tau}},\forall x\in U_{p}.
\] Let \(
\mO_{E,U}^{\times}:=\{x\in \mO_{E}^{\times}:x\in U_{\ell},\forall \ell\}\subset \mO_{E,U^{p}}^{\times}.
\)
Then \(x\mapsto 
\prod_{\tau}\tau(x)^{a_{\tau}}
\) vanishes on \( 
\mO_{E,U}^{\times} \). 

Using \( \iota:\bar{\QQ}_{p}\cong \CC \), we know that \(x\mapsto 
\prod_{\tau:E\hookrightarrow \CC}\tau(x)^{a_{\tau}}
\) vanishes on \( 
\mO_{E,U}^{\times} \). 
We consider the morphism \[
\ti{\eta}:
(\RR^{\times})^{\Sigma_{\infty,E}}\to R^{\times},\;
((x_{\nu})_{E_{\nu}\cong \CC},
(x_{\nu})_{E_{\nu}\cong \RR})_{\nu\in\Sigma_{\infty,E}}\mapsto \prod_{E_{\nu}\cong \CC}x_{\nu}^{a_{\tau_{\nu}}+a_{\tau_{\nu}^{c}}}\cdot \prod_{E_{\nu}\cong \RR}x_{\nu}^{a_{\tau_{\nu}}}.
\]
By the argument above, \( \ti{\eta}|_{\mO^{\times}_{E,U}}=1 \), and thus it is also trivial on the Zariski closure of \( \mO^{\times}_{E,U} \).
On the other hand, by Minkowski's theorem, \( \mO_{E}^{\times} \)
is a lattice in \( (\RR^{\times})^{\Sigma_{\infty,E},\prod=1} \), and  \( \mO_{E,U}^{\times} \) is a subgroup of \( \mO_{E}^{\times} \) of finite index. Thus the Zariski closure of \( \mO^{\times}_{E,U} \) is \( (\RR^{\times})^{\Sigma_{\infty,E},\prod=1} \). Hence, there exists \( w\in \ZZ \), such that \( a_{\tau_{v}}=w \) for \( E_{v}\cong \RR \), and \( a_{\tau_{v}}+a_{\tau_{v}^{c}}=w \) for \( E_{v}\cong \CC \).
\end{proof}
In the strongly irreducible
case, we can fix the infinitesimal character on the Hecke eigenspace using Theorem \ref{thmPadicPlectic}. 
\begin{prop}\label{propFixInfCharc}
	Let \(\rho:\Gal_{F,S}\to \GL_{2}(L)\) be a continuous strongly irreducible representation, with \(\chi:\TT^{S}[1/p]\to L\) the associated character determined by the Eichler-Shimura relation (\ref{alignEicherShimuraDetermin}).
	We assume that \( \rho \) has distinct \( \tau \)-Hodge-Tate weights for all \( \tau \). 

	Then there exists a regular multiweight  \(\kw\), \(\rho\) is \(\tau\)-de Rham of \( \tau \)-Hodge-Tate weights \(\frac{k_{\tau}-w}{2}\) and \(\frac{2-k_{\tau}-w}{2}\) for all \( \tau\in\Sigma_{\infty} \), and \( Z(\fg) \) acts on \(
	\ti{H}^{d}(K^{p},L)^{\la}[\chi]
	\) via \( \lambda^{\kw} \).
\end{prop}
\begin{proof}
Denote \(\mfk{z}:=\Lie(Z(G))\otimes L\cong \bigoplus_{\tau\in \Sigma_{\infty}}\mfk{z}_{\tau}\) with \(\mfk{z}_{\tau}\cong L\). Denote by \(\mrm{Z}_{\tau}\) the generator of \(\mfk{z}_{\tau}\),
	given by \(\mrm{diag}\{1,1\}\)
	on the \(\tau\)-th component and zero on the other components.
	\begin{lem}\label{lem:smallZaction}
		For each \(\tau\in\Sigma_{\infty}\),
		the action of \(\mrm{Z}_{\tau}\)
		on \(\ti{H}^{d}(K^{p},L)^{\la}[\chi]\) is given by the constant \(-w\). 
	\end{lem}
	\begin{proof}[Proof of the lemma]
	By assumption (2), \( \det(\rho) \) is de Rham. By Lemma \ref{lem:characterCMfield}, there exists \( w\in \ZZ \) such that \( \det(\rho) \)
	is parallel of Hodge-Tate weight \( 1-w \).
	The action of \( Z_{\tau} \) is given by \cite[Theorem 4.35]{JiangPan2025Plectic}.
	\end{proof}

Now since \(\rho\)
is de Rham, and 
\(\det(\rho)\) is of weight \(w-1\),
we can find \(k_{\tau}\in\ZZ_{\le 0}\) for \(\tau\in \Sigma_{\infty}\),
such that \(\rho\)
is Hodge-Tate of \(\tau\)-weights \(\frac{k_{\tau}-w}{2}\) and \(\frac{2-k_{\tau}-w}{2}\) for \(\tau\in \Sigma_{\infty}\).
We now claim
the tuple \(((k_{\tau})_{\tau\in\Sigma_{\infty}},w)\) meets the requirement. 
By Theorem \ref{thmPadicPlectic},
\(\theta(d\mu_{\tau})\)
acts with eigenvalues
\(\frac{w-k_{\tau}}{2}\)
and \(\frac{w+k_{\tau}}{2}-1\).
By Lemma \ref{lem:smallZaction}, \(\theta(\mrm{Z}_{\tau})=w\). Thus
we know that \(\theta(\mrm{H}_{\tau})\) also 
acts semisimply with eigenvalues \(-k_{\tau}\)
or \(k_{\tau}-2\).
Thus
\(\Omega_{\tau}\) acts via \[\Omega_{\tau}=\frac{1}{2}\theta(\mrm{H}_{\tau})^{2}+\theta(\mrm{H}_{\tau})=\frac{1}{2}k_{\tau}^{2}-k_{\tau}=\lambda^{\kw}(\Omega_{\tau}),\]
as desired.
\end{proof}
\subsection{Proof of the classicality theorem}\label{subsecProofMainThm}
We now start to prove Theorem \ref{thmClassicalityCompleteCoho}. The key input is Theorem \ref{thmFontainComplexClassical}, which is related to the completed cohomology via the following primitive comparison of \cite{Scholze13}.
\begin{prop}\label{propPrimitiveComparisonBdR}
	There exists a canonical isomorphism 
	\[
	R\Gamma(K^{p},\QQ_{p})^{R-\la}\hatotimes_{\QQ_{p}} B_{\dR}^{+}/t^{n}\cong R\Gamma_{\an}(\mX^{\tor}_{K^{p}},\BdR^{+,\la}/t^{n}).
	\]
\end{prop}
\begin{proof}
	There exists a natural map that is compatible with the \(t\)-adic filtration on both sides.
	Then it suffices to check the isomorphism on the level of graded pieces, which reduces to \cite[Theorem 6.2.6]{Juan2022.09locallyShi}
\end{proof}


\begin{proof}[Proof of Theorem \ref{thmClassicalityCompleteCoho}]
The case where \(\rho\) admits CM can be reduced to global class field theory. Therefore, by Lemma \ref{lemCMorStrongIrr}, we assume without loss of generality that \(\rho\) is strongly irreducible.

We will denote by the subscript \( L \) the base change \( -\otimes_{\QQ_{p}}L \).
Let \[
-[\chi]:=\RHom_{R^{ps}}(\chi,-),\;(-)^{\lambda^{\kw}}:=\RHom_{Z\fg}(\lambda^{\kw},-).
\] 
Since \( R^{ps} \) is Noetherian, \( \chi \)
is pseudo-coherent as a \( R^{ps} \)-module (\cite[Tag 066E]{stacks-project}), and thus \( -[\chi] \)
commutes with colimits when restricted to uniformly left bounded objects. Let \( \p_{\chi}:=\Ker(R^{ps}[1/p]\to L) \).
Then by Theorem \ref{thmPadicPlectic}, 
\begin{align*}
R\Gamma(K^{p},L)^{R-\la,\lambda^{\kw}}[\chi]
\cong (R\Gamma(K^{p},L)^{R-\la,\lambda^{\kw}}[\chi])_{\p_{\chi}}
&\cong ((R\Gamma(K^{p},L)^{R-\la})_{\p_{\chi}})^{\lambda^{\kw}}[\chi]\\
&\cong (\ti{H}^{d}(K^{p},L)^{\la}[\chi])^{\lambda^{\kw}}.
\end{align*}

By primitive comparison (Proposition \ref{propPrimitiveComparisonBdR}),
we have that 
\begin{align*}
R\Gamma_{\an}(\SGTKpTo{},\BdR^{\la,\lambda^{\kw}})_{L}[\chi]&\cong (\ti{H}^{d}(K^{p},L)^{\la}[\chi])^{\lambda^{\kw}}
\hatotimes_{\QQ_{p}}B_{\dR}.
\end{align*}
On the 
other hand, we know by Theorem \ref{thmPadicPlectic} that 
\[
	\ti{H}^{d}(K^{p},L)^{\la}[\chi]\cong \oInd_{F}^{\QQ}\rho\otimes \Pi^{\la}(\rho).\]
Since \(\rho\) is de Rham, \(\oInd^{\QQ}_{F}\rho\)
is also de Rham.
Thus by Lemma \ref{lem:Fontaine0de Rham},
the morphism \(R\Gamma_{\an}(\SGTKpTo{},N_{I_{0},\tau})[\chi]\) is homotopic to \( 0 \), for \(N_{I_{0},\tau}\) as in Notation \ref{dfnFontaineOpe}.

Let us assume by contradiction that \(\chi\) is not associated to a classical Hilbert modular form of weight \((\max(2-\Bbbk,\Bbbk),w)\).
By Theorem \ref{thmFontainComplexClassical}, 
we know that \[R\Gamma_{\an}(\SGTKpTo{},\mrm{Font}(\BdR^{\la,\lambda^{\kw}}))_{L}[\chi]\cong 0.\]
However, by the definition of Bernstein-Gelfand-Gelfand-Fontaine complex (Definition \ref{dfnFontainComplex}), \begin{align*}
&\mrm{Font}(\BdR^{\la,\lambda^{\kw}}):= \\&\left[
	\mO^{\la,\kw}_{\SGTKpTo{}}\to\bigoplus_{\tau\in \Sigma_{\infty}}\mO^{\la,(s_{\tau}\cdot\Bbbk,w)}_{\SGTKpTo{}}\to \cdots \bigoplus_{|I_{0}|=i,I_{0}\subset \Sigma_{\infty}}\mO^{\la,(s_{I_{0}}\cdot\Bbbk,w)}_{\SGTKpTo{}}\to \cdots \to \mO^{\la,(s_{\Sigma_{\infty}}\cdot\Bbbk,w)}_{\SGTKpTo{}}\right],
\end{align*}
where the differentials are defined using \(N_{I_{0},\tau}\). 

We have showed that
the differentials are homotopic to zero after taking \(R\Gamma_{\an}(\SGTKpTo{},-)_{L}[\chi]\). 
As a result,
we know that \[R\Gamma_{\an}\left(\SGTKpTo{},\mrm{Font}(\BdR^{\la,\lambda^{\kw}})\right)_{L}[\chi]\cong 
\bigoplus_{I_{0}\subset \Sigma_{\infty}}R\Gamma_{\an}\left(\SGTKpTo{},\mO^{\la,(s_{I_{0}}\cdot \Bbbk,w)}_{\SGTKpTo{}}\right)_{L}[\chi][-\#I_{0}]\cong 0.
\]
By Proposition \ref{propInfiniteCharToHorizontal},  \(\mO^{\la,\lambda^{\kw}}_{\SGTKpTo{}}\cong \bigoplus_{s\in W}\mO_{\SGTKpTo{}}^{\la,(s\cdot \Bbbk,w)},\) and thus
\[R\Gamma_{\an}(\SGTKpTo{},\mO^{\la,\lambda^{\kw}}_{\SGTKpTo{}})_{L}[\chi]\cong 0.\]
However, by primitive comparison (by \cite[Theorem 6.2.6]{Juan2022.09locallyShi} or by taking \(n=1\) in Proposition \ref{propPrimitiveComparisonBdR}),
\[R\Gamma(\SGTKpTo{},\mO^{\la,\lambda^{\kw}}_{\SGTKpTo{}})_{L}[\chi]\cong R\Gamma(\SGTKp{T},L)^{R-\la,\lambda^{\kw}}[\chi]\otimes_{\QQ_{p}}\CC_{p},\]
and the RHS is non-zero by Proposition \ref{propFixInfCharc}, which gives a contradiction.
\end{proof}

\subsection{Langlands-Clozel-Fontaine-Mazur conjecture}
In \cite{BreuilHerzigHuMorraSchraen2023GKdim}, the authors prove the following remarkable theorem:
\begin{thm}[{\cite[Corollary 8.5.1]{BreuilHerzigHuMorraSchraen2023GKdim}}]\label{thmBHHMSFlatness}
Assume that \(p\) is inert in \(F\). Let \(D/F\) be a quaternion algebra which is split at all places over \(p\)
and at precisely \(d_{D}\) infinite place of \(F\) for \(d_{D}\in\{0,1\}\). Let \(\bar{\rho}:\Gal_{F}\to \GL_{2}(\FF)\) for a finite field \(\FF/\FF_{p}\). Let \(\psi:\Gal_{F}\to W(\FF)^{\times}\) be the Teichmüller lift of \(\chi_{\mrm{cycl}}\det(\bar{\rho})\). 
Let \(S_{D}\) (resp. \(S_{\bar{\rho}}\)) denote the set of finite places of \(F\) where \(D\) (resp. \(\bar{\rho}\)) ramifies. Let \(S\) and \(w_{1}\)
be as in \cite[100 (ii)-(iii)]{BreuilHerzigHuMorraSchraen2023GKdim}. 

Let \(\{\mrm{Sh}_{V}(G_{D})\}_{V\subset G_{D}(\mA_{f})}\) be the tower of Shimura varieties (of dimension \(d_{D}\)) associated to \(G_{D}:=\mrm{Res}_{F/\QQ}D^{\times}\).
We fix \(V^{p}\subset G_{D}(\mA_{f}^{p})\) that is unramified away from \(S\cup\{w_{1}\}\), and let \(\ti{H}^{i}(\mrm{Sh}_{V^{p}}(G_{D}),C)\)
be the completed cohomology of the tower defined in \cite[Definition \ref{1-dfnCompleteCoh}]{Jiang2026Shla} for any \(C/\ZZ_{p}\). 
Let \(\TT^{S\cup\{w_{1}\}}\) denote the spherical Hecke algebra over \(W(\FF)\), and let \(\mm\subset \TT^{S}\)
be the kernel of the character \(\TT^{S\cup\{w_{1}\}}\to \FF\) determined by \(\bar{\rho}\) via
(\ref{alignEicherShimuraDetermin}). 
Assume that \begin{align}\label{alignResiduModularity}
\ti{H}^{d_{D}}(\mrm{Sh}_{V^{p}}(G_{D}),\FF)[\mm]\ne 0.
\end{align}

Let \(\hat{\TT}^{\psi}_{\bar{\rho}}\)
denote the ``big'' spherical Hecke algebra acting on \(\ti{H}^{d_{D}}(\mrm{Sh}_{V^{p}}(G_{D}),W(\FF)/p^{n})^{\psi}_{\mm}\) as in \cite[124]{BreuilHerzigHuMorraSchraen2023GKdim}. Let \(R^{\psi}_{\bar{\rho},S\cup \{w_{1}\}}\)
denotes the deformation ring of \(\bar{\rho}\), parametrizing deformations \(\rho:\Gal_{F}\to \GL_{2}(R)\) with which have the fixed determinant \(\psi\chi_{\mrm{cycl}}^{-1}\)
and are unramified away from \(S\cup \{w_{1}\}\). 

Assume that 


(1) \(\bar{\rho}|_{\Gal_{F(\sqrt[p]{1})}}\)
is absolutely irreducible. 

(2) 
\(\bar{\rho}|_{\Gal_{F_{p}}}\) is semisimple generic in the sense of \cite[p.100 (i)]{BreuilHerzigHuMorraSchraen2023GKdim}.

(3) For any \(w\in S_{D}\cup S_{\bar{\rho}}\) with \(w\nmid p\),
the universal framed deformation ring of \(\bar{\rho}|_{\Gal_{F_{w}}}\) over \(W(\FF)\)
is formally smooth over \(W(\FF)\).

Then \(R^{\psi}_{\bar{\rho},S\cup\{w_{1}\}}\cong \hat{\TT}^{\psi}_{\bar{\rho}}\), 
and \(\Hom_{W(\FF)}\left(\ti{H}^{d_{V}}(\mrm{Sh}_{V^{p}}(G_{D}),W(\FF))^{\psi}_{\mm},W(\FF)\right)\)
is \emph{faithfully flat} as a \(R^{\psi}_{\bar{\rho},S\cup \{w_{1}\}}\)-module. 
Moreover, \(R^{\psi}_{\bar{\rho},S\cup \{w_{1}\}}\) is a complete intersection.
\end{thm}
\begin{rmk}
To be consistent with our normalization, we will use the normalization \(\epsilon=1\)
rather than \(\epsilon=-1\) to define the canonical model. This explains why we have take \(\ti{H}^{d_{D}}(-)^{\psi}\)
instead of \((-)^{\psi^{-1}}\)
as in \cite{BreuilHerzigHuMorraSchraen2023GKdim}. See \cite[Lemma 3.17, Lemma 3.12]{CornutVatsal2005cm} for a detailed argument around the normalization.
\end{rmk}
\begin{cor}[Promodularity]\label{corPromodular}
Assume that \(p\) is inert in \(F\), \(F\) is totally real, and \(d:=[F:\QQ]\).
We take \(T=\Sigma_{\infty}\)
if \(2|d\)
and taking \(T=\Sigma_{\infty}\backslash \{\tau_{0}\}\)
for some \(\tau_{0}\in\Sigma_{\infty}\) if \(2\nmid d\). We will take \(D=B\)
in Theorem \ref{thmBHHMSFlatness}. 

Let \(\FF\) be a finite extension of \(\FF_{p}\), \(L/\QQ_{p}\) be a finite extension with residue field \(\FF\), and let
\(\rho:\Gal_{F}\to \GL_{2}(L)\) be a continuous representation.
Assume that \(\bar{\rho}:\Gal_{F}\to \GL_{2}(\FF)\) satisfies (\ref{alignResiduModularity})
and the assumptions (1)(2)(3) of Theorem \ref{thmBHHMSFlatness}. 
Let \(S\),
\(w_{1},V^{p}\) be as in Theorem \ref{thmBHHMSFlatness}, and assume that \(\rho\) is unramified away from \(S\cup w_{1}\), and \(\det(\rho)=\psi\chi_{\mrm{cycl}}^{-1}\)
for \(\psi\) as in Theorem \ref{thmBHHMSFlatness}. 

Let \(\chi:\TT^{S\cup \{w_{1}\}}\to \mO_{L}\)
be the character determined by \(\rho\)
via (\ref{alignEicherShimuraDetermin}). Then \[\ti{H}^{d-\#T}(\mrm{Sh}_{V^{p}}(G),\mO_{L})^{\psi}[\chi]\ne 0.
\]
\end{cor}
\begin{proof}
By universal property,
\(\rho\) determines a character \(\chi:R^{\psi}_{\bar{\rho},S\cup \{w_{1}\}}\to \mO_{L}\). By 
\(R^{\psi}_{\bar{\rho},S\cup\{w_{1}\}}\cong \hat{\TT}^{\psi}_{\bar{\rho}}\)
showed in Theorem \ref{thmBHHMSFlatness}, this gives a character \(\chi:\hat{\TT}^{\psi}_{\bar{\rho}}\to \mO_{L}\), which is compatible with \(\chi:\TT^{S\cup \{w_{1}\}}\to \mO_{L}\).
By Theorem \ref{thmBHHMSFlatness}, \(\Hom_{W(\FF)}\left(\ti{H}^{d-\#T}(\mrm{Sh}_{V^{p}}(G),W(\FF))^{\psi}_{\mm},W(\FF)\right)\) is faithfully flat over \(\hat{\TT}^{\psi}_{\bar{\rho}}\). 
Therefore, \[
\Hom_{W(\FF)}\left(\ti{H}^{d-\#T}(\mrm{Sh}_{V^{p}}(G),W(\FF))^{\psi}_{\mm},W(\FF)\right)\otimes_{\hat{\TT}^{\psi}_{\bar{\rho}},\chi}\mO_{L}\cong 
\ne 0,
\] and taking \(\Hom(-,\mO_{L})\) in \(\mrm{Solid}_{W(\FF)}\), 
we see that \[\ti{H}^{d-\#T}(\mrm{Sh}_{V^{p}}(G),\mO_{L})^{\psi}[\chi]\ne 0.
\]
\end{proof}

We can thus prove some cases of the Langlands-Clozel-Fontaine-Mazur conjecture:
\begin{thm}\label{thmConjecFontaineMazur}
Let \(F\) be totally real number field of degree \(d\). Assume that \(p\) is inert in \(F\).
Let \(L/\QQ_{p}\) be a finite extension with residue field \(\FF\), and let
\(\rho:\Gal_{F,S}\to \GL_{2}(L)\) be a continuous representation.
Let \(\bar{\rho}:\Gal_{F}\to \GL_{2}(\FF)\) be the reduction of \(\rho\).

Assume that 

(1) \(\bar{\rho}\) is residual modular, that is, there exists a Hilbert modular eigenform \(f_{0}\)
such that \(\bar{\rho}\cong \bar{\rho}_{f_{0}}\).

(2) \(\bar{\rho}|_{\Gal_{F(\sqrt[p]{1})}}\)
is absolutely irreducible.

(3) \(\bar{\rho}|_{\Gal_{F_{p}}}\) is semisimple generic in the sense of \cite[p.100 (i)]{BreuilHerzigHuMorraSchraen2023GKdim}.

(4) For any \(w\in S\backslash \{p\}\),
the universal framed deformation ring of \(\bar{\rho}|_{\Gal_{F_{w}}}\) over \(W(\FF)\)
is formally smooth over \(W(\FF)\).

(5) Let \(\psi\) denote the Teichmüller lift of \(\chi_{\mrm{cycl}}\det(\bar{\rho})\), and we 
assume that \(\det(\rho)=\psi \chi_{\mrm{cycl}}^{-1}\).

Let \(\chi:\TT^{S}\to \mO_{L}\) be the character determined by \(\rho\) via (\ref{alignEicherShimuraDetermin}).

Then there exists a neat open compact subgroup \(K^{p}\subset G_{\emptyset}(\mA_{f}^{p})\), and a finite set \(S'\supset S\) such that \(K^{p}\) is unramified away from \(S'\), and \[\ti{H}^{d}(\mrm{Sh}_{K^{p}}(G_{\emptyset}),L)[\chi|_{\TT^{S'}}]\ne 0.
\]

If we assume in addition that 

(6) \(\rho\) is de Rham of regular parallel weights.

Then there exists a Hilbert modular eigenforms \(f\)
such that \(\rho\cong \rho_{f}\).
\end{thm}
\begin{rmk}
It should be sufficient to assume that \(p\) is unramified (rather than inert) in \(F\). See \cite[Remark 8.5.2]{BreuilHerzigHuMorraSchraen2023GKdim}.
\end{rmk}
\begin{rmk}
Following the strategy of \cite{GeeNewton2022patching},
\cite{BreuilHerzigHuMorraSchraen2023GKdim} proves the ``big \( R= \) big \( T \)'' theorem by calculating the Gelfand-Kirillov dimension of the completed cohomolgy of Shimura sets and curves. We then deduce the promodularity for the completed cohomolgy of Hilbert modular varieties using \( \ell \)-adic Hecke action of \cite{FarguesScholze2021geometrization}. It is also possible to calculate directly the Gelfand-Kirillov dimension of Hilbert modular varieties, which will also implies the ``big \( R= \) big \( T \)'' theorem. This project is the subject of Nischay Reddy's research, supervised by Florian Herzig.
\end{rmk}
\begin{proof}
Let \(\mm:=\Ker(\TT^{S}\xrightarrow{\chi}\mO_{L}\twoheadrightarrow\FF)\).
If \(\bar{\rho}\)
is residual modular for a form \(f\) of level \(\Kpp\), 
by twisting using Hasse invariant, we can assume that \(f\) is of regular weights.
Then by comparing with complex Hodge decomposition, there exists \(K_{p}'\) such that \[H^{d}_{\et}(\mrm{Sh}_{K^{p}K_{p}'}(G_{\emptyset})_{\bar{\QQ}},\FF)_{\mm}\ne 0. 
\] 
By Poincaré duality,
 we know that 
\[H^{d}_{c,\et}(\mrm{Sh}_{K^{p}K_{p}'}(G_{\emptyset}),\FF)_{\mm}\ne 0. 
\] 
Then applying Corollay \ref{corDifferentSVfromHecke} to an auxiliary place \(w\ne p\), we know that
\[H^{d-\#T}_{\et}
(\mrm{Sh}_{K^{p}K_{p}''}
(G)_{\bar{\QQ}},\FF)_{\mm}\ne 0
\] for \(T\) as in Corollary \ref{corPromodular}. Then by Corollary \ref{corPromodular}, we know that \[H^{d-\#T}
(\mrm{Sh}_{K^{p}}
(G),L)[\chi]\ne 0.
\]
Now we can apply Corollary \ref{corDirectSummandVaryL} to show that (up to shrinking \(K^{p}\), and restricting \(\chi\) to \(\TT^{S'}\) for a large finite set \(S'\supset S\)) \[
H^{d}(\mrm{Sh}_{K^{p}}(G),L)[\chi|_{\TT^{S'}}]\cong
H^{d}_{c}(\mrm{Sh}_{K^{p}}(G),L)[\chi|_{\TT^{S'}}]\ne 0,
\] 
where the last isomorphism follows by noting that the completed cohomology of the boundary is supported in degree \([0,\floor*{\frac{d-1}{2}}]\) by \cite{Waldschmidt1981transcendance} (in fact, the trivial bounds \([0,d-1]\) suffices for this proof). See \cite{JiangPan2025Plectic} for details.

If \(\rho\) is in addition de Rham of regular parallel weights, we can then apply Theorem \ref{thmClassicalityCompleteCoho} to conclude. 
\end{proof}

\appendix
\section{Universal determinants on Hilbert modular varieties}\label{secUniversalDeterminant}

In this appendix, we will define the \(p\)-adic Hecke algebra \(\hat{\TT}^{S}\) acting on completed cohomology of Hilbert modular varieties, 
and define a \(2\)-dimensional determinant of \(\Gal_{F}\)
valued in \(\hat{\TT}^{S}\) in the sense of \cite{Chenevier2014determinants}. The result should be well-known following the strategy of \cite[Theorem 4.3.1]{Scholze15}, but we give an account here due to the lack of references, especially in the abelian type case. 

Let \( G \), \( K^{p} \), \( S \), and \( \TT^{S} \) be as in \S \ref{subsec:ClassiThm}. 
\begin{dfn}[\cite{Emerton06}, \cite{Emerton2014completed}]\label{dfnPadcHeckeAlg}
We fix an open compact subgroup \(K_{p,0}\subset G(\QQ_{p})\). 
For any open compact subgroup \(K_{p}\subset K_{p,0}\) and \(n\in\NN\), 
let \(\TT^{S}_{K_{p},n}\)
denote the image of \(\TT^{S}\) in \[\End_{\ZZ/p^{n}[K_{p,0}/K_{p}]}\left(R\Gamma_{\et,c}(\mrm{Sh}_{\Kpp}(G)_{\bar{\QQ}},\ZZ/p^{n})\right),\] and define \[\hat{\TT}^{S}:=\varprojlim_{K_{p},n}\TT^{S}_{K_{p},n},
\] equipped with the limit topology. 
Note that by \cite[Corollary 9.11]{DospinescuPaskunasSchaen2020infinitesimalFamily}, it has only finitely many open maximal ideals \(\mm\). 
Also note that by Poincaré duality, we can also use \(R\Gamma_{\et}(\mrm{Sh}_{\Kpp}(G)_{\bar{\QQ}},\ZZ/p^{n})\) instead
in the definition.
\end{dfn}
Then \(\hat{\TT}^{S}\) acts continuously 
on the completed cohomology \(R\Gamma(K^{p},\ZZ_{p})\) (Definition \ref{dfn:CompletedCohomo}). 

\begin{thm}\label{thmConstrDeterminantHilbert}
There exists a \(2\)-dimensional continuous determinant \(D:\ZZ_{p,\square}[\Gal_{F,S}]\to \hat{\TT}^{S}\) in the sense of \cite{Chenevier2014determinants}, such that for any \(\mfk{l}\notin S\),
\begin{align}\label{alignEicherShimuraDetermin}
D(t-\Fr_{\mfk{l}})=t^{2}-
T_{\mfk{l}}\cdot t+\Nm(\mfk{l})S_{\mfk{l}},
\end{align}
Here, concretely \(\ZZ_{p,\square}[\Gal_{F,S}]\cong \varprojlim_{U}\ZZ_{p}[\Gal_{F,S}/U]\),
where the limit is taken over all the open subgroups \(U\subset \Gal_{F,S}\),
and \(S_{\mfk{l}}=\mrm{diag}(\varpi_{\mfk{l}}^{-1},\varpi_{\mfk{l}}^{-1})\) and \(T_{\mfk{l}}=G(\mO_{F_{\mfk{l}}})\mrm{diag}(1,\varpi_{\mfk{l}}^{-1})G(\mO_{F_{\mfk{l}}})\).
\end{thm}
We will use the compactification \(\mrm{Sh}^{*}_{\overline{K}}\)(G) of \(\mrm{Sh}_{K}(G)\) considered in \cite[\S 4.4]{BoxerPilloni2021higherColeman} which admits a finite surjection \(\mrm{Sh}^{*}_{K}(G)\to \mrm{Sh}^{*}_{\overline{K}}(G)\), where \(\mrm{Sh}^{*}_{K}(G)\) is the minimal compactification of \(\mrm{Sh}^{*}_{\overline{K}}(G)\). Concretely, 
let \(G^{*}:=\mrm{det}^{-1}(T_{\QQ})\)
for \(\mrm{det}:G\to T_{F}\), and \(G^{*}\) defines a tower of Shimura varieties of PEL type \(\mrm{Sh}_{K}(G^{*})\), which admits a closed immersion \(\mrm{Sh}_{K}(G^{*})\hookrightarrow \mrm{Sh}_{K'}(\mrm{GSp}_{2d})\) for \(K\subset K'\cap G^{*}(\mA_{f})\). Then \(\mrm{Sh}^{*}_{\overline{K}}(G^{*})\)
denotes the Zariski closure of \(\mrm{Sh}_{K}(G^{*})\) in \(\mrm{Sh}_{K'}^{*}(\mrm{GSp}_{2d})\). Then using the standard construction relating \(\mrm{Sh}_{K'}(G^{*})\)
and \(\mrm{Sh}_{K}(G)\) as in \cite[\S 2]{Deligne1979varietesInterpretationModulaire}, we construct the compactification \(\mrm{Sh}^{*}_{\overline{K}}(G)\). 
Let \(\mS_{\overline{K}}^{*}(G)\)
be the analytification of \(\mrm{Sh}_{\overline{K}}^{*}(G)\)
over \(\CC_{p}\).

By \cite[Proposition 4.4.53]{BoxerPilloni2021higherColeman} (generalizing \cite[Theorem 4.1.1]{Scholze15}), this compactification has the benefit that there is a perfectoid space \(\mS_{\overline{K^{p}}}^{*}(G)\) such that  \[\mS_{\overline{K^{p}}}^{*}(G)
\sim \varprojlim_{K_{p}}\mS_{\overline{\Kpp}}^{*}(G).
\] 
Note that this is in general stronger than an isomorphism as diamonds. 
By
\cite[Proposition 4.4.53]{BoxerPilloni2021higherColeman}, there is a Hodge-Tate period map \[\pi_{\HT}:\mS^{*}_{\overline{K^{p}}}(G)\to \Fl_{G,\mu}\cong \prod_{\tau\in\Sigma_{\infty}}\mP^{1},
\] such that \(\Fl_{G,\mu}\)
admits a finite analytic cover \(\{U_{i}\}\)
by open affinoid subspaces, such that \(\forall i\), \(\pi_{\HT}^{-1}(U_{i})\)
is a \emph{good affinoid perfectoid open subspace} in the sense of \cite[Definition 4.4.3]{BoxerPilloni2021higherColeman}, i.e. for \(K_{p}\) sufficiently small,
\(\pi_{\HT}^{-1}(U_{i})\)
descends to an open affinoid subspace \(U_{i,\Kpp}\subset \mS^{*}_{\overline{\Kpp}}(G)\),
and we have an isomorphism over \(\mO_{\CC_{p}}/p^{n}\) \begin{align}\label{alignFunctionIsDens}
H^{0}(\pi_{\HT}^{-1}(U_{i}),\mO^{+}_{\mS^{*}_{\overline{K^{p}}}(G)}/p^{n})\cong \varinjlim_{K_{p}}
H^{0}(U_{i,\Kpp},\mO^{+}_{\mS^{*}_{\overline{\Kpp}}(G)}/p^{n}). 
\end{align} Note that \(U_{i,\Kpp}=\Spa(R_{i,K_{p}},R_{i,K_{p}}^{+})\) and \(\pi_{\HT}^{-1}(U_{i})\) are uniform, and \(R_{i,K_{p}}^{+}=R_{i,K_{p}}^{\circ}\) by \cite[Proposition 2.4.15]{Huber1993bewertungsspektrum}, so this definition indeed coincides with the definition in \cite{BoxerPilloni2021higherColeman}.

 We will refer to such a cover \(\{U_{i}\}\)
as \emph{a good cover} for \(\mS^{*}_{\overline{K^{p}}}(G)\).
\begin{construction}[A good cover for Hilbert modular varieties]
\label{constructionGoodCoverHilbert}
For \(\mS^{*}_{\overline{K^{p}}}(G^{*})\),
we have the following concrete choice of the cover \(\{U_{i}\}\):
we have \(\Fl_{G,\mu}\cong \prod_{\tau\in \Sigma_{\infty}}\mP^{1}\),
and fixing the maximal torus \(T\subset G\), we denote by \(0_{\tau}\) and \(\infty_{\tau}\)
the unique fixing points for the action of \(T\) on the \(\tau\)-component \(\mP^{1}_{\tau}\). We can choose two sections \(s_{\tau,0},s_{\tau,\infty}\in H^{0}(\mP^{1},\omega_{\mP^{1}})\) 
such that \(\infty_{\tau}\in U_{\tau,\infty}:=\{\frac{|s_{\tau,0}|}{|s_{\tau,\infty}|}\le 1\}\), \(0_{\tau}\in U_{\tau,0}:=\{\frac{|s_{\tau,\infty}|}{|s_{\tau,0}|}\le 1\}\), and \(\{U_{\tau,\infty},U_{\tau,0}\}\)
covers \(\mP^{1}_{\tau}\).

Let \(\omega_{\Fl_{G,\mu}}\) be the tautological ample line bundle,
and for any \(i=(i_{\tau})_{\tau\in\Sigma_{\infty}}\in \{0,\infty\}^{\Sigma_{\infty}}\), \(s_{i}:=\prod_{\tau}s_{\tau,i_{\tau}}\in H^{0}(\Fl_{G,\mu},\omega_{\Fl_{G,\mu}})\), and define \[U_{i}:=\prod_{\tau\in\Sigma_{\infty}}U_{\tau,i_{\tau}}.
\] Then \(\{U_{i}\}_{i\in\{0,\infty\}^{\Sigma_{\infty}}}\)
forms an open cover of \(\Fl_{G,\mu}\),
\(s_{i}\)
is invertible on \(U_{i}\), and for any \(i,j\), the open inclusion
\(U_{ij}:=U_{i}\cap U_{j}\subset U_{i}\)
is characterized by the condition \(|s_{i}|=|s_{j}|\). 

We claim that \(\{U_{i}\}_{i\in\{0,\infty\}^{\Sigma_{\infty}}}\) gives a good cover for \(\mS^{*}_{\overline{K^{p}}}(G^{*})\). Comparing with \cite[Theorem 4.1.1 (i)]{Scholze15}, it suffices to show that along \(\Fl_{G,\mu}\hookrightarrow \Fl_{\mrm{GSp}_{2d}}\), each \(U_{i}\)
is mapped into \(\Fl_{\mrm{GSp}_{2d},J}\)
for some \(J\subset\{1,2,\cdots,2d\}\), where the latter is defined in \cite[1008]{Scholze15}. 
We can identify \(\{1,2,\cdots,2d\}\)
with \(\coprod_{\tau}\{0_{\tau},\infty_{\tau}\}\)
and we see easily that \(U_{i}\)
is mapped into \(\Fl_{\mrm{GSp}_{2d},\coprod_{\tau}\{i_{\tau}\}}\).

By the construction of \(\pi_{\HT}\)
for \(\mS^{*}_{\overline{K^{p}}}(G)\) in \cite{BoxerPilloni2021higherColeman},
the cover \(\{U_{i}\}\) is also a good cover for \(\mS^{*}_{\overline{K^{p}}}(G)\). 
\end{construction}

\begin{construction}[Descent of canonical line bundle]
By \cite[Theorem 4.1.1 (v)]{Scholze15}, we have a canonical line bundle \(\omega_{\mS^{*}_{\overline{\Kpp}}(G^{*})}\) over \(\mS^{*}_{\overline{\Kpp}}(G^{*})\), whose pull-back to \(\mS^{*}_{\overline{K^{p}}}(G^{*})\) 
is \(\mscr{A}_{G^{*}}\)-equivariantly isomorphic to the pull-back of \(\omega_{\Fl}\) along \(\pi_{\HT}:\mS^{*}_{\overline{\Kpp}}(G^{*})\to \Fl_{G,\mu}\). Concretely, \(\omega_{\mS^{*}_{\overline{\Kpp}}(G^{*})}|_{\mS_{\Kpp}(G^{*})}\) is the top exterior product of the relative differentials of the universal abelian schemes.  


For \(i\in\{0,\infty\}^{\Sigma_{\infty}}\),
we denote by \(U_{i,\Kpp}\) the descent of \(\pi_{\HT}^{-1}(U_{i})\cap \mS^{*}(G^{*})\)
to \(\mS^{*}_{\overline{\Kpp}}(G^{*})\). Then by (\ref{alignFunctionIsDens}), for any \(n\ge 1\), \(i,j\in \{0,\infty\}^{\Sigma_{\infty}}\),
there exist a small enough subgroup \(K_{p}\), and sections \(s_{j}^{(i)}\in H^{0}(U_{i,\Kpp},\omega_{\mS^{*}_{\overline{\Kpp}}(G^{*})})\)
such that \(\frac{|s_{j}^{(i)}-s_{j}|}{|s_{i}|}\le |p|^{n}\). 
This implies in particular that \(s_{i}^{(i)}\)
is an invertible section. 

We claim that the line bundle \(\omega_{\mS^{*}_{\overline{\Kpp}}(G^{*})}\) 
desends to a line bundle \(\omega_{\mS^{*}_{\overline{\Kpp}}(G)}\)
over \(\mS^{*}_{\overline{\Kpp}}(G)\). For this, we can fix a connected component, and verify the descent locally. Then by the construction of \(\mS^{*}_{\overline{\Kpp}}(G)\)
in \cite{BoxerPilloni2021higherColeman}, \[\mS^{*,0}_{\overline{\Kpp}}(G)\cong 
\mS^{*,0}_{\overline{\Kpp}}(G^{*})/\Delta,
\] where \(\Delta\) is a finite group, which we can assume to be independent of \(K_{p}\) by \cite[Theorem 4.4.43 (6)]{BoxerPilloni2021higherColeman} by taking \(K_{p}\) sufficiently small. 
Here, the quotient of adic spaces exists by \cite{Hansen2016quotients}, and \(\omega_{\mS^{*}_{\overline{\Kpp}}(G^{*})}|_{\mS^{*,0}_{\overline{\Kpp}}(G^{*})}\)
is \(\Delta\)-equivariant. 
We denote \(U_{i,\Kpp}^{0}:=U_{i}\cap \mS^{*,0}_{\overline{\Kpp}}(G^{*})\). 
Since \(s_{j}\) is \(\Delta\)-invariant,
by averaging over \(\Delta\), we can also choose \(s_{j}^{(i)}\) such that \(s_{j}^{(i)}|_{U_{i,\Kpp}^{0}}\)
is \(\Delta\)-invariant; in particular \(s_{i}^{(i)}\) is an invertible \(\Delta\)-invariant section. This implies the desired descent, that is, there is a unique line bundle \(\omega_{\mS^{*}_{\overline{\Kpp}}(G)}\) over \(\mS^{*,0}_{\overline{\Kpp}}(G)\)
such that its pull-back to \(\mS^{*,0}_{\overline{\Kpp}}(G^{*})\)
is \(\Delta\)-equivariantly isomorphic to \(\omega_{\mS^{*}_{\overline{\Kpp}}(G^{*})}\). We can then extend it to a line bundle \(\omega_{\mS^{*}_{\overline{\Kpp}}(G)}\)
over \(\mS^{*}_{\overline{\Kpp}}(G)\) by the standard construction following \cite[\S 2]{Deligne1979varietesInterpretationModulaire}.
\end{construction}
\begin{construction}[A formal model of the minimal compactification]\label{constructionFormalModelMinimalCompactification}
For \(\Kpp\subset G(\mA_{f})\) with \(
K_{p}\) sufficiently small, 
by abuse of notation, we denote the open subspace \(U_{i,\Kpp}\subset \mS^{*}_{\overline{\Kpp}}(G)\)
to be the descent of \(\pi_{\HT}^{-1}(U_{i})\subset \mS^{*}_{\overline{K^{p}}}(G)\). 
Then using (\ref{alignFunctionIsDens}),
we can find a sequence of open subgroups \(\{K_{p,n}\}_{n\in\ZZ_{\ge 1}}\), and
there exist \(s^{(i)}_{j,n}\in H^{0}(U_{i,K^{p}K_{p,n}},\omega_{\mS^{*}_{\overline{\Kpp}}(G)})\) for \(i,j\in \{0,\infty\}^{\Sigma_{\infty}}\), such that \(\frac{|s_{j,n}^{(i)}-s_{j}|}{|s_{i}|}\le |p|^{n}\). 

We denote by \(\ti{U}_{i,\Kpp}\) the preimage of \(U_{i,\Kpp}\)
along the finite morphism \(\mS^{*}_{\Kpp}(G)\to \mS^{*}_{\overline{\Kpp}}(G)\), and denote by \(\omega_{\Kpp}\)
the pull-back of \(\omega_{\mS^{*}_{\overline{\Kpp}}(G)}\). We also denote by \(s^{(i)}_{j,n}\in H^{0}(\ti{U}_{i,K^{p}K_{p,n}},\omega_{K^{p}K_{p,n}})\)
the pull-back of \(s^{(i)}_{j,n}\). 

We can now define a formal model of \(\mS^{*}_{K^{p}K_{p,n}}(G)\) using \cite[Lemma 2.1.1]{Scholze15} as in \cite[1029-1030]{Scholze15}. This implies that \(\mS^{*}_{K^{p}K_{p,n}}(G)\)
equipped with the line bundle \(\omega_{K^{p}K_{p,n}}\), 
and the cover \(\{\ti{U}_{i,K^{p}K_{p,n}}\}\)
admits a flat proper formal model \(\mfk{S}^{*}_{K^{p}K_{p,n}}(G)\) over \(\Spf(\mO_{\CC_{p}})\) equipped with an ample line bundle \(\mfk{w}_{K^{p}K_{p,n}}\) and an affine cover \(\{\mfk{U}_{i,K^{p}K_{p,n}}\}\). Moreover, for fixed \(j\in\{0,\infty\}^{\Sigma_{\infty}}\), \(s^{(i)}_{j,n}\in H^{0}(\mfk{U}_{i,K^{p}K_{p,n}},\mfk{w}_{K^{p}K_{p,n}})\) and they
glue to a section (``fake Hasse invariant'') \[\bar{s}_{j,n}\in H^{0}(\mfk{S}^{*}_{K^{p}K_{p,n}}(G),\mfk{w}_{K^{p}K_{p,n}}/p^{n})
\] Note that \(\mfk{U}_{j,K^{p}K_{p,n}}=U(\bar{s}_{j,n})\), i.e. the locus where \(\bar{s}_{j,n}\)
is invertible. 
\end{construction}
We can compare the ideal sheaves of \(\mS^{*}_{\overline{K}}(G)\) and \(\mS^{*}_{K}(G)\) that define the boundaries.
\begin{lem}\label{lemIdealSheafDense}
Denote by \(\mcal{Z}_{\overline{K}}\), 
and \(\mcal{Z}_{K}\) the reduced complements of \(\mS_{K}(G)\)
in \(\mS^{*}_{\overline{K}}(G)\) and \(\mS^{*}_{K}(G)\). We define \(\mcal{Z}_{\overline{K^{p}}}:=\varprojlim_{K_{p}}\mcal{Z}_{\overline{\Kpp}}\).
Let \(\mcal{I}^{+}_{K}\subset \mO^{+}_{\mS^{*}_{\overline{K}}(G)}\) (resp. \(\ti{\mcal{I}}^{+}_{K}\subset \mO^{+}_{\mS^{*}_{K}(G)}\), resp. 
\(I^{+}_{K^{p}}\subset \mO^{+}_{\mS^{*}_{\overline{K^{p}}}(G)}\)) be the ideal of functions vanishing on \(\mcal{Z}_{\overline{K}}\) (resp. \(\mcal{Z}_{K}\), resp. \(\mcal{Z}_{\overline{K^{p}}}\)). 

Let \(U_{i}\), \(U_{i,\Kpp}\), and \(\ti{U}_{i,\Kpp}\) be as in Construction \ref{constructionGoodCoverHilbert} and Construction \ref{constructionFormalModelMinimalCompactification}. Then we have 
isomorphisms \[
\varinjlim_{K_{p}}H^{0}_{\an}(U_{i,\Kpp},\mcal{I}^{+}_{\Kpp}/p^{n})
\cong 
\varinjlim_{K_{p}}H^{0}_{\an}(\ti{U}_{i,\Kpp},\ti{\mcal{I}}^{+}_{\Kpp}/p^{n})
\cong H^{0}_{\an}(U_{i},\mcal{I}_{K^{p}}^{+}/p^{n}).
\] Moreover, the last isomorphism is compatible with \(\TT^{S}\)-actions.
\end{lem}
\begin{proof}
This follows the same proof of \cite[Lemma IV.3.3]{Scholze15}. The only non-formal input in the proof loc. cit. is the almost surjectivity of the composition \[\varinjlim_{K_{p}}H^{0}(U_{i,\Kpp},\mcal{I}^{+}_{\Kpp}/p^{n})
\to H^{0}(U_{i},\mcal{I}_{K^{p}}^{+}/p^{n}).
\] By the construction in \cite[\S 4.4]{BoxerPilloni2021higherColeman},
this follows from the surjectivity in the Hodge type case (\cite[Lemma IV.3.3]{Scholze15}), and \cite[Theorem 4.4.43 (6)]{BoxerPilloni2021higherColeman}. 

More precisely, if we denote the corresponding ideal sheaves for \(\mS^{*}(G^{*})\) by \(\mcal{I}^{+,\prime}_{K^{p,\prime}K_{p}'}\) and \(\mcal{I}^{+,\prime}_{K^{p,\prime}}\), then by \cite[Theorem 4.4.43 (6)]{BoxerPilloni2021higherColeman}, there is a finite group \(\Delta\), such that \(I^{+}_{\Kpp}\)
is the \(\Delta\)-invariant subspace of \(I^{+,\prime}_{K^{p,\prime}(K_{p}\cap G^{*}(\QQ_{p}))}\), and \(I^{+}_{K^{p}}\)
is the \(\Delta\)-invariant subspace of \(I^{+,\prime}_{K^{p,\prime}}\).
By \cite[Lemma IV.3.3]{Scholze15}, \(\varinjlim_{K_{p}'}I^{+}_{K^{p,\prime}K_{p}'}[1/p]\) is dense in \(I^{+}_{K^{p,\prime}}[1/p]\), and then
we have the density of \(\varinjlim_{K_{p}}I^{+}_{\Kpp}[1/p]\) in \(I^{+}_{K^{p}}[1/p]\)
by taking \(\Delta\)-coinvariant projection. Intersecting with \(\mO^{+}_{\mS^{*}_{\overline{K^{p}}}(G)}\), we obtain the desired almost surjection.
\end{proof}

\begin{prop}[Primitive comparison]\label{propPrimitiveComparisonScholze}
We 
have an almost isomorphism \[\varinjlim_{K_{p}}R\Gamma_{\et,c}(\mrm{Sh}_{\Kpp}(G)_{\bar{\QQ}},\ZZ/p^{n})\otimes_{\ZZ/p^{n}}\mO_{\CC_{p}}/p^{n}\cong R\Gamma_{\an}(\mS_{\overline{K^{p}}}^{*}(G),\mcal{I}^{+}_{K^{p}}/p^{n}).\]
\end{prop}
\begin{proof}
This follows the same proof of \cite[Theorem 4.2.1]{Scholze15}. Note that it is necessary to know that the boundary is strongly Zariski closed, which follows from \cite[Remark 7.5]{BhattScholze2022prisms}.
\end{proof}

\begin{lem}\label{lemInjectiveFromToAlmost}
Let \(F,F'\in D(\ZZ/p^{n})\) be (derived) discrete \(\ZZ/p^{n}\)-modules. Then \[\Hom_{\ZZ/p^{n}}(F,F')\to \Hom_{\mO_{\CC_{p}}^{a}/p^{n}}(F\otimes_{\ZZ/p^{n}}\mO_{\CC_{p}}/p^{n},F'\otimes_{\ZZ/p^{n}}\mO_{\CC_{p}}/p^{n}) \]
is injective and continuous, i.e. the map of condensed \(\ZZ/p^{n}\)-modules. Here the source is endowed with the weak topology, and the target is endowed with the weak topology for the \(!\)-realization of almost modules as in \cite[1028]{Scholze15}, i.e. by identifying \[\Hom_{\mO_{\CC_{p}}^{a}/p^{n}}(M,N)\cong \Hom_{\mO_{\CC_{p}}/p^{n}}(M\otimes_{\mO_{\CC_{p}}}\mm_{\CC_{p}},N\otimes_{\mO_{\CC_{p}}}\mm_{\CC_{p}}),
\] and endowing \(M\otimes_{\mO_{\CC_{p}}}\mm_{\CC_{p}}\)
and \(N\otimes_{\mO_{\CC_{p}}}\mm_{\CC_{p}}\)
with discrete topology. 
\end{lem}
\begin{rmk}
In this setting, the weak topology is the correct topology for the internal Hom in the category of condensed sets. 
\end{rmk}
\begin{proof}
We note that \[
\Hom_{\mO_{\CC_{p}}^{a}/p^{n}}(F\otimes_{\ZZ/p^{n}}\mO_{\CC_{p}}/p^{n},F'\otimes_{\ZZ/p^{n}}\mO_{\CC_{p}}/p^{n})\cong \Hom_{\mO_{\CC_{p}}/p^{n}}(F\otimes_{\ZZ/p^{n}}\mm_{\CC_{p}}/p^{n},F'\otimes_{\ZZ/p^{n}}\mm_{\CC_{p}}/p^{n}),
\] and that \(\mm_{\CC_{p}}/p^{n}\)
is free over \(\ZZ/p^{n}\) by lifting a basis of \(\mm_{\CC_{p}}/p\)
over \(\FF_{p}\). In particular, we can fix a direct summand \[\ZZ/p^{n}\xhookrightarrow{\iota}\mm_{\CC_{p}}/p^{n}\xrightarrow{p}\ZZ/p^{n},
\] such that the composition is the identity, and
then we have a splitting to the map above by \[\Hom_{\mO_{\CC_{p}}/p^{n}}(F\otimes_{\ZZ/p^{n}}\mm_{\CC_{p}}/p^{n},F'\otimes_{\ZZ/p^{n}}\mm_{\CC_{p}}/p^{n})\to \Hom_{\ZZ/p^{n}}(F,F'),\;f\mapsto p\circ f\circ \iota. 
\] This implies the injectivity. 
For continuity,
we can reduce to the case when \(F\in \Perf(\ZZ/p^{n})\), and then the continuity is trivial since the source is endowed with discrete topology.
\end{proof}

\begin{proof}[Proof of Theorem \ref{thmConstrDeterminantHilbert}]
The proof is verbatim the same as that of \cite[Theorem 4.3.1, Corollary 5.1.11]{Scholze15} except that we choose to work with the complexes, which annhilates the appearance of the nilpotent ideals in \cite[Corollary 5.2.6]{Scholze15}. 

By Proposition \ref{propPrimitiveComparisonScholze}, we have a \(G(\QQ_{p})\)-equivariant almost isomorphism \begin{align}\label{alignPrimitiveEg}
R\Gamma_{c}(K^{p},\ZZ/p^{n})\otimes_{\ZZ/p^{n}}\mO_{\CC_{p}}/p^{n}\cong R\Gamma_{\an}(\mS_{\overline{K^{p}}}^{*}(G),\mcal{I}^{+}_{K^{p}}/p^{n}).
\end{align} 
We have
\[\TT^{S}\to \End_{\ZZ/p^{n}[K_{p,0}]}\left(R\Gamma_{c}(K^{p},\ZZ/p^{n})\right)
\xhookrightarrow{i} \End_{(\mO_{\CC_{p}}/p^{n})^{a}[K_{p,0}]}\left(R\Gamma_{\an}(\mS_{\overline{K^{p}}}^{*}(G),\mcal{I}^{+}_{K_{p}}/p^{n})\right).
\]
Note that \(i\) is injective and continuous by Lemma \ref{lemInjectiveFromToAlmost}. 

By affinoidness of \(\pi_{\HT}^{-1}(U_{i})\) and the perfectoidness of \(\mS^{*}_{\overline{K^{p}}}(G)\) (\cite[Proposition 4.4.53]{BoxerPilloni2021higherColeman}), \(R\Gamma_{\an}(\mS_{\overline{K^{p}}}^{*}(G),\mcal{I}^{+}_{K^{p}}/p^{n})\)
can be calculated by a Cech complex given by the cover \(\{\pi_{\HT}^{-1}(U_{i})\}\).
For any \(I\subset\{0,\infty\}^{\Sigma_{\infty}}\), write \(U_{I}:=\bigcap_{i\in I}U_{i}\), and \(\ti{U}_{I,\Kpp}:=\bigcap_{i\in I}\ti{U}_{i,\Kpp}\),
then we have an almost isomorphism  \begin{align}\label{alignALmostPurityToCech}
R\Gamma_{\an}(\pi_{\HT}^{-1}(U_{I}),\mcal{I}^{+}_{K^{p}}/p^{n})
\cong H^{0}_{\an}(\pi_{\HT}^{-1}(U_{I}),\mcal{I}^{+}_{K^{p}}/p^{n})\cong \varinjlim_{K_{p}}H^{0}_{\an}(\ti{U}_{I,\Kpp},\ti{\mcal{I}}^{+}_{\Kpp}/p^{n}),
\end{align}
 by \cite[Theorem 6.3]{Scholze12} and the fact that the boundary is strongly Zariski closed by \cite[Remark 7.5]{BhattScholze2022prisms}.
The second isomorphism
is given by Lemma \ref{lemIdealSheafDense}.

By \cite[Lemma 2.1.2]{Scholze15}, \(\ti{\mcal{I}}_{K^{p}K_{p,n}}\) has a formal model as a coherent ideal sheaf \(\ti{\mfk{I}}_{K^{p}K_{p,n}}\) over \(\mfk{S}^{*}_{K^{p}K_{p,n}}(G)\) as in Construction \ref{constructionFormalModelMinimalCompactification}. Then
\begin{align}\label{alignUItoWholeS*}
H^{0}_{\an}(\ti{U}_{I,K^{p}K_{p,n}},\ti{\mcal{I}}^{+}_{K^{p}K_{p,n}}/p^{n})
\cong 
H^{0}(\mfk{U}_{I,K^{p}K_{p,n}},\ti{\mfk{I}}^{+}_{K^{p}K_{p,n}}/p^{n})\\
\cong \varinjlim_{m} H^{0}(\mfk{S}^{*}_{K^{p}K_{p,n}},\ti{\mfk{I}}^{+}_{K^{p}K_{p,n}}\otimes \mfk{w}_{K^{p}K_{p,n}}^{\otimes m\cdot\#(I)}/p^{n})\\
\cong \varinjlim_{m} H^{0}(\mfk{S}^{*}_{K^{p}K_{p,n}},\ti{\mfk{I}}^{+}_{K^{p}K_{p,n}}\otimes \mfk{w}_{K^{p}K_{p,n}}^{\otimes m\cdot\#(I)})/p^{n},
\end{align} 
 where the transition maps are given by multiplication by \(\prod_{i\in I}s_{I}\), and the last isomorphism follows from the ampleness of \(\mfk{w}_{K^{p}K_{p,n}}\). Moreover, 
\(H^{0}(\mfk{S}^{*}_{K^{p}K_{p,n}},\ti{\mfk{I}}^{+}_{K^{p}K_{p,n}}\otimes \mfk{w}_{K^{p}K_{p,n}}^{\otimes m\cdot\#(I)})\)
is \(p\)-torsion free, and \[H^{0}(\mfk{S}^{*}_{K^{p}K_{p,n}},\ti{\mfk{I}}^{+}_{K^{p}K_{p,n}}\otimes \mfk{w}_{K^{p}K_{p,n}}^{\otimes m\cdot\#(I)})[1/p]\cong H^{0}(\mS^{*}_{K^{p}K_{p,n}},
\ti{\mcal{I}}_{K^{p}K_{p,n}}\otimes
\omega_{K^{p}K_{p,n}}^{\otimes m\cdot\#(I)}),
\] with \(\ti{\mcal{I}}_{K^{p}K_{p,n}}:=\ti{\mcal{I}}_{K^{p}K_{p,n}}^{+}[1/p]\). 


\begin{notation}
Let \(\TT^{S,\prime}_{K_{p},m}\) denote the \(\ZZ_{p}\)-algebra generated by the image of \(\TT^{S}\) in \[\End\left(H^{0}(\mfk{S}^{*}_{K^{p}K_{p,n}},\ti{\mfk{I}}^{+}_{K^{p}K_{p,n}}\otimes \mfk{w}_{K^{p}K_{p,n}}^{\otimes m})\right),\] 
which is flat over \(\ZZ_{p}\).
\end{notation}
Note that \(H^{0}(\mS^{*}_{\Kpp},\ti{I}_{\Kpp}\otimes \omega^{\otimes m}_{\Kpp})\)
admits a finite dimensional model over \(\QQ_{p}\) on which \(\TT^{S}\) acts. 
By descent from Petersson inner product product (over \(\CC\)), we see that \(\TT^{S,\prime}_{K_{p},m}[1/p]\) is finite \'etale over \(\QQ_{p}\), and in particular \(\TT^{S,\prime}_{K_{p},m}\)
is reduced and finite flat over \(\ZZ_{p}\).

For an eigenform \(f\in H^{0}(\mS^{*}_{\Kpp},\ti{I}_{\Kpp}\otimes \omega^{\otimes m}_{\Kpp})\), \(f\) gives rise to a cuspidal Hilbert modular forms of parallel weight \(m\) by pulling back to the toroidal compactified Hilbert modular varieties. Note that the pull-back of \(\omega_{\Kpp}\) is precisely the automorphic line bundle on \(\mS^{\tor}_{\Kpp}(G)\)
corresponding to parallel weight 1 forms, which follows from the similar statement for \(\mS^{*}_{\Kpp}(G^{*})\). 
Therefore, by 
\cite{Carayol86rep}, \cite{RT83}, \cite{BR93}, \cite{Taylor89galois}, \(f\) has an associated \(2\)-dimensional Galois representation \(\rho_{f}\)
satisfying the Eichler-Shimura relation. 
Thus by \cite[Example 2.32]{Chenevier2014determinants}, we obtain the following:
\begin{lem}\label{lemExistDeterminantMiddle}
There is a continuous \(2\)-dimensional determinant of \(\ZZ_{p,\square}[\Gal_{F,S}]\) valued in \(\TT^{S,\prime}_{K_{p},m}\)
satisfying (\ref{alignEicherShimuraDetermin}).
\end{lem}

\begin{dfn}[{\cite[Theorem 4.3.1]{Scholze15}}]
We define \(\TT^{S}_{\mrm{cl}}\)
to be the ring \(\TT^{S}\otimes_{\ZZ}\ZZ_{p}\)
endowed with the weakest topology such that \(\TT^{S}\otimes_{\ZZ}\ZZ_{p}\twoheadrightarrow \TT^{S,\prime}_{K_{p},m}\) is continuous for all \(K_{p}\) and \(m\).
\end{dfn}
Then by (\ref{alignUItoWholeS*}), \(\TT^{S}_{\mrm{cl}}\) acts continously on \(
H^{0}_{\an}(\ti{U}_{I,K^{p}K_{p,n}},\ti{\mcal{I}}^{+}_{K^{p}K_{p,n}}/p^{n}),\)
and thus using the \v Cech complex and (\ref{alignALmostPurityToCech}),
we obtain a continuous map \[\TT^{S}_{\mrm{cl}}\to \End_{\mO_{\CC_{p}}^{a}/p^{n}[K_{p,0}]}\left(
R\Gamma_{\an}(\mS_{\overline{K^{p}}}^{*}(G),\mcal{I}^{+}_{K^{p}}/p^{n})
\right).
\]
Using Lemma \ref{lemInjectiveFromToAlmost} and (\ref{alignPrimitiveEg}), we obtain a continuous map \[\TT^{S}_{\mrm{cl}}\to \End_{\ZZ/p^{n}[K_{p,0}]}\left(
R\Gamma_{c}(K^{p},\ZZ/p^{n})
\right), 
\]
For open normal subgroup \(K_{p}\subset K_{p,0}\), by finite \'etale descent, we have a continuous map \[\TT^{S}_{\mrm{cl}}\to \End_{\ZZ/p^{n}[K_{p,0}/K_{p}]}\left(
R\Gamma_{\et,c}(\mrm{Sh}_{\Kpp}(G),\ZZ/p^{n})
\right), 
\] i.e. we have a continuous map \[f:\TT^{S}_{\mrm{cl}}\to \TT^{S}_{K_{p},n}.
\] By the definition of \(\TT^{S}_{\mrm{cl}}\), there exist \(K_{p}'\)
and a finite set \(\{m_{i}\}\)
such that \[\bigcap_{i}\Ker(\TT^{S}_{\mrm{cl}}\to \TT^{S,\prime}_{K_{p}',m_{i}})\subset \Ker(\TT^{S}_{\mrm{cl}}\to \TT^{S}_{K_{p},n}).\]
 and thus Theorem \ref{thmConstrDeterminantHilbert}
follows from Lemma \ref{lemExistDeterminantMiddle} and \cite[Example 2.32]{Chenevier2014determinants}.


\end{proof}

\printbibliography

\end{document}